\documentclass[letterpaper]{article}
\usepackage{arxiv}
\usepackage[round]{natbib}
\usepackage[T1]{fontenc}
\usepackage[utf8]{inputenc}
\usepackage{amsmath,amssymb,amsthm,mathtools}
\usepackage{booktabs,array,graphicx,longtable}
\usepackage{alphalph}
\usepackage{algorithm}
\usepackage[noend]{algpseudocode}
\usepackage{titletoc,microtype,xurl,balance,hyperref}
\newtheorem{theorem}{Theorem}[section]
\newtheorem{lemma}[theorem]{Lemma}
\newtheorem{proposition}[theorem]{Proposition}
\newtheorem{corollary}[theorem]{Corollary}
\theoremstyle{definition}
\newtheorem{assumption}[theorem]{Assumption}

\newcommand{\R}{\mathbb R}
\newcommand{\E}{\mathbb E}
\newcommand{\Prb}{\mathbb P}
\newcommand{\norm}[1]{\lVert #1\rVert}
\newcommand{\ip}[2]{\langle #1,#2\rangle}
\newcommand{\floor}[1]{\lfloor #1\rfloor}
\newcommand{\ceil}[1]{\lceil #1\rceil}

\newcommand{\argmin}{\operatorname*{arg\,min}}
\newcommand{\Ball}{\mathbb B}

\newcommand{\proj}{\Pi_Q}
\newcommand{\pos}[1]{[#1]_+}
\newcommand{\gammaMax}{\gamma_{\max}}
\newcommand{\FreshMean}{\textsc{SampleMean}}
\newcommand{\PaidBatch}{\textsc{SelectBatch}}
\hypersetup{colorlinks=true,linkcolor=blue,citecolor=blue,urlcolor=blue,pdfauthor={},pdftitle={Adaptive Two-Point Zeroth-Order Accelerated Methods for Stochastic Convex Optimization Problems}}
\allowdisplaybreaks

\begin{document}
%\runningtitle{Adaptive Two-Point Zeroth-Order Acceleration}
\title{Adaptive two-points zero-order accelerated methods for stochastic convex optimization problems}
%\twocolumn[
%\aistatstitle{Adaptive Two-Point Zeroth-Order Accelerated Methods\\for Stochastic Convex Optimization Problems}
%\aistatsauthor{Anonymous Authors}
%\aistatsaddress{Anonymous Institution}
%]
\author{
  Yuriy Dorn \\
  IAI MSU \\
  \texttt{dornyv@my.msu.ru}
  \And
  Osman Osmanov \\
  Innopolis University\\
  \And
  Timofei Loginov \\
  MIRIAI \\
  \And
   Nazarii Tupitsa\\
   MBZUAI\\
   \AND
  Darina Dvinskikh \\
  HSE University \\
  \And
  Alexander Gasnikov \\
  Innopolis University, Trusted AI Research Center, RAS, \\
  Steklov Mathematical Institute, RAS \\
}
\maketitle
\begin{abstract}
We study stochastic convex optimization, where each sampled function can be
evaluated at two arbitrary points. Our accelerated method chooses
step sizes, batch sizes and smoothing radii from function values.
A forward difference between smoothing scales controls the radius;
we account for changes in the smoothed objective. Under the stated
regularity and second-moment assumptions, exact function evaluations give
\begin{multline*}
  \mathbb{E}\bigl[f(x_N)-f(x^\star)\bigr]
  \le \widetilde{O}\biggl(
    \frac{\sqrt{d}\,M_2 D}{\sqrt{T}} 
  +\min\biggl(
    \frac{L D^2}{N^2},\;
    \frac{d^{1/4} M D}{N}+\frac{\sqrt{d}\,M D^2}{\gamma_{\max} N^2}
  \biggr)\biggr).
\end{multline*}
Here $N,T,d,D$ denote the numbers of optimization steps and paired
evaluations, the dimension and a feasible-set diameter bound, respectively. Evaluations are available within Euclidean distance $\gamma_{\max}$ of the feasible set. For the population objective $f(x)=\mathbb{E}_{\xi} F(x,\xi)$, $M$ satisfies
$|f(x)-f(y)|\le M\|x-y\|_2$, and $M_2$ satisfies
$\mathbb{E}_{\xi}\|\nabla F(x,\xi)\|_2^2\le M_2^2$ for sample weak gradients in
the interior of the evaluation domain. The method requires no $L,M,M_2$, where $L$ is a Lipschitz constant
of $\nabla f$; set $L=+\infty$ without smoothness. The method uses exactly $2T$ values in
$3N$ sequential sampling rounds with predetermined batches, or $5N$ with
batches chosen from observations. An extension selects a search radius without a supplied
distance bound. Fixed radius parameters can leave
nonvanishing error; suitable choices depending on the
budgets yield convergence.
\end{abstract}

\section{Introduction}
\label{sec:intro}
\label{sec:comparison} 

We study stochastic convex optimization using only function values:
\begin{equation}
 \min_{x\in Q}f(x),\qquad f(x)=\E_\xi F(x,\xi),
 \label{eq:problem}
\end{equation}
%where $Q\subset\R^d$ is compact and convex. Each sampled loss can be evaluated at two arbitrary points. Differences along random directions estimate gradients of a smoothed objective, while averaging independent estimates reduces sampling error. Evaluations within a batch can run in parallel, so we distinguish updates $N$ from sampled losses $T$ ($2T$ function values). We also count sequential sampling rounds and track admissible evaluation errors.
%\na{``track admissible evaluation errors'' is vague; the rewrite changes only the last sentence.}
%\na{Terminology: ``sampled loss'', ``sample loss'', ``sampled function'' (abstract), ``mean loss''. Pick one term for $F(\cdot,\xi)$ and one for $f$ throughout.}
%\na{The rewrite also adds citations for two-point estimation and smoothing (entries in \texttt{additions.bib}). Rewrite below replaces the line above these notes.}
where $Q\subset\R^d$ is compact and convex. Each sampled loss can be evaluated at two arbitrary points. Differences along random directions estimate gradients of a smoothed objective \citep{nemirovski1983problem,flaxman2005online,nesterov2017random,duchi2015optimal,ghadimi2013stochastic}, while averaging independent estimates reduces sampling error; see \citet{akhavan2020exploiting,berahas2022theoretical} for other estimators and comparisons. Evaluations within a batch can run in parallel, so we distinguish updates $N$ from sampled losses $T$ ($2T$ function values). We also count sequential sampling rounds and allow bounded errors in function values.

%The approach of \citet{Gasnikov2022} combines two-point estimation \citep{Shamir2017} with smoothing, acceleration, and batching. Let $M$ bound the mean-loss Lipschitz constant, $M_2$ control the sample gradient second moments, and $R$ be the initial distance to a minimizer. The smoothing bias is at most $M\gamma$. For $M>0$, choosing $\gamma=\varepsilon/(2M)$ gives the tuned tradeoff
%\na{Definitions of $M$ and $M_2$ are awkward (``mean-loss Lipschitz constant'', ``control the \dots second moments'').}
%\na{$\gamma$ is not introduced before ``smoothing bias is at most $M\gamma$''.}
%\na{``tuned tradeoff'' $\to$ ``the bounds''.}
%\na{The rewrite also adds citations. Check that ``tolerates adversarial noise in function values'' describes \citet{Gasnikov2022} correctly: \citet{kornilov2023accelerated} say its optimality includes the admissible level of adversarial noise. Then Section~\ref{sec:protected-main} should compare our allowance for $\Delta$ with that level. Rewrite below replaces the line above these notes.}
The approach of \citet{Gasnikov2022} combines two-point estimation \citep{duchi2015optimal,Shamir2017} with smoothing \citep{nesterov2005smooth,yousefian2012stochastic,duchi2012randomized}, acceleration \citep{lan2012optimal,GasnikovNesterov2018}, and batching, and also tolerates adversarial noise in function values; related accelerated zeroth-order methods include \citet{dvurechensky2021accelerated,kornilov2023accelerated}. Let $M$ be the Lipschitz constant of $f$, let $M_2^2$ bound the second moment of sample gradients, and let $R$ be the initial distance to a minimizer. For smoothing radius $\gamma$, the smoothing bias is at most $M\gamma$. For $M>0$, choosing $\gamma=\varepsilon/(2M)$ gives the bounds
\begin{equation}
 \begin{split}
 N_\star&=\widetilde O\!\left(1+
 d^{1/4}MR/\varepsilon\right),\\
 T_\star&=\widetilde O\!\left(N_\star+
 dM_2^2R^2/\varepsilon^2\right)
 \end{split}
 \label{eq:gasnikov-target}
\end{equation}
for expected error at most $\varepsilon$, provided $\varepsilon\le2M\gammaMax$, where $\gammaMax$ bounds the permitted evaluation distance from $Q$. The step, batch, and radius choices use $M,M_2,R$, motivating adaptation to unknown parameters.

Prior adaptive methods remove different subsets of these inputs. First-order step sizes can adapt to the gradient scale and regularity \citep{duchi2011adaptive,nesterov2013gradient,nesterov2015universal,malitsky2020adaptive}, to unknown smoothness while keeping acceleration on bounded domains \citep{levy2018online,kavis2019unixgrad}, and to the unknown distance \citep{orabona2016coin,carmon2022making,ivgi2023dog,defazio2023learning,khaled2023dowg}. For stochastic gradients, \citet{Rodomanov2024} adapts to unknown H\"older regularity and variance on a domain with a known diameter. U-DoG \citep{Kreisler2024} adapts to the unknown distance while preserving acceleration, but requires a noise bound and an initial distance scale. For function-value feedback, POEM \citep{POEM2025} adapts steps and smoothing radii under sample convexity and a uniform sample Lipschitz bound, and evaluates functions throughout $\R^d$; \citet{chen2026parameter} extend it to decentralized networks. We require only convexity of $f$, second-moment bounds and evaluations within distance $\gammaMax$ of $Q$. Removing inputs can incur costs, as \citet{CarmonHinder2024} show for missing distance and Lipschitz information \citep[see also][]{cutkosky2017online}. On unbounded domains, tuning-free stochastic optimization is impossible in general \citep{khaled2024tuning}, which motivates a known diameter. Oracle access also matters: the method of \citet{KreislerCarmonHinder2026} modifies model outputs before evaluating the loss, whereas our oracle evaluates fixed sample losses. None of these methods combines function-value feedback, acceleration and adaptation to $M$, $M_2$ and $L$.
% the oracle that we use in our research evaluates fixed sample losses.

%Combining these ideas requires accounting for two additional effects. Parameter selection consumes samples that must be charged to $T$, while changing the smoothing radius alters the objective being optimized. The latter also affects distance guarantees: the initial distance to a smoothed minimizer can greatly exceed the distance to the original minimizer (Proposition~\ref{prop:smoothed-distance}).
%\na{Related work on adaptive batch sizes and on choosing finite-difference intervals is missing; the rewrite adds it. ``The latter'' would then be ambiguous, so the last two sentences are merged. Rewrite below replaces the line above.}
Combining these ideas requires accounting for two additional effects. Parameter selection consumes samples that must be charged to $T$: adaptive batch rules choose sample sizes from observed variance \citep{byrd2012sample,friedlander2012hybrid,bollapragada2018adaptive,bollapragada2023adaptive}, and finite-difference intervals are often chosen from estimated noise levels \citep{more2011estimating,berahas2019derivative,shi2022adaptive}. Changing the smoothing radius also alters the objective being optimized, which affects distance guarantees: the initial distance to a smoothed minimizer can greatly exceed the distance to the original minimizer (Proposition~\ref{prop:smoothed-distance}).

%Overall, the costs depend on smoothing radii, batch sizes, and step sizes, which jointly control approximation error, sampling error, and progress per update. Tuning requires regularity, variance, and distance information that may be unavailable. We therefore ask: \emph{Can steps, batches, and smoothing radii be chosen from observations while retaining accelerated iteration bounds, controlling total evaluations, and tolerating bounded evaluation errors? What further costs arise when the initial distance and domain diameter are unknown?}
%\na{``Tuning requires \dots'' repeats ``The step, batch, and radius choices use $M,M_2,R$'' above; the rewrite drops it.}
%\na{The results tolerate errors only with a known bound $\Delta$; the rewrite says so in the question. Rewrite below replaces the line above these notes.}
Overall, the costs depend on smoothing radii, batch sizes, and step sizes, which jointly control approximation error, sampling error, and progress per update. We therefore ask: \emph{Can steps, batches, and smoothing radii be chosen from observations while retaining accelerated iteration bounds, controlling total evaluations, and tolerating bounded evaluation errors with a known bound? What further costs arise when the initial distance and domain diameter are unknown?}

\paragraph{Results and contributions.}
Our method allocates samples to both parameter selection and optimization updates. A forward difference between two smoothing scales controls the model correction relative to the original objective. This comparison and observed gradient differences set the radius and a scalar step coefficient without restarts when the smoothing radius changes. We obtain the following guarantees.
\begin{enumerate}
 \item \textbf{Acceleration with adaptive parameters.}
 With a supplied diameter bound $D$, Algorithm~\ref{alg:protected-main} attains the smooth $N^{-2}$ and nonsmooth $d^{1/4}N^{-1}$ orders in one execution, without supplied regularity or variance constants. For accuracy $\varepsilon\le2M\gammaMax$, the budgets of Section~\ref{sec:scope} match \eqref{eq:gasnikov-target} up to logarithmic factors, with $R$ replaced by $D$. The guarantee includes an additional term $\sqrt d\,MD^2/(\gammaMax N^2)$ due to the bounded evaluation domain, and requires vector and scalar second moments (Assumptions~\ref{ass:vector} and~\ref{ass:forward-moments}) and an initial step size $1/h_0$ with sufficiently small $h_0>0$ (Theorem~\ref{thm:protected-main} and Corollary~\ref{cor:protected-main-clean}).
 \item \textbf{Sampling costs.}
 Counting all parameter-selection samples, the method uses exactly $T$ pairs, $2T$ function values and $N$ projections. Batch sizes fixed in advance for each iteration require $3N$ sequential sampling rounds; batch sizes chosen from observations require $5N$. Both attain the same worst-case orders. Corollary~\ref{cor:epsilon-clean-budgets} gives explicit budgets $N,T$ for accuracy $\varepsilon$ without hidden logarithmic factors.
 \item \textbf{Distance and evaluation errors.}
 With exact evaluations, an initial batch selects a search ball without supplied $D$ or $R$. The bound includes additional error from the initial gradient estimate, its smoothing and the restricted search region, and requires projection onto the ball intersected with $Q$ (Theorem~\ref{thm:joint-main}). For fixed $p,r_0$, this error need not vanish as the budgets increase; suitable choices of $p$ and the initial smoothing radius yield convergence (Appendix~\ref{subsec:joint-convergence}). With a supplied diameter, fixed smoothing tolerates bounded function-value errors without knowing $\Delta$ (Theorem~\ref{thm:corruption}); adaptive smoothing uses a supplied $\Delta$ (Theorem~\ref{thm:protected-main}). On inexact oracles, see \citet{devolder2014first, RisteskiLi2016}.
 \item \textbf{Limits of simultaneous adaptation.}
 To delimit these guarantees, we quantify the sampling cost of constructing upper confidence bounds for mean function-value differences, and establish limits on joint distance--variance adaptation and corruption tolerance for arbitrary pairs (Theorem~\ref{thm:observable-certification-cost} and Corollaries~\ref{cor:universal-distance-main} and~\ref{cor:noise-frontier}). These complement classical oracle lower bounds \citep{agarwal2012information,jamieson2012query,duchi2015optimal}.
\end{enumerate}

Table~\ref{tab:comparison} compares these bounds and their assumptions, inputs and resource counts, distinguishing our known-diameter and selected-search-radius guarantees. Table~\ref{tab:setup-bounds} also records the fixed-smoothing-radius and corrupted-value variants.

\begin{table*}[!t]
\caption{Bounds on $f(\widehat x)-f(x^\star)$ after $N$ updates,
where $\widehat x$ is each method's output ($x_N$ in our rows).
Each gradient call returns an unbiased stochastic
gradient. Queries within a round can run in parallel; rounds are
sequential. Function-value rows assume exact evaluations.
Bounds are in expectation unless marked otherwise. Notation is detailed in Appendix~\ref{sec:result-setups}.
}
\label{tab:comparison}
\centering
\small
\setlength{\tabcolsep}{3pt}
\renewcommand{\arraystretch}{1.0}
\begin{tabular}{@{}
  >{\raggedright\arraybackslash}p{0.16\textwidth}
  p{0.42\textwidth}
  >{\raggedright\arraybackslash}p{0.21\textwidth}
  >{\raggedright\arraybackslash}p{0.17\textwidth}
@{}}
\toprule
Method & Error bound & Assumptions and inputs & Calls and rounds\\
\midrule
Tuned smoothing \citep{Gasnikov2022}
& $\displaystyle O\!\left(
    \frac{\sqrt d\,MR^2}{\gamma N^2}
    +\frac{\sqrt d\,M_2R}{\sqrt T}+M\gamma\right)$
& Uses $M,M_2,R$ to choose steps, batch size $B$, and
  $0<\gamma\le\gammaMax$.
& $T=NB$ pairs; $N$ rounds.\\
Universal fast gradient \citep{Rodomanov2024}
& $\displaystyle O\!\left(
    \inf_{\nu\in[0,1]}\frac{L_\nu D^{1+\nu}}{N^{(1+3\nu)/2}}
    +\frac{\sigma D}{\sqrt N}\right)$
& Requires $D$; does not require $L_\nu,\nu,\sigma$.
& $2N$ gradient calls; $2N$ rounds.\\
U-DoG \citep{Kreisler2024}
& $\displaystyle\widetilde O\!\left(
    \frac{LR^2}{N^2}+\frac{\sigma_{\rm sg}R}{\sqrt N}
    +\frac{\widehat\sigma R}{N}\right)$
  \par\emph{High probability.}
& Smooth, unconstrained; requires
  $\widehat\sigma\ge\sigma_{\rm sg}$ and an initial distance scale
  $0<r_{\rm init}\le R$.
& $3N$ gradient calls; $2N$ rounds.\\
POEM \citep[Thm.~4.9]{POEM2025}
& $\displaystyle\widetilde O\!\left(
    \frac{dM_2D}{T}+\frac{\sqrt d\,M_2D}{\sqrt T}\right)$
  \par\emph{Conditional mean on a high-probability event.}
& Choose $0<r_{\rm init}\le D$; $D,M_2$ need not be given.
& $T=N$ pairs; $N$ rounds.\\
This work: known diameter, Cor.~\ref{cor:protected-main-clean}
& $\displaystyle\widetilde O\!\Bigg(
    \frac{\sqrt d\,M_2D}{\sqrt T}$\par
  $\displaystyle\quad+\min\!\left\{
      \frac{LD^2}{N^2},\frac{d^{1/4}MD}{N}
      +\frac{\sqrt d\,MD^2}{\gammaMax N^2}\right\}\Bigg)$
& Requires $D,\gammaMax$; choose $h_0>0$ and, for adaptive batches,
  $s_0>0$ before sampling. $L,M,M_2,R$ are not required.
& Fixed batches: $T\ge8N$, $3N$ rounds.\par
  Adaptive batches: $T\ge16N$, $5N$ rounds.\\
This work: selected search radius, Thm.~\ref{thm:joint-main}
& $\displaystyle\mathcal A_*(x^\star)+\widetilde O\!\Bigg(
    \frac{\sqrt d\,M_2S_2}{\sqrt T}$\par
  $\displaystyle\quad+\min\!\left\{
      \frac{LS_2^2}{N^2},\frac{d^{1/4}MS_2}{N}
      +\frac{\sqrt d\,MS_2^2}{\gammaMax N^2}\right\}\Bigg)$
& Requires $\gammaMax$; choose
  $p,r_0,h_{\rm ref},\gamma_{\rm a}$ and, for adaptive batches,
  $s_0>0$ before sampling. $D,R,L,M,M_2$ are not required.
  Needs projection onto $Q$ intersected with a ball.
& Fixed batches: $T\ge16N$, $3N+1$ rounds.\par
  Adaptive batches: $T\ge32N$, $5N+1$ rounds.\\
\bottomrule
\end{tabular}
\par\smallskip
\begin{minipage}{\textwidth}\small
\textbf{Batch choices.}
The total $T$ includes all samples used for batch selection, radius
tests, the initial batch, and optimization updates.

\textbf{Problem constants.}
POEM additionally assumes convex sample functions with Lipschitz
constant at most $M_2$ almost surely, and evaluations throughout
$\mathbb R^d$.
Both of our rows require gradient-estimate variance at most
$CdM_2^2$, where $C\ge1$ is universal.
The additional scalar second-moment requirements are those of
Assumption~\ref{ass:forward-moments} for the forward difference
in the known-diameter row, and
Assumption~\ref{ass:adaptive-moments} for the radius-derivative
estimate in the selected-radius row.

\textbf{Small initial coefficients and remaining error.}
Write $\mathcal B$ for the sum inside $\widetilde O$ in each
of our rows. The known-diameter row assumes
$2h_0D^2/N^2\le\mathcal B$; the selected-radius row assumes
$4h_{\rm ref}r_p^2/N^2\le\mathcal B$ for the term outside
$\mathcal A_*$.

$O$ hides universal constants; $\widetilde O$ also hides the
logarithmic factors specified in each cited result.
\end{minipage}
\end{table*}

\paragraph{Organization.}
Section~\ref{sec:problem} specifies the oracle and assumptions. Sections~\ref{sec:method} and~\ref{sec:theory} develop the accelerated method and its fixed-smoothing-radius bounds, which form the basis for adaptive smoothing in Section~\ref{sec:protected-main}. Section~\ref{sec:scope} then studies distance adaptation and its limits. The appendices provide all proofs, additional variants and the notation
summary.
\section{Problem and oracle model}
\label{sec:problem}

Write $\Ball_2^d=\{u\in\R^d:\|u\|\le1\}$ and fix $\gammaMax>0$. The following geometry applies to the known-diameter methods; Section~\ref{sec:scope} selects a search radius without a supplied $D$.

\begin{assumption}[Geometry and population]\label{ass:geometry}
The set $Q$ is nonempty, compact and convex, has an exact Euclidean projector, and has a known diameter bound $D>0$. Function values are available throughout $Q+\gammaMax\Ball_2^d$, the evaluation domain. The population $f$ is convex and $M$-Lipschitz on this evaluation domain. Fix $x^0\in Q$ and a minimizer $x^\star\in\argmin_Q f$ (one exists since $Q$ is compact), and let $R=\norm{x^0-x^\star}$.
\end{assumption}

The method uses only $D$, $\gammaMax$, the initial step size $1/h_0$ and, for adaptive batch sizes, $s_0$ (and $\Delta$, $g$ for adaptive smoothing with inexact evaluations); it does not use $R$, $M$, $M_2$ or $L$. If $\nabla f$ is Lipschitz on the evaluation domain, $L$ denotes its Lipschitz constant; otherwise, $L=+\infty$. We count projections onto $Q$ but not their computational cost.

For each independent sample, the oracle accepts arbitrary points $z_1,z_2$
in the evaluation domain and returns
\begin{align}
 \bigl(\widetilde F(z_1,\xi),\widetilde F(z_2,\xi)\bigr),
 \qquad\norm{\delta}_{\infty}&\le\Delta,\nonumber\\*
 \widetilde F(z,\xi)&=F(z,\xi)+\delta(z).
 \label{eq:oracle}
\end{align}

Here $\delta$ is a fixed function of the query point; exact evaluations correspond to $\Delta=0$. A common random offset in $F(\cdot,\xi)$ cancels in the difference and needs no second-moment bound. The population $f$ and the differences in \eqref{eq:record} are assumed to have well-defined expectations.

Define ball smoothing and a symmetric gradient estimate by
\begin{align}
 f_\gamma(z)&=\E_{U\sim\mathrm{Unif}(\Ball_2^d)}f(z+\gamma U),
 \label{eq:smoothing}\\
 G_\gamma(z;e,\xi)&=\frac{d}{2\gamma}
 [\widetilde F(z+\gamma e,\xi)-\widetilde F(z-\gamma e,\xi)]e,
 \label{eq:record}
\end{align}
where $e$ is uniform on the unit sphere and independent of $\xi$. Symmetry is an estimator choice: the oracle also permits forward differences, which compare function values at two smoothing radii.
For $z\in Q$ and $0<\gamma\le\gammaMax$, every point
$z+\gamma U$ with $\|U\|\le1$ lies in the evaluation domain.
To define derivatives with respect to both $z$ and $\gamma$
at boundary parameter values, the analysis uses a convex
$M$-Lipschitz extension of $f$ beyond this domain.
The resulting derivatives depend only on function values within
the evaluation domain, so the method requires no evaluations
of the extension
(Lemmas~\ref{lem:smoothing} and~\ref{lem:radial-affine}).

\begin{assumption}[Gradient-estimate moments]
\label{ass:vector}
Let $G_\gamma^0$ denote the estimate \eqref{eq:record} computed
with exact function values ($\Delta=0$).
At every queried center and admissible radius,
\begin{align*}
 \E G_\gamma^0(z)&=\nabla f_\gamma(z),\\
 \E\|G_\gamma^0(z)-\nabla f_\gamma(z)\|^2
 &\le \sigma_\gamma^2<\infty.
\end{align*}
Whenever a uniform variance bound is used, assume
$\sigma_0:=\sup_\gamma\sigma_\gamma<\infty$, where the supremum
covers all positive radii permitted by the method, including initialization.
\end{assumption}

For the bounds expressed in terms of sample gradients, $M_2$
satisfies $\E_\xi\|\nabla F(z,\xi)\|^2\le M_2^2$
for sample weak gradients in the interior of the evaluation
domain. Thus $M_2^2$ bounds an uncentered second moment,
whereas $\sigma_\gamma^2$ bounds the variance of the gradient
estimate.
Under the local Lipschitz and integrability hypotheses of
Lemma~\ref{lem:oracle-moments},
\begin{align*}
 \E\|G_\gamma^0(z)-\nabla f_\gamma(z)\|^2
 \le \E\|G_\gamma^0(z)\|^2 \le 2dM_2^2.
\end{align*}
Hence the variance bounds can be chosen with
$\sigma_0^2\le2dM_2^2$; $C=2$ suffices in the
dimension-dependent bounds.
A random sample Lipschitz constant $m(\xi)$ on the evaluation
domain with $\E m(\xi)^2\le M_2^2$ is another sufficient condition.
Individual estimates need not be bounded.
The analysis stated in terms of gradient-estimate moments
requires only the specified conditional moment conditions
and population convexity; sample losses need not be convex.

Ball smoothing satisfies $f\le f_\gamma\le f+M\gamma$, and $f_\gamma$ has gradient Lipschitz constant $L_\gamma\le\sqrt dM/\gamma$. Since $f$ is convex, $f_\gamma(z)$ is nondecreasing in $\gamma$, so a predictable nonincreasing radius sequence produces pointwise nonincreasing convex targets $\phi_t=f_{\gamma_t}$. 

\begin{assumption}[Conditional update estimates]
\label{ass:conditional}
For the iterates and update estimates defined in
Section~\ref{sec:method}, let $\mathcal F_t^y$ and $\mathcal F_t^x$
denote the histories
immediately before the update batches at $y_t$ and $x_t$,
respectively. Each history contains the corresponding center,
the selected radius, the batch size, and all observations used
to choose them.
The history $\mathcal F_t^x$ additionally contains the
$y$-batch from iteration $t$, its estimate $g_t^y$, the resulting
projected point $v_t$ and averaged point $x_t$, and the
observations used to select the $x$-batch size.

Conditionally on the corresponding history, the function samples
and random directions in each update batch are independent and
follow the oracle's sampling distributions.
For each batch label $\iota\in\{y,x\}$, assume
$\E\|g_t^\iota\|<\infty$.
Define the conditional biases by
\begin{align*}
 b_t^y
 &=\E[g_t^y\mid\mathcal F_t^y]-\nabla\phi_t(y_t),\\
 b_t^x
 &=\E[g_t^x\mid\mathcal F_t^x]-\nabla\phi_t(x_t).
\end{align*}
The centered errors are defined by
\begin{align}
 g_t^y
 &=\nabla\phi_t(y_t)+b_t^y+\zeta_t^y,\nonumber\\
 g_t^x
 &=\nabla\phi_t(x_t)+b_t^x+\zeta_t^x,\nonumber\\
 \E[\zeta_t^\iota\mid\mathcal F_t^\iota]
 &=0,\qquad \iota\in\{y,x\}.
 \label{eq:conditional}
\end{align}
We assume
$\E[\|\zeta_t^\iota\|^2\mid\mathcal F_t^\iota]<\infty$
almost surely.
Additional unconditional moment requirements are specified
in the results that use them; in particular, a finite bound
in Lemma~\ref{lem:growth} requires its quantity $W$ to be finite.
\end{assumption}

For batch means of \eqref{eq:record} with exact evaluations and
targets $\phi_t=f_{\gamma_t}$, Assumptions~\ref{ass:vector}
and~\ref{ass:conditional} give $b_t^y=b_t^x=0$.
With fixed spatial corruption, these biases are the mean
contributions of $\delta$, bounded in
Lemma~\ref{lem:oracle-moments}.
They are measured relative to the smoothed objective
$\phi_t$ and do not include its approximation error relative
to $f$.

\begin{assumption}[Forward-difference moments]
\label{ass:forward-moments}
For Algorithm~\ref{alg:protected-main}, at each center $z\in Q$
and radius $0<\gamma\le\gammaMax/2$ chosen from preceding
observations, the clean estimates $\widehat q_\gamma^0(z)$,
defined by \eqref{eq:protected-main-record} with exact function values,
are conditionally independent, with conditional mean
$f_{2\gamma}(z)-f_\gamma(z)$ and conditional second moment
at most $\nu_dM_2^2\gamma^2$, where $\nu_d=d/(d+2)$.
\end{assumption}

Lemma~\ref{lem:corrupt-forward-secant} derives this bound under
its sample regularity conditions.

\begin{assumption}[Moments for radius-derivative estimates]
\label{ass:adaptive-moments}
For the selected-search-radius extension with exact evaluations,
assume Assumption~\ref{ass:vector} with $\sigma_0<\infty$.
At each admissible $(z,\gamma)$ chosen from preceding observations,
draw mutually independent $e,r,\xi$, independently of those
observations, with $e$ uniform on the unit sphere and
$r\sim\mathrm{Beta}(d,1)$. Define
\begin{equation}
 J_\gamma(z)
 =d[F(z+\gamma e,\xi)-F(z+\gamma r e,\xi)].
 \label{eq:radial-estimate-main}
\end{equation}
Conditionally on preceding observations, its mean is
$h_\gamma(z)=\gamma\partial_\gamma f_\gamma(z)$ and its second
moment is at most $\chi_dM_2^2\gamma^2$, where
$\chi_d=2d^2/[(d+1)(d+2)]$.
\end{assumption}

Lemma~\ref{lem:radial-moment} derives this scalar moment bound from a random
sample Lipschitz constant, or from its stated weak-gradient and segment
conditions.
Uniformity across selected positive radii is essential because radius selection
can give an arbitrarily small value.

\section{Scalar acceleration and adaptive batches}
\label{sec:method}

The method maintains a projected point $v_t$, an averaged point $x_t$, and a positive scalar coefficient $H_t$. Set $a_t=t$ and $A_t=t(t+1)/2$, with $A_0=0$. The center $y_t$ combines $x_{t-1}$ and $v_{t-1}$. Averaged gradient estimates at $y_t$ determine a projected step, and estimates at $x_t$ measure the gradient change along it. The updates are
\begin{align}
 y_t&=(A_{t-1}x_{t-1}+a_tv_{t-1})/A_t,
 \label{eq:center-y}\\
 v_t&=\proj(v_{t-1}-a_tg_t^y/H_{t-1}),\nonumber\\
 x_t&=(A_{t-1}x_{t-1}+a_tv_t)/A_t,
 \label{eq:step}\\
 r_t&=\norm{v_t-v_{t-1}},\quad
 \widehat\beta_t=\ip{g_t^x-g_t^y}{x_t-y_t},
 \label{eq:curvature}\\
 H_t&=H_{t-1}+
 \frac{\pos{A_t\widehat\beta_t-H_{t-1}r_t^2/2}}
      {D^2+r_t^2/2}.
 \label{eq:coefficient}
\end{align}
We initialize $v_0=x_0=x^0$ and $H_0=h_0>0$. This positive value sets the first step size and keeps the update well defined. We choose $h_0$ without a regularity or variance input; it has units of function value divided by squared distance. In the bound of Theorem~\ref{thm:potential} divided by $A_N$, it contributes $h_0R_u^2/(2A_N)$, where $R_u=\|u-x^0\|$ for a comparator $u\in Q$. Each iteration uses the same target $\phi_t$ in both batches (which matters for the adaptive radius in later sections). The output is $x_N$.

For a fixed radius and a chosen integer budget $T\ge8N$, reserve $4N$ pairs to select batch sizes and set
\begin{align}
 \mathsf U&=T-4N,\qquad
 B_t^0=1+\floor{\frac{(\mathsf U-2N)t}{2N(N+1)}},
 \label{eq:reserve}\\
 m_{\mathrm{sur}}&=\mathsf U-2\sum_{t=1}^NB_t^0.
 \label{eq:surplus}
\end{align}

The surplus satisfies $m_{\mathrm{sur}}\ge\mathsf U/4$ (Lemma~\ref{lem:paid-accounting}). At each center, two independent gradient estimates give an unbiased variance estimate for choosing the update batch:
\begin{equation}
 \widehat v_j=\tfrac12\norm{G_j^1-G_j^2}^2,
 \quad \lambda_j=t\sqrt{\widehat v_j},
 \quad \Lambda_j=\Lambda_{j-1}+\lambda_j.
 \label{eq:pilot}
\end{equation}
The center order is $j=2t-1$ at $y_t$ and $j=2t$ at $x_t$,
with $B_t^y=B_{2t-1}$ and $B_t^x=B_{2t}$.
For $j<2N-1$ allocate
\begin{equation}
 k_j=\floor{m_j\min\{\tfrac12,\tfrac{\lambda_j}{s_0+\Lambda_j}\}},
 \quad B_j=B_t^0+k_j,
 \label{eq:extras}
\end{equation}
Here $s_0>0$ is chosen before sampling and has the units
of a gradient norm. It controls surplus allocation;
every positive choice preserves the stated worst-case
bounds because $B_j\ge B_t^0$.
At $j=2N-1$ allocate all
remaining surplus, and at $j=2N$ allocate zero surplus. Set
$m_{j+1}=m_j-k_j$, starting from $m_1=m_{\mathrm{sur}}$ and
$\Lambda_0=0$. The final surplus is therefore spent before the final projected step. Batch selection and gradient updates use separate samples.

\begin{algorithm}[t]
\caption{Scalar acceleration with adaptive batch sizes}
\label{alg:scalar}
\small
\begin{algorithmic}[1]
\Require $Q,x^0,D$; $N\in\mathbb N_+,T\in\mathbb N$ with $T\ge8N$;
 $h_0>0,s_0>0$, fixed $\gamma\in(0,\gammaMax]$
\State Set $v_0=x_0=x^0$, $H_0=h_0$, $\Lambda_0=0$
\State Compute $\mathsf U,B_t^0,m_1$ by \eqref{eq:reserve}--\eqref{eq:surplus}
\For{$t=1,\ldots,N$}
 \State $a_t=t$, $A_t=t(t+1)/2$; compute $y_t$ by \eqref{eq:center-y}
 \State $B_t^y\gets\PaidBatch(y_t,t,2t-1,\gamma)$
 \State $g_t^y\gets\FreshMean(y_t,B_t^y,\gamma)$
 \State Compute $v_t,x_t$ by \eqref{eq:step}
 \State $B_t^x\gets\PaidBatch(x_t,t,2t,\gamma)$
 \State $g_t^x\gets\FreshMean(x_t,B_t^x,\gamma)$
 \State Update $r_t,\widehat\beta_t,H_t$ by \eqref{eq:curvature}--\eqref{eq:coefficient}
\EndFor
\State \Return $x_N$
\Function{\PaidBatch}{$z,t,j,\gamma$}
 \State Draw two independent estimates $G_j^1,G_j^2$ at $z$ using $G_\gamma$
 \State Compute $\widehat v_j,\lambda_j,\Lambda_j$ by \eqref{eq:pilot}
 \If{$j<2N-1$}
  \State Compute $k_j$ by \eqref{eq:extras}
 \ElsIf{$j=2N-1$}
  \State $k_j\gets m_j$
 \Else
  \State $k_j\gets0$
 \EndIf
 \State $m_{j+1}\gets m_j-k_j$; \Return $B_t^0+k_j$
\EndFunction
\end{algorithmic}
\end{algorithm}

In Algorithm~\ref{alg:scalar}, $\FreshMean(z,B,\gamma)$ averages $B$
independent estimates \eqref{eq:record}, using exactly $B$ samples
and $2B$ values in one sampling round.
The algorithm stores $O(d)$ real numbers
and uses $O(d)$ arithmetic for each scalar update outside projection.
Generating directions and averaging all estimates cost $O(dT)$ arithmetic.

\section{Expected error and resource bounds} \label{sec:theory}

The fixed allocation below uses only the two vector
sample means per iteration.

\begin{corollary}[Batch sizes fixed before sampling]\label{cor:public}
For a fixed radius $\gamma\in(0,\gammaMax]$, exact function evaluations, and $T\ge4N$, replace the adaptive batch rule by
\begin{equation}
 B_t^y=B_t^x=1+
 \floor{\frac{(T-2N)t}{N(N+1)}}.
 \label{eq:public-batch}
\end{equation}
At most $T$ pairs, $2T$ values and $2N$ sequential rounds suffice, and
\begin{equation}
 \begin{aligned}
 \E[f(x_N)-f(x^\star)]
 \le \frac{h_0R^2}{2A_N}+\frac{16L_\gamma D^2}{N^2}+\frac{4\sqrt2\sigma_\gamma D}{\sqrt T}+M\gamma.
 \end{aligned}
 \label{eq:public-risk}
\end{equation}
The full proof is in Appendix~\ref{subsec:proof-public}.
\end{corollary}

\begin{theorem}[Adaptive batch guarantee]\label{thm:paid}
For a fixed radius $\gamma\in(0,\gammaMax]$ and exact function evaluations, Algorithm~\ref{alg:scalar} satisfies
\begin{align}
 \E[f(x_N)-f(x^\star)]
 \le \frac{h_0R^2}{2A_N}
 +\frac{16L_\gamma D^2}{N^2}
 \nonumber +\frac{8\sqrt2\sigma_\gamma D}{\sqrt T}
 +M\gamma. \label{eq:paid-risk}
\end{align}
No variance parameter enters the batch rule or the scalar update. The full proof is in Appendix~\ref{subsec:proof-paid}.
\end{theorem}

Theorem~\ref{thm:paid} establishes a guarantee for batch sizes chosen from observations, including their cost; it does not quantify a gain over the fixed allocation \eqref{eq:public-batch}. Its worst-case constants are larger; Theorem~\ref{thm:batch-variance-actual} and Proposition~\ref{prop:batch-variance-paid-limit} compare the allocation with the variances at the actual centers.

A radius chosen from $d,D,N,\gammaMax$ already gives the nonsmooth rate under vector moments alone (Corollary~\ref{cor:horizon} in Appendix~\ref{subsec:proof-horizon}). Section~\ref{sec:protected-main} adapts the radius from function values and obtains both regularity regimes.

For a fixed $\gamma\in(0,\gammaMax]$,
Algorithm~\ref{alg:scalar} requires no knowledge of $\Delta$;
the error bound \eqref{eq:corruption-risk} holds whenever
$\|\delta\|_\infty\le\Delta$.
Algorithm~\ref{alg:protected-main} uses a supplied bound
$\Delta$ to select the smoothing radius.

\section{Forward-difference adaptive smoothing}
\label{sec:protected-main}

To compare an update for $f_\gamma$ with $f$,
Lemma~\ref{lem:radial-affine} gives the affine lower bound
\[
 f(u)\ge f_\gamma(y)+\ip{\nabla f_\gamma(y)}{u-y}-h_\gamma(y),
\]
where $h_\gamma(y)=\gamma\partial_\gamma f_\gamma(y)$ is the correction.
Draw a uniform unit-ball vector $U$ and a sampled function independently
of one another and of preceding observations. We use the two-value estimate
\begin{equation}
 \widehat q_\gamma(y)=\widetilde F(y+2\gamma U,\xi)
                    -\widetilde F(y+\gamma U,\xi).
 \label{eq:protected-main-record}
\end{equation}
The clean counterpart $\widehat q_\gamma^0$, obtained by replacing
$\widetilde F$ with $F$, has mean
$q_\gamma(y)=f_{2\gamma}(y)-f_\gamma(y)\ge h_\gamma(y)\ge0$.
Since $h_{\lambda\gamma}(y)\le\lambda h_\gamma(y)$ for $0<\lambda\le1$,
dividing the radius by $\max\{1,q_\gamma(y)/\tau\}$ would bound the correction
by any prescribed $\tau>0$. The method estimates $q_\gamma(y)$ from a batch.
Lemma~\ref{lem:corrupt-forward-secant} proves
$\E(\widehat q_\gamma^0)^2\le\nu_dM_2^2\gamma^2$, where
$\nu_d=d/(d+2)$, and the corruption sensitivity
$|\widehat q_\gamma-\widehat q_\gamma^0|\le2\Delta$.
Algorithm~\ref{alg:protected-main} uses $\widehat q_\gamma$
with clean second-moment factor $\nu_d$; the extension in
Section~\ref{sec:scope} uses $J_\gamma$ from
Assumption~\ref{ass:adaptive-moments}, with factor $\chi_d$.
Assumption~\ref{ass:forward-moments} is additional to the vector variance bound;
a sample Lipschitz constant with second moment $M_2^2$ suffices.
The outer endpoint requires $2\gamma\le\gammaMax$.

\subsection{Choosing the smoothing radius}
Let $h_{\rm init}\ge h_{\rm ref}>0$ and $\mathsf K_t=h_{\rm ref}+H_t-h_{\rm init}$.
A chosen lower radius bound $g\in[0,\gammaMax/2]$ determines
\[
 \Gamma=\max\{g,\min\{\gammaMax/2,d^{1/4}D/N\}\}.
\]
The supplied error bound $\Delta\ge0$ enters the radius rule in Algorithm~\ref{alg:protected-main}; for exact evaluations, set $\Delta=g=0$.

\begin{algorithm}[t]
\caption{Acceleration with adaptive smoothing}
\label{alg:protected-main}
\small
\begin{algorithmic}[1]
\Require $Q,x^0,D$, $N\ge1$, $\gammaMax>0$, $\Delta\ge0$,
$0\le g\le\gammaMax/2$, $h_{\rm init}\ge h_{\rm ref}>0$;
fixed batches: $T\ge8N$; adaptive batches: $T\ge16N,s_0>0$
\State $x_0=v_0=x^0$, $H_0=h_{\rm init}$,
$\mathsf K_0=\underline K_0=h_{\rm ref}$, $\gamma_0=\Gamma$
\State $B^{\rm rad}\gets\floor{T/(8N)}$; initialize the chosen $\mathsf{Batch}$ rule
\For{$t=1,\ldots,N$}
 \State $a_t=t$, $A_t=t(t+1)/2$; compute $y_t$ by \eqref{eq:center-y}
 \If{$\mathsf K_{t-1}\ge2\underline K_{t-1}$}
  \State $\underline K_t\gets\mathsf K_{t-1}$, $\gamma_t^{\rm try}\gets\Gamma$
 \Else
  \State $\underline K_t\gets\underline K_{t-1}$,
  $\gamma_t^{\rm try}\gets\gamma_{t-1}$
 \EndIf
 \State $\tau_t\gets\underline K_tD^2/N^2$
 \State Draw mean $\widehat q_t$ of $B^{\rm rad}$ records
 \eqref{eq:protected-main-record} at $(y_t,\gamma_t^{\rm try})$
 \State $\gamma_t\gets\max\{g,\gamma_t^{\rm try}/
       \max\{1,[\widehat q_t-2\Delta]_+/\tau_t\}\}$
 \State $B_t^y\gets\mathsf{Batch}(y_t,t,2t-1,\gamma_t)$
 \State $g_t^y\gets\FreshMean(y_t,B_t^y,\gamma_t)$
 \State Compute $v_t,x_t$ by \eqref{eq:step}
 \State $B_t^x\gets\mathsf{Batch}(x_t,t,2t,\gamma_t)$
 \State $g_t^x\gets\FreshMean(x_t,B_t^x,\gamma_t)$
 \State Update $H_t$ by \eqref{eq:curvature}--\eqref{eq:coefficient}
 \State $\mathsf K_t\gets h_{\rm ref}+H_t-h_{\rm init}$
\EndFor
\State \Return $x_N$
\end{algorithmic}
\end{algorithm}

The rule $\mathsf{Batch}(z,t,j,\gamma)$ is chosen before sampling. For fixed batches, use \eqref{eq:public-batch} with $T$ replaced by $T-NB^{\rm rad}$ and add all unused pairs to the final $y$ batch. For adaptive batches, call $\PaidBatch(z,t,j,\gamma)$ from
Algorithm~\ref{alg:scalar}. Initialize $\Lambda_0=0$, with $\mathsf U=T-NB^{\rm rad}-4N$, $B_t^0$ and $m_1$ given by \eqref{eq:reserve}--\eqref{eq:surplus} at this budget. Both batch-selection estimates and update means use the selected radius, with independent samples as in Algorithm~\ref{alg:scalar}. The complete counts are
\eqref{eq:corrupt-public-counts}--\eqref{eq:corrupt-paid-counts}. All optimizer state, weights and remaining sample budgets persist.

\subsection{Regularity and corruption guarantees}
The affine model replaces the comparator smoothing bias in Theorem~\ref{thm:potential} by charges $h_{\gamma_t}(y_t)$ at queried centers. Radius increases add positive changes of the target value at $x_{t-1}$. Theorem~\ref{thm:radial-potential} gives this potential, and Lemma~\ref{lem:corrupt-growth-drift} bounds those changes. The resulting bound is \eqref{eq:corrupt-adaptive-risk}. Set $c_d=d\E|[e]_1|\le\sqrt d$; the index $\mathrm{reg}\in\{\mathrm{s},\mathrm{ns}\}$ denotes the smooth or nonsmooth case. Write
$\ell_{\mathrm{reg}}^\Delta=\log_2(1+\mathcal H_{\mathrm{reg}}^\Delta/h_{\rm ref})$, with
$\mathcal H_{\mathrm{reg}}^\Delta$ defined in \eqref{eq:corrupt-curvature-envelopes}.

\begingroup
\interdisplaylinepenalty=10000
\begin{theorem}[Accelerated bounds with adaptive smoothing]
\label{thm:protected-main}
Under Assumptions~\ref{ass:geometry} and~\ref{ass:forward-moments}, assume uniform clean vector variance
$\sigma_0^2\le CdM_2^2$ for a fixed numerical constant $C\ge1$ and fixed spatial corruption $\|\delta\|_\infty\le\Delta$. Set $h_{\rm init}=h_{\rm ref}=h_0>0$. For an objective-error tolerance $\rho>0$, choose $g=c_d\Delta D/\rho\le\gammaMax/2$. Algorithm~\ref{alg:protected-main} satisfies
\begin{equation}
 \begin{aligned}
 &\E[f(x_N)-f(x^\star)]\le\\
 &\widetilde O\!\Biggl(
 \frac{h_0D^2}{N^2}+\frac{\sqrt dM_2D}{\sqrt T}+M_2g\sqrt{\frac NT}+\rho\left(1+\sqrt{\frac dT}\right)+\Delta
 \\
 &+\min\!\Biggl\{L\left(\frac{D^2}{N^2}+g^2\right),\frac{d^{1/4}MD}{N}+\frac{\sqrt dMD^2}{\gammaMax N^2}+Mg\Biggr\}\Biggr).
 \end{aligned}
 \label{eq:protected-main-order}
\end{equation}
The smooth branch requires $L$-smoothness on the evaluation domain; otherwise set $L=+\infty$. In each branch, $\widetilde O$ hides only the factor $1+\ell_{\mathrm{reg}}^\Delta$ of that branch. If $\Delta=0$, take $g=0$ and remove all $\rho$ terms. The method uses exactly $T$ pairs, $2T$ values and $N$ projections in $3N$ rounds with fixed batches or $5N$ rounds with adaptive batches. The full proof is in Appendix~\ref{subsec:proof-protected-main}.
\end{theorem}
\endgroup

Subtracting $2\Delta$ from $\widehat q_t$ prevents corruption from shrinking the radius; the floor $g$ on the radius controls the bias of the gradient estimates. With
$h_0=4\varepsilon N^2/(77D^2)$ and $\rho=\varepsilon/64$, Corollary~\ref{cor:epsilon-corruption-budgets} gives the sufficient nonsmooth allowance $\Delta=O(\varepsilon^2/[c_dMD])$, together with explicit evaluation-domain and sampling conditions. A smooth population permits a quadratic floor cost $Lg^2$ in place of $Mg$. Since $(d/c_d)^2\le3d$, the additional condition $T\ge150d$ suffices to control the added vector variance at this bias allowance. The extra scalar term is absorbed by the usual sampling term when $g\sqrt N\le\sqrt dD$; it remains explicit for larger horizons. The method uses a supplied $\Delta$ in both the radius update and the choice of $g$.

\begin{corollary}[Controlling the initialization error]
\label{cor:protected-main-clean}
For exact evaluations, define
\begin{equation}
\begin{aligned}
 \mathcal B_{N,T}&=\frac{\sqrt dM_2D}{\sqrt T}+\min\!\left\{\frac{LD^2}{N^2},
 \frac{d^{1/4}MD}{N}+\frac{\sqrt dMD^2}{\gammaMax N^2}\right\}.
 \label{eq:protected-main-base}
\end{aligned}
\end{equation}
If $\mathcal B_{N,T}>0$ and
$0<h_0\le N^2\mathcal B_{N,T}/(2D^2)$, then
$\E[f(x_N)-f(x^\star)]\le\widetilde O(\mathcal B_{N,T})$.
No positive problem-dependent lower bound on $h_0$ is imposed.
For any target $\varepsilon>0$, the choice made before sampling
$h_0\le\varepsilon N^2/(8D^2)$ makes the combined initialization and reference
cost at most $\varepsilon/4$.
The full proof is in Appendix~\ref{subsec:proof-protected-main-clean}.
\end{corollary}

The first condition explains the bound; the accuracy-based choice uses only
$D,N,\varepsilon$. For arbitrary $h_0$, the bounds retain logarithmic dependence
on $h_0$. The specific initialization and sufficient budgets in
Corollary~\ref{cor:epsilon-clean-budgets} bound these factors by universal constants.
\section{Distance and joint-adaptation tradeoffs}
\label{sec:scope}

With exact evaluations, an independent initial batch selects a search ball.

Before sampling, choose $r_0,h_{\rm ref}>0$, $p\ge1$ and the initial smoothing radius
$0<\gamma_{\rm a}\le\gammaMax$. Draw $B_0=\floor{T/2}$ clean vector records at
$(x^0,\gamma_{\rm a})$, with mean $g_0$. Put $A=A_N=N(N+1)/2$, $T'=T-B_0$ and
\begin{align}
 h_{\rm a}&=h_{\rm ref}+A\norm{g_0}/(p r_0),\nonumber\\
 S&=r_0+2A\norm{g_0}/h_{\rm a},\qquad r_p=r_0(1+2p).
 \label{eq:joint-main-anchor}
\end{align}
Thus $r_0\le S<r_p$. Set
$Q_S=Q\cap\overline{\mathbb B}(x^0,S)$ and require exact
projection onto $Q_S$. We use $J_\gamma$ in
\eqref{eq:radial-estimate-main} to estimate $h_\gamma$ directly,
with both evaluation points within distance $\gamma$ of their center.
Run Algorithm~\ref{alg:threshold}
(Appendix~\ref{sec:app-threshold}) once on $Q_S$, with $D_S=2S$,
$H_0=h_{\rm init}=h_{\rm a}$ and budget $T'$.
Radius selection tracks $\mathsf K_t=h_{\rm ref}+H_t-h_{\rm a}$,
while projected updates use step size $a_t/H_{t-1}$.
Thus $\mathsf K_0=h_{\rm ref}$ and $H_0=h_{\rm a}$.

For fixed $u\in Q$, write $R_u=\|x^0-u\|$.
The bound $\mathcal A_*(u)$ accounts for initialization, initial smoothing
and restricting the search to $Q_S$. The deterministic radius bound $S_2$
satisfies $(\E S^2)^{1/2}\le S_2$.
Both are analysis quantities defined in
\eqref{eq:joint-main-moments}--\eqref{eq:joint-main-saturated-price}.

\begin{theorem}[Guarantee with a selected search radius]\label{thm:joint-main}
Under Assumptions~\ref{ass:vector} and~\ref{ass:adaptive-moments} and exact projection onto the selected intersection, the execution on the selected ball satisfies
\begin{equation}
 \begin{aligned}
 &\E[f(x_N)-f(u)]\le\mathcal A_*(u)+\frac{4h_{\rm ref}S_2^2}{N^2}\\
 &\quad+\widetilde O\!\Biggl(\frac{\sigma_0S_2}{\sqrt T}
 +\min\!\Biggl\{\frac{LS_2^2}{N^2}+\frac{d^{1/4}M_2S_2}{\sqrt{NT}},\\
 &\qquad\frac{d^{1/4}MS_2}{N}
 +\frac{\sqrt d\,MS_2^2}{\gammaMax N^2}
 +\frac{d^{1/4}M_2S_2}{\sqrt T}\Biggr\}\Biggr).
 \end{aligned}
 \label{eq:joint-main-order}
\end{equation}
Set $L=+\infty$ without population smoothness.
The notation $\widetilde O$ hides only universal constants and the
logarithmic factors specified in Appendix~\ref{subsec:proof-joint-main},
which also gives the exact bound.
Fixed batches require $T\ge16N$ and $3N+1$ rounds; adaptive batches require $T\ge32N$ and $5N+1$ rounds. Both use exactly $T$ pairs, $2T$ values and $N$ intersection projections, with no supplied $D,R,L,M,M_2$ or variance. The full proof is in Appendix~\ref{subsec:proof-joint-main}.
\end{theorem}

The alternative error bound $\mathcal A_{\rm sat}(u)$ in
\eqref{eq:joint-main-saturated-price} satisfies
$\mathcal A_*(u)\le\mathcal A_{\rm sat}(u)$ and grows at most
linearly with $R_u$. Using this bound, choosing
$0<h_{\rm ref}\le\varepsilon N^2/(20r_p^2)$ makes the explicit
reference contribution at most $\varepsilon/4$
(Corollary~\ref{cor:run5-distance-reference-absorption}).
For fixed $p,r_0$, the contribution $\mathcal A_*(u)$
need not vanish as $N,T$ increase, even when
$g_0=\nabla f_{\gamma_{\rm a}}(x^0)$.
Indeed, for $F(x,\xi)=f(x)=-x$ on $Q=[0,R]$,
$x^0=0$, and $R>r_p$, every output satisfies
$f(x_N)-f(R)\ge R-r_p>0$
(Appendix~\ref{subsec:comparison-parameters}).
Increasing $p$ or $r_0$ reduces the factor $1/(pr_0)$
in $h_{\rm a}/A$ and enlarges $r_p$, but can increase
the terms depending on $S_2$.
Convergence follows with $p$ increasing and
$\gamma_{\rm a}$ decreasing according to the budgets
(Appendix~\ref{subsec:joint-convergence}).

\begin{corollary}[Unknown variance and original distance]
\label{cor:universal-distance-main}
Fix $N\ge1$, $T\ge4$, $D,L>0$, and $0<r\le D/2$. Set $\gammaMax=D$. There are two convex sample-function laws on $Q=[0,D]$, with common supplied $x^0=r,D,L,N,T,\gammaMax$ and finite uniform second-moment bounds, such that any method that receives only these inputs, uses at most $T$ sampled functions and satisfies on both laws
\[
 \E_i[f_i(\widehat x)-f_i(x_i^\star)]\le
 C\left(\frac{LR_i^2}{N^2}+
       \frac{M_{2,i}R_i}{\sqrt T}\right)
\]
requires $C\ge\min\{N^2,\sqrt T,D/r\}/96$. The conclusion holds even if the method queries arbitrary points of each sampled function or observes it in full. The full proof is in Appendix~\ref{subsec:proof-universal-distance-main}.
\end{corollary}

Without a supplied moment or range bound, no procedure can
report a finite upper confidence bound for a mean
function-value difference almost surely and maintain
uniform coverage over all laws with finite second moments,
even with adaptive sampling and stopping
(Corollary~\ref{cor:observable-certification-no-cap}).

\subsection{Reading the resource guarantee}
For exact evaluations and the known-diameter method with
$\sigma_0^2\le CdM_2^2$, choose the initial coefficient $h_0=4\varepsilon N^2/(77D^2)$.
Corollary~\ref{cor:epsilon-clean-budgets} gives explicit budgets $N_\varepsilon,T_\varepsilon$ for accuracy $\varepsilon$ without the logarithmic factors of adaptation. In the smooth case,
\begin{equation}
 N_\varepsilon=O\!\left(1+
 D\sqrt{L/\varepsilon}\right).
 \label{eq:adaptive-smooth-accuracy}
\end{equation}
In the nonsmooth case, including a finite permitted evaluation radius,
\begin{equation}
 N_\varepsilon=O\!\left(1+d^{1/4}MD/\varepsilon
 +d^{1/4}D\sqrt{M/(\gammaMax\varepsilon)}\right).
 \label{eq:adaptive-nonsmooth-accuracy}
\end{equation}
For either case,
\[
 T_\varepsilon=O\!\left(N_\varepsilon+
 dM_2^2D^2/\varepsilon^2\right)
\]
suffices, together with $T\ge8N$ (fixed batches) or $T\ge16N$ (adaptive batches); see Appendix~\ref{subsec:epsilon-clean}. When $\varepsilon\le2M\gammaMax$, the nonsmooth budgets match
\eqref{eq:gasnikov-target} with $R$ replaced by $D$,
without additional adaptation logarithms. The method does not use
$M,M_2,L$.

The budgets guarantee error bounds, without an observable
stopping test. For nonsmooth problems, the largest admissible
function-value error remains undetermined at finer accuracy:
our sufficient allowance is of order
$\varepsilon^2/(\sqrt d\,MD)$ under
Corollary~\ref{cor:epsilon-corruption-budgets}, whereas the
impossibility threshold in Corollary~\ref{cor:noise-frontier}
is of order $\varepsilon\log(2T+1)/d$ in its finer-accuracy
regime. Obtaining these adaptive-smoothing guarantees without
a supplied $\Delta$, sharp achievable distance-dependent
bounds with unknown variance, and an optimal tradeoff between
sequential sampling rounds and evaluations also remain open.

\label{page:main-end}
\clearpage
\subsection*{AI Use Statement}
Generative AI tools assisted with proof development and manuscript preparation. The authors are responsible for the mathematical claims and final presentation.
\bibliographystyle{apalike}
\bibliography{references}
\clearpage
\clearpage\appendix\onecolumn
\renewcommand{\thesection}{\AlphAlph{\value{section}}}
\renewcommand{\theHsection}{\AlphAlph{\value{section}}}
\title{Adaptive Two-Point Zeroth-Order Accelerated Methods\\for Stochastic Convex Optimization Problems:\\Supplementary Material}
\startcontents[appendices]
\section*{Contents of the supplementary material}
\begingroup\small
\titlecontents{section}[2.3em]{\addvspace{0.5em}\bfseries}
 {\contentslabel{2.3em}}{}{\titlerule*[0.75pc]{.}\contentspage}
\titlecontents{subsection}[5.4em]{\addvspace{0.1em}}
 {\contentslabel{3.1em}}{}{\titlerule*[0.75pc]{.}\contentspage}
\printcontents[appendices]{}{1}{\setcounter{tocdepth}{2}}
\endgroup
\clearpage
\section{Detailed assumptions and bounds}
\label{sec:result-setups}

\textbf{Gradient noise.}
For the first-order rows, $\sigma$ bounds the root mean square
gradient error at each query, and $\sigma_{\rm sg}$ bounds its tail:
$\Prb(\|g-\nabla f\|\ge s)\le
2\exp(-s^2/\sigma_{\rm sg}^2)$ for $s>0$.
The supplied $\widehat\sigma$ bounds this tail scale.

\textbf{Starting values and selected radius.}
$S_2$ bounds the root mean square selected radius.
$\mathcal A_*(x^\star)$ bounds contributions from initialization,
smoothing, and restricting the search to the selected ball.
Neither is an input; see
\eqref{eq:joint-main-moments}--\eqref{eq:joint-main-saturated-price}.

$L=L_1$ and $L_\nu$ satisfy
$\|\nabla f(x)-\nabla f(y)\|\le L_\nu\|x-y\|^\nu$
(with subgradients when $\nu=0$); set $L=+\infty$ without smoothness.

Fixed batch sizes are specified for each iteration before sampling.
Adaptive batch sizes use additional samples to select the number of
evaluations to average; $s_0>0$ controls their allocation.

For known diameter, $h_0>0$ is the initial step coefficient,
giving first step size $1/h_0$.
For the selected radius, $r_0>0$ is its minimum,
$p\ge1$ sets its upper bound $r_p=(1+2p)r_0$, and
$0<\gamma_{\rm a}\le\gammaMax$ is the initial-batch smoothing radius.
The reference coefficient $h_{\rm ref}>0$ initializes
smoothing-radius selection; the initial step coefficient
$h_{\rm a}$ also depends on the initial gradient estimate,
as defined in \eqref{eq:joint-main-anchor}.

For a target $\varepsilon>0$ fixed before sampling, choosing
$h_0\le\varepsilon N^2/(8D^2)$ and
$h_{\rm ref}\le\varepsilon N^2/(20r_p^2)$ makes the explicit
contributions $2h_0D^2/N^2$ and $5h_{\rm ref}r_p^2/N^2$
at most $\varepsilon/4$ in
Corollaries~\ref{cor:protected-main-clean} and
\ref{cor:run5-distance-reference-absorption}, respectively.

Table~\ref{tab:setup-bounds} records the assumptions and complete bounds
for the algorithms and lower bounds. Optimization upper bounds concern the original population risk, except
for the explicitly identified exact-record target row. Write
$A=A_N=N(N+1)/2$ for the default scalar recurrences, $R_u=\norm{u-x^0}$ and
$\nu_d=d/(d+2)$. The supplied $D$ bounds the diameter of $Q$;
it is distinct from solution distance $R$. Each original pair uses one
sampled function at exactly two points and contributes two scalar values.
Positive scalar references, budgets and the evaluation domain are supplied. Regularity
and variance scales in the bounds are analysis quantities unless the
setup explicitly states otherwise. Sampling stages are distinct from
the $N$ optimizer updates.

For the forward-secant rows, define the complete remainder
\begin{align}
 \mathcal P_{\mathrm{reg}}^\Delta={}&k_ND^2\mathcal H_{\mathrm{reg}}^\Delta
 +h_{\rm ref}D^2/N^2
 +(\eta+4\Delta+w_{\mathrm{reg}})(1+\ell_{\mathrm{reg}}^\Delta)
 +q_{\mathrm{reg}}\ell_{\mathrm{reg}}^\Delta+2D\beta.
 \label{eq:setup-protected-remainder}
\end{align}
All its quantities are defined in
\eqref{eq:corrupt-diagnostic-scales}--\eqref{eq:corrupt-curvature-envelopes}.
At $\Delta=0$, $g=0$, $\beta=w_{\mathrm{reg}}=0$ and
$\bar\sigma=\sigma_0$. This is an exact specialization of
\eqref{eq:corrupt-adaptive-risk}.

\begingroup\small
\setlength{\tabcolsep}{5pt}
\renewcommand{\arraystretch}{1.16}
\begin{longtable}{@{}p{0.30\textwidth}p{0.66\textwidth}@{}}
\caption{Assumptions, error bounds and sampling and projection costs.
The smooth method with fixed batches uses $2N$ sampling rounds under vector
moment bounds. Other rows adapt batch sizes or smoothing
radii from function values under the stated conditions.
The same radial controller has both clean regularity guarantees.
Lower bounds state their quantifiers and scope separately.}
\label{tab:setup-bounds}\\
\toprule
Setup and result & Bound and resource conditions\\
\midrule
\endfirsthead
\multicolumn{2}{l}{\textit{Table~\thetable\ continued}}\\
\toprule
Setup and result & Bound and resource conditions\\
\midrule
\endhead
\bottomrule
\endfoot

\textbf{Smooth population, fixed batch sizes.}
Convex $L$-smooth population on the evaluation domain; clean vector records with
variance at most $\sigma_\gamma^2$ at the selected fixed radius.
Take $u=x^\star$ and
$\gamma=\min\{\gammaMax,D/N\}$.
Use \eqref{eq:public-batch}.
Neither $L$ nor $\sigma_\gamma$ is supplied.
&
$\displaystyle
 \frac{h_0R^2}{2A}
 +\left(16+\frac{\nu_d}{2}\right)\frac{LD^2}{N^2}
 +\frac{4\sqrt2\sigma_\gamma D}{\sqrt T}$.
\par\smallskip
$T\ge4N$; at most $T$ pairs and $2T$ values,
$2N$ sampling stages, $N$ projections.
No radial sample condition is needed.
The bound follows from Corollary~\ref{cor:public} and
Proposition~\ref{prop:protected-smoothing-smooth}; its derivation is below.\\[7pt]

\textbf{Nonsmooth population, adaptive batch sizes with all samples counted.}
Convex $M$-Lipschitz population; clean vector moment assumptions
$\sigma_\gamma^2\le CdM_2^2$, $C\ge1$.
Set $u=x^\star$ and
$\gamma=\min\{\gammaMax,d^{1/4}D/N\}$.
Algorithm~\ref{alg:scalar} learns the batch surplus from pilot observations.
Neither $M$ nor $M_2$ is supplied.
&
$\displaystyle \frac{h_0R^2}{2A}
 +\frac{17d^{1/4}MD}{N}
 +\frac{16\sqrt d\,MD^2}{\gammaMax N^2}$
\par\smallskip
$\displaystyle\hfill
 +\frac{8\sqrt{2Cd}\,M_2D}{\sqrt T}$.
\par\smallskip
$T\ge8N$; exactly $T$ pairs and $2T$ values,
$4N$ stages, $N$ projections.
Corollary~\ref{cor:horizon}. Fixed batch sizes use $2N$ stages
and replace the stochastic coefficient $8\sqrt{2Cd}$ by
$4\sqrt{2Cd}$, without a surplus decision based on samples.\\[7pt]

\textbf{One clean forward-secant policy, two regimes.}
Convex $M$-Lipschitz population, uniform clean vector variance
$\sigma_0^2\le CdM_2^2$, and Assumption~\ref{ass:forward-moments}. Use
Algorithm~\ref{alg:protected-main} with $\Delta=g=0$ and
$h_{\rm init}=h_{\rm ref}=h_0>0$.
Neither regularity nor variance scales are supplied.
&
$\displaystyle \frac{h_0R_u^2}{2A}+\mathcal P_{\mathrm{reg}}^0$,
\par\smallskip
where $\mathcal P_{\mathrm{reg}}^\Delta$ is defined in \eqref{eq:setup-protected-remainder} from the complete finite
bound \eqref{eq:corrupt-adaptive-risk}.
For the initialization of Corollary~\ref{cor:protected-main-clean},
$\displaystyle\widetilde O(\mathcal B_{N,T})$
with $\mathcal B_{N,T}$ in \eqref{eq:protected-main-base}.
The two branches are evaluated with their own finite reference logs.
With supplied target $\varepsilon>0$ and
$h_0=4\varepsilon N^2/(77D^2)$,
Corollary~\ref{cor:epsilon-clean-budgets} gives gap below $\varepsilon$
for log-free sufficient orders
\par\smallskip
$\displaystyle N_{\rm s}=O(1+D\sqrt{L/\varepsilon})$,
\par\smallskip
$\displaystyle N_{\rm ns}=O\!\left(1+\frac{d^{1/4}MD}{\varepsilon}
+d^{1/4}D\sqrt{\frac M{\gammaMax\varepsilon}}\right)$,
\par\smallskip
$\displaystyle T=O\!\left(N+
\frac{(\sigma_0^2+dM_2^2)D^2}{\varepsilon^2}\right)$.
The sufficient conditions use population scales only in the analysis.
\par\smallskip
Fixed batch sizes: $T\ge8N$, $3N$ stages.
Adaptive batch sizes: $T\ge16N$, $5N$ stages.
Both: exactly $T$ pairs, $2T$ values, $N$ projections.
Theorem~\ref{thm:protected-main}. The bounds include the model error for the
original objective and the cost of radius increases; no optimizer state is reset.\\[7pt]

\textbf{Selected search ball and adaptive smoothing.}
Uniform vector and radial moment assumptions, clean initial-record variance
$\sigma_{\rm a}^2$, and an exact ball-intersection projector.
Choose $r_0,h_{\rm ref}>0$, $p\ge1$, and legal $\gamma_{\rm a}>0$ before sampling;
$B_0=\floor{T/2}$, $T'=T-B_0$.
Use \eqref{eq:joint-main-anchor} once and the adjusted radius reference.
&
$\displaystyle\mathcal A_*(u)+\min\{\mathcal J_{\rm s}(S_2),\mathcal J_{\rm ns}(S_2)\}$,
\par\smallskip
where $\mathcal A_*$ is the minimum of
\eqref{eq:joint-main-price} and \eqref{eq:joint-main-saturated-price};
the costs are defined in
\eqref{eq:joint-main-smooth-cost}--\eqref{eq:joint-main-nonsmooth-cost},
with $\mathcal J_{\rm s}(S_2)=+\infty$ without population smoothness.
Here $\E S^2\le S_2^2$, $\E S\le S_2<r_p$.
The direct nonsmooth moment substitution requires
$\gammaMax\ge2d^{1/4}r_p/N$; otherwise the displayed finite-radius
minimum applies. All coefficient bounds use $T'$; initial smoothing bias remains.
\par\smallskip
Fixed batch sizes: $T\ge16N$, $3N+1$ stages.
Adaptive batch sizes: $T\ge32N$, $5N+1$ stages.
Both: exactly $T$ pairs, $2T$ values, $N$ intersection projections.
The exact bound is \eqref{eq:joint-main-risk}; Theorem~\ref{thm:joint-main}
states its explicit order bound \eqref{eq:joint-main-order}.
No supplied $D,R,L,M,M_2$ or variance.
The explicit reference contribution can be absorbed by
Corollary~\ref{cor:run5-distance-reference-absorption};
initial-estimate and search-restriction errors remain.\\[7pt]

\textbf{Adaptive smoothing with a supplied error bound.}
The forward-secant assumptions above hold uniformly.
A fixed spatial error satisfies $\|\delta\|_\infty\le\Delta$;
the bound $\Delta$ is supplied. Choose $\rho>0$ and
$g=c_d\Delta D/\rho\le\gammaMax/2$.
Use Algorithm~\ref{alg:protected-main}. For calibrated accuracy,
set $h_0=4\varepsilon N^2/(77D^2)$ and $\rho=\varepsilon/64$.
&
$\displaystyle\frac{h_0R_u^2}{2A}+\mathcal P_{\mathrm{reg}}^\Delta$,
\par\smallskip
or the order bound \eqref{eq:protected-main-order}.
For calibrated accuracy, Corollary~\ref{cor:epsilon-corruption-budgets}
gives the sufficient nonsmooth allowance
$\displaystyle\Delta\le\frac{\varepsilon^2}
 {2560c_d\mu_dMD}$
with its common scalar, radius and sampling conditions.
The smooth floor condition is
$\displaystyle\Delta\le\frac{\varepsilon^{3/2}}
 {64\sqrt{40}\,c_dD\sqrt{\nu_dL}}$.
The reference log is below $3/2$; the additional diagnostic term
$M_2g\sqrt{N/T}$ and its budget remain explicit.
\par\smallskip
Fixed batch sizes: $T\ge8N$, $3N$ stages.
Adaptive batch sizes: $T\ge16N$, $5N$ stages.
Both: exactly $T$ pairs, $2T$ values, $N$ projections.
Obtaining these adaptive-smoothing guarantees without a supplied corruption
bound remains open, as does matching arbitrary-pair optimality.\\[7pt]

\textbf{Bounded spatial corruption, fixed radius.}
Convex $M$-Lipschitz population with $L_\gamma$-smooth ball average;
clean single-record variance at most $\sigma_\gamma^2$.
A fixed corruption satisfies $\norm\delta_\infty\le\Delta$.
Algorithm~\ref{alg:scalar} uses any fixed $\gamma\in(0,\gammaMax]$
and does not require $\Delta$.
Set $u=x^\star$ and $c_d=d\E|[e]_1|\le\sqrt d$.
&
$\displaystyle\frac{h_0R^2}{2A}
 +\frac{16L_\gamma D^2}{N^2}+M\gamma
 +\frac{16\sigma_\gamma D}{\sqrt T}$
\par\smallskip
$\displaystyle\hfill
 +\frac{6c_dD\Delta}{\gamma}
 +\frac{16dD\Delta}{\gamma\sqrt T}$.
\par\smallskip
$T\ge8N$; exactly $T$ pairs, $2T$ values and $4N$ stages.
Theorem~\ref{thm:corruption}. If $T\ge d$, its last two terms
are at most $22\sqrt dD\Delta/\gamma$.
With a legal $\gamma\asymp\varepsilon/M$ and clean terms
within the target, this gives the sufficient scale
$\Delta=O(\varepsilon^2/(\sqrt dMD))$.\\[7pt]

\textbf{Selected search ball with fixed smoothing.}
A clean convex $L_\gamma$-smooth $f_\gamma$ dominates $f$;
uniform vector variance is at most $\sigma_\gamma^2$.
An exact ball-intersection projector is supplied.
Use \eqref{eq:frontier-anchor-rule}, with
$p\ge1,r_0,h_{\rm ref}>0$ chosen before sampling and $\gamma_{\rm a}=\gamma$.
&
$\displaystyle\mathcal A_*(u)+\frac{32L_\gamma S_2^2}A
 +\frac{\nu\sigma_\gamma S_2}{\sqrt{T'}}$.
\par\smallskip
$\mathcal A_*,S_2$ are
\eqref{eq:run5-distance-best-price} and
\eqref{eq:run5-distance-moment-radius}; initial smoothing bias remains.
Theorem~\ref{thm:run5-distance-fixed}.
Fixed batch sizes: $T\ge8N$, $\nu=8\sqrt2$, $2N+1$ stages.
Adaptive batch sizes: $T\ge16N$, $\nu=16\sqrt2$, $4N+1$ stages.
Both: exactly $T$ pairs, $2T$ values, $N$ intersection projections.
Only vector second moments are needed.\\[7pt]

\textbf{A smooth original stationary minimizer.}
The clean joint assumptions hold and $\nabla f(u)=0$ for the
original comparator, with $R_u=\|u-x^0\|$ and population smoothness $L$.
&
$\displaystyle\mathcal A_*^{\rm stat}(R_u)
 +\mathcal U_{\rm s}(2S_{R_u},\ell_{\rm s}^0)$,
\par\smallskip
where $S_{R_u}$ is \eqref{eq:run5-distance-stationary-radius} and
$\mathcal A_*^{\rm stat}$ is the substitution specified in
Corollary~\ref{cor:run5-distance-stationary}.
It retains anchor bias $L\nu_d\gamma_{\rm a}^2/2$ and finite moments.
This stationarity condition need not hold at a constrained minimizer.
The counts are those of the clean joint row.\\[7pt]

\textbf{A quadratic limitation of the selected-ball method.}
Fix $N\ge1,p\ge1,r_0,h_{\rm ref},G>0$; set
$h=h_{\rm ref}+AG/(pr_0)$, $q=G/h$ and $S=r_0+2Aq$.
Use the one-dimensional deterministic quadratic on $Q=[0,S]$,
$x^0=0$, $R=G/L\le r_0$, with $L\ge\max\{2Ah,G/r_0\}$
and fixed $0<\gamma\le\min\{\gammaMax,q/(2A)\}$.
&
$\displaystyle f_L(x_N)-f_L(x^\star)\ge Lq^2/(8A^2)$,
\par\smallskip
while $LR^2=G^2/L\to0$.
The exact trajectory is $x_N=q/A$, for fixed and selected radii alike;
vector variance is zero.
Proposition~\ref{prop:run5-distance-curved} and
Corollary~\ref{cor:run5-distance-curved-transfer}.
This limitation applies to the specified initialization.\\[7pt]

\textbf{Error under the relaxed affine conditions.}
$p\ge5$, $T\ge16N$, $R_0=2p r_0/5$, $Q=[0,r_0+R_0]$,
$x^0=r_0$, $m\ge h_{\rm ref}p r_0/A$.
Common-law risk is at most
$h_{\rm ref}\Psi(R_0)/A+m\Psi(R_0)/(p r_0)$.
An algorithm may use arbitrary adaptive endpoints or even
complete access to at most $T$ sampled functions.
The same algorithm and known inputs return $\widehat x\in Q$ on both laws.
&
A rare law has $R=r_0$, $G_1=m(1+1/T)$ and
$\sigma_1^2\le4Tm^2$. A proposed upper bound
$h_{\rm ref}\Psi(r_0)/A+G_1\Psi(r_0)/(p r_0)
 +P\sigma_1\mathcal H_p(r_0)/\sqrt T$
with $P\ge0$ requires $P\ge p/88$.
Here $\mathcal H_p(r_0)=(2+1/p)r_0$.
\par\smallskip
Theorem~\ref{thm:frontier-pareto-lower} gives the exact constant.
The upper multiplier for fixed batch sizes is $\sqrt3+16(1+2p)$
(Corollary~\ref{cor:frontier-noise-multiplier}); the order is attained
under these affine conditions. Curved-objective optimality remains open.\\[7pt]

\textbf{Constraint from initial progress on affine objectives.}
$r_0>0$, $N\ge2$, $T\ge8N$, $h_0=1/r_0$,
$S=r_0(1+A)$, $Q=[0,S]$, $x^0=r_0$.
An algorithm uses at most $T$ independently sampled functions and arbitrary endpoints,
with the deterministic bound in Corollary~\ref{cor:startup-price}.
The same algorithm and known inputs return $\widehat x\in Q$ on both laws.
&
An alternative finite-moment law has solution distance $R=r_0$
and $\sigma_1^2\le4T$.
A proposed bound
$h_0R^2/(2A)+P\sigma_1(R+r_0)/\sqrt T$
with $P\ge0$ requires $P\ge A/32$.
Allowing deterministic residual at most $r_0A/4$ still requires
$P\ge A/64$.
This lower constraint depends on the displayed initialization requirements;
it permits even full sample-function access.\\[7pt]

\textbf{Corruption constraint at a pair budget.}
$T\ge1$, $\ell_T=\log(8(2T+1))$, $d\ge8\ell_T$,
$M,R>0$, $0<\varepsilon\le MR/8$.
The algorithm may use arbitrary query points and any sequential
schedule, with at most $T$ independent pairs.
&
If
$\displaystyle\Delta\ge\frac{8\varepsilon\ell_T}{d}
 +\frac{16\sqrt2\varepsilon^2\sqrt{\ell_T}}{MR\sqrt d}$,
an admissible convex $M$-Lipschitz instance on $R\mathbb B_2^d$
with original distance $R$ has expected gap at least $3\varepsilon/2$.
Corollary~\ref{cor:noise-frontier}.
This is a pair-budget lower bound; joint stage/value/corruption
optimality and the finer-accuracy gap remain open.\\

\textbf{Actual variance along one fixed-target trajectory.}
Convex $L$-smooth target $\phi\ge f$; conditional record means equal
$\nabla\phi$; counts are chosen before independent update records.
Set $\omega_{2t-1}=\omega_{2t}=t$ and let $V_j$ be the variance
at the actual exposed center.
&
$\displaystyle W=\E\sum_{j=1}^{2N}\omega_j^2V_j/B_j$,
\par\smallskip
$\displaystyle
 \frac{h_0R_u^2}{2A}+\frac{2D^2(4L-h_0)_+}{A}
 +\phi(u)-f(u)+\frac{2D\sqrt W}{A}$.
\par\smallskip
For the adaptive batch rule $W\le16A^2\overline{\mathcal V}_N/(T-4N)$;
for fixed batch sizes $W\le8A^2\overline{\mathcal V}_N/T$, where
$\overline{\mathcal V}_N=\E\sum_j\omega_jV_j/(2A)$.
With adaptive batch sizes, $T\ge8N$ is required and the method uses exactly $T$ pairs, $2T$ values and $4N$ stages;
the version with fixed batch sizes requires $T\ge4N$ and uses at most $T$ pairs and $2N$ stages.
Theorem~\ref{thm:batch-variance-actual} and
Corollary~\ref{cor:batch-variance-public}.\\[7pt]

\textbf{Scalar pre-step testing with exact records.}
Convex $L$-smooth fixed target; exact gradient and relative-value
records; $L,R_u$ are unknown; $h_0,\mathcal L_0>0$ chosen before sampling.
One accepted trajectory persists across rejected trials.
&
$\displaystyle \phi(x_N)-\phi(u)
 \le2\max\{\mathcal L_0,2L\}R_u^2/N^2$.
\par\smallskip
At most $J_\star=\ceil{\log_2\max\{1,L/\mathcal L_0\}}$
rejected trials. A cap of $N+J_\star$ trials suffices.
For deterministic one-dimensional quadratics, exact records use two
pairs, four values, two stages and one projection per trial.
In the general stochastic model, the potential retains the rejected
error charge and a random terminal weight.
Theorem~\ref{thm:prestep-exact-distance} and
Proposition~\ref{prop:prestep-rejection-charge}.\\[7pt]

\textbf{Confidence-controlled scalar pre-step execution.}
Fixed legal smoothing radius, supplied $D$ and uniform record
moment caps $\bar\sigma,\overline M_2>0$; unknown smoothness $L_\phi$
of $\phi=f_\gamma$ and original comparator distance $R_u$.
Use robust record estimates, tolerance $\varepsilon>0$,
confidence $\delta_{\rm conf}\in(0,1)$ and a finite trial cap.
&
$\displaystyle
 \E[f(\widehat x)-f(u)]\le
 \frac{2\overline{\mathcal L}R_u^2}{N^2}
 +\frac{\varepsilon}{2}+\beta_\gamma(u)+\delta_{\rm conf}MD$,
\par\smallskip
where $\overline{\mathcal L}=\max\{\mathcal L_0,4L_\phi\}$ and
$\beta_\gamma(u)=f_\gamma(u)-f(u)$.
The cap satisfies
$K_{\rm tr}\ge N+\ceil{\log_2\max\{1,2L_\phi/\mathcal L_0\}}$.
For the actual executed trials,
$T_{\rm exec}=k_{\rm mom}\sum_j(b_j^g+b_j^q)$ pairs,
$2T_{\rm exec}$ values, $2J_{\rm exec}$ stages and
$J_{\rm exec}$ projections.
The sufficient cap computable from known inputs and deterministic sample budget
are in Corollary~\ref{cor:preconf-public-cap} and
Theorem~\ref{thm:preconf-paid-risk}.
The confidence sample cost exceeds the tuned joint work bound;
the supplied caps are retained.\\[7pt]

\textbf{Unknown variance and original distance: common curvature.}
Two convex quadratic sample laws share supplied $L,D,N,T$ and
$x^0=r$, with $\gammaMax=D$ and $0<r\le D/2$. Their intrinsic distance and variance
are unknown. Require $T\ge4$. Arbitrary pairs and full realized functions are allowed.
&
A bound $\E_i[f_i(\widehat x)-f_i(x_i^\star)]\le Cb_i$
on both laws requires
\par\smallskip
$\displaystyle C\ge\min\{N^2,D/r\}/64$
for $b_i=LR_i^2/N^2+\sigma_iR_i/\sqrt T$,
\par\smallskip
$\displaystyle C\ge\min\{N^2,\sqrt T,D/r\}/96$
for $b_i=LR_i^2/N^2+M_{2,i}R_i/\sqrt T$.
The exact matching factor is
$q_TLD^2/[2(\sqrt{b_1}+\sqrt{q_Tb_0})^2]$ on the disclosed
two-law family; its upper uses one stage and exactly $T$ pairs.
Theorem~\ref{thm:universal-distance-menu}. With unknown curvature,
Corollary~\ref{cor:prestep-rare-price} gives a necessary
diameter-scale charge for any method making the stated linear progress.\\[7pt]

\textbf{A finite-moment limitation of the adaptive allocation.}
The benchmark freezes the adaptive policy's actual queried centers and
reallocates only its surplus $S_{\rm sur}$; assume $N\ge2,T\ge8N$.
&
The exact worst-case variance-energy ratio is
$1+S_{\rm sur}/b_1$, of order $\min\{\mathsf U,N^2\}$,
relative to the floor-respecting allocation on the same trace.
The lower construction uses a global iid sample function with legal
two-value records. This ratio does not bound optimization risk from below.
Proposition~\ref{prop:batch-variance-paid-limit}.\\[7pt]

\textbf{Adaptive relative-mean certification.}
Global affine sample losses $F(x,Z)=Zx$; arbitrary endpoints,
adaptive sampling and stopping. Require a uniformly valid upper
mean report with false-report probability at most
$\delta_{\rm conf}$. Under the null $Z=0$, put
$C=\{\tau<\infty,\widehat U\le\eta\}$ and $q=\Pr_0(C)$.
&
Without a supplied moment cap,
$\Pr_0(\tau<\infty,\widehat U<\infty)\le\delta_{\rm conf}$.
With supplied $\E Z^2\le\overline M_2^2$, $0<\eta<\overline M_2$
and $q>\delta_{\rm conf}$,
\par\smallskip
$\displaystyle
 \E_0[\tau\mathbf1_C]\ge
 \frac{q\log(q/\delta_{\rm conf})}
 {-\log(1-\eta^2/\overline M_2^2)}$.
\par\smallskip
For $\overline M_2^2\ge2\eta^2$, $\delta_{\rm conf}\le1/4$,
and $q=1$, the order
$\overline M_2^2\eta^{-2}\log(1/\delta_{\rm conf})$
is attained by a median of block means using one sampling stage,
even though the null variance is zero.
Corollary~\ref{cor:observable-certification-no-cap} and
Theorem~\ref{thm:observable-certification-capped}.\\[7pt]

\end{longtable}
\endgroup

For completeness, the smooth row with fixed batch sizes retains the exact
comparator bias in the proof of Corollary~\ref{cor:public}.
Proposition~\ref{prop:protected-smoothing-smooth} gives
$L_\gamma\le L$ and
$f_\gamma(x^\star)-f(x^\star)\le L\nu_d\gamma^2/2$.
That proof therefore yields
\[
 \E[f(x_N)-f(x^\star)]\le
 \frac{h_0R^2}{2A}
 +\frac{16LD^2}{N^2}
 +\frac{4\sqrt2\sigma_\gamma D}{\sqrt T}
 +\frac{L\nu_d\gamma^2}{2}.
\]
Substituting $\gamma\le D/N$ proves the stated row. The accounting for fixed batch sizes
is unchanged. Its two vector update batches per iteration
give $2N$ sequential stages; no radial estimate is used.

The earlier fixed-threshold derivative-record execution requires
$\tau>2d\Delta$ and its own reciprocal-radius terms
(Theorem~\ref{thm:radial-corruption}). The protected forward-secant
execution instead subtracts $2\Delta$ from its diagnostic and charges its
floor, bias, variance and positive target drift in the displayed bounds.
These are distinct policies with their assumptions stated above.
Every algorithm in the table retains one optimizer state and uses no optimizer restart.

\section{Smoothing and the paired oracle}\label{sec:app-oracle}
Let $U$ be uniform on the Euclidean unit ball and $e$ uniform on its sphere. Write $[e]_i$ for the $i$th coordinate of $e$. Set $\mu_d=d/(d+1)$, $\nu_d=d/(d+2)$, and
\[
 c_d=d\E|[e]_1|=\frac{d\Gamma(d/2)}{\sqrt\pi\Gamma((d+1)/2)}\le\sqrt d.
\]
Throughout this section, $x\in Q$ and $0<\gamma\le\gammaMax$.
Other positive radii are also at most $\gammaMax$; where a zero radius
is explicitly allowed, we set $f_0=f$.
For centers on the boundary of $Q$ or radii equal to $\gammaMax$, we take
smoothing derivatives after a globally convex $M$-Lipschitz extension of $f$.
The proof below constructs this extension and shows that the displayed
derivatives depend only on values in the evaluation domain.

\begin{lemma}[Ball smoothing]\label{lem:smoothing}
For a convex $M$-Lipschitz population objective, $f_\gamma(x)=\E f(x+\gamma U)$ is convex and differentiable, with
\[
 \nabla f_\gamma(x)=\frac d\gamma\E[f(x+\gamma e)e],\quad
 0\le f_\gamma(x)-f(x)\le M\mu_d\gamma,\quad
 L_\gamma\le\frac{c_d M}{\gamma}\le\frac{\sqrt d M}{\gamma}.
\]
For each fixed $x$, $f_\gamma(x)$ is nondecreasing in $\gamma$. If $f$ is additionally $L$-smooth, then $L_\gamma\le L$ and
\[
 0\le f_\gamma(x)-f(x)\le L\nu_d\gamma^2/2.
\]
\end{lemma}
\begin{proof}
Let $C=Q+\gammaMax\Ball_2^d$ and define
\[
 \bar f(x)=\min_{z\in C}\{f(z)+M\|x-z\|\},\qquad x\in\R^d.
\]
The minimum exists because $C$ is nonempty and compact and the integrand is
continuous. For $x\in C$, Lipschitzness gives $f(z)+M\|x-z\|\ge f(x)$;
the choice $z=x$ proves $\bar f(x)=f(x)$. Combining minimizers at two
points and using convexity of $C$, $f$ and the norm proves convexity of
$\bar f$. Evaluating at a minimizer for $y$ and applying the triangle
inequality gives $\bar f(x)\le\bar f(y)+M\|x-y\|$. Interchanging $x,y$
proves that $\bar f$ is globally $M$-Lipschitz, including when $M=0$.

For a smooth function, differentiation under the ball integral and the
divergence theorem give the displayed boundary identity: the
sphere-to-ball volume ratio is $d$. Convolve $\bar f$ with smooth kernels.
The approximations converge uniformly on compact sets, as do their ball
averages. For fixed $\gamma>0$, the boundary expressions converge uniformly
on compact sets because their difference is at most $d/\gamma$ times the
uniform approximation error. Integrating these gradients along line segments
and passing to the limit proves that the limiting boundary expression is the
gradient of the ball average. This argument applies at boundary centers and
at $\gamma=\gammaMax$. On $Q$, the formula uses only values in $C$ and is
therefore independent of the chosen extension.

Convexity is preserved by translation and expectation. Jensen's inequality gives $f\le f_\gamma$. The Lipschitz inequality and the radial density $dr^{d-1}$ give $f_\gamma-f\le M\gamma\E\|U\|=M\mu_d\gamma$. For every unit vector $v$, the boundary formula gives
\[
 |\langle\nabla f_\gamma(x)-\nabla f_\gamma(y),v\rangle|
 \le \frac{dM}{\gamma}\|x-y\|\E|\langle e,v\rangle|
 =\frac{c_d M}{\gamma}\|x-y\|.
\]
Take the supremum over $v$. Rotational symmetry and Cauchy--Schwarz give $c_d\le d\sqrt{\E [e]_1^2}=\sqrt d$; integration of the spherical coordinate density gives the stated gamma-function expression (for $d=1$, both sides equal one).

The function $r\mapsto\E\bar f(x+rU)$ is even and convex, so it is
nondecreasing on $r\ge0$. In the smooth case, almost every ball sample has
$\|U\|<1$, and hence $x+\gamma U$ lies in the interior of $C$, even at
the largest legal radius. There $\bar f$ agrees with $f$ on a neighborhood,
so their gradients agree. The global Lipschitz bound justifies
differentiation under the integral and gives
$\nabla f_\gamma(x)=\E\nabla f(x+\gamma U)$. Averaging the original
gradient's $L$-Lipschitz bound proves $L_\gamma\le L$ on $Q$; no smoothness
of $\bar f$ outside $C$ is required. Averaging
$f(x+\gamma U)\le f(x)+\gamma\langle\nabla f(x),U\rangle+L\gamma^2\|U\|^2/2$
and using $\E U=0$, $\E\|U\|^2=\nu_d$ proves the last claim.
\end{proof}

\begin{lemma}[Moments and corruption of the symmetric estimate]\label{lem:oracle-moments}
Suppose the sample functions are locally Lipschitz on a neighborhood of the evaluation domain and their weak gradients satisfy $\E_\xi\|\nabla F(z,\xi)\|^2\le M_2^2$ on its interior, with the integrability needed for the following sphere and segment integrals. The clean estimate
\[
 G_\gamma^0(x)=\frac d{2\gamma}[F(x+\gamma e,\xi)-F(x-\gamma e,\xi)]e
\]
has mean $\nabla f_\gamma(x)$ and $\E\|G_\gamma^0(x)\|^2\le2dM_2^2$. A random sample Lipschitz constant $m(\xi)$ with $\E m(\xi)^2\le M_2^2$ is a sufficient condition. With fixed spatial corruption $|\delta|\le\Delta$, write
$G_\gamma=\nabla f_\gamma+b_\gamma+\zeta_\gamma$, where $\E\zeta_\gamma=0$. Then
\[
 \|b_\gamma\|\le c_d\Delta/\gamma,\qquad
 \E\|\zeta_\gamma\|^2\le2\sigma_\gamma^2+2d^2\Delta^2/\gamma^2
\]
whenever the clean centered second moment is at most $\sigma_\gamma^2$.
For a uniform bound, replace $\sigma_\gamma$ by $\sigma_0$. An average of $B$ conditionally independent records, drawn after the center, radius and batch size are fixed, has the same mean bias and centered second moment at most this bound divided by $B$.
\end{lemma}
\begin{proof}
We first bound the clean estimate, then separate the corruption bias from
its centered fluctuation.

\emph{Mean.}
Taking expectation in the sample and using sphere symmetry gives
\[
 \E G_\gamma^0(x)
 =\frac d{2\gamma}\E_e[(f(x+\gamma e)-f(x-\gamma e))e]
 =\frac d\gamma\E_e[f(x+\gamma e)e]
 =\nabla f_\gamma(x),
\]
where the last equality is Lemma~\ref{lem:smoothing}. The integrability
hypotheses justify the exchanges of expectations.

\emph{Second moment for $d\ge2$.}
First suppose the samples are smooth and $\gamma<\gammaMax$. The function
$h_\xi(e)=F(x+\gamma e,\xi)-F(x-\gamma e,\xi)$ is odd and therefore has
zero spherical mean. The spherical Poincar\'e inequality, with first
nonzero eigenvalue $d-1$, gives
\[
 \E_e h_\xi(e)^2\le\frac1{d-1}\E_e\|\nabla_S h_\xi(e)\|^2.
\]
The tangent gradient is
\[
 \nabla_S h_\xi(e)
 =\gamma(I-ee^\top)
 [\nabla F(x+\gamma e,\xi)+\nabla F(x-\gamma e,\xi)].
\]
Since $I-ee^\top$ is an orthogonal projector, the squared norm of this
expression is at most twice $\gamma^2$ times the sum of the two squared
gradient norms. Taking expectation in $\xi$ and $e$ yields
\[
 \E h_\xi(e)^2\le\frac{4\gamma^2M_2^2}{d-1},
 \qquad
 \E\|G_\gamma^0(x)\|^2
 =\frac{d^2}{4\gamma^2}\E h_\xi(e)^2
 \le\frac{d^2}{d-1}M_2^2\le2dM_2^2.
\]

For locally Lipschitz samples, convolve with a smooth nonnegative kernel
of unit mass supported in a ball of radius less than
$\gammaMax-\gamma$. Its translates at the queried endpoints remain in
the interior of the evaluation domain. Jensen's inequality and the
weak-gradient moment assumption preserve the bound $M_2^2$ for the
expected squared gradient of each convolved sample. The smooth argument
therefore applies. As the kernel support shrinks to zero, sample
continuity gives convergence of the endpoint differences. Fatou's lemma
then proves the same second-moment bound for the original samples.
This argument avoids evaluating weak gradients on a prescribed sphere.
The resulting dimension dependence also appears in the two-point
analysis of \citet{Shamir2017}.

\emph{Dimension one and the largest radius.}
For $d=1$, the fundamental theorem along the interval $[-\gamma,\gamma]$
and Jensen's inequality give
\[
 \E_\xi|F(x+\gamma,\xi)-F(x-\gamma,\xi)|^2
 \le 2\gamma\int_{-\gamma}^{\gamma}
       \E_\xi|F'(x+s,\xi)|^2\,ds
 \le4\gamma^2M_2^2.
\]
Thus $\E\|G_\gamma^0(x)\|^2\le M_2^2$ in this case. For any dimension,
let inner radii increase to $\gammaMax$. Sample continuity gives
pointwise convergence of the estimates, and Fatou's lemma preserves
the second-moment bound at the limiting radius. The mean formula above
also holds there by Lemma~\ref{lem:smoothing} and the assumed
integrability.

The random-Lipschitz alternative requires Lipschitzness only on the closed
evaluation domain $C=Q+\gammaMax\Ball_2^d$. Extend each sample by
\[
 \bar F(z,\xi)=\inf_{w\in C}\{F(w,\xi)+m(\xi)\|z-w\|\}.
\]
This extension preserves all queried values and is globally
$m(\xi)$-Lipschitz. Thus the preceding argument applies, and
$\E m(\xi)^2\le M_2^2$ bounds its gradient moments and endpoint differences.

\emph{Corruption bias and variance.}
Write
\[
 C_\gamma=\frac d{2\gamma}
 [\delta(x+\gamma e)-\delta(x-\gamma e)]e,
 \qquad b_\gamma=\E C_\gamma.
\]
For every unit vector $v$,
\[
 |\langle b_\gamma,v\rangle|
 \le\frac{d\Delta}{\gamma}\E|\langle e,v\rangle|
 =\frac{c_d\Delta}{\gamma}.
\]
Taking the supremum over $v$ proves the bias bound. Since
$\|C_\gamma\|\le d\Delta/\gamma$ and
$\zeta_\gamma=(G_\gamma^0-\E G_\gamma^0)+(C_\gamma-\E C_\gamma)$, we have
\[
 \E\|\zeta_\gamma\|^2
 \le2\E\|G_\gamma^0-\E G_\gamma^0\|^2
       +2\E\|C_\gamma-\E C_\gamma\|^2
 \le2\sigma_\gamma^2+\frac{2d^2\Delta^2}{\gamma^2}.
\]

\emph{Conditional batching.}
Condition on the observations fixing the center, radius and batch size
$B$. The $B$ centered update records have conditional mean zero and are
conditionally independent, so cross terms in the squared norm of their
average vanish. Its centered second moment is therefore at most
$(2\sigma_\gamma^2+2d^2\Delta^2/\gamma^2)/B$, while its bias remains
$b_\gamma$. The conditioning includes any separate diagnostic
observations used to select $B$.
\end{proof}

\begin{proposition}[Nonsymmetric directional endpoints]\label{prop:asymmetric}
For $\alpha>0$, $\beta\ge0$, define
\[
 G_{\alpha,\beta}(x)=\frac d{\alpha+\beta}
 [\widetilde F(x+\alpha e,\xi)-\widetilde F(x-\beta e,\xi)]e.
\]
Its clean mean is the gradient of
$\phi_{\alpha,\beta}=(\alpha f_\alpha+\beta f_\beta)/(\alpha+\beta)$, with $f_0=f$. Its smoothing bias is at most
\[
 M\mu_d\frac{\alpha^2+\beta^2}{\alpha+\beta},
 \quad\text{or, under $L$-smoothness,}\quad
 \frac{L\nu_d}{2}\frac{\alpha^3+\beta^3}{\alpha+\beta}.
\]
Its mean corruption is at most $2c_d\Delta/(\alpha+\beta)$; this is attained at a center when both radii are positive. For $\beta=0$ the stronger bound is $c_d\Delta/\alpha$. At the center of $f(z)=M\|z\|$, the clean second moment is exactly
\[
 d^2M^2\left(\frac{\alpha-\beta}{\alpha+\beta}\right)^2.
\]
\end{proposition}
\begin{proof}
The boundary identity gives
$\E[f(x+\alpha e)e]=(\alpha/d)\nabla f_\alpha(x)$ and
$\E[f(x-\beta e)e]=-(\beta/d)\nabla f_\beta(x)$ when $\beta>0$.
For $\beta=0$, the latter mean equals $f(x)\E e=0$. Addition proves the target identity. Mixing the two smoothing bounds of Lemma~\ref{lem:smoothing} proves the bias claims. The dual norm argument in Lemma~\ref{lem:oracle-moments}, now with the coefficient $d/(\alpha+\beta)$, proves the corruption bound. At a fixed center $x_c$, the spatial function $\delta(z)=\Delta\operatorname{sign}\langle z-x_c,v\rangle$ attains it for any unit $v$ and positive radii. If $\beta=0$, the center contribution has mean zero, so only one endpoint contributes. Finally, the norm objective has endpoint values $M\alpha,M\beta$ at zero. Its estimate is the constant $dM(\alpha-\beta)/(\alpha+\beta)$ times $e$, proving the last identity. This instance concerns this estimator family; it imposes no restriction on other legal pairs.
\end{proof}

\begin{proposition}[An adjacent-scale paired diagnostic]\label{prop:adjacent-scale}
Under the sample-gradient moment and sphere/segment integrability hypotheses of Lemma~\ref{lem:oracle-moments}, one legal nonsymmetric pair at $x+\gamma U$ and $x+(\gamma/2)U$ gives
\[
 V_\gamma=\widetilde F(x+\gamma U,\xi)-\widetilde F(x+(\gamma/2)U,\xi).
\]
Without corruption its mean is $v_\gamma=f_\gamma(x)-f_{\gamma/2}(x)$, and
\[
 \tfrac12[f_\gamma(x)-f(x)]\le v_\gamma\le f_\gamma(x)-f(x),\qquad
 \E (V_\gamma^0)^2\le M_2^2\nu_d\gamma^2/4.
\]
Corruption shifts the mean by at most $2\Delta$ and gives
$\E V_\gamma^2\le M_2^2\nu_d\gamma^2/2+8\Delta^2$.
\end{proposition}
\begin{proof}
The mean identity follows from the same-sample sample expectation. Convexity gives $f(x+(\gamma/2)U)\le[f(x+\gamma U)+f(x)]/2$, and Jensen's inequality gives $f_{\gamma/2}(x)\ge f(x)$. These imply the bracket. Along the segment from $x+(\gamma/2)U$ to $x+\gamma U$, the fundamental theorem and Jensen bound the squared sample difference by $(\gamma^2\|U\|^2/4)$ times the average squared gradient norm on that segment. Integrating in $\xi,U$ yields the displayed moment bound. The two corruption values differ by at most $2\Delta$; applying $(p+q)^2\le2p^2+2q^2$ proves the last assertion.

The bracket concerns bias at the queried center. It does not bound smoothing bias at an optimizer: for $f(x)=|x|$ in dimension one and $x\ge\gamma$, $v_\gamma=0$, whereas $f_\gamma(0)-f(0)=\gamma/2$.
\end{proof}

\section{The scalar accelerated potential}
\label{sec:app-core}

We prove the results in Section~\ref{sec:theory} for arbitrary vector
errors. The proof uses convexity of each mean target and Euclidean
projection. Individual realized vector records need not be gradients of
convex functions.

\subsection{Proof of Theorem~\ref{thm:potential}}
\label{subsec:proof-potential}

The following potential is valid for predictable nonincreasing targets $f\le\phi_{t+1}\le\phi_t$. It separates the smoothing error from the errors of the two vector sample means. Write
\[
 e_t^y=g_t^y-\nabla\phi_t(y_t),\qquad
 e_t^x=g_t^x-\nabla\phi_t(x_t),
\]
and define
\begin{equation}
 Z_t=t\ip{e_t^y}{u-v_{t-1}}-
       t\ip{e_t^x}{v_t-v_{t-1}}.
 \label{eq:martingale-term}
\end{equation}

\begin{theorem}[Accelerated potential]\label{thm:potential}
For every fixed $u\in Q$, the iterates
\eqref{eq:center-y}--\eqref{eq:coefficient} satisfy
\begin{equation}
 A_N[f(x_N)-f(u)]
 \le \tfrac{h_0}{2}\norm{u-x^0}^2
       +2D^2(H_N-h_0)
       +\sum_{t=1}^Nt[\phi_t(u)-f(u)]
       +\sum_{t=1}^NZ_t.
 \label{eq:potential}
\end{equation}
If the total errors $e_t^y,e_t^x$ are integrable and satisfy
the conditional decomposition \eqref{eq:conditional}, then
$\E\sum_tZ_t\le D\sum_tt\,\E(\norm{b_t^y}+\norm{b_t^x})$.
For batch means of \eqref{eq:record} with exact evaluations and
targets $\phi_t=f_{\gamma_t}$, $b_t^y=b_t^x=0$, and hence
$\E\sum_{t=1}^N Z_t=0$.

\end{theorem}

Fix $u\in Q$ and write $p_t^y=\nabla\phi_t(y_t)$, $p_t^x=\nabla\phi_t(x_t)$,
$\delta_t=H_t-H_{t-1}$ and $w_t=v_t-v_{t-1}$. The scalar update has the
exact balance identity
\begin{equation}
 \delta_tD^2=
 \pos{A_t\widehat\beta_t-H_tr_t^2/2}.
 \label{eq:balance-app}
\end{equation}
Indeed, if the numerator in \eqref{eq:coefficient} is positive, multiply
that equality by $D^2+r_t^2/2$ and move the term
$\delta_tr_t^2/2$ to the right. If it is nonpositive, $\delta_t=0$ and
both sides of \eqref{eq:balance-app} vanish.

Optimality of the projection in \eqref{eq:step} gives
\begin{align}
 a_t\ip{g_t^y}{v_t-u}
 \le \frac{H_{t-1}}2\bigl(
 \norm{u-v_{t-1}}^2-\norm{u-v_t}^2-r_t^2\bigr).
 \label{eq:projection-app}
\end{align}
Convexity at $y_t$, together with the definition of $y_t$, yields
\begin{align}
 A_{t-1}\phi_t(x_{t-1})+a_t\phi_t(u)
 \ge A_t\phi_t(y_t)+a_t\ip{p_t^y}{u-v_{t-1}}.
 \label{eq:convex-y-app}
\end{align}
The identity $A_t(x_t-y_t)=a_tw_t$ and convexity at $x_t$ imply
\begin{equation}
 A_t[\phi_t(x_t)-\phi_t(y_t)]
 \le a_t\ip{p_t^x}{w_t}.
 \label{eq:convex-x-app}
\end{equation}
Adding these inequalities and substituting $p_t^\iota=g_t^\iota-e_t^\iota$ for $\iota\in\{y,x\}$ gives
\begin{align}
 A_t\phi_t(x_t)-A_{t-1}\phi_t(x_{t-1})-a_t\phi_t(u)
 \le a_t\ip{g_t^y}{v_t-u}+A_t\widehat\beta_t+Z_t.
 \label{eq:one-step-app}
\end{align}
Add $H_t\norm{u-v_t}^2/2$ to both sides and use
\eqref{eq:projection-app}. The curvature remainder is
\begin{align}
 A_t\widehat\beta_t-\frac{H_t r_t^2}{2}
 +\frac{\delta_t}{2}\bigl(\norm{u-v_t}^2+r_t^2\bigr)
 \le 2D^2\delta_t. \label{eq:curvature-remainder-app}
\end{align}
Here \eqref{eq:balance-app} bounds the first term by $D^2\delta_t$,
and both distances in the second term are at most $D$. Therefore
\begin{align}
 &A_t\phi_t(x_t)+\frac{H_t}{2}\norm{u-v_t}^2 \nonumber\\
 &\quad\le A_{t-1}\phi_t(x_{t-1})
 +\frac{H_{t-1}}2\norm{u-v_{t-1}}^2
 +a_t\phi_t(u)+2D^2\delta_t+Z_t.
 \label{eq:potential-step-app}
\end{align}
Sum over $t$. The target terms on the left telescope according to
\begin{align}
 &\sum_{t=1}^N\bigl[A_t\phi_t(x_t)
                   -A_{t-1}\phi_t(x_{t-1})\bigr] \nonumber\\
 &=A_N\phi_N(x_N)+
 \sum_{t=1}^{N-1}A_t[\phi_t(x_t)-\phi_{t+1}(x_t)]
 \ge A_N\phi_N(x_N).
 \label{eq:moving-target-telescope-app}
\end{align}
The inequality follows from the assumed target order. Discard the
nonnegative terminal distance, use $\phi_N\ge f$, and subtract
$A_Nf(u)=\sum_ta_tf(u)$. This proves \eqref{eq:potential}.

For the expectation statement, $u-v_{t-1}$ is known before the $y$
update batch. The displacement $w_t$ is known before the $x$ update
batch. Conditional centering in \eqref{eq:conditional} therefore removes
both $\zeta$ inner products separately. The two remaining bias inner
products are at most
$a_tD(\norm{b_t^y}+\norm{b_t^x})$, since $r_t\le D$.
Integrability of the total errors, together with these bounded
displacements, makes both centered inner products integrable. For
smoothing targets, it follows from Assumption~\ref{ass:conditional}
and $\|\nabla f_{\gamma_t}\|\le M$.
Thus summation and expectation are justified. For batch means of
\eqref{eq:record} with exact evaluations and targets
$\phi_t=f_{\gamma_t}$, both biases vanish and $\E\sum_t Z_t=0$.
\qed

\subsection{Proof of Lemma~\ref{lem:growth}}
\label{subsec:proof-growth}

\begin{lemma}[Scalar coefficient growth]\label{lem:growth}
If each $\phi_t$ has gradient Lipschitz constant at most $L_\circ$, then
\begin{equation}
 \pos{H_N-4L_\circ}^2\le\pos{h_0-4L_\circ}^2
 +\frac{1}{D^2}\sum_{t=1}^Nt^2\norm{e_t^x-e_t^y}^2.
 \label{eq:growth}
\end{equation}
Consequently, if
$W=\sum_{t=1}^N t^2\E\norm{e_t^x-e_t^y}^2<\infty$, then
$\E(H_N-h_0)\le4L_\circ+\sqrt W/D$.

\end{lemma}

The scalar update uses no $L_\circ$; its growth is controlled by curvature
and realized error. 

Put $\mathcal E_N=\sum_tt^2\norm{e_t^x-e_t^y}^2$ and
$K_t=\pos{H_t-4L_\circ}$. Smoothness and
$A_t(x_t-y_t)=t(v_t-v_{t-1})$ give
\begin{equation}
 A_t\widehat\beta_t
 \le2L_\circ r_t^2+t\norm{e_t^x-e_t^y}r_t.
 \label{eq:growth-curvature-app}
\end{equation}
For $K_t>0$, the balance identity \eqref{eq:balance-app} implies
\[
 (H_t-H_{t-1})D^2
 \le\pos{-K_tr_t^2/2+t\norm{e_t^x-e_t^y}r_t}
 \le\frac{t^2\norm{e_t^x-e_t^y}^2}{2K_t}.
\]
The last step completes a scalar square. Since
$0\le K_t-K_{t-1}\le H_t-H_{t-1}$,
\begin{equation}
 D^2(K_t^2-K_{t-1}^2)
 \le2D^2K_t(K_t-K_{t-1})
 \le t^2\norm{e_t^x-e_t^y}^2.
 \label{eq:growth-square-app}
\end{equation}
For $K_t=0$, both positive parts vanish and the same inequality holds.
Summation proves \eqref{eq:growth}. Taking square roots gives
$H_N-h_0\le(4L_\circ-h_0)_++\sqrt{\mathcal E_N}/D$:
if $H_N\le4L_\circ$ the conclusion is immediate; otherwise
$K_N\le K_0+\sqrt{\mathcal E_N}/D$.
When $W=\E\mathcal E_N<\infty$, Jensen's inequality gives the
expected bound in the lemma.
\qed

\subsection{Proof of Corollary~\ref{cor:public}}
\label{subsec:proof-public}

The sum of the counts in \eqref{eq:public-batch} is at most
\begin{equation}
 2N+\frac{2(T-2N)}{N(N+1)}\sum_{t=1}^Nt=T.
\end{equation}
The inequality $1+\floor z\ge z$, and $T-2N\ge T/2$, imply
\begin{equation}
 B_t^y=B_t^x\ge\frac{Tt}{2N(N+1)}.
 \label{eq:public-floor-app}
\end{equation}
For $\iota\in\{y,x\}$, let $e_t^\iota$ be the error of the clean
update mean at the corresponding center. Its conditional second moment
is at most $\sigma_\gamma^2/B_t^\iota$. Moreover,
$e_t^y$ is known before the second update mean and that second mean is
conditionally unbiased. Hence
$\E\ip{e_t^x}{e_t^y}=0$, even though its center depends on $e_t^y$.
It follows that
\begin{align}
 W&\le\sum_{t=1}^Nt^2\sigma_\gamma^2
       (1/B_t^y+1/B_t^x)
 \le\frac{8\sigma_\gamma^2A_N^2}{T}.
 \label{eq:public-energy-app}
\end{align}
Apply Theorem~\ref{thm:potential} with $u=x^\star$ and
$\phi_t=f_\gamma$. The expectation of $Z_t$ vanishes, and
$f_\gamma(x^\star)-f(x^\star)\le M\gamma$.
Lemma~\ref{lem:growth} now bounds the curvature contribution by
$8L_\gamma D^2/A_N+4\sqrt2\sigma_\gamma D/\sqrt T$.
The inequality $A_N\ge N^2/2$ proves \eqref{eq:public-risk}.
There are two vector sampling stages and one projection per iteration. Each
record consumes one pair and two values, which proves the resource
statement.
\qed

\subsection{A radius from the iteration budget}
\label{subsec:proof-horizon}

\begin{corollary}[A radius from the iteration budget]\label{cor:horizon}
For exact function evaluations, assume $\sigma_\gamma^2\le CdM_2^2$ and set
\begin{equation}
 \gamma=\min\{\gammaMax,d^{1/4}D/N\}.
 \label{eq:horizon-radius}
\end{equation}
The risk of Algorithm~\ref{alg:scalar} is at most
\begin{align}
 \frac{h_0R^2}{2A_N}
 &+\frac{17d^{1/4}MD}{N}
 +\frac{16\sqrt d\,MD^2}{\gammaMax N^2}\nonumber\\
 &
 +\frac{8\sqrt{2Cd}\,M_2D}{\sqrt T}. \label{eq:horizon-risk}
\end{align}
For fixed batch sizes replace $8$ by $4$ in the final term.
\end{corollary}

Use Theorem~\ref{thm:paid} for adaptive batches with $T\ge8N$, and
Corollary~\ref{cor:public} for fixed batches with $T\ge4N$.
Put $\rho=d^{1/4}D/N>0$. The choice
$\gamma=\min\{\gammaMax,\rho\}$ is positive and satisfies
\begin{equation}
 \frac1\gamma
 =\max\left\{\frac1{\gammaMax},\frac1\rho\right\}
 \le\frac1{\gammaMax}+\frac1\rho,
 \qquad \gamma\le\rho.
 \label{eq:horizon-reciprocal-app}
\end{equation}
The smoothing inequalities in Lemma~\ref{lem:smoothing} give
$L_\gamma\le\sqrt dM/\gamma$. Hence
\begin{align*}
 \frac{16L_\gamma D^2}{N^2}+M\gamma
 &\le\frac{16\sqrt d\,MD^2}{N^2}
       \left(\frac1{\gammaMax}+\frac1\rho\right)+M\rho\\
 &=\frac{16\sqrt d\,MD^2}{\gammaMax N^2}
       +\frac{(16+1)d^{1/4}MD}{N}.
\end{align*}
The vector moment assumptions give
$\sigma_\gamma\le\sqrt{Cd}\,M_2$. Substituting these bounds into
\eqref{eq:paid-risk} proves \eqref{eq:horizon-risk}.
Using \eqref{eq:public-risk} instead replaces $8$ by $4$ in its
stochastic term.
\qed

\section{Deterministic fixed-objective trajectories}
\label{sec:formal-trajectory}

Let $E$ be a real Hilbert space and $Q\subset E$ be nonempty and convex.
Assume that the metric projections in the execution below exist. Let
$D,h_0>0$, $x^0\in Q$, and
$\norm{z-w}\le D$ for every $z,w\in Q$.
Let $\phi:E\to\mathbb R$ be convex on $Q$ and differentiable at the
points used below. Fix $N\ge0$, initialize $x_0=v_0=x^0$ and $H_0=h_0$,
and set $a_t=t$ and $A_t=t(t+1)/2$. For $t=1,\ldots,N$, use
\begin{align}
 y_t&=\frac{A_{t-1}}{A_t}x_{t-1}+\frac{a_t}{A_t}v_{t-1},
 \label{eq:formal-trajectory-y}\\*
 v_t&=\operatorname{Proj}_Q\left(v_{t-1}-\frac{a_t}{H_{t-1}}g_t^y\right),
 \qquad
 x_t=\frac{A_{t-1}}{A_t}x_{t-1}+\frac{a_t}{A_t}v_t,
 \label{eq:formal-trajectory-x}\\*
 \widehat\beta_t&=\ip{g_t^x-g_t^y}{x_t-y_t},
 \qquad
 H_t=H_{t-1}+
 \frac{\pos{A_t\widehat\beta_t-H_{t-1}\norm{v_t-v_{t-1}}^2/2}}
 {D^2+\norm{v_t-v_{t-1}}^2/2}.
 \label{eq:formal-trajectory-H}
\end{align}
The vectors $g_t^y,g_t^x$ may have arbitrary deterministic errors
$e_t^y=g_t^y-\nabla\phi(y_t)$ and
$e_t^x=g_t^x-\nabla\phi(x_t)$. No distributional assumption is used here.

\begin{theorem}[Fixed-objective trajectory potential]
\label{thm:formal-trajectory-potential}
For every $u\in Q$, the execution
\eqref{eq:formal-trajectory-y}--\eqref{eq:formal-trajectory-H} satisfies
$x_t,v_t,y_t\in Q$ at their defined indices and $H_t\ge h_0>0$.
Moreover,
\begin{align}
 A_N[\phi(x_N)-\phi(u)]+\frac{H_N}{2}\norm{u-v_N}^2
 \le{}&\frac{h_0}{2}\norm{u-x^0}^2+2D^2(H_N-h_0)
       +\sum_{t=1}^N Z_t(u),
 \label{eq:formal-trajectory-potential}\\
 Z_t(u)={}&a_t\ip{e_t^y}{u-v_{t-1}}
           -a_t\ip{e_t^x}{v_t-v_{t-1}}.
 \label{eq:formal-trajectory-charge}
\end{align}
\end{theorem}

\begin{proof}
The weights in both centers are nonnegative and sum to one. Starting
from $x_0=v_0\in Q$, convexity and projection therefore imply all stated
memberships. The coefficient denominator is strictly positive and its
numerator is nonnegative, so induction gives $H_t\ge h_0>0$.

Write $r_t=\norm{v_t-v_{t-1}}$ and $\delta_t=H_t-H_{t-1}\ge0$.
Projection optimality gives
\[
 a_t\ip{g_t^y}{v_t-u}\le\frac{H_{t-1}}2
 \left(\norm{u-v_{t-1}}^2-\norm{u-v_t}^2-r_t^2\right).
\]
The two center definitions imply
\[
 A_{t-1}(x_{t-1}-y_t)+a_t(u-y_t)=a_t(u-v_{t-1}),
 \qquad A_t(x_t-y_t)=a_t(v_t-v_{t-1}).
\]
The convex gradient supporting inequalities at $y_t$ with
$x_{t-1},u$, and at $x_t$ with $y_t$, consequently yield
\[
 A_t\phi(x_t)-A_{t-1}\phi(x_{t-1})-a_t\phi(u)
 \le a_t\ip{g_t^y}{v_t-u}+A_t\widehat\beta_t+Z_t(u).
\]
The scalar recurrence has the exact post-update identity
\[
 \delta_tD^2=\pos{A_t\widehat\beta_t-H_tr_t^2/2}.
\]
To verify it, the nonpositive-numerator case has $\delta_t=0$;
otherwise multiply the update by its denominator and use
$H_t=H_{t-1}+\delta_t$.
Since $r_t,\norm{u-v_t}\le D$, the resulting curvature remainder obeys
\[
 A_t\widehat\beta_t-\frac{H_t r_t^2}{2}
 +\frac{\delta_t}{2}\left(\norm{u-v_t}^2+r_t^2\right)
 \le 2D^2\delta_t.
\]
Combining these inequalities gives
\begin{align*}
 A_t[\phi(x_t)-\phi(u)]+\frac{H_t}{2}\norm{u-v_t}^2
 \le{}&A_{t-1}[\phi(x_{t-1})-\phi(u)]
 +\frac{H_{t-1}}2\norm{u-v_{t-1}}^2
 +2D^2\delta_t+Z_t(u).
\end{align*}
Summation, $A_0=0$, $v_0=x^0$, and $H_0=h_0$ prove
\eqref{eq:formal-trajectory-potential}.
\end{proof}

\begin{corollary}[Exact-gradient accelerated comparator bound]
\label{cor:formal-trajectory-rate}
In addition, let $L\ge0$ and assume
$\norm{\nabla\phi(z)-\nabla\phi(w)}\le L\norm{z-w}$ for
$z,w\in Q$. If $g_t^y=\nabla\phi(y_t)$ and
$g_t^x=\nabla\phi(x_t)$, then
$H_t\le\max\{h_0,4L\}$. For $N\ge1$ and every $u\in Q$,
\[
 \phi(x_N)-\phi(u)\le
 \frac{h_0\norm{u-x^0}^2/2+8LD^2}{A_N}.
\]
\end{corollary}

\begin{proof}
Cauchy--Schwarz and the gradient Lipschitz condition give
$\widehat\beta_t\le L\norm{x_t-y_t}^2$.
The center identity implies
$A_t\norm{x_t-y_t}=a_t r_t$ and
$a_t^2/A_t=2t/(t+1)\le2$, hence
$A_t\widehat\beta_t\le2Lr_t^2$.
If $H_{t-1}\ge4L$, the update numerator vanishes.
If $H_{t-1}<4L$, its positive part is at most
$(4L-H_{t-1})r_t^2/2$, which gives
$H_t-H_{t-1}\le4L-H_{t-1}$.
Thus $H_t\le\max\{H_{t-1},4L\}$, and induction proves the cap.
In particular $H_N-h_0\le4L$. Exact gradients make every
$Z_t(u)$ zero. Drop the nonnegative terminal term in
\eqref{eq:formal-trajectory-potential}, substitute the coefficient cap,
and divide by $A_N>0$.
\end{proof}

\section{Conditional-expectation trajectory potential}
\label{sec:stochastic-trajectory}

Consider the fixed-objective execution
\eqref{eq:formal-trajectory-y}--\eqref{eq:formal-trajectory-H} on a
probability space $(\Omega,\mathcal F,\mathbb P)$. The proximal set,
population target $\phi$, initial point, diameter bound and initial
coefficient are deterministic. Every realization satisfies the displayed
projection and gradient-error assumptions. The errors may have general
probability laws. For each $t$, let $\mathcal F_t^y,\mathcal F_t^x$
be sub-$\sigma$-algebras of $\mathcal F$. The vector $v_{t-1}$ is
measurable with respect to $\mathcal F_t^y$, and $v_t-v_{t-1}$ is
measurable with respect to $\mathcal F_t^x$. Assume that the vector
errors are integrable and satisfy
\begin{equation}
 \E[e_t^y\mid\mathcal F_t^y]=0,\qquad
 \E[e_t^x\mid\mathcal F_t^x]=0.
 \label{eq:stochastic-trajectory-centering}
\end{equation}
Strong measurability is sufficient; almost-everywhere strong
measurability gives the same conclusions. The second history includes
the first update and the projected displacement.

\begin{theorem}[Expected fixed-objective comparator potential]
\label{thm:stochastic-trajectory-potential}
Let $Q$ be convex with diameter at most $D>0$, let $h_0>0$, and let
$\phi$ be convex on $Q$ and differentiable at the points used in the
execution. Fix $x^0,u\in Q$ and $N\ge1$. Under the preceding trajectory,
predictability and conditional-mean assumptions, additionally assume that
$H_N$ and $\phi(x_N)-\phi(u)$ are integrable. Then all defined
$x_t,v_t,y_t$ belong to $Q$, $H_t\ge h_0$, and
\begin{equation}
 \E[\phi(x_N)-\phi(u)]\le
 \frac{h_0\norm{u-x^0}^2/2+2D^2(\E H_N-h_0)}{A_N}.
 \label{eq:stochastic-trajectory-expected}
\end{equation}
\end{theorem}

\begin{proof}
The membership and coefficient conclusions follow for each realization
from Theorem~\ref{thm:formal-trajectory-potential}. Consequently both
$u-v_{t-1}$ and $v_t-v_{t-1}$ have norm at most $D$.
Cauchy--Schwarz and integrability of the vector errors make both
error inner products integrable. The conditional-expectation pull-out
identity for the continuous bilinear inner product gives
\begin{align*}
 \E[\ip{e_t^y}{u-v_{t-1}}\mid\mathcal F_t^y]
 &=\ip{\E[e_t^y\mid\mathcal F_t^y]}{u-v_{t-1}}=0,\\
 \E[\ip{e_t^x}{v_t-v_{t-1}}\mid\mathcal F_t^x]
 &=\ip{\E[e_t^x\mid\mathcal F_t^x]}{v_t-v_{t-1}}=0.
\end{align*}
The two identities use their respective histories. Taking expectations
and multiplying by the deterministic weights yields $\E Z_t(u)=0$
for the charge in \eqref{eq:formal-trajectory-charge}.

Drop the nonnegative terminal distance in
\eqref{eq:formal-trajectory-potential}. Its remaining terms are
integrable by the assumptions and the preceding bounds. Take
expectations, use the finite sum of zero-mean charges, and divide by
$A_N>0$ to obtain \eqref{eq:stochastic-trajectory-expected}.
\end{proof}

The conditional-mean assumptions specify the required oracle behavior.
The result covers general probability measures and the actual
accelerated centers and projections for a fixed target. Sampling laws
for two-value records and changes in the smoothing target require
their additional moment and target-drift arguments.

\section{Diagnostic costs and batch guarantees}
\label{sec:app-batches}

\subsection{Proof of Lemma~\ref{lem:paid-accounting}}
\label{subsec:proof-paid-accounting}

\begin{lemma}[Adaptive batch sample counts]\label{lem:paid-accounting}
Algorithm~\ref{alg:scalar} uses exactly $T$ pairs and $2T$
scalar values, $4N$ sequential sampling rounds, and $N$ projections. Moreover, $m_{\mathrm{sur}}\ge \mathsf U/4$ and $B_t^\iota\ge B_t^0\ge \mathsf U t/[4N(N+1)]$ for $\iota\in\{y,x\}$. At a fixed center, $\E[\widehat v_j\mid\text{history}]$ equals the single-record variance. 
\end{lemma}

Set $\mathsf U=T-4N\ge4N$. The reserve satisfies
\begin{align}
 2\sum_{t=1}^NB_t^0
 &\le2N+\frac{\mathsf U-2N}{N(N+1)}\sum_{t=1}^Nt
 =\frac{\mathsf U}{2}+N\le\frac{3\mathsf U}{4}.
 \label{eq:paid-reserve-cost-app}
\end{align}
Thus $m_{\mathrm{sur}}\ge \mathsf U/4$. For $j<2N-1$ the allocated extra is an
integer between zero and $m_j/2$; at $j=2N-1$ it equals the entire
remaining surplus. Consequently every $m_j$ is nonnegative,
$\sum_{j=1}^{2N}k_j=m_{\mathrm{sur}}$, and the update counts sum to
$\mathsf U$. There are two pilot records at each of $2N$ centers, so their
total cost is $4N$ pairs. The complete cost is exactly $T$ pairs and
$2T$ values. At each center the pilot stage precedes the update stage.
The $x$ center depends on the $y$ update. These are four sequential
stages per iteration and one projection.

From $1+\floor z\ge z$ and $\mathsf U-2N\ge \mathsf U/2$,
\begin{equation}
 B_t^\iota\ge B_t^0\ge\frac{\mathsf U t}{4N(N+1)},\qquad \iota\in\{y,x\}.
 \label{eq:paid-reserve-lower-app}
\end{equation}
Finally, condition on the selected center and radius and the previous
history. The two pilots are conditionally independent and identically
distributed. If their mean is $\mu$ and variance is
$V=\E\norm{G-\mu}^2$, independence gives
\[
 \E\bigl[\tfrac12\norm{G^1-G^2}^2\bigr]
 =\tfrac12(2V)=V.
\]
The claim applies to the observed law, including fixed value corruption.
\qed

\subsection{Proof of Theorem~\ref{thm:paid}}
\label{subsec:proof-paid}

Each count is fixed after its pilots and before the update records,
which use sampled functions and directions independent of that history.
Thus, conditionally on that history, the clean update mean is unbiased
and has variance at most $\sigma_\gamma^2/B_t^\iota$ for $\iota\in\{y,x\}$. The first update error
is known before the second update mean, so conditional centering removes
the cross term between their errors. Equation
\eqref{eq:paid-reserve-lower-app} therefore gives
\begin{align}
 W&=\sum_{t=1}^Nt^2\E\norm{e_t^x-e_t^y}^2 \nonumber\\
 &\le\sum_{t=1}^Nt^2\sigma_\gamma^2
       \E(1/B_t^y+1/B_t^x)
 \le\frac{16\sigma_\gamma^2A_N^2}{\mathsf U}.
 \label{eq:paid-energy-app}
\end{align}
Use Theorem~\ref{thm:potential} and Lemma~\ref{lem:growth} with
$u=x^\star$ and $\phi_t=f_\gamma$. The stochastic contribution is
\[
 \frac{2D\sqrt W}{A_N}
 \le\frac{8\sigma_\gamma D}{\sqrt{\mathsf U}}
 \le\frac{8\sqrt2\sigma_\gamma D}{\sqrt T},
\]
since $\mathsf U\ge T/2$. The deterministic curvature contribution is
$8L_\gamma D^2/A_N\le16L_\gamma D^2/N^2$, and the comparator
smoothing error is at most $M\gamma$. This proves
\eqref{eq:paid-risk}. Neither inequality used the accuracy of an
individual pilot. All random batch choices are covered by the reserve.
\qed

\subsection{Proof of Theorem~\ref{thm:corruption}}
\label{subsec:proof-corruption}

\begin{theorem}[Bounded value corruption]\label{thm:corruption}
Suppose the clean single-record variance is at most $\sigma_\gamma^2$ and
$\norm{\delta}_\infty\le\Delta$. With
$c_d=d\E|[e]_1|\le\sqrt d$, where $[e]_1$ denotes the first
coordinate of $e$, Algorithm~\ref{alg:scalar} satisfies
\begin{equation}
 \E[f(x_N)-f(x^\star)]
 \le \frac{h_0R^2}{2A_N}+\frac{16L_\gamma D^2}{N^2}+M\gamma
 +\frac{16\sigma_\gamma D}{\sqrt T}
 +\frac{6c_dD\Delta}{\gamma}+\frac{16dD\Delta}{\gamma\sqrt T}.
 \label{eq:corruption-risk}
\end{equation}
All batch-selection and update costs remain as in Lemma~\ref{lem:paid-accounting}. 
\end{theorem}

When $T\ge d$, the last two terms are at most $22\sqrt d\,D\Delta/\gamma$. For $M>0$, choosing $\gamma\le\gammaMax$ proportional to $\varepsilon/M$ gives the sufficient noise
scale $\Delta$ proportional to
$\varepsilon^2/(DM\sqrt d)$ once the clean terms meet the target.

We separate the bias from the centered sampling error. The corruption
calculation in the proof of Lemma~\ref{lem:oracle-moments} requires only
the clean record's mean and variance and $\norm{\delta}_\infty\le\Delta$.
It bounds the bias of each observed record by $b=c_d\Delta/\gamma$
and its centered second moment by
\begin{equation}
 V=2\sigma_\gamma^2+\frac{2d^2\Delta^2}{\gamma^2}.
 \label{eq:corrupt-single-var-app}
\end{equation}
Averaging the $B_t^\iota$ update records preserves the bias and divides the
conditional centered second moment by $B_t^\iota$, for $\iota\in\{y,x\}$.
The conditioning includes the diagnostic observations used to choose
that batch size.

Write $e_t^\iota=b_t^\iota+\zeta_t^\iota$, where $\norm{b_t^\iota}\le b$ and
$\E[\zeta_t^\iota\mid\mathcal F_t^\iota]=0$. The $x$ history contains
$\zeta_t^y$, so conditional centering of $\zeta_t^x$ gives
$\E\ip{\zeta_t^x}{\zeta_t^y}=0$. Consequently, the minimum batch sizes
in Lemma~\ref{lem:paid-accounting} imply
\begin{align}
 W_\zeta
 &:=\sum_{t=1}^Nt^2\E\norm{\zeta_t^x-\zeta_t^y}^2\nonumber\\
 &\le V\sum_{t=1}^Nt^2\E(1/B_t^y+1/B_t^x)
 \le\frac{16VA_N^2}{\mathsf U}.
 \label{eq:corrupt-energy-app}
\end{align}
The bias $b_t^x$ may depend on $\zeta_t^y$; the argument requires no
cancellation between these vectors.

Let $W=\sum_{t=1}^Nt^2\E\norm{e_t^x-e_t^y}^2$.
The potential bias contribution is at most $2Db$ after division by
$A_N$. Applying Theorem~\ref{thm:potential} and Lemma~\ref{lem:growth}
gives
\begin{align}
 \E[f(x_N)-f(x^\star)]
 &\le\frac{h_0R^2}{2A_N}+\frac{8L_\gamma D^2}{A_N}
       +M\gamma+2Db+\frac{2D\sqrt W}{A_N}.
 \label{eq:corrupt-before-constants-app}
\end{align}
To control $W$, regard the weighted sequences
$(t(b_t^x-b_t^y))_{t=1}^N$ and
$(t(\zeta_t^x-\zeta_t^y))_{t=1}^N$ as random vectors in the product
Euclidean space. Minkowski's inequality in $L_2$, followed by
$\norm{b_t^x-b_t^y}\le2b$ and \eqref{eq:corrupt-energy-app}, yields
\begin{equation}
 \sqrt W\le
 2b\sqrt{\sum_{t=1}^Nt^2}+\sqrt{W_\zeta}
 \le2b\sqrt{\sum_{t=1}^Nt^2}
       +\frac{4A_N\sqrt V}{\sqrt{\mathsf U}}.
 \label{eq:corrupt-minkowski-app}
\end{equation}
Substitution into \eqref{eq:corrupt-before-constants-app} bounds the
combined bias contribution by
\[
 Db\left(2+\frac{4\sqrt{\sum_{t=1}^Nt^2}}{A_N}\right)\le6Db,
\]
since $\sum_{t=1}^Nt^2\le(\sum_{t=1}^Nt)^2=A_N^2$.
The centered-error contribution is $8D\sqrt V/\sqrt{\mathsf U}$.
Using $\mathsf U\ge T/2$ and
$\sqrt V\le\sqrt2\sigma_\gamma+\sqrt2d\Delta/\gamma$ bounds it by
$16\sigma_\gamma D/\sqrt T+16dD\Delta/(\gamma\sqrt T)$.
Finally, $8L_\gamma D^2/A_N\le16L_\gamma D^2/N^2$.
This proves \eqref{eq:corruption-risk}. The sample and stage counts in
Lemma~\ref{lem:paid-accounting} are deterministic identities and remain
valid for corrupted diagnostic observations.
\qed

\section{Conditional variance and adaptive batch sizes}
\label{sec:batch-variance}

We bound the variance contribution along the trajectory generated by the
scalar method. We then quantify a limitation of its adaptive batch rule under
the original two-value oracle. The allocation comparison concerns variance
energy at the exposed centers.

Minimum batch sizes preserve the error bound when variance estimates are inaccurate. For a fixed target, Theorem~\ref{thm:batch-variance-actual} expresses the bound through conditional variances at the actual centers. Proposition~\ref{prop:batch-variance-paid-limit} quantifies how much the allocation can lose relative to redistributing its surplus at those same centers. This variance comparison fixes the centers because changing batch sizes can change the trajectory.

\subsection{The variance energy of the actual trajectory}
\label{subsec:batch-variance-actual}

Fix a convex $L$-smooth target $\phi$ on a compact convex set $Q$ of
diameter at most $D>0$, and use the scalar recurrence
\begin{align}
 a_t&=t,\quad A_t=t(t+1)/2,\quad A_0=0,\quad H_0=h_0>0,
 \nonumber\\
 y_t&=(A_{t-1}x_{t-1}+a_tv_{t-1})/A_t,
 \nonumber\\
 v_t&=\Pi_Q(v_{t-1}-a_tg_t^y/H_{t-1}),\qquad
 x_t=(A_{t-1}x_{t-1}+a_tv_t)/A_t,
 \nonumber\\
 r_t&=\|v_t-v_{t-1}\|,\qquad
 \widehat\beta_t=\langle g_t^x-g_t^y,x_t-y_t\rangle,
 \nonumber\\
 H_t&=H_{t-1}+
 \frac{[A_t\widehat\beta_t-H_{t-1}r_t^2/2]_+}
 {D^2+r_t^2/2}.
 \label{eq:batch-variance-scalar}
\end{align}
Here $x_0=v_0=x^0\in Q$. At stage $j=2t-1$ the center is $y_t$,
and at stage $j=2t$ it is $x_t$. Write $\omega_{2t-1}=\omega_{2t}=t$.
Let $\mathcal F_j$ contain the history before the current pilots,
including the exposed center $z_j$ and radius, and define
\begin{equation}
 V_j=\mathbb E[\|G_j-\nabla\phi(z_j)\|^2\mid\mathcal F_j].
 \label{eq:batch-variance-predictable}
\end{equation}
The center and radius remain fixed during the pilots and the update
batch. The integer count $B_j\ge1$ is chosen after its pilots and before
sampling the functions and directions used in the update records,
independently of that history. Their conditional mean and
variance, given this enlarged history, remain $\nabla\phi(z_j)$ and
$V_j$. Thus the pilots can affect the count without changing the
frozen record law. We assume integrability of the displayed second
moments.

\begin{theorem}[Expected risk in terms of conditional variance]
\label{thm:batch-variance-actual}
Let $e_t^y=g_t^y-\nabla\phi(y_t)$, $e_t^x=g_t^x-\nabla\phi(x_t)$, $A=A_N$, and let $u\in Q$ be fixed.
Then
\begin{equation}
 W:=\sum_{t=1}^Nt^2\mathbb E\|e_t^x-e_t^y\|^2
 =\mathbb E\sum_{j=1}^{2N}\frac{\omega_j^2V_j}{B_j}.
 \label{eq:batch-variance-exact}
\end{equation}
If $f\le\phi$ on $Q$ and $R_u=\|u-x^0\|$, then
\begin{align}
 \mathbb E[f(x_N)-f(u)]
 &\le\frac{h_0R_u^2}{2A}
 +\frac{2D^2(4L-h_0)_+}{A}
 +\phi(u)-f(u)+\frac{2D\sqrt W}{A}.
 \label{eq:batch-variance-risk}
\end{align}
For the adaptive batch rule with $T\ge8N$, $\mathsf U=T-4N$, and the reserve
\begin{equation}
 b_t=1+\left\lfloor
 \frac{(\mathsf U-2N)t}{2N(N+1)}
 \right\rfloor,
 \label{eq:batch-variance-reserve}
\end{equation}
define
\begin{equation}
 \overline{\mathcal V}_N
 =\frac1{2A}\mathbb E\sum_{j=1}^{2N}\omega_jV_j.
 \label{eq:batch-variance-average}
\end{equation}
Its actual counts satisfy
\begin{equation}
 W\le\frac{16A^2\overline{\mathcal V}_N}{\mathsf U}.
 \label{eq:batch-variance-paid-energy}
\end{equation}
Consequently the variance term in \eqref{eq:batch-variance-risk} is at
most $8D\sqrt{\overline{\mathcal V}_N/\mathsf U}$.
The complete execution uses exactly $T$ sampled functions and $2T$ values,
uses $4N$ sequential sampling stages, and performs $N$ projections.
\end{theorem}

\begin{proof}
Condition on the enlarged history after the current pilots. Its count
is fixed, and conditional independence and centering of the update
records imply
$\mathbb E[\|e_j\|^2\mid\text{enlarged history}]=V_j/B_j$.
Before the $x$ update, $e_t^y$ is known. Conditional centering of the
$x$ update therefore gives
$\mathbb E\langle e_t^x,e_t^y\rangle=0$. Expand each square and take
expectations to obtain \eqref{eq:batch-variance-exact}.

The recurrence \eqref{eq:batch-variance-scalar} is the fixed-target
execution of Theorem~\ref{thm:stochastic-trajectory-potential}.
Its two conditional-mean assumptions hold with the respective histories
before each update batch, including the current pilots and count.
Lemma~\ref{lem:growth}, with $L_\circ=L$, gives
\[
 H_N-h_0\le(4L-h_0)_+
 +\frac1D\sqrt{\sum_{t=1}^Nt^2\|e_t^x-e_t^y\|^2}.
\]
The assumed second moments make this upper bound integrable.
The target is continuous on compact $Q$, so its objective differences
are also integrable. Thus all assumptions of
Theorem~\ref{thm:stochastic-trajectory-potential} are satisfied.
Apply that theorem, use Jensen's inequality and
\eqref{eq:batch-variance-exact}, and then use $f\le\phi$ to obtain
\eqref{eq:batch-variance-risk}.

For the adaptive batch rule, $\mathsf U\ge4N$ and
$1+\lfloor z\rfloor\ge z$ imply
$B_j\ge b_t\ge\mathsf U t/[4N(N+1)]=\mathsf U \omega_j/(8A)$.
Substitution into \eqref{eq:batch-variance-exact} gives
\[
 W\le\frac{8A}{\mathsf U}\mathbb E\sum_j\omega_jV_j
 =\frac{16A^2\overline{\mathcal V}_N}{\mathsf U}.
\]
The baseline total satisfies
$2\sum_tb_t\le\mathsf U/2+N\le3\mathsf U/4$.
All extras are nonnegative, and the remaining surplus is allocated
at final $y$; hence the update counts sum to $\mathsf U$.
Two pilots at each of $2N$ centers require $4N$ sampled functions, giving exactly
$T$ sampled functions and $2T$ values. Each iteration has a $y$ pilot stage, a $y$
update stage, an $x$ pilot stage, an $x$ update stage, and one projection.
\end{proof}

\begin{corollary}[Fixed batch sizes]
\label{cor:batch-variance-public}
With the fixed batch sizes
$B_t^y=B_t^x=1+\lfloor(T-2N)t/[N(N+1)]\rfloor$ and $T\ge4N$,
the same actual-trajectory identity satisfies
\[
 W\le\frac{8A^2\overline{\mathcal V}_N}{T}.
\]
The variance term in \eqref{eq:batch-variance-risk} is at most
$4\sqrt2D\sqrt{\overline{\mathcal V}_N/T}$.
The execution uses at most $T$ sampled functions and $2T$ values, $2N$ sampling
stages, and $N$ projections.
\end{corollary}

\begin{proof}
There are no pilots, so the proof of
\eqref{eq:batch-variance-exact} applies with an empty pilot history.
The floor satisfies
$B_j\ge T\omega_j/[2N(N+1)]=T\omega_j/(4A)$.
Substitution gives
$W\le(4A/T)\mathbb E\sum_j\omega_jV_j
=8A^2\overline{\mathcal V}_N/T$.
The sum of the counts is at most
$2N+2(T-2N)\sum_tt/[N(N+1)]=T$.
Each record uses two values; each iteration has two update sampling rounds
and one projection.
\end{proof}

The average \eqref{eq:batch-variance-average} is defined at the actual
centers and conditional record laws. It can be substituted into the risk
bound without replacing those centers by another allocation's trajectory.
It is a population-dependent analysis quantity; the theorem does not
give a data-derived upper confidence bound for the variance under finite second moments.

\subsection{A limitation of the adaptive batch rule under the original oracle}
\label{subsec:batch-variance-paid-limit}

Use the adaptive batch rule with the reserve \eqref{eq:batch-variance-reserve}
and surplus $S_{\mathrm{sur}}=\mathsf U-2\sum_tb_t$.
Its pilot statistic is
$\widehat v_j=\|G_j^1-G_j^2\|^2/2$, its allocation signal is
$\lambda_j=\omega_j\sqrt{\widehat v_j}$, and its cumulative signal is
$\Lambda_j=\sum_{i\le j}\lambda_i$. For $j<2N-1$ it allocates
\[
 k_j=\left\lfloor m_j
       \min\{1/2,\lambda_j/(s_0+\Lambda_j)\}\right\rfloor,
 \qquad B_j=b_{t(j)}+k_j.
\]
It allocates all remaining surplus at final $y$ and none at final $x$.
Here $s_0>0$ and $m_j$ is the remaining surplus.
For the variance sequence of its actual trajectory, define
\begin{equation}
 \mathcal C_\star^{\mathrm{tr}}(v)
 =\min_{\substack{n_j\ge b_{t(j)}\ {\rm integer}\\
                  \sum_jn_j=\mathsf U}}
       \sum_jv_j/n_j,\qquad v_j=\omega_j^2V_j.
 \label{eq:batch-variance-trace-comparator}
\end{equation}
This retrospective allocation benchmark fixes the observed centers and
redistributes only the adaptive policy's surplus. It does not assert that
an optimizer with different counts would expose the same centers.

\begin{proposition}[Exact worst-case ratio of the adaptive batch rule]
\label{prop:batch-variance-paid-limit}
For $N\ge2$, $T\ge8N$, and positive $h_0,s_0$, put
$\mathfrak R_{\mathrm{batch}}=1+S_{\mathrm{sur}}/b_1$.
Every actual trajectory satisfies
\begin{equation}
 \sum_jv_j/B_j
 \le\mathfrak R_{\mathrm{batch}}\mathcal C_\star^{\mathrm{tr}}(v).
 \label{eq:batch-variance-paid-upper}
\end{equation}
There are convex
affine populations on a compact one-dimensional domain, a global iid
sample-function law with a finite-second-moment Lipschitz envelope,
and a legal fixed-radius two-value record such that
\begin{equation}
 \frac{\mathbb E\sum_jv_j/B_j}
      {\mathbb E\mathcal C_\star^{\mathrm{tr}}(v)}
 \ge\frac{\mathfrak R_{\mathrm{batch}}}{2}
 \ge\frac12+\min\left\{\frac{\mathsf U}{16},
                         \frac{N(N+1)}8\right\}.
 \label{eq:batch-variance-paid-loss}
\end{equation}
There is a sequence of these laws, with the same population and a
uniformly bounded second-moment envelope, for which
\begin{equation}
 \lim_{p\downarrow0}
 \frac{\mathbb E\sum_jv_j/B_j}
      {\mathbb E\mathcal C_\star^{\mathrm{tr}}(v)}
 =\mathfrak R_{\mathrm{batch}}.
 \label{eq:batch-variance-paid-limit}
\end{equation}
Thus $\mathfrak R_{\mathrm{batch}}$ is the exact supremum of the ratio
in this oracle class, allowing the compact domain and legal radius
to vary in the matching construction. It satisfies
\begin{equation}
 1+\min\left\{\frac{\mathsf U}8,\frac{N(N+1)}4\right\}
 \le\mathfrak R_{\mathrm{batch}}
 \le1+\min\{\mathsf U,4N(N+1)\}.
 \label{eq:batch-variance-paid-scale}
\end{equation}
It equals $1+2N$ when $T=8N$, and converges to $1+N(N+1)$ as
$\mathsf U\to\infty$ at fixed $N$.
The numerator is exactly $W$ in \eqref{eq:batch-variance-exact}.
\end{proposition}

\begin{proof}
Every feasible trace-comparator count satisfies
$n_j\le b_{t(j)}+S_{\mathrm{sur}}$, and $B_j\ge b_{t(j)}\ge b_1$.
Consequently $n_j/B_j\le1+S_{\mathrm{sur}}/b_1$.
Multiply by $v_j/n_j$, sum, and minimize over comparator counts to
obtain \eqref{eq:batch-variance-paid-upper}.

For the matching example, choose $m>0$, put $q_0=m/h_0$, and let
$Q=[0,D]$ with $D>q_0A_N$.
The population is $f(x)=-mx$, whose ball smoothing equals itself.
On a history of records all equal to $-m$, all pilot differences
vanish, the coefficient stays $h_0$, and no surplus is allocated
before final $y$. Projection is inactive. Directly from the scalar
step and averaging recurrence,
\begin{equation}
 v_t=q_0A_t,\qquad
 x_t=q_0(3t^2+7t+2)/12,\qquad
 y_t=x_t-\frac{2q_0t}{t+1}.
 \label{eq:batch-variance-affine-trajectory}
\end{equation}
Indeed, $v_t-v_{t-1}=q_0t$ and
$A_tx_t=\sum_{s\le t}sv_s$; the square and cube sums give the formula
for $x_t$, and $A_t(x_t-y_t)=t(v_t-v_{t-1})$ gives $y_t$.
For $t\ge3$,
\[
 y_t-x_{t-1}=\frac{q_0(3t-1)(t-2)}{6(t+1)}>0,
\]
while $x_t-y_t>0$, $y_1=0$, and $y_2=x_1=q_0$.
Thus the common centers have order
$y_1<x_1=y_2<x_2<y_3<x_3<\cdots<y_N<x_N$.
The initial $y$ center $c=y_1=x^0=0$ is unique among these centers.
Every other common center is at distance at least $q_0$ from $c$.
Choose a legal fixed $\gamma>0$ with $8\gamma\le q_0$.
Function values below are globally defined, so the required evaluation
domain is available.

For $r=x-c$ define
\[
 \psi_c(x)=
 \begin{cases}
 r,& |r|\le2\gamma,\\
 4\gamma-r,&2\gamma\le r\le4\gamma,\\
 -4\gamma-r,&-4\gamma\le r\le-2\gamma,\\
 0,&|r|\ge4\gamma.
 \end{cases}
\]
The formulas agree at their endpoints and have slopes of magnitude
at most one. Hence $\psi_c$ is globally $1$-Lipschitz.
Its symmetric secant at $c$ has slope one. Every other common center
has both radius-$\gamma$ endpoints outside the support, so its secant
has slope zero.

Fix $V>0$ and let the independently sampled variable $Z_p$ equal zero with probability
$1-p$ and equal $\pm\sqrt{V/p}$ with probability $p/2$ each.
Define the single time-independent sample function
\[
 F(x,Z_p)=-mx+Z_p\psi_c(x).
\]
Its mean is $f$, and its Lipschitz envelope satisfies
$\mathbb E(m+|Z_p|)^2\le2m^2+2V$ uniformly in $p$.
Population convexity is satisfied, and samplewise convexity is not
required. The one-dimensional symmetric record uses one sampled function
at exactly two points and equals
\[
 G_\gamma(z;Z_p)=-m+Z_ph(z),\qquad
 h(z)=\frac{\psi_c(z+\gamma)-\psi_c(z-\gamma)}{2\gamma}.
\]
It is independent of the sign of the spherical direction.
Its conditional mean is $-m$ and its conditional variance is
$Vh(z)^2\le V$ at every center.

Let $E_0$ be the event that all $T$ samples have $Z_p=0$, and put
$q_p=(1-p)^T$. On this event the complete trajectory and allocation
equal the common history. The only center on that history with
nonzero conditional variance is $c=y_1$; its weight is one.
Consequently
\begin{equation}
 \mathbb E\sum_jv_j/B_j\ge q_pV/b_1.
 \label{eq:batch-variance-loss-numerator}
\end{equation}
On $E_0$, the trace comparator places all surplus at $c$, so its
value is $V/(b_1+S_{\mathrm{sur}})$.
On every trajectory, a feasible redistribution can only improve on
the baseline cost. Since $V_j\le V$ and $b_t\ge\mathsf Ut/(8A_N)$,
\[
 \mathcal C_\star^{\mathrm{tr}}
 \le\sum_j \omega_j^2V_j/b_{t(j)}
 \le16VA_N^2/\mathsf U.
\]
It follows that
\begin{align}
 \mathbb E\mathcal C_\star^{\mathrm{tr}}
 &\le\frac{q_pV}{b_1+S_{\mathrm{sur}}}
 +(1-q_p)\frac{16VA_N^2}{\mathsf U}.
 \label{eq:batch-variance-loss-denominator}
\end{align}
Choose $p>0$ with $pT\le1/(32A_N^2)$.
Then $1-q_p\le pT$ and $q_p\ge1/2$. The second term in
\eqref{eq:batch-variance-loss-denominator} is at most
$V/(2\mathsf U)$, which is at most the first term since
$b_1+S_{\mathrm{sur}}\le\mathsf U$.
Dividing \eqref{eq:batch-variance-loss-numerator} by
\eqref{eq:batch-variance-loss-denominator} proves the first inequality
of \eqref{eq:batch-variance-paid-loss}.
The baseline total is at most
$\mathsf U/2+N\le3\mathsf U/4$, so $S_{\mathrm{sur}}\ge\mathsf U/4$.
Also $b_1\le1+\mathsf U/[2N(N+1)]$, and hence
\[
 \frac{S_{\mathrm{sur}}}{b_1}
 \ge\frac{\mathsf U}{4+2\mathsf U/[N(N+1)]}
 \ge\min\left\{\frac{\mathsf U}8,\frac{N(N+1)}4\right\}.
\]
The reserve lower bound gives $b_1\ge\mathsf U/[4N(N+1)]$.
Together with $b_1\ge1$ and $S_{\mathrm{sur}}\le\mathsf U$, this
also proves the upper bound in \eqref{eq:batch-variance-paid-scale}.
Letting $p\downarrow0$ in the numerator and denominator bounds gives
a lower limit of at least $\mathfrak R_{\mathrm{batch}}$.
Equation \eqref{eq:batch-variance-paid-upper} bounds the upper limit
by the same value, proving \eqref{eq:batch-variance-paid-limit}.
The population and uniform second-moment envelope stay fixed.
When $\mathsf U=4N$, every $b_t=1$ and $S_{\mathrm{sur}}=2N$.
For fixed $N$ as $\mathsf U\to\infty$,
$S_{\mathrm{sur}}/\mathsf U\to1/2$ and
$b_1/\mathsf U\to1/[2N(N+1)]$, giving the final limit.
\end{proof}

The predictable variance cost can be positive on the event where
all realized update errors vanish. The equality to $W$ is the
unconditional identity \eqref{eq:batch-variance-exact}.
The proposition establishes a limitation of the variance allocation achieved by
the specified adaptive batch rule. It does not provide a lower bound on convex
optimization risk or a minimax impossibility for adaptive batching.

\section{The price of joint adaptation}\label{sec:app-obstructions}
The rare-sample construction of \citet[Section~1.1 and Theorem~3]{CarmonHinder2024} explains a basic difficulty in learning distance and noise simultaneously: an unobserved sample can change both the optimizer and the noise scale. The original-distance constraints in Appendices~\ref{sec:universal-distance-frontier} and~\ref{subsec:prestep-rare-curvature} apply this transcript-coupling argument to the two-value oracle. We first record the corresponding limitation for variance estimates from finitely many preliminary samples.

\begin{proposition}[A finite sample cannot bound an unknown variance uniformly]\label{prop:variance-certificate}
Fix $k<\infty$ and $\alpha\in(0,1)$. No statistic that is finite almost surely on the all-zero sample, including a randomized statistic, can be an upper confidence bound with coverage $1-\alpha$ for the variance of every real distribution having finite second moment, using $k$ independent observations.
\end{proposition}
\begin{proof}
Let $U$ denote the proposed bound. On the all-zero sample choose a finite $K$ with $p_0=\Prb(U\le K)>\alpha$; such a $K$ exists by almost-sure finiteness. Choose $p\in(0,1)$ small enough that $(1-p)^k p_0>\alpha$. Consider a variable equal to zero with probability $1-p$ and $A$ with probability $p$. Its variance is $p(1-p)A^2$, which is finite and exceeds $K$ for sufficiently large finite $A$. The probability of observing only zero and producing $U\le K$ is $(1-p)^k p_0>\alpha$. On this event the proposed bound is smaller than the true variance. This contradicts the required coverage.
\end{proof}
The batch policy uses preliminary differences to allocate the samples remaining after the minimum batch sizes have been reserved. Its guarantee follows from these minimum batch sizes and conditional centering of update estimates formed from subsequent independent samples; it does not require an upper confidence bound on the variance.

\section{Data-selected smoothing and trajectory curvature}
\label{sec:app-smoothing}

We use the same-sample two-value oracle and the scalar accelerated updates from the main text. All norms in this appendix are Euclidean. We write $a_t=t$, $A_t=t(t+1)/2$, $r_t=\|v_t-v_{t-1}\|$, and $H_0=h_0>0$. The known diameter bound is $D$. The radius is chosen before the first vector batch of an iteration and remains fixed for its first- and second-center vector batches. Each batch uses functions and directions sampled independently of the preceding history. Each sampled function is used in one record; diagnostics and update means use separate batches.

\subsection{Radius selection and complete resource accounting}
\label{subsec:smoothing-accounting}

Let $\gamma_{\max}$ be the permitted evaluation distance from $Q$ and define
\begin{equation}
 \bar\gamma=\min\{\gamma_{\max},d^{1/4}D/N\},\qquad
 \underline\gamma=\bar\gamma/2,\qquad \gamma_0=\bar\gamma.
 \label{eq:protected-radius-range}
\end{equation}
At the beginning of iteration $t$, the point $y_t$ is determined by the preceding state. Draw a uniform ball direction $U_t$ and a sampled function independently of each other and of the preceding history, and observe
\begin{equation}
 V_t=\widetilde F(y_t+\gamma_{t-1}U_t,\xi_t)
       -\widetilde F(y_t+\gamma_{t-1}U_t/2,\xi_t).
 \label{eq:radius-diagnostic-app}
\end{equation}
For a reference $s_\gamma>0$ chosen before sampling, set
\begin{equation}
 q_t=\frac{[V_t]_+}{s_\gamma+[V_t]_+},\qquad
 \gamma_t=\max\{\underline\gamma,\gamma_{t-1}/(1+q_t)\}.
 \label{eq:radius-policy-app}
\end{equation}
Both endpoints in \eqref{eq:radius-diagnostic-app} lie in the evaluation domain. They use exactly two values of the same sampled function. A positive observed scale difference requests a smaller radius; the lower radius bound chosen before sampling protects the subsequent finite-difference estimates.

Take $T\ge10N$ and reserve the $N$ radius pairs and $4N$ vector-pilot pairs. Put
\begin{equation}
 \mathsf U=T-5N,\qquad
 B_t^0=1+\left\lfloor\frac{(\mathsf U-2N)t}{2N(N+1)}\right\rfloor,
 \qquad m_{\mathrm{sur}}=\mathsf U-2\sum_{t=1}^N B_t^0.
 \label{eq:smoothing-budget-app}
\end{equation}
At each center $j=2t-1,2t$, after the radius has been selected, obtain two independent vector pilots $G_j^1,G_j^2$ at that center and radius. Starting from $\Lambda_0=0$ and $m_1=m_{\mathrm{sur}}$, define
\[
 \widehat v_j=\tfrac12\|G_j^1-G_j^2\|^2,\qquad
 \lambda_j=t\sqrt{\widehat v_j},\qquad
 \Lambda_j=\Lambda_{j-1}+\lambda_j.
\]
For $s_0>0$ chosen before sampling, spend
\[
 k_j=\left\lfloor m_j\min\{1/2,\lambda_j/(s_0+\Lambda_j)\}\right\rfloor
 \quad(j<2N-1),\qquad
 k_{2N-1}=m_{2N-1},\qquad k_{2N}=0.
\]
Set $m_{j+1}=m_j-k_j$ and $B_j=B_t^0+k_j$, and obtain the update mean of size $B_j$. The final surplus is spent before the final projected step. The terminal second-center mean receives its protected count and is also charged.

\begin{lemma}[Sampling costs and monotone radii]
\label{lem:smoothing-accounting}
The specified execution uses exactly $T$ original pairs, $2T$ scalar values, $5N$ sequential stochastic stages, and $N$ projections. Its radii are nonincreasing and lie in $[\underline\gamma,\bar\gamma]$. Each clean update mean is conditionally unbiased for the gradient of its selected smoothed objective at its chosen center. Moreover,
\begin{equation}
 B_j\ge B_t^0\ge\frac{\mathsf U t}{4N(N+1)},\qquad \mathsf U\ge T/2.
 \label{eq:radius-batch-lower-app}
\end{equation}
\end{lemma}

\begin{proof}
Since $0\le q_t<1$, \eqref{eq:radius-policy-app} gives
$\underline\gamma\le\gamma_t\le\gamma_{t-1}\le\bar\gamma$ by induction. The resource condition gives $\mathsf U\ge T/2\ge5N\ge4N$. Summing the protected counts yields
\[
 2\sum_tB_t^0\le2N+(\mathsf U-2N)/2=\mathsf U/2+N\le3\mathsf U/4.
\]
Thus $m_{\mathrm{sur}}\ge \mathsf U/4\ge0$. Every surplus expenditure is an integer between zero and the remaining account. The account is exhausted at the last first center, so the update means consume $2\sum_t B_t^0+m_{\mathrm{sur}}=\mathsf U$ pairs. The $N$ radius diagnostics and $4N$ vector pilots give a total of $T$ pairs, each consisting of two scalar values.

The five successive stages in each iteration are the radius diagnostic, the first-center vector pilots, the first-center update mean, the second-center vector pilots, and the second-center update mean. Each stage uses functions and directions sampled independently of the completed history. Conditional on the completed radius and relevant vector diagnostics, the current center, radius, and count are fixed; the clean update mean is therefore conditionally unbiased for its smoothed gradient. The projected update uses one projection per iteration. Finally, $1+\lfloor z\rfloor\ge z$ and $\mathsf U-2N\ge \mathsf U/2$ prove \eqref{eq:radius-batch-lower-app}.
\end{proof}

The clean conditional mean of \eqref{eq:radius-diagnostic-app} is
$f_{\gamma_{t-1}}(y_t)-f_{\gamma_{t-1}/2}(y_t)$. For convex $f$, it lies between one half and one times the smoothing bias at $y_t$. This follows by applying even convexity to $r\mapsto\mathbb E_U f(y_t+rU)$: its value at $\gamma/2$ lies between its value at zero and the midpoint of its values at zero and $\gamma$. The diagnostic can therefore inform scale selection, while the theorem accounts separately for the smoothing bias at the comparator.

\subsection{Actual secant curvature}
\label{subsec:secant-growth}

For the current differentiable convex target $\phi_t$, write
$p_t^y=\nabla\phi_t(y_t)$, $p_t^x=\nabla\phi_t(x_t)$, and
$e_t^\iota=g_t^\iota-p_t^\iota$ for $\iota\in\{y,x\}$. Define
\begin{equation}
 \kappa_t=\begin{cases}
 A_t\langle p_t^x-p_t^y,x_t-y_t\rangle/r_t^2,&r_t>0,\\
 0,&r_t=0,
 \end{cases}
 \qquad
 \mathcal C_N=\sum_{t=1}^N\frac{r_t^2}{D^2}
                   [\kappa_t-H_{t-1}/4]_+,
 \label{eq:secant-charge-app}
\end{equation}
and $\mathcal E_N=\sum_{t=1}^Nt^2\|e_t^x-e_t^y\|^2$. Convexity gives $\kappa_t\ge0$. These quantities describe the actual execution; they are used only in the analysis.

\begin{lemma}[Secant-curvature growth]
\label{lem:secant-growth}
For every execution,
\begin{align}
 H_N&\le\mathcal C_N+
       \sqrt{\mathcal C_N^2+h_0^2+2\mathcal E_N/D^2},
       \label{eq:secant-growth-exact}\\
 H_N-h_0&\le2\mathcal C_N+\sqrt{2\mathcal E_N}/D.
       \label{eq:secant-growth-simple}
\end{align}
If all targets are $L_*$-smooth, then also
\begin{equation}
 H_N-h_0\le4L_*+\sqrt{\mathcal E_N}/D.
 \label{eq:secant-growth-uniform}
\end{equation}
\end{lemma}

\begin{proof}
Put $\epsilon_t=\|e_t^x-e_t^y\|$. The scalar update's exact balance identity is
\[
 (H_t-H_{t-1})D^2=
 [A_t\langle g_t^x-g_t^y,x_t-y_t\rangle-H_t r_t^2/2]_+.
\]
If the update numerator is nonpositive, both sides are zero. Otherwise multiplication of the increment by $D^2+r_t^2/2$ and rearrangement prove the identity. Since $A_t(x_t-y_t)=t(v_t-v_{t-1})$, its right-hand side is at most
$[(\kappa_t-H_t/2)r_t^2+t\epsilon_t r_t]_+$.
Let $q=[\kappa_t-H_{t-1}/4]_+$. Monotonicity $H_t\ge H_{t-1}\ge h_0$ gives
$\kappa_t-H_t/2\le q-H_t/4$. Thus
\begin{align*}
 (H_t-H_{t-1})D^2
 &\le q r_t^2+[-H_t r_t^2/4+t\epsilon_t r_t]_+\\
 &\le q r_t^2+t^2\epsilon_t^2/H_t.
\end{align*}
The last inequality maximizes a scalar quadratic over $r_t\ge0$. Multiplying by $(H_t+H_{t-1})/D^2$, using $H_t+H_{t-1}\le2H_t\le2H_N$, and summing gives
\[
 H_N^2-h_0^2\le2H_N\mathcal C_N+2\mathcal E_N/D^2.
\]
Solving this quadratic inequality proves \eqref{eq:secant-growth-exact}. Subadditivity of the square root gives \eqref{eq:secant-growth-simple}.

For the uniform bound, apply Lemma~\ref{lem:growth} to the selected
targets with common bound $L_*$. It gives
$\pos{H_N-4L_*}^2\le\pos{h_0-4L_*}^2+\mathcal E_N/D^2$,
and hence $H_N-h_0\le4L_*+\sqrt{\mathcal E_N}/D$.
This proves \eqref{eq:secant-growth-uniform}.
\end{proof}

\subsection{Proof of the learned-radius guarantee}
\label{subsec:proof-learned-radius}

\begingroup
\interdisplaylinepenalty=10000
\paragraph{Proof of Theorem~\ref{thm:learned-radius}.}
Fix $u\in Q$ and write $R_u=\|u-x_0\|$. Let
\[
 \mathcal S_N(u)=\frac1{A_N}\sum_{t=1}^Nt\,
                 \mathbb E[f_{\gamma_t}(u)-f(u)],\qquad
 W=\mathbb E\mathcal E_N.
\]
The radius policy is nonincreasing, so convex symmetric ball smoothing gives $f\le f_{\gamma_t}$ and $f_{\gamma_{t+1}}\le f_{\gamma_t}$ pointwise on $Q$. Lemma~\ref{lem:smoothing-accounting} supplies the required pre-batch measurability and conditional centering of the update errors. The potential from Theorem~\ref{thm:potential}, proved in Appendix~\ref{subsec:proof-potential}, consequently yields
\begin{align}
 \mathbb E[f(x_N)-f(u)]
 \le{}&\frac{h_0R_u^2}{2A_N}
 +\frac{2D^2\mathbb E(H_N-h_0)}{A_N}
 +\mathcal S_N(u)\nonumber\\
 &+\frac D{A_N}\sum_t t\,
       \mathbb E[\|b_t^y\|+\|b_t^x\|],
 \label{eq:learned-radius-potential-app}
\end{align}
where $b_t^\iota$, for $\iota\in\{y,x\}$, is the mean vector bias conditional on that center's completed diagnostics. Martingale cancellation takes place before division by the deterministic $A_N$.

For the symmetric sphere update estimate, bounded spatial corruption gives
\[
 \|b_t^\iota\|\le b:=c_d\Delta/\underline\gamma,
 \qquad
 V:=2\sigma_0^2+2d^2\Delta^2/\underline\gamma^2
\]
as a bound on the centered single-record variance. The bias bound follows by taking a scalar product with a unit vector and bounding the corruption difference by $2\Delta$:
$\|(d/(2\gamma_t))\mathbb E_e[(\delta(z+\gamma_t e)-\delta(z-\gamma_t e))e]\|
\le c_d\Delta/\gamma_t$.
For the variance bound, the corruption vector has norm at most $d\Delta/\gamma_t$; centering cannot increase its second moment, and $\|a+b\|^2\le2\|a\|^2+2\|b\|^2$ proves the asserted $V$.

Every update mean has centered conditional variance at most
$V/B_j$. Conditional centering of the second error removes its inner
product with the first centered error. Their accumulated variance sum
is at most $16VA_N^2/\mathsf U$. The difference of the two mean biases
has norm at most $2b$. Minkowski's inequality in the weighted product
$L_2$ space therefore gives
\[
 \sqrt W\le2b\sqrt{\sum_tt^2}+4A_N\sqrt{V/\mathsf U}.
\]
This argument permits the second bias to depend on the first mean;
only the centered second error is canceled.
Lemma~\ref{lem:growth} and Jensen's inequality bound the curvature
term in \eqref{eq:learned-radius-potential-app} by
$8L_*D^2+2D\sqrt W$. Since $\sqrt{\sum_tt^2}\le A_N$,
the bias and stochastic contributions after division by $A_N$ are
at most
\[
 2Db+2D\sqrt W/A_N
 \le6Db+8D\sqrt{V/\mathsf U}
 \le6Db+16\sigma_0D/\sqrt T
           +16dD\Delta/(\underline\gamma\sqrt T).
\]
We used $\mathsf U\ge T/2$ and
$\sqrt V\le\sqrt2\sigma_0+\sqrt2d\Delta/\underline\gamma$.
Thus
\begin{align}
 \E[f(x_N)-f(u)]\le{}&\frac{h_0R_u^2}{2A_N}
 +\frac{8L_*D^2}{A_N}+\mathcal S_N(u)\nonumber\\
 &+\frac{16\sigma_0D}{\sqrt T}
 +\frac{6c_dD\Delta}{\underline\gamma}
 +\frac{16dD\Delta}{\underline\gamma\sqrt T}.
 \label{eq:learned-radius-uniform-app}
\end{align}
The inequality $A_N\ge N^2/2$ gives the stated curvature term.
Lipschitzness and Jensen give
$0\le f_{\gamma_t}(u)-f(u)\le M\gamma_t\le M\bar\gamma$.
Lemma~\ref{lem:smoothing} supplies
$L_*=\sqrt dM/\underline\gamma$.

When $\Delta=0$, both biases vanish and the clean variance is at most
$\sigma_0^2/B_j$. The same conditional centering gives
$W\le16\sigma_0^2A_N^2/\mathsf U$.
Consequently $2D\sqrt W/A_N\le8\sqrt2\sigma_0D/\sqrt T$.
Together with the same curvature and comparator terms this proves
the clean guarantee.\hfill$\square$
\endgroup

\subsection{Nonsmooth and smooth specializations}
\label{subsec:smoothing-specializations}

\paragraph{Nonsmooth dimensional order.}
If the clean variance satisfies $\sigma_0^2\le CdM_2^2$ uniformly over the protected radii, then
\[
 \mathbb E[f(x_N)-f(u)]\le
 \frac{h_0R_u^2}{2A_N}
 +\frac{33d^{1/4}MD}{N}
 +\frac{32\sqrt d MD^2}{\gamma_{\max}N^2}
 +\frac{8\sqrt{2Cd}M_2D}{\sqrt T}
\]
in the clean case. If $T\ge d$ and $\gamma_{\max}\ge d^{1/4}D/N$, the corruption terms in Theorem~\ref{thm:learned-radius} sum to at most $44d^{1/4}\Delta N$.
To obtain these constants, the uniform curvature term is
$32\sqrt d MD^2/(\bar\gamma N^2)$ because $\underline\gamma=\bar\gamma/2$. Use
$1/\bar\gamma\le1/\gamma_{\max}+N/(d^{1/4}D)$ and
$\mathcal S_N(u)\le M d^{1/4}D/N$, and substitute the variance envelope into the clean bound. For corruption, $c_d\le\sqrt d$ and $T\ge d$ bound its sum by
$22\sqrt dD\Delta/\underline\gamma$. In the stated evaluation domain regime,
$\underline\gamma=d^{1/4}D/(2N)$, which gives the asserted constant.

\begin{proposition}[Smooth populations]
\label{prop:protected-smoothing-smooth}
If $f$ is differentiable and $L$-smooth on the convex evaluation domain, every $f_\gamma$ is $L$-smooth on $Q$ and
\[
 0\le f_\gamma(u)-f(u)\le L\nu_d\gamma^2/2,
 \qquad \nu_d=d/(d+2).
\]
Thus the learned-radius guarantee applies with
\[
 L_*=\min\{L,\sqrt dM/\underline\gamma\},\qquad
 \mathcal S_N(u)\le\frac{L\nu_d}{2A_N}
                       \sum_t t\,\mathbb E\gamma_t^2.
\]
In particular, its clean smooth bound is
\[
 \mathbb E[f(x_N)-f(u)]\le
 \frac{h_0R_u^2}{2A_N}+\frac{16LD^2}{N^2}
 +\frac{L\nu_d}{2A_N}\sum_t t\,\mathbb E\gamma_t^2
 +\frac{8\sqrt2\sigma_0D}{\sqrt T}.
\]
\end{proposition}

\begin{proof}
Differentiation of the bounded-ball average gives
$\nabla f_\gamma(x)=\mathbb E_U\nabla f(x+\gamma U)$. Smoothness on the evaluation domain implies
\[
 \|\nabla f_\gamma(x)-\nabla f_\gamma(y)\|
 \le\mathbb E\|\nabla f(x+\gamma U)-\nabla f(y+\gamma U)\|
 \le L\|x-y\|.
\]
The smooth upper model gives
$f(u+\gamma U)\le f(u)+\gamma\langle\nabla f(u),U\rangle
+L\gamma^2\|U\|^2/2$. Symmetry makes its linear expectation zero, and Jensen supplies the lower bound. Since $\mathbb E\|U\|^2=\nu_d$, the quadratic bias follows. Average this bound over the selected radii and substitute it and the valid smoothness bound in \eqref{eq:learned-radius-uniform-app}. No higher derivative is required.
\end{proof}

The protected interval provides the uniform upper bound
$A_N^{-1}\sum_t t\mathbb E\gamma_t^2\le\bar\gamma^2\le\sqrt dD^2/N^2$.
The resulting uniform smooth deterministic term is therefore
$(16+\nu_d\sqrt d/2)LD^2/N^2$. The actual-radius statement gives a smaller term whenever the selected weighted squared radius is smaller. The forward-secant controller in Section~\ref{sec:protected-main} permits a wider radius range and removes this extra dimension factor while retaining the nonsmooth $d^{1/4}$ iteration dependence. The protected-interval variant here requires only vector moment bounds; its radius diagnostic needs no separate moment bound.

\subsection{Proof of Theorem~\ref{thm:trajectory}}
\label{subsec:proof-trajectory}
Use the pathwise inequality \eqref{eq:potential}. The integrable total
errors and bounded displacements make its error inner products
integrable. Conditioning at the two histories removes the centered
parts separately, so the expected gap is bounded by the startup term,
$\mathcal S_N(u)$, $2D^2\E(H_N-h_0)/A_N$, and the stated sum of bias norms.
Put $W=\E\mathcal E_N$. Lemma~\ref{lem:secant-growth} gives
\[
 \E(H_N-h_0)\le2\E\mathcal C_N+\sqrt{2W}/D,
 \qquad
 \E(H_N-h_0)\le4L_*+\sqrt W/D.
\]
Jensen bounds the expected square roots. Both inequalities hold
for this same execution and have identical startup, comparator and
bias terms. Multiply each coefficient bound by $2D^2/A_N$ and take
the smaller full numerator. This proves
\eqref{eq:trajectory-risk-main}. Finiteness of $W$ makes $H_N$
integrable, while $0\le\kappa_t\le2L_*$ gives
$0\le\mathcal C_N\le2NL_*$. The comparator is fixed and $A_N$
is deterministic, so no future-dependent normalization is used.\qed

\subsection{Proof of Corollary~\ref{cor:weighted-variance}}
\label{subsec:proof-weighted-variance}
By conditional unbiasedness for the current target, both biases
vanish. Conditional on each completed pre-update history, the average
of $B_t^\iota$ independent records, for $\iota\in\{y,x\}$, has
squared error expectation exactly $V_t^\iota/B_t^\iota$.
The first error $e_t^y$ is measurable before the second update mean
is drawn. Since $\E\mathcal E_N<\infty$ and
$e_t^y=-\E[e_t^x-e_t^y\mid\mathcal F_t^x]$, conditional Jensen
implies that $e_t^y$, and hence $e_t^x$, is square-integrable.
Thus conditional centering of the second mean gives
$\E\langle e_t^x,e_t^y\rangle=0$. Expanding the squared error
difference and taking expectations yields
\[
 \E\norm{e_t^x-e_t^y}^2
 =\E\left(V_t^x/B_t^x+V_t^y/B_t^y\right).
\]
Multiply by $t^2$ and sum. This proves
$\E\mathcal E_N=\mathcal V_N$ even when each radius, count and
conditional variance is random. The assumed conditional unbiasedness
therefore allows substitution in Theorem~\ref{thm:trajectory}, proving the
risk claim. The second center's variance need not be independent of
the first error; only conditional centering of its update mean was
used.\qed

\subsection{Proof of Corollary~\ref{cor:resources}}
\label{subsec:proof-resources}
Take, for example, $c=2^{12}$ in \eqref{eq:resource-conditions}.
Then $T\ge10N$, $T\ge d$, and
$N\ge d^{1/4}D/\gammaMax$, so
$\bar\gamma=d^{1/4}D/N$ and
$\underline\gamma=d^{1/4}D/(2N)$. The robust guarantee of
Theorem~\ref{thm:learned-radius} and the variance envelope give
\[
 \E[f(x_N)-f(x^\star)]
 \le \frac{h_0R^2}{N^2}
 +\frac{33d^{1/4}MD}{N}
 +\frac{16\sqrt{Cd}M_2D}{\sqrt T}
 +44d^{1/4}\Delta N.
\]
We used $A_N\ge N^2/2$, $c_d\le\sqrt d$, and $T\ge d$.
The four displayed terms are respectively at most
$\varepsilon/c^2$, $33\varepsilon/c$,
$16\varepsilon/\sqrt c$, and $44\varepsilon/c$.
Since $c=4096$, their sum is less than $\varepsilon$.
The exact stage and scalar-value counts follow from
Lemma~\ref{lem:smoothing-accounting}. The parameters in these
sufficient inequalities are used only to assess the supplied budgets;
no parameter estimate or stopping certificate is invoked.\qed

\section{Fixed-threshold radial executions}
\label{sec:extensions}

\subsection{An affine lower model of the original objective}
The smoothing parameter supplies an additional derivative that can be
estimated within the two-value interface. Define
\begin{equation}
 h_\gamma(y)=\gamma\,\partial_\gamma f_\gamma(y)
 =d\bigl[\E_e f(y+\gamma e)-f_\gamma(y)\bigr].
 \label{eq:radial-main-definition}
\end{equation}
Joint convexity of $(y,\gamma)\mapsto f_\gamma(y)$ gives the global model
\begin{equation}
 f(u)\ge f_\gamma(y)+\ip{\nabla f_\gamma(y)}{u-y}-h_\gamma(y).
 \label{eq:radial-main-model}
\end{equation}
Lemma~\ref{lem:radial-affine} proves this statement for Lipschitz populations,
including at the boundary of the evaluation domain. It also proves
$0\le h_\gamma(y)\le M\gamma$ and
$h_{\lambda\gamma}(y)\le\lambda h_\gamma(y)$ for $0<\lambda\le1$.
For a smooth population, $h_\gamma(y)\le Ld\gamma^2/(d+2)$.

Let $e$ be uniform on the sphere and let $r$ have density $dr^{d-1}$
on $[0,1]$, independently of $e$ and the sampled function. One radial record is
\begin{equation}
 J_\gamma(y)=d\bigl[\widetilde F(y+\gamma e,\xi)
                  -\widetilde F(y+\gamma r e,\xi)\bigr].
 \label{eq:radial-main-record}
\end{equation}
The endpoints are nonsymmetric and use exactly two values of one sampled function.
Its clean mean equals $h_\gamma(y)$, and its second moment is at most
$2M_2^2\gamma^2$ under the radial regularity conditions of
Lemma~\ref{lem:radial-moment}. A random Lipschitz envelope on the evaluation domain with
second moment $M_2^2$ suffices; that lemma also gives the weaker
sample-gradient and segment conditions. The vector variance assumption
alone does not imply this different scalar moment bound.

The potential now replaces comparator smoothing bias by the local charge
$\sum_t t h_{\gamma_t}(y_t)$
(Theorem~\ref{thm:radial-potential}). Each original comparator satisfies
\eqref{eq:radial-main-model}; no distance to a minimizer of the smoothed
objective enters. If radii increase, the exact additional target drift is
$\sum_{t<N}A_t[f_{\gamma_{t+1}}(x_t)-f_{\gamma_t}(x_t)]$.
Our shrinking rule makes this drift nonpositive.

The key weighted inequality follows by combining ordinary convexity
at $x_{t-1}$ with \eqref{eq:radial-main-model} at $u$:
\begin{align}
 &A_{t-1}\phi_t(x_{t-1})+a_tf(u)\nonumber\\
 &\quad\ge A_t\phi_t(y_t)
 +a_t\ip{\nabla\phi_t(y_t)}{u-v_{t-1}}
 -a_th_{\gamma_t}(y_t),
 \label{eq:radial-main-weighted-model}
\end{align}
where $\phi_t=f_{\gamma_t}$. The original comparator therefore enters
before any target values are telescoped. Projection and the observed
curvature balance then give
\begin{align}
 A_N[f(x_N)-f(u)]\le{}&h_0R_u^2/2+2D^2(H_N-h_0)\nonumber\\
 &+\sum_ta_th_{\gamma_t}(y_t)+\sum_tZ_t(u),
 \label{eq:radial-main-potential}
\end{align}
for nonincreasing radii. Here $Z_t(u)$ is the stochastic error term of
Theorem~\ref{thm:potential}. Theorem~\ref{thm:radial-potential}
proves the full pathwise statement, including arbitrary radius changes
and their signed drift. This is why a current-center radial diagnostic
can control original-objective risk without certifying smoothing bias
at an unknown solution.

\subsection{A persistent shrinking-radius execution}
Fix $T\ge16N$, a positive threshold $\tau$ chosen before sampling, and
$\Gamma\in(0,\gammaMax]$. Set
\begin{equation}
 B^{\rm rad}=\floor{T/(8N)},\quad
 \mathsf U=T-NB^{\rm rad}-4N,\quad\gamma_0=\Gamma.
 \label{eq:radial-main-budget}
\end{equation}
At the fixed first center $y_t$, average $B^{\rm rad}$ conditionally independent records
\eqref{eq:radial-main-record} using $\gamma_{t-1}$, obtaining
$\widehat h_t$. Select
\begin{equation}
 \gamma_t=\frac{\gamma_{t-1}}
                 {\max\{1,[\widehat h_t]_+/\tau\}}.
 \label{eq:radial-main-selection}
\end{equation}
Every vector pilot and update mean in that iteration uses this selected
radius. The update reserve is $\mathsf U$ in
\eqref{eq:reserve}--\eqref{eq:surplus}; all surplus decisions retain
\eqref{eq:pilot}--\eqref{eq:extras}. The radial records are disjoint from
all vector records, and radius selection precedes both update means.

\begin{algorithm}[t]
\caption{Scalar acceleration with observed radial smoothing}
\label{alg:radial}
\small
\begin{algorithmic}[1]
\Require $Q,x^0,D$, integers $N\ge1$, $T\ge16N$;
$h_0,s_0,\tau>0$, legal $\Gamma>0$
\State $x_0=v_0=x^0$, $H_0=h_0$, $\Lambda_0=0$, $\gamma_0=\Gamma$
\State Set $B^{\rm rad},\mathsf U$ by \eqref{eq:radial-main-budget}
\State Compute $B_t^0,m_1$ by \eqref{eq:reserve}--\eqref{eq:surplus}, using this $\mathsf U$
\For{$t=1,\ldots,N$}
 \State Set $a_t=t$, $A_t=t(t+1)/2$; compute $y_t$ by \eqref{eq:center-y}
 \State $\widehat h_t\gets$ mean of $B^{\rm rad}$ conditionally independent radial pairs at
 $(y_t,\gamma_{t-1})$
 \State Select $\gamma_t$ by \eqref{eq:radial-main-selection}
 \State $B_t^y\gets\PaidBatch(y_t,t,2t-1,\gamma_t)$
 \State $g_t^y\gets\FreshMean(y_t,B_t^y,\gamma_t)$
 \State Compute $v_t,x_t$ by \eqref{eq:step}
 \State $B_t^x\gets\PaidBatch(x_t,t,2t,\gamma_t)$
 \State $g_t^x\gets\FreshMean(x_t,B_t^x,\gamma_t)$
 \State Update $r_t,\widehat\beta_t,H_t$ by \eqref{eq:curvature}--\eqref{eq:coefficient}
\EndFor
\State \Return $x_N$
\end{algorithmic}
\end{algorithm}

The procedures in Algorithm~\ref{alg:radial} are fully specified in
Algorithm~\ref{alg:scalar}. The new execution retains one optimizer state
and has no restart. Its radius is positive and nonincreasing, with no
prescribed positive lower floor. Lemma~\ref{lem:radial-counts} proves
that the five stages per iteration spend exactly the supplied budget.

Put $\chi_d=2d^2/((d+1)(d+2))\le2$,
$\omega_{\mathrm{rad}}=\sqrt{\chi_d}M_2\sqrt{N/B^{\rm rad}}$,
$\eta_{\mathrm{rad}}=\sqrt{\chi_d}M_2\Gamma/\sqrt{B^{\rm rad}}$, and
$\mathfrak X_0=\Gamma^{-1}+(M+\omega_{\mathrm{rad}})/\tau$.
These are analysis quantities. Let $L_{\rm pop}=L$ when the population
is $L$-smooth, and $L_{\rm pop}=+\infty$ otherwise.

\begin{samepage}
\interdisplaylinepenalty=10000
\begin{theorem}[Observed radial smoothing]\label{thm:radial-clean}
Assume the uniform clean vector moment assumptions, the radial moment conditions
of Lemma~\ref{lem:radial-moment}, and integrability of the displayed
potential. For any $\tau,\Gamma$ in Algorithm~\ref{alg:radial} and fixed
$u\in Q$, its clean expected risk satisfies
\begin{align}
 \E[f(x_N)-f(u)]\le{}&\frac{h_0R_u^2}{2A_N}
 +\frac{8D^2}{A_N}\min\{L_{\rm pop},\sqrt dM\mathfrak X_0\}
 \nonumber\\
 &+\frac{8\sqrt2\sigma_0D}{\sqrt T}+\tau+\eta_{\mathrm{rad}}.
 \label{eq:radial-main-risk}
\end{align}
The method spends exactly $T$ pairs, $2T$ values and $5N$ sequential
sampling stages, with $N$ projections. The full proof is in
Appendix~\ref{subsec:proof-radial-clean}.
\end{theorem}
\end{samepage}

The same proof gives a refinement along the actual trajectory:
replace $\tau+\eta_{\mathrm{rad}}$ by
\[
 \frac1{A_N}\sum_ta_t\E h_{\gamma_t}(y_t)
\]
and use $\E\min\{L_{\rm pop},\sqrt dM/\gamma_N\}$ in place of the
minimum in \eqref{eq:radial-main-risk}. The inverse-radius bound is proved
from the conditional radial moments:
\begin{equation}
 \bigl(\E\gamma_N^{-2}\bigr)^{1/2}\le\mathfrak X_0.
 \label{eq:radial-main-inverse}
\end{equation}
Lemma~\ref{lem:radial-radius} derives this envelope from the exact
maximum recurrence of $1/\gamma_t$. It holds under second moments,
without independent errors across iterations or a supplied smallest radius.
Thus a large positive diagnostic may select a much smaller radius,
while the resulting curvature cost remains in the same bound.

A choice made before sampling that preserves the smooth order is
\begin{equation}
 \Gamma=\min\{\gammaMax,D/\sqrt N\},\qquad\tau=h_0D^2/N^2.
 \label{eq:radial-main-smooth-choices}
\end{equation}
Indeed, the smooth branch of Theorem~\ref{thm:radial-clean},
$A_N\ge N^2/2$ and
$\eta_{\rm rad}\le4\sqrt2M_2D/\sqrt T$ from
\eqref{eq:radial-public-moment-scales} give, when
$\sigma_0^2\le CdM_2^2$,
\[
 \E[f(x_N)-f(u)]\le\frac{h_0R_u^2}{2A_N}
 +\frac{(16L+h_0)D^2}{N^2}
 +(8\sqrt{2Cd}+4\sqrt2)\frac{M_2D}{\sqrt T}.
\]
This is a specialization of the arbitrary-threshold guarantee.

The same smooth order is available with a radius fixed before sampling.
For a clean oracle, take $\gamma=\min\{\gammaMax,D/N\}$ in
Algorithm~\ref{alg:scalar}, with $T\ge8N$ and
$\sigma_\gamma\le\sigma_0$.
Proposition~\ref{prop:protected-smoothing-smooth} gives $L_\gamma\le L$
and comparator bias at most $L\nu_d\gamma^2/2$, where
$\nu_d=d/(d+2)$. The proof of Theorem~\ref{thm:paid} retains this
actual bias before using $M\gamma$. Substituting $\gamma\le D/N$ yields
\begin{align}
 \E[f(x_N)-f(x^\star)]\le{}&\frac{h_0R^2}{2A_N}
 +(16+\nu_d/2)\frac{LD^2}{N^2}\nonumber\\
 &+\frac{8\sqrt2\sigma_0D}{\sqrt T}.
 \label{eq:fixed-smooth-comparison}
\end{align}
This execution spends exactly $T$ pairs in $4N$ stages and needs
no additional radial regularity. The radial construction supplies the
locally estimable original-objective affine model and permits
unrestricted observed shrinking, using $5N$ stages.

The nonsmooth envelope in \eqref{eq:radial-main-risk} displays the tradeoff
between local radial charge $\tau$ and curvature growth proportional
to $M^2/\tau$. Taking $\tau$ at order $d^{1/4}MD/N$ recovers the
nonsmooth deterministic order, with the explicit initial-radius correction in
Appendix~\ref{sec:app-radial}. This specialization uses the unknown $M$.
The smooth choice \eqref{eq:radial-main-smooth-choices} alone therefore
does not establish the optimal nonsmooth rate.

With spatial corruption, the inverse-radius moments remain bounded if
$2d\Delta<\tau$; Theorem~\ref{thm:radial-corruption} gives the complete
risk bound. Above the margin, fixed spatial corruption can cause
exponential inverse-radius growth even for a flat objective
(Proposition~\ref{prop:radial-collapse}); this is a stability limitation.
The protected-radius method in
Appendix~\ref{sec:protected-extensions} retains the earlier sufficient
allowance $\Delta\asymp\varepsilon^2/(\sqrt dMD)$ under its stated
budget and evaluation domain conditions. These are separate radius policies with
different corruption guarantees. Their common lower frontier is
Corollary~\ref{cor:noise-frontier}.

\section{Unrestricted shrinking radii and a radial affine model}
\label{sec:app-radial}

We retain the original scalar state, projection and curvature update.
The radius is selected after a separate radial diagnostic and before
either vector update mean in the iteration. A local radial derivative
provides an affine lower model for the original objective at every
comparator. This model permits a wider radius range while retaining
the original optimizer as comparator.

Throughout, $a_t=t$, $A_t=t(t+1)/2$, $H_0=h_0>0$ and
$R_u=\norm{u-x^0}$ for a fixed $u\in Q$. The geometry and population
conditions are those of Assumption~\ref{ass:geometry}.
All expectations and inner products below are assumed integrable.
The clean vector-record variance bound $\sigma_0^2$ is uniform over
all the positive radii used, rather than only a protected interval.

\subsection{The radial affine lower bound}
\label{subsec:radial-affine}

We use the extension convention of Lemma~\ref{lem:smoothing} for derivatives
at centers on the boundary of $Q$ and at $\gamma=\gammaMax$. Define
\begin{equation}
 \Phi(x,\gamma)=\E_{U\sim\mathrm{Unif}(\mathbb B_2^d)}
                         f(x+\gamma U),\qquad
 h_\gamma(x)=\gamma\partial_\gamma\Phi(x,\gamma),
 \label{eq:radial-function}
\end{equation}
and set $\mu_d=d/(d+1)$ and $\nu_d=d/(d+2)$.

\begin{lemma}[Joint smoothing and a radial affine model]
\label{lem:radial-affine}
For $y\in Q$ and $0<\gamma\le\gammaMax$, the derivative in
\eqref{eq:radial-function} exists, and
\begin{align}
 h_\gamma(y)
 &=d\bigl[\E_e f(y+\gamma e)-f_\gamma(y)\bigr],
 &0\le h_\gamma(y)&\le M\mu_d\gamma,
 \label{eq:radial-boundary}\\
 f(u)&\ge f_\gamma(y)+
       \ip{\nabla f_\gamma(y)}{u-y}-h_\gamma(y)
 &&(u\in Q),\label{eq:radial-affine}\\
 h_{\lambda\gamma}(y)&\le\lambda h_\gamma(y)
 &&(0<\lambda\le1).
 \label{eq:radial-shrink}
\end{align}
If $f$ is differentiable and $L$-smooth on the evaluation domain, then also
\begin{equation}
 h_\gamma(y)\le L\nu_d\gamma^2.
 \label{eq:radial-smooth-charge}
\end{equation}
\end{lemma}

\begingroup
\interdisplaylinepenalty=10000
\begin{proof}
Use the globally convex $M$-Lipschitz extension $\bar f$ constructed in
Lemma~\ref{lem:smoothing}. Its ball average defines $\Phi$ for all centers
and radii and agrees with the original smoothing on the legal domain.
All ball and sphere values in the asserted formulas belong to that domain.

For fixed $y$, let $q_y(s)=\E_e\bar f(y+se)$. This is continuous and
$M$-Lipschitz in $s$. Polar integration gives
\begin{equation}
 \Phi(y,\gamma)=d\gamma^{-d}\int_0^\gamma
                      s^{d-1}q_y(s)\,ds.
 \label{eq:radial-polar}
\end{equation}
The fundamental theorem of calculus applies to its continuous
integrand. Differentiating proves the equality in
\eqref{eq:radial-boundary} for every $\gamma>0$.
Lemma~\ref{lem:smoothing} gives the spatial derivative of $\Phi$.
Its sphere expression and the radial derivative just obtained are
continuous jointly in $x$ and positive $\gamma$: continuity of $\bar f$
gives uniform convergence of the integrands on compact parameter sets.
Thus $\Phi$ is continuously differentiable in the joint variables.
Both derivative formulas use only legal values when $x\in Q$ and
$0<\gamma\le\gammaMax$, so they are independent of the extension.

For each fixed $U$, the function $(x,\gamma)\mapsto\bar f(x+\gamma U)$
is convex. Its expectation is jointly convex and
$\Phi(u,0)=f(u)$. Apply its affine lower model at $(y,\gamma)$
to $(u,0)$ to obtain \eqref{eq:radial-affine}. The segment in
this application uses only points of the original evaluation domain.

The radial function $\gamma\mapsto\Phi(y,\gamma)$ is convex and
even. Its derivative on the positive half-line is nonnegative and
nondecreasing. Its Lipschitz constant is at most
$M\E\norm U=M\mu_d$, which proves the inequality in
\eqref{eq:radial-boundary}. Monotonicity of the derivative gives
\[
 h_{\lambda\gamma}(y)
 =\lambda\gamma\partial_\gamma\Phi(y,\lambda\gamma)
 \le\lambda\gamma\partial_\gamma\Phi(y,\gamma)
 =\lambda h_\gamma(y).
\]
This proves \eqref{eq:radial-shrink}.

In the smooth case, almost every sampled ball point is interior to the
original evaluation domain, as in Lemma~\ref{lem:smoothing}. Differentiation
under the integral therefore uses the original gradient, without requiring
a smooth extension. Symmetry gives
\begin{align*}
 \partial_\gamma\Phi(y,\gamma)
 &=\E\ip{\nabla f(y+\gamma U)}{U}\\
 &=\E\ip{\nabla f(y+\gamma U)-\nabla f(y)}{U}
 \le L\gamma\E\norm U^2=L\nu_d\gamma.
\end{align*}
Multiplication by $\gamma$ proves
\eqref{eq:radial-smooth-charge}, without a higher derivative.
\end{proof}
\endgroup

The inequality in \eqref{eq:radial-affine} concerns an affine
lower model of the original objective at all comparators. It does
not identify smoothing bias at the optimizer from a small
pointwise bias at the current center.

\subsection{Two-value radial records and their regularity}
\label{subsec:radial-moments}

The radial moment analysis requires a sample regularity condition
in addition to the vector variance assumption. A sufficient condition
is a random Lipschitz envelope on the evaluation domain $m(\xi)$ with
$\E m(\xi)^2\le M_2^2$. Another sufficient condition consists of
measurable locally Lipschitz samples on a neighborhood of the closed
evaluation domain, a measurable weak gradient with the pointwise envelope
$\E_\xi\norm{\nabla F(z,\xi)}^2\le M_2^2$ for $z$ in the interior
of the evaluation domain, and the integrability
conditions in Lemma~\ref{lem:oracle-moments}. In the latter setting,
the weak-gradient fundamental theorem must hold on almost every
sampled radial segment and its integrals must be integrable. These
conditions follow by polar Fubini from local Lipschitzness and the
stated moment envelope; they can equivalently be imposed explicitly
as the segment regularity assumption. Samples need not be convex.

Draw a uniform sphere direction $e$, a radius $r$ with density
$dr^{d-1}$ on $[0,1]$, and an independent sample $\xi$.
The clean radial record \eqref{eq:radial-estimate-main} is
\begin{equation}
 J_\gamma(y;e,r,\xi)
 =d[F(y+\gamma e,\xi)-F(y+\gamma re,\xi)].
 \label{eq:radial-record}
\end{equation}
The same sampled function is evaluated at both endpoints. Set
$\chi_d=2d^2/[(d+1)(d+2)]\le2$.

\begin{lemma}[Radial mean and second moment]
\label{lem:radial-moment}
Under either sufficient regularity condition above,
\begin{equation}
 \E J_\gamma(y)=h_\gamma(y),\qquad
 \E J_\gamma(y)^2\le\chi_dM_2^2\gamma^2.
 \label{eq:radial-record-moment}
\end{equation}
For a mean $\widehat h^{\,0}$ of $B$ clean records at a
center and radius fixed by the preceding history, with functions and
directions sampled independently of that history and of one another,
\begin{equation}
 \E[\widehat h^{\,0}\mid\mathcal F]=h_\gamma(y),\qquad
 \E[|\widehat h^{\,0}-h_\gamma(y)|^2\mid\mathcal F]
 \le\chi_dM_2^2\gamma^2/B.
 \label{eq:radial-mean-moment}
\end{equation}
The corruption contribution to a radial record or its mean has
absolute value at most $2d\Delta$.
\end{lemma}

\begin{proof}
The random point $re$ is uniform in the unit ball and independent
of $\xi$. Taking the expectation of the population difference in
\eqref{eq:radial-record} and applying
\eqref{eq:radial-boundary} gives its asserted mean. This uses the
difference, so a common random sample offset cancels.

The random Lipschitz envelope gives
$|J_\gamma|\le dm(\xi)\gamma(1-r)$. Under the weak-gradient
condition, the fundamental theorem on the radial segment gives
\[
 J_\gamma=d\gamma(1-r)\int_0^1
 \ip{\nabla F(y+\gamma[r+s(1-r)]e,\xi)}{e}\,ds.
\]
Jensen in this integral, followed by the gradient moment envelope
and Fubini, gives the same conditional squared bound
$d^2M_2^2\gamma^2(1-r)^2$. To justify the weak-gradient use,
local Lipschitzness makes every segment restriction absolutely
continuous. For almost every $\xi$, Rademacher differentiability
holds at almost every spatial point. Polar Fubini about each fixed
center shows that its weak gradient and the derivative along a
radial line agree for almost every direction and almost every
positive radius. Applying this statement to the sampled segment
gives the displayed fundamental theorem almost surely.
The moment envelope supplies the integrability used in Jensen and
Fubini. Finally,
\[
 \E(1-r)^2=1-\frac{2d}{d+1}+\frac d{d+2}
          =\frac2{(d+1)(d+2)}.
\]
This proves \eqref{eq:radial-record-moment}.
Conditional independence of the records gives the variance and mean of
their average; its variance is at most the single-record second
moment divided by $B$. This proves
\eqref{eq:radial-mean-moment}. Each corruption difference is at
most $2\Delta$, so multiplication by $d$ proves the last claim.
\end{proof}

\subsection{Construction and exact sampling cost}
\label{subsec:radial-construction}

Before sampling choose $N\ge1$, integer $T\ge16N$, threshold $\tau>0$ and
legal initial radius $0<\Gamma\le\gammaMax$. Define
\begin{equation}
 B^{\mathrm{rad}}=\left\lfloor\frac{T}{8N}\right\rfloor,
 \qquad\mathsf U=T-NB^{\mathrm{rad}}-4N,
 \qquad\gamma_0=\Gamma.
 \label{eq:radial-budget}
\end{equation}
After computing $y_t$ from the previous scalar state, obtain
$B^{\mathrm{rad}}$ conditionally independent observations of
\eqref{eq:radial-record} at radius $\gamma_{t-1}$ and center $y_t$.
Let $\widehat h_t$ be their mean, and set
\begin{equation}
 \gamma_t=\frac{\gamma_{t-1}}
                 {\max\{1,\pos{\widehat h_t}/\tau\}}.
 \label{eq:radial-policy}
\end{equation}
The selected radius is strictly positive, nonincreasing and at
most $\Gamma$. It is fixed before the vector update means at
both centers, which use the same objective $f_{\gamma_t}$.

Use the batch-selection procedures of Algorithm~\ref{alg:scalar} at this
radius with reserve $\mathsf U$ and fixed minimum batch sizes
\begin{equation}
 B_t^0=1+\left\lfloor
          \frac{(\mathsf U-2N)t}{2N(N+1)}\right\rfloor.
 \label{eq:radial-vector-floor}
\end{equation}
In each iteration, after the radial stage, the order is first-center
pilots, its update mean, the projected step and second center,
second-center pilots, and its update mean. The scalar
curvature update is unchanged. No diagnostic record is reused in
an update mean. The optimizer state and remaining sample budgets persist across iterations.

\begin{lemma}[Sampling costs of radial selection]
\label{lem:radial-counts}
The method uses exactly $T$ original pairs, $2T$ scalar values,
$5N$ sequential sampling stages and $N$ projections. It satisfies
\begin{equation}
 B^{\mathrm{rad}}\ge\frac{T}{16N},\qquad
 \mathsf U\ge\frac{5T}{8}\ge\frac T2,\qquad
 B_t^\iota\ge B_t^0\ge\frac{\mathsf U t}{4N(N+1)},\quad \iota\in\{y,x\}.
 \label{eq:radial-floors}
\end{equation}
\end{lemma}
\begin{proof}
Since $T/(8N)\ge2$, its floor is at least half its value. Also
$NB^{\mathrm{rad}}\le T/8$ and $4N\le T/4$, so
$\mathsf U\ge5T/8\ge10N$. The protected counts obey
\[
 2\sum_tB_t^0\le2N+(\mathsf U-2N)/2
 =\mathsf U/2+N\le3\mathsf U/4.
\]
Thus the surplus is nonnegative. The original adaptive allocation
exhausts this surplus at the last first center, so its update means
spend exactly $\mathsf U$ pairs. Its two vector pilots at each of
the two centers cost $4N$ pairs. The radial stages cost
$NB^{\mathrm{rad}}$ pairs. Their sum is $T$ by
\eqref{eq:radial-budget}. Each sampled function supplies two values.

The radial stage precedes the four vector stages just specified.
Each stage permits parallel sampling of all its pairs; the next
stage waits for the preceding observations. Thus there are $5N$
sequential stages and one projection per iteration. Finally,
$1+\lfloor z\rfloor\ge z$ and
$\mathsf U-2N\ge\mathsf U/2$ prove the last floor bound.
\end{proof}

\subsection{The local affine bound and target drift in the potential}
\label{subsec:radial-potential}

The next inequality holds for each realization of the vector estimates;
unbiasedness and moment bounds enter only when taking expectations.
Fix $N\ge1$ and $u\in Q$. For arbitrary radii
$0<\gamma_t\le\gammaMax$, selected before either vector update batch at
iteration $t$, let
$\phi_t=f_{\gamma_t}$, $e_t^y=g_t^y-\nabla\phi_t(y_t)$, $e_t^x=g_t^x-\nabla\phi_t(x_t)$ and
\begin{align}
 \mathcal E_N&=\sum_ta_t^2\norm{e_t^x-e_t^y}^2,\nonumber\\*
 Z_t(u)&=a_t\ip{e_t^y}{u-v_{t-1}}
              -a_t\ip{e_t^x}{v_t-v_{t-1}},
 \label{eq:radial-errors}\\*
 \mathcal D_N&=\sum_{t<N}A_t
          [f_{\gamma_{t+1}}(x_t)-f_{\gamma_t}(x_t)],\nonumber\\*
 \mathcal D_N^+&=\sum_{t<N}A_t
          \pos{f_{\gamma_{t+1}}(x_t)-f_{\gamma_t}(x_t)}.
 \label{eq:radial-drift}
\end{align}
Here $\sum_t$ means $\sum_{t=1}^N$, and $\sum_{t<N}$ means
$\sum_{t=1}^{N-1}$. The difference in $\mathcal D_N$ compares the two
targets at the same point $x_t$; it records the effect of changing the
radius. The quantity $\mathcal D_N^+$ retains only positive changes.

\begin{samepage}
\begin{theorem}[Local radial potential]
\label{thm:radial-potential}
The scalar projection and curvature update satisfy
\begin{align}
 A_N[f(x_N)-f(u)]\le{}&h_0R_u^2/2+2D^2(H_N-h_0)
 +\sum_ta_th_{\gamma_t}(y_t)+\sum_tZ_t(u)\nonumber\\
 &+\mathcal D_N-A_N[f_{\gamma_N}(x_N)-f(x_N)].
 \label{eq:radial-potential}
\end{align}
The last line can be replaced by $\mathcal D_N^+$.
For nonincreasing radii this term can be omitted. For arbitrary
radii,
\begin{align}
 \mathcal D_N^+&\le M\mu_d\sum_{t<N}A_t
                            \pos{\gamma_{t+1}-\gamma_t},
 \label{eq:radial-lipschitz-drift}\\
 \mathcal D_N^+&\le\frac{L\nu_d}{2}\sum_{t<N}A_t
                            \pos{\gamma_{t+1}^2-\gamma_t^2}
 \label{eq:radial-smooth-drift}
\end{align}
in the Lipschitz and smooth cases respectively.
\end{theorem}
\end{samepage}

\begin{proof}
Write $p_t^y=\nabla\phi_t(y_t)$, $p_t^x=\nabla\phi_t(x_t)$ and $\delta_t=H_t-H_{t-1}$.
We first derive an inequality for one update, then sum it while retaining
the changes in the smoothed target. The ordinary affine lower
model of $\phi_t$ at $y_t$, applied to $x_{t-1}$ with weight
$A_{t-1}$, and \eqref{eq:radial-affine}, applied to $u$ with
weight $a_t$, give
\[
 A_{t-1}\phi_t(x_{t-1})+a_tf(u)
 \ge A_t\phi_t(y_t)+a_t\ip{p_t^y}{u-v_{t-1}}
                                  -a_th_{\gamma_t}(y_t).
\]
Here $A_{t-1}x_{t-1}+a_tu-A_ty_t=a_t(u-v_{t-1})$.
Convexity at $x_t$ and the center identity give
\[
 A_t[\phi_t(x_t)-\phi_t(y_t)]
 \le a_t\ip{p_t^x}{v_t-v_{t-1}}.
\]
Projection optimality gives
\[
 a_t\ip{g_t^y}{v_t-u}\le\frac{H_{t-1}}2
 [\norm{u-v_{t-1}}^2-\norm{u-v_t}^2-r_t^2].
\]
The curvature balance gives
$\delta_tD^2=\pos{A_t\widehat\beta_t-H_tr_t^2/2}$,
with $\delta_t\ge0$. Since both $\norm{u-v_t}$ and $r_t$ are at most $D$,
the curvature remainder satisfies
\[
 A_t\widehat\beta_t-H_tr_t^2/2
 +\delta_t(\norm{u-v_t}^2+r_t^2)/2\le2D^2\delta_t.
\]
Substitute $p_t^\iota=g_t^\iota-e_t^\iota$ for $\iota\in\{y,x\}$ and combine these inequalities to obtain
\begin{align*}
 A_t\phi_t(x_t)+\frac{H_t}{2}\norm{u-v_t}^2
 \le{}&A_{t-1}\phi_t(x_{t-1})
       +\frac{H_{t-1}}2\norm{u-v_{t-1}}^2+a_tf(u)\\
 &+2D^2\delta_t+Z_t(u)+a_th_{\gamma_t}(y_t).
\end{align*}
Summing cancels the intermediate distance terms. For the target terms,
\begin{align*}
 \sum_{t=1}^N\bigl[A_t\phi_t(x_t)-A_{t-1}\phi_t(x_{t-1})\bigr]
 &=A_N\phi_N(x_N)-\sum_{t=1}^{N-1}A_t
             [\phi_{t+1}(x_t)-\phi_t(x_t)]\\
 &=A_N\phi_N(x_N)-\mathcal D_N.
\end{align*}
Use $A_0=0$, $v_0=x^0$, $\sum_ta_t=A_N$ and
$\sum_t\delta_t=H_N-h_0$. Discard the nonnegative terminal distance
term and write $\phi_N(x_N)=f(x_N)+[f_{\gamma_N}(x_N)-f(x_N)]$
to obtain \eqref{eq:radial-potential}.
The smoothing excess is nonnegative, and replacing the signed
drift by its positive part gives the stated simpler version.

Nonincreasing radii make the signed drift nonpositive by convex
evenness of the radial mean. To prove
\eqref{eq:radial-lipschitz-drift}, integrate
$\partial_\gamma\Phi\le M\mu_d$ on each increasing interval.
Integration of $\partial_\gamma\Phi\le L\nu_d\gamma$ gives
\eqref{eq:radial-smooth-drift}. Decreasing intervals have zero
positive contribution.
\end{proof}

For clean update means, condition on the history before each update batch.
The vector $u-v_{t-1}$ is known before the first mean, and
$v_t-v_{t-1}$ is known before the second. Each corresponding error has
conditional mean zero, so iterated expectation gives $\E Z_t(u)=0$.
The radius diagnostic precedes both batches; selecting the target from its
observations preserves these conditional cancellations.

\subsection{Shrinking and reciprocal-radius moments}
\label{subsec:radial-reciprocal}

Let $\widehat h_t^{\,0}$ be the mean of the clean parts of the
radial records actually drawn, and set
\begin{align}
 X_t&=1/\gamma_t,\qquad
 \mathsf s_t=\widehat h_t^{\,0}/\gamma_{t-1},\nonumber\\
 \omega_{\mathrm{rad}}&=\sqrt{\chi_d}\,M_2
                       \sqrt{N/B^{\mathrm{rad}}},\qquad
 \eta_{\mathrm{rad}}=\sqrt{\chi_d}\,M_2
                       \Gamma/\sqrt{B^{\mathrm{rad}}}.
 \label{eq:radial-moment-scales}
\end{align}
The clean parts and normalized slopes are analysis quantities;
the algorithm uses the observed $\widehat h_t$.

\begin{lemma}[Local charge and reciprocal moments]
\label{lem:radial-radius}
The rule \eqref{eq:radial-policy} satisfies
\begin{equation}
 \E h_{\gamma_t}(y_t)\le\tau+2d\Delta+\eta_{\mathrm{rad}}.
 \label{eq:radial-selected-charge}
\end{equation}
If $2d\Delta<\tau$, define
\begin{equation}
 \tau_- =\tau-2d\Delta,\qquad
 \mathfrak X=\Gamma^{-1}+(M+\omega_{\mathrm{rad}})/\tau_-.
 \label{eq:radial-reciprocal-scale}
\end{equation}
Then, pathwise and in expectation respectively,
\begin{align}
 X_N&\le\max\{\Gamma^{-1},
                    \max_{t\le N}\pos{\mathsf s_t}/\tau_-\},
 \label{eq:radial-reciprocal-path}\\
 \E X_N&\le\mathfrak X,\qquad
 (\E X_N^2)^{1/2}\le\mathfrak X.
 \label{eq:radial-reciprocal-moment}
\end{align}
For clean observations these statements hold for every $\tau>0$
with $\tau_- =\tau$. Moreover
\begin{equation}
 \omega_{\mathrm{rad}}\le4\sqrt2M_2N/\sqrt T,
 \qquad\eta_{\mathrm{rad}}\le
                   4\sqrt2M_2\Gamma\sqrt N/\sqrt T.
 \label{eq:radial-public-moment-scales}
\end{equation}
\end{lemma}

\begin{proof}
Condition on the history before the radial stage. Its center and
old radius are fixed. Write $h=h_{\gamma_{t-1}}(y_t)$ and
$\lambda=1/\max\{1,\pos{\widehat h_t}/\tau\}$.
Lemma~\ref{lem:radial-affine} gives
$h_{\gamma_t}(y_t)\le\lambda h$. Since $\lambda\le1$ and
$\lambda\pos{\widehat h_t}\le\tau$,
\[
 \lambda h\le\tau+|\widehat h_t-h|.
\]
Lemma~\ref{lem:radial-moment} and Cauchy--Schwarz bound its
conditional expectation by
$\tau+2d\Delta+\sqrt{\chi_d}M_2\gamma_{t-1}/
\sqrt{B^{\mathrm{rad}}}$. Using $\gamma_{t-1}\le\Gamma$
proves \eqref{eq:radial-selected-charge}.

The reciprocal update is exactly
$X_t=\max\{X_{t-1},\pos{\widehat h_t}X_{t-1}/\tau\}$.
The normalized corruption is bounded by $2d\Delta X_{t-1}$,
so
\[
 X_t\le\max\{X_{t-1},
              (\pos{\mathsf s_t}+2d\Delta X_{t-1})/\tau\}.
\]
Fix a realized execution and let
$B=\max\{\Gamma^{-1},\max_t\pos{\mathsf s_t}/\tau_-\}$.
If $X_{t-1}\le B$, the second term above is at most
$(\tau_-B+2d\Delta B)/\tau=B$. Induction proves
\eqref{eq:radial-reciprocal-path}.

Let $\bar s_t$ denote the conditional mean of $\mathsf s_t$
before the radial stage. Equations
\eqref{eq:radial-boundary} and \eqref{eq:radial-mean-moment}
give
$0\le\bar s_t\le M$ and
$\E[(\mathsf s_t-\bar s_t)^2\mid\mathcal F]
\le\chi_dM_2^2/B^{\mathrm{rad}}$. Therefore
\[
 \max_t\pos{\mathsf s_t}\le M+\max_t|\mathsf s_t-\bar s_t|,
 \qquad
 \E\max_t|\mathsf s_t-\bar s_t|^2
 \le\sum_t\E|\mathsf s_t-\bar s_t|^2\le \omega_{\mathrm{rad}}^2.
\]
Minkowski in $L_2$ bounds the norm of the slope maximum by
$M+\omega_{\mathrm{rad}}$. Apply
\eqref{eq:radial-reciprocal-path},
$\max\{a,b\}\le a+b$ and Minkowski once more to obtain the
second statement of \eqref{eq:radial-reciprocal-moment}.
Its first statement follows by Cauchy--Schwarz. This does not
require independence across iterations. Finally,
$B^{\mathrm{rad}}\ge T/(16N)$ and $\chi_d\le2$ prove
\eqref{eq:radial-public-moment-scales}.
\end{proof}

\subsection{Proof of the fixed-threshold risk bound}
\label{subsec:proof-radial-clean}

For an originally smooth population write $L_{\mathrm{pop}}=L$;
otherwise write $L_{\mathrm{pop}}=+\infty$. This is an analysis
convention and supplies no parameter to the computation.

\begin{proof}
The radii are nonincreasing, so
Theorem~\ref{thm:radial-potential} has no positive drift cost.
Conditional centering gives $\E Z_t(u)=0$ separately at the two
histories preceding the respective sampling stages. For each realized execution, its targets have
the common smoothness bound
\[
 \overline L=\min\{L_{\mathrm{pop}},\sqrt dM/\gamma_N\}.
\]
The first alternative follows from averaging an $L$-smooth
population, as in Proposition~\ref{prop:protected-smoothing-smooth}.
The second uses Lemma~\ref{lem:smoothing} and the realized
minimum radius. The square-growth proof of Lemma~\ref{lem:growth}
is pathwise. Apply it after fixing the realized execution with
this common bound to obtain
$H_N-h_0\le4\overline L+\sqrt{\mathcal E_N}/D$.

The second centered update error has zero expected cross term
with the first, which is known before the second update mean. The
protected vector floors yield
\begin{align*}
 W:=\E\mathcal E_N
 &\le\sum_ta_t^2\sigma_0^2\E(1/B_t^y+1/B_t^x)
 \le16\sigma_0^2A_N^2/\mathsf U
 \le32\sigma_0^2A_N^2/T.
\end{align*}
The potential thus gives curvature contribution
$8D^2\E\overline L/A_N$ and stochastic contribution
$2D\sqrt W/A_N\le8\sqrt2\sigma_0D/\sqrt T$.
Lemma~\ref{lem:radial-radius} and $\sum_ta_t=A_N$ bound the
local charge by $\tau+\eta_{\mathrm{rad}}$. Finally,
\[
 \E\overline L\le
 \min\{L_{\mathrm{pop}},\sqrt dM\E(1/\gamma_N)\}
 \le\min\{L_{\mathrm{pop}},\sqrt dM\mathfrak X_0\}.
\]
This proves \eqref{eq:radial-main-risk} and its actual-trajectory
refinement. The required inverse-radius moments are proved in
Lemma~\ref{lem:radial-radius}; no unexplained floor is imposed.
\end{proof}

\subsection{Positive corruption and the threshold tradeoff}
\label{subsec:radial-corruption}

\begin{theorem}[Risk under spatial corruption]
\label{thm:radial-corruption}
Suppose $2d\Delta<\tau$ and use
\eqref{eq:radial-reciprocal-scale}. Under the uniform clean vector
variance bound $\sigma_0^2$ and the radial sample regularity,
\begin{align}
 \E[f(x_N)-f(u)]\le{}&\frac{h_0R_u^2}{2A_N}
 +\frac{8D^2}{A_N}
       \min\{L_{\mathrm{pop}},\sqrt dM\mathfrak X\}
 \nonumber\\
 &+\tau+2d\Delta+\eta_{\mathrm{rad}}
           +6c_dD\Delta\mathfrak X\nonumber\\
 &+\frac{16\sigma_0D}{\sqrt T}
       +\frac{16dD\Delta\mathfrak X}{\sqrt T},
 \label{eq:radial-corrupt-risk}
\end{align}
where $c_d=d\E_e|[e]_1|\le\sqrt d$.
\end{theorem}

\begin{proof}
At each selected vector center the corruption mean bias has norm
at most $c_d\Delta/\gamma_t$, and its centered single-record
variance is at most
$2\sigma_0^2+2d^2\Delta^2/\gamma_t^2$, by the symmetric-record
calculation in Appendix~\ref{sec:app-batches}. Write
$e_t^\iota=b_t^\iota+\zeta_t^\iota$ for $\iota\in\{y,x\}$, with conditionally centered $\zeta_t^\iota$.
Lemma~\ref{lem:radial-radius} supplies the reciprocal $L_2$ bound
$\mathfrak X$. Therefore the weighted product-space $L_2$ norm
of $(b_t^x-b_t^y)_t$ is at most
$2c_d\Delta\mathfrak X\sqrt{\sum_ta_t^2}$.

The second centered error has zero expected cross term with the
first. The deterministic floor in \eqref{eq:radial-floors}
bounds their accumulated variance by
\[
 \E\sum_ta_t^2\norm{\zeta_t^x-\zeta_t^y}^2
 \le\frac{16A_N^2}{\mathsf U}
        (2\sigma_0^2+2d^2\Delta^2\mathfrak X^2).
\]
This uses the fixed minimum batch size, without requiring independence between
the final inverse-radius maximum and a data-selected count. Set
$\overline V=2\sigma_0^2+2d^2\Delta^2\mathfrak X^2$.
Minkowski in the weighted product space gives
\[
 \sqrt{\E\mathcal E_N}\le
 2c_d\Delta\mathfrak X\sqrt{\sum_ta_t^2}
                       +4A_N\sqrt{\overline V/\mathsf U}.
\]
The expected potential bias, after division by $A_N$, is at most
$2c_dD\Delta\mathfrak X$. The first term of the error-energy
bound adds at most $4c_dD\Delta\mathfrak X$, because
$\sqrt{\sum_ta_t^2}\le A_N$. Their sum gives coefficient $6$.
The variance term is $8D\sqrt{\overline V/\mathsf U}$.
Use $\mathsf U\ge T/2$ and
$\sqrt{\overline V}\le\sqrt2\sigma_0+\sqrt2d\Delta\mathfrak X$
to obtain the last two terms in
\eqref{eq:radial-corrupt-risk}. The curvature and local charge are
bounded exactly as in Theorem~\ref{thm:radial-clean} and
Lemma~\ref{lem:radial-radius}. No cross-term cancellation between
a second-center mean bias and the first error is assumed.
\end{proof}

For any threshold chosen before sampling, the clean theorem also gives
\begin{align}
 \E[f(x_N)-f(u)]\le{}&\frac{h_0R_u^2}{2A_N}
 +\frac{16\sqrt dMD^2}{\Gamma N^2}
 +\frac{16\sqrt dM^2D^2}{\tau N^2}\nonumber\\*
 &+\frac{64\sqrt{2d}MM_2D^2}{\tau N\sqrt T}
 +\tau+\frac{8\sqrt2\sigma_0D}{\sqrt T}
 +\frac{4\sqrt2M_2\Gamma\sqrt N}{\sqrt T}.
 \label{eq:radial-nonsmooth-envelope}
\end{align}
This follows by expanding \eqref{eq:radial-main-risk}, using
$A_N\ge N^2/2$ and \eqref{eq:radial-public-moment-scales}.
The threshold scale $\tau\asymp d^{1/4}MD/N$ yields the desired
nonsmooth orders for its threshold-dependent terms, but uses the
unknown $M$. Its initial-radius correction remains explicit.
For the smooth choice \eqref{eq:radial-main-smooth-choices}, this
nonsmooth envelope instead contains the nonvanishing term
$16\sqrt dM^2/h_0$. The smooth guarantee therefore does not
establish an optimal simultaneous nonsmooth guarantee.

The condition $2d\Delta<\tau$ and the reciprocal-dependent bias
in \eqref{eq:radial-corrupt-risk} are sufficient conditions for
this unrestricted shrinking policy. They differ from the
protected-radius noise allowance. A threshold policy preserving
the best smooth and nonsmooth orders and the full admissible
corruption frontier requires further control. For radius increases,
the positive target drift in \eqref{eq:radial-drift} remains an
additional cost.

\subsection{The radial corruption margin}
\label{subsec:radial-collapse}

\begin{proposition}[Collapse beyond the radial corruption margin]
\label{prop:radial-collapse}
For $\tau,\Gamma>0$ and $2d\Delta>\tau$, there is a deterministic
flat population and a fixed spatial corruption of magnitude
$\Delta$ for which the proposed execution satisfies, almost surely,
\[
 \gamma_t=\Gamma\bigl(\tau/(2d\Delta)\bigr)^t,
 \qquad t=0,\ldots,N.
\]
Its vector records are all zero and its centers remain $x^0$.
Consequently the radial corruption bound $2d\Delta$ is sharp in
the stated spatial-corruption class. The conclusion concerns
inverse-radius stability; the original objective gap is zero.
\end{proposition}
\begin{proof}
Translate $x^0$ to zero and set $\rho=\tau/(2d\Delta)\in(0,1)$.
Take $F(z,\xi)=f(z)=0$ on any compact convex feasible set
containing zero and any evaluation domain of width at least $\Gamma$.
Fix before the execution the Borel corruption
\[
 \delta(z)=\begin{cases}
 +\Delta,&\norm z\in\{\Gamma\rho^j:j=0,1,\ldots\},\\
 -\Delta,&\text{otherwise}.
 \end{cases}
\]
It is globally bounded by $\Delta$ and is independent of the
sampled functions and observation history. If the center is zero and the old radius
is $\Gamma\rho^{t-1}$, the radial outer endpoint is on a fixed
sphere and has value $+\Delta$. The inner endpoint has radius
$r\Gamma\rho^{t-1}$. Its continuous radial density gives zero
probability of hitting the countable set of fixed radii. Thus its
value is $-\Delta$ almost surely, and every radial record is
$2d\Delta$. The mean is the same, so the radius is multiplied
by $\rho$.

At the selected radius both endpoints of any symmetric vector
record have equal norm and equal corruption $+\Delta$. Their
difference vanishes. Hence the pilots and update means are zero,
the projection leaves $v_t=0$, and the averaged centers are zero.
Induction proves the displayed formula and the center assertion.
The samples have zero weak gradients and meet every stated clean
moment and regularity requirement. The flat population has zero
objective gap, so this is the claimed stability obstruction.
\end{proof}

\section{Curvature-controlled smoothing thresholds}
\label{sec:app-threshold}

We retain one scalar accelerated state and permit radius reexpansion
when an adjusted curvature reference doubles. The selected radial
derivative controls both the affine-model error and the positive target
drift. The controller uses no population Lipschitz, smoothness or
sample-moment parameter.

The geometry and population conditions are
Assumption~\ref{ass:geometry}. The oracle is clean, its conditional
vector variance is at most $\sigma_0^2$ uniformly at the positive radii
and centers used, and the radial records satisfy the additional sample
regularity in Lemma~\ref{lem:radial-moment}. Fix $u\in Q$ and
$R_u=\norm{u-x^0}$. Set $a_t=t$, $A_t=t(t+1)/2$ and retain the
projection, two centers and scalar update
\eqref{eq:center-y}--\eqref{eq:coefficient}.

The initial algorithm coefficient may be $H_0=h_{\rm init}\ge
h_{\rm ref}>0$, where $h_{\rm ref}$ is a reference chosen before sampling. Define
\begin{equation}
 \mathsf K_t=h_{\rm ref}+H_t-h_{\rm init},\qquad
 \underline K_0=h_{\rm ref}.
 \label{eq:threshold-adjusted-curvature}
\end{equation}
The default choice is $h_{\rm init}=h_{\rm ref}=h_0$.
An initial coefficient chosen from observations is fixed in the initial
history and is assumed integrable; its preceding oracle cost is separate
from the subsequent execution counted here. In this case the startup
term below is $\E[h_{\rm init}R_u^2]/(2A_N)$.
The adjusted coefficient satisfies $\mathsf K_t\le H_t$ and has
exactly the same increments as $H_t$.

\subsection{Executions with fixed or adaptive batch sizes}
\label{subsec:threshold-construction}

Choose $N\ge1$ before sampling. For later substitutions of the geometry, write
$\Gamma(D)=\min\{\gammaMax,d^{1/4}D/N\}$. Define
\begin{equation}
 \Gamma=\min\{\gammaMax,d^{1/4}D/N\},\qquad
 B^{\rm rad}=\floor{T/(8N)},\qquad \gamma_0=\Gamma.
 \label{eq:threshold-initial}
\end{equation}
The execution with fixed batch sizes requires integer $T\ge8N$. Put
\begin{align}
 \mathsf U_{\rm pub}&=T-NB^{\rm rad},\nonumber\\
 B_t^{\rm pub}&=1+
 \floor{\frac{(\mathsf U_{\rm pub}-2N)t}{N(N+1)}},\qquad
 m_{\rm pub}=\mathsf U_{\rm pub}-2\sum_{t=1}^NB_t^{\rm pub}.
 \label{eq:threshold-public-batches}
\end{align}
Use $B_t^y=B_t^x=B_t^{\rm pub}$, except
$B_N^y=B_N^{\rm pub}+m_{\rm pub}$. These counts are known initially.
The execution with adaptive batch sizes requires $T\ge16N$ and retains
\PaidBatch\ and the surplus accounts of Algorithm~\ref{alg:scalar},
initialized once with
\begin{equation}
 \mathsf U_{\rm emp}=T-NB^{\rm rad}-4N,\qquad
 B_t^0=1+
 \floor{\frac{(\mathsf U_{\rm emp}-2N)t}{2N(N+1)}}.
 \label{eq:threshold-empirical-batches}
\end{equation}

Before the radial stage in iteration $t$, compute
\begin{align}
 I_t&=\mathbf1\{\mathsf K_{t-1}\ge2\underline K_{t-1}\},\nonumber\\
 \underline K_t&=
 \begin{cases}\mathsf K_{t-1},&I_t=1,\\
 \underline K_{t-1},&I_t=0,
 \end{cases}
 &\gamma_t^{\rm try}&=
 \begin{cases}\Gamma,&I_t=1,\\\gamma_{t-1},&I_t=0.
 \end{cases}
 \label{eq:threshold-reference}
\end{align}
At the conditionally fixed $y_t,\gamma_t^{\rm try}$, draw
$B^{\rm rad}$ conditionally independent records
\eqref{eq:radial-estimate-main}:
\begin{equation}
 J_{\gamma_t^{\rm try}}(y_t)
 =d[F(y_t+\gamma_t^{\rm try}e,\xi)
       -F(y_t+\gamma_t^{\rm try}r e,\xi)],
 \qquad r\sim\mathrm{Beta}(d,1).
 \label{eq:threshold-radial-record}
\end{equation}
The direction, radial scalar and sampled function in each record are
independent of one another and of the preceding history; the function
is evaluated at the two endpoints. Let $\widehat h_t$ be their mean
and set
\begin{equation}
 \tau_t=\underline K_tD^2/N^2,\qquad
 \gamma_t=\frac{\gamma_t^{\rm try}}
                 {\max\{1,\pos{\widehat h_t}/\tau_t\}}.
 \label{eq:threshold-rule}
\end{equation}
Both vector update means use the same selected radius. Each batch uses
functions and directions sampled independently of all preceding observations. The positive radius can increase
only when $I_t=1$. The optimizer vectors, weights, coefficient $H_t$
and all sampling accounts persist.

\begin{algorithm}[t]
\caption{Scalar acceleration with curvature-controlled smoothing}
\label{alg:threshold}
\small
\begin{algorithmic}[1]
\Require $Q,x^0,D$, $N\ge1$, $h_{\rm init}\ge h_{\rm ref}>0$,
$\gammaMax>0$; fixed batch sizes with $T\ge8N$, or adaptive batch sizes with
$T\ge16N$ and $s_0>0$
\State Set $x_0=v_0=x^0$, $H_0=h_{\rm init}$,
$\mathsf K_0=\underline K_0=h_{\rm ref}$
\State Compute $\Gamma,B^{\rm rad}$ by
\eqref{eq:threshold-initial} and set $\gamma_0=\Gamma$
\State Initialize the fixed batch sizes
\eqref{eq:threshold-public-batches} or the adaptive-batch accounts
\eqref{eq:threshold-empirical-batches}
\For{$t=1,\ldots,N$}
 \State Set $a_t=t$, $A_t=t(t+1)/2$; compute $y_t$
 by \eqref{eq:center-y}
 \State Compute $I_t,\underline K_t,\gamma_t^{\rm try}$
 by \eqref{eq:threshold-reference}
 \State Draw $\widehat h_t$ from $B^{\rm rad}$ conditionally independent
 records \eqref{eq:threshold-radial-record}
 \State Compute $\tau_t,\gamma_t$ by \eqref{eq:threshold-rule}
 \If{batch sizes are adaptive}
  \State $B_t^y\gets\PaidBatch(y_t,t,2t-1,\gamma_t)$
 \EndIf
 \State $g_t^y\gets\FreshMean(y_t,B_t^y,\gamma_t)$
 \State Compute $v_t,x_t$ by \eqref{eq:step}
 \If{batch sizes are adaptive}
  \State $B_t^x\gets\PaidBatch(x_t,t,2t,\gamma_t)$
 \EndIf
 \State $g_t^x\gets\FreshMean(x_t,B_t^x,\gamma_t)$
 \State Update $r_t,\widehat\beta_t,H_t$ by
 \eqref{eq:curvature}--\eqref{eq:coefficient}
 \State Set $\mathsf K_t=h_{\rm ref}+H_t-h_{\rm init}$
\EndFor
\State \Return $x_N$
\end{algorithmic}
\end{algorithm}

\subsection{Exact costs and conditional moments}
\label{subsec:threshold-accounting}

Set $\chi_d=2d^2/((d+1)(d+2))\le2$ and define
\begin{align}
 \omega_{\rm rad}&=\sqrt{\chi_d}M_2\sqrt{N/B^{\rm rad}},&
 \eta_{\rm rad}&=\sqrt{\chi_d}M_2\Gamma/\sqrt{B^{\rm rad}},
 \label{eq:threshold-moment-scales}\\
 c_{\rm vec}&=
 \begin{cases}8/\sqrt7,&\text{fixed batches},\\
 4\sqrt2,&\text{adaptive batches}.
 \end{cases}\nonumber
\end{align}
Write $\phi_t=f_{\gamma_t}$,
$e_t^y=g_t^y-\nabla\phi_t(y_t)$, $e_t^x=g_t^x-\nabla\phi_t(x_t)$ and
$\mathcal E_N=\sum_t t^2\norm{e_t^x-e_t^y}^2$.
The normalized diagnostic and its maximum are
\begin{equation}
 \mathsf s_t=\widehat h_t/\gamma_t^{\rm try},\qquad
 \mathsf S_N=\max_{t\le N}\pos{\mathsf s_t}.
 \label{eq:threshold-slope}
\end{equation}

\begin{lemma}[Accounting and diagnostic moments]
\label{lem:threshold-accounting}
Algorithm~\ref{alg:threshold} spends exactly $T$ pairs, $2T$ values
and $N$ projections. It uses $3N$ sequential sampling stages with fixed
batch sizes and $5N$ with adaptive batch sizes. Moreover
\begin{align}
 \frac{\sqrt{\E\mathcal E_N}}D
 &\le\frac{c_{\rm vec}\sigma_0A_N}{D\sqrt T},
 \label{eq:threshold-energy}\\*
 \omega_{\rm rad}&\le4\sqrt2M_2N/\sqrt T,&
 \eta_{\rm rad}&\le
 \frac{4\sqrt2d^{1/4}M_2D}{\sqrt N\sqrt T},
 \label{eq:threshold-public-moments}\\*
 \E\mathsf S_N&\le M+\omega_{\rm rad},&
 (\E\mathsf S_N^2)^{1/2}&\le M+\omega_{\rm rad}.
 \label{eq:threshold-slope-moments}
\end{align}
Conditionally on the history $\mathcal F_t^{\rm rad}$ before its
radial stage,
\begin{equation}
 \E[h_{\gamma_t}(y_t)\mid\mathcal F_t^{\rm rad}]
 \le\tau_t+\eta_{\rm rad}.
 \label{eq:threshold-local-charge}
\end{equation}
\end{lemma}

\begin{proof}
For $x\ge1$, $\floor x\ge x/2$. Thus in both modes
$B^{\rm rad}\ge T/(16N)$ and $NB^{\rm rad}\le T/8$.
With fixed batch sizes $\mathsf U_{\rm pub}\ge7T/8\ge7N$, and
\[
 2\sum_tB_t^{\rm pub}\le
 2N+\frac{2(\mathsf U_{\rm pub}-2N)}{N(N+1)}\sum_t t
 =\mathsf U_{\rm pub}.
\]
The leftover for fixed batch sizes is nonnegative and its predetermined addition
makes the update counts sum exactly to $\mathsf U_{\rm pub}$.
Since $\mathsf U_{\rm pub}\ge4N$, their floors satisfy
$B_t^\iota\ge\mathsf U_{\rm pub}t/(2N(N+1))$ for $\iota\in\{y,x\}$.
With adaptive batch sizes $\mathsf U_{\rm emp}\ge5T/8\ge10N$.
The adaptive-batch accounting proof gives exactly $\mathsf U_{\rm emp}$ update
pairs and $B_t^\iota\ge\mathsf U_{\rm emp}t/(4N(N+1))$; the vector pilots
cost $4N$ pairs. Adding the $NB^{\rm rad}$ radial pairs gives $T$
in both modes. Each sampled function supplies exactly two values. Each iteration
has one radial stage and two vector stages with fixed batch sizes, or one radial
stage and four vector stages with adaptive batch sizes, and one projection.

The second vector error is centered conditionally on its history,
which contains the first error. Their expected cross term is zero.
The floors therefore give
\[
 \E\mathcal E_N\le
 \begin{cases}
 8\sigma_0^2A_N^2/\mathsf U_{\rm pub}
 \le(64/7)\sigma_0^2A_N^2/T,\\
 16\sigma_0^2A_N^2/\mathsf U_{\rm emp}
 \le32\sigma_0^2A_N^2/T.
 \end{cases}
\]
This proves \eqref{eq:threshold-energy}.

Let $\bar s_t$ be the conditional mean of $\mathsf s_t$ before its
radial stage. The radial moment lemma gives
$0\le\bar s_t\le M$ and conditional variance at most
$\chi_dM_2^2/B^{\rm rad}$. Hence
\[
 \mathsf S_N\le M+\max_t|\mathsf s_t-\bar s_t|,\qquad
 \E\max_t|\mathsf s_t-\bar s_t|^2
 \le\sum_t\E|\mathsf s_t-\bar s_t|^2\le\omega_{\rm rad}^2.
\]
Minkowski and Cauchy--Schwarz prove the two maximum-moment bounds,
without independence across iterations.
For the local charge set $\lambda=\gamma_t/\gamma_t^{\rm try}\le1$.
Lemma~\ref{lem:radial-affine} and the threshold rule give
\[
 h_{\gamma_t}(y_t)\le
 \lambda h_{\gamma_t^{\rm try}}(y_t)
 \le\tau_t+|\widehat h_t-h_{\gamma_t^{\rm try}}(y_t)|,
\]
because $\lambda\pos{\widehat h_t}\le\tau_t$.
Conditional Cauchy--Schwarz bounds the last expectation by
$\sqrt{\chi_d}M_2\gamma_t^{\rm try}/\sqrt{B^{\rm rad}}$
and hence by $\eta_{\rm rad}$.
The diagnostic floor, $\chi_d\le2$ and
$\Gamma\le d^{1/4}D/N$ prove
\eqref{eq:threshold-public-moments}.
\end{proof}

\subsection{Reference growth and reciprocal radii}
\label{subsec:threshold-reference-proof}

\begin{lemma}[Reciprocal radii and geometric references]
\label{lem:threshold-reference}
For every iteration,
\begin{align}
 \underline K_t&\le\mathsf K_{t-1}<2\underline K_t,&
 \gamma_t^{-1}&\le
 \max\{\Gamma^{-1},\mathsf S_NN^2/(\underline K_tD^2)\},
 \label{eq:threshold-reciprocal}\\
 L_t&\le\frac{\sqrt dM}{\Gamma}
 +\frac{2\sqrt dM\mathsf S_NN^2}{\mathsf K_{t-1}D^2},
 \label{eq:threshold-smoothness}
\end{align}
where $L_t$ is a valid smoothness bound for the selected target.
If $J=\sum_t I_t$, then
\begin{equation}
 J\le\log_2\!\left(1+\frac{H_N-h_{\rm init}}{h_{\rm ref}}\right),
 \qquad
 \sum_{t:I_t=1}\underline K_t\le2(H_N-h_{\rm init}).
 \label{eq:threshold-geometric}
\end{equation}
\end{lemma}

\begin{proof}
The first relation follows from the trigger and monotonicity of
$\mathsf K_t$. Within a sequence of iterations having an unchanged
reference, the threshold is constant and the exact reciprocal update is
\[
 \gamma_t^{-1}=
 \max\{(\gamma_t^{\rm try})^{-1},\pos{\mathsf s_t}/\tau_t\}.
\]
Its first trial radius is $\Gamma$; thereafter the trial radius is the
previous selected radius. Induction proves the inverse bound.
The smoothing bound $\sqrt dM/\gamma_t$ and
$\underline K_t>\mathsf K_{t-1}/2$ prove
\eqref{eq:threshold-smoothness}.
Each replacement raises the reference by a factor of at least two,
and its last value is at most
$\mathsf K_N=h_{\rm ref}+H_N-h_{\rm init}$.
This proves the event-count bound.
If its successive values are $c_{j-1},c_j$, then
$c_j\ge2c_{j-1}$ and $c_j\le2(c_j-c_{j-1})$.
Summation from the initial reference $h_{\rm ref}$ proves the
geometric-sum bound.
\end{proof}

\subsection{Curvature growth across reference changes}
\label{subsec:threshold-growth-proof}

For explicit geometry dependence, write $\mathcal H_{\mathrm{reg}}=\mathcal H_{\mathrm{reg}}(D)$
at the supplied diameter and define
\begin{align}
 \mathfrak H&=
 4\sqrt dM/\Gamma+
        2\sqrt2d^{1/4}\sqrt{M\mathsf S_N}\,N/D,
 \label{eq:threshold-growth-shift}\\
 \mathcal H_{\rm ns}(D)&=
 4\sqrt dM/\Gamma(D)+(2\sqrt2d^{1/4}+2)MN/D
 +\sqrt2d^{1/4}\omega_{\rm rad}N/D
 +\frac{c_{\rm vec}\sigma_0A_N}{D\sqrt T},
 \label{eq:adaptive-main-Hns}\\
 \mathcal H_{\rm s}(D)&=4L+
              \frac{c_{\rm vec}\sigma_0A_N}{D\sqrt T}.
 \label{eq:adaptive-main-Hs}
\end{align}

\begin{lemma}[Curvature under adaptive thresholds]
\label{lem:threshold-growth}
Pathwise,
\begin{equation}
 H_N-h_{\rm init}\le
 (\mathfrak H+2NM/D-h_{\rm ref})_+
                         +\sqrt{\mathcal E_N}/D.
 \label{eq:threshold-growth-path}
\end{equation}
Consequently $\E(H_N-h_{\rm init})\le\mathcal H_{\rm ns}$.
For an $L$-smooth population it is also at most $\mathcal H_{\rm s}$.
\end{lemma}

\begin{proof}
Fix a complete realization. Its $\mathsf S_N$ and $\mathfrak H$
are fixed analysis quantities in this argument. Put
$a=\sqrt dM/\Gamma$ and $b=\sqrt dM\mathsf S_NN^2/D^2$.
The threshold $\mathfrak H=4a+2\sqrt2\sqrt b$ satisfies
\[
 \mathfrak H^2-4a\mathfrak H-8b=8\sqrt2a\sqrt b\ge0.
\]
Consequently, whenever $\mathsf K_{t-1}\ge\mathfrak H$,
\eqref{eq:threshold-smoothness} gives
$4L_t\le4a+8b/\mathfrak H\le\mathfrak H$.
If $\mathfrak H=0$, the population is constant and $L_t=0$,
so this conclusion holds directly.
The secant identity implies
$A_t\widehat\beta_t\le2L_tr_t^2+
t\norm{e_t^x-e_t^y}r_t$.
Use the single shift $\mathfrak B=\mathfrak H+2NM/D$ and
$k_t=\pos{\mathsf K_t-\mathfrak B}$.
The balance identity has negative term $H_tr_t^2/2$.
Since $H_t\ge\mathsf K_t$, it is at least as strong as the negative
term using $\mathsf K_t$. Thus, when
$\mathsf K_{t-1}\ge\mathfrak H$ and $k_t>0$,
\[
 (H_t-H_{t-1})D^2
 \le\pos{-k_tr_t^2/2+t\norm{e_t^x-e_t^y}r_t}
 \le\frac{t^2\norm{e_t^x-e_t^y}^2}{2k_t}.
\]
As $0\le k_t-k_{t-1}\le H_t-H_{t-1}$,
\begin{equation}
 D^2(k_t^2-k_{t-1}^2)
 \le t^2\norm{e_t^x-e_t^y}^2.
 \label{eq:threshold-growth-square}
\end{equation}
The case $k_t=0$ is immediate.
If $\mathsf K_{t-1}<\mathfrak H$, both true gradients have norm
at most $M$, and $r_t\le D$ gives
\[
 A_t\widehat\beta_t\le
 t(2M+\norm{e_t^x-e_t^y})r_t,\qquad
 H_t-H_{t-1}\le
 t(2M+\norm{e_t^x-e_t^y})/D.
\]
Then $k_{t-1}=0$ and
$k_t\le t\norm{e_t^x-e_t^y}/D$, so
\eqref{eq:threshold-growth-square} holds in this case too.
Summation yields $k_N^2\le k_0^2+\mathcal E_N/D^2$.
Use $k_N\le k_0+\sqrt{\mathcal E_N}/D$, together with
$\mathsf K_0=h_{\rm ref}$ and
$\mathsf K_N-h_{\rm ref}=H_N-h_{\rm init}$, to prove
\eqref{eq:threshold-growth-path}; the case
$\mathsf K_N\le\mathfrak B$ is direct.

By the diagnostic maximum bound and concavity,
\[
 \E\sqrt{M\mathsf S_N}
 \le\sqrt{M(M+\omega_{\rm rad})}
 \le M+\omega_{\rm rad}/2.
\]
Take expectations, drop the negative $h_{\rm ref}$ in the positive
part, and use \eqref{eq:threshold-energy}.
This proves the nonsmooth bound.
For a smooth population every target has the common bound $L$.
The original scalar-growth lemma, with initial coefficient
$h_{\rm init}$, gives
$H_N-h_{\rm init}\le4L+\sqrt{\mathcal E_N}/D$ and proves the
smooth bound. No independence between terminal maxima, chosen radii,
counts and curvature coefficients is required.
\end{proof}

\subsection{Positive target drift}
\label{subsec:threshold-drift-proof}

\begin{lemma}[Drift at radius reexpansions]
\label{lem:threshold-drift}
Let
$\mathcal D_N^+=\sum_{t=2}^N A_{t-1}
\pos{\phi_t(x_{t-1})-\phi_{t-1}(x_{t-1})}$.
Then
\begin{equation}
 \frac{\E\mathcal D_N^+}{A_N}
 \le\frac{2D^2}{N^2}\E(H_N-h_{\rm init})
 +\eta_{\rm rad}\E J+\frac{2MDN}{A_N}\E J.
 \label{eq:threshold-drift-bound}
\end{equation}
For an $L$-smooth population replace the last term by
$LD^2\E J/A_N$.
\end{lemma}

\begin{proof}
When $I_t=0$, the radius cannot increase, so its positive drift is
zero. On an event $I_t=1$, use $\phi_{t-1}\ge f$ and the radial
affine model at $(y_t,\gamma_t)$ applied to $x_{t-1}$. They give
\[
 \pos{\phi_t(x_{t-1})-\phi_{t-1}(x_{t-1})}
 \le h_{\gamma_t}(y_t)+D_{\phi_t}(x_{t-1},y_t),
\]
where $D_\phi(x,y)=\phi(x)-\phi(y)-\ip{\nabla\phi(y)}{x-y}$
is a nonnegative Bregman remainder. The same right-hand side
bounds $\phi_t(x_{t-1})-f(x_{t-1})$.
The general upper bound for this remainder is
$2M\norm{x_{t-1}-y_t}$; the smooth bound is
$L\norm{x_{t-1}-y_t}^2/2$.
The center identity gives $\norm{x_{t-1}-y_t}\le2D/(t+1)$.
Consequently
\[
 A_{t-1}D_{\phi_t}(x_{t-1},y_t)\le
 \begin{cases}2MDN,&\text{general case},\\LD^2,&\text{smooth case}.
 \end{cases}
\]
The exact coefficients are respectively
$2MDt(t-1)/(t+1)\le2MDN$ and
$LD^2t(t-1)/(t+1)^2\le LD^2$.

The event, center, trial radius and threshold are known before the
radial stage. Multiply \eqref{eq:threshold-local-charge}
by $I_tA_{t-1}$ before taking expectations. Since $A_{t-1}\le A_N$,
\begin{align*}
 \frac1{A_N}\E\sum_t I_tA_{t-1}h_{\gamma_t}(y_t)
 &\le\frac{D^2}{N^2}\E\sum_t I_t\underline K_t
                 +\eta_{\rm rad}\E J\\
 &\le\frac{2D^2}{N^2}\E(H_N-h_{\rm init})
                 +\eta_{\rm rad}\E J.
\end{align*}
The second inequality is the geometric-reference bound.
Add the Bregman remainder. This eventwise averaging makes no
independence assumption about the selected radius.
\end{proof}

\subsection{The clean simultaneous guarantee}
\label{subsec:proof-threshold-risk}

For use with the anchor-selected diameter, retain the geometry dependence
of the envelopes above. Put $A=A_N$, $k_N=2/A+3/N^2$,
$\eta(D)=\sqrt{\chi_d}M_2\Gamma(D)/\sqrt{B^{\rm rad}}$, and
$\ell_{\mathrm{reg}}(D)=\log_2(1+\mathcal H_{\mathrm{reg}}(D)/h_{\rm ref})$.
Set $q_{\rm ns}(D)=2MDN/A$ and $q_{\rm s}(D)=LD^2/A$, and define
\begin{align}
 \mathcal U_{\mathrm{reg}}(D,\ell)
 &=k_ND^2\mathcal H_{\mathrm{reg}}(D)+h_{\rm ref}D^2/N^2\nonumber\\
 &\quad+\eta(D)(1+\ell)+q_{\mathrm{reg}}(D)\ell.
 \label{eq:adaptive-main-U}
\end{align}
These are analysis quantities; none is an algorithm input.

\begin{theorem}[A simultaneous threshold guarantee]
\label{thm:threshold-risk}
Assume the geometry of Assumption~\ref{ass:geometry}, the uniform clean
moments of Assumption~\ref{ass:adaptive-moments}, and the radial regularity
of Lemma~\ref{lem:radial-moment}. Fix $u\in Q$ and let
$h_{\rm init}\ge h_{\rm ref}>0$ be fixed in the initial history.
Choose $\mathcal H=\mathcal H_{\rm ns}$, or under population
smoothness $\mathcal H=\mathcal H_{\rm s}$, and put
$\ell=\log_2(1+\mathcal H/h_{\rm ref})$.
Algorithm~\ref{alg:threshold} satisfies
\begin{align}
 \E[f(x_N)-f(u)]\le{}&
 \frac{\E[h_{\rm init}R_u^2]}{2A_N}
 +\left(\frac{2D^2}{A_N}+\frac{3D^2}{N^2}\right)\mathcal H
 +\frac{h_{\rm ref}D^2}{N^2}\nonumber\\
 &+\eta_{\rm rad}(1+\ell)+\frac{2MDN}{A_N}\ell.
 \label{eq:threshold-risk}
\end{align}
For an $L$-smooth population its last term may be replaced by
$LD^2\ell/A_N$.
Equivalently, each applicable branch is bounded by
$\E[h_{\rm init}R_u^2]/(2A)+\mathcal U_{\mathrm{reg}}(D,\ell_{\mathrm{reg}}(D))$.
The method spends exactly $T$ pairs, $2T$ values and $N$ projections,
in $3N$ sampling stages with fixed batch sizes or $5N$ with adaptive batch sizes.
\end{theorem}

\begin{proof}
The radial potential of Theorem~\ref{thm:radial-potential}
holds with initial coefficient $h_{\rm init}$ by its projection
proof. Retain its positive drift term. Its two error inner
products have zero conditional mean at their respective sampling histories.
The weighted local radial charge satisfies
\[
 \frac1{A_N}\sum_t a_t\E h_{\gamma_t}(y_t)
 \le\frac{D^2}{N^2}\E\mathsf K_N+\eta_{\rm rad}
 =\frac{D^2}{N^2}
   [h_{\rm ref}+\E(H_N-h_{\rm init})]+\eta_{\rm rad}.
\]
This uses the conditional charge bound and
$\underline K_t\le\mathsf K_N$ pathwise.
Add the drift bound to the expected potential, obtaining exactly
the curvature coefficient in \eqref{eq:threshold-risk}.
Lemma~\ref{lem:threshold-growth} bounds the expected increase by
$\mathcal H$. Jensen and the event-count bound give
\[
 \E J\le
 \E\log_2\!\left(1+\frac{H_N-h_{\rm init}}{h_{\rm ref}}\right)
 \le\log_2(1+\mathcal H/h_{\rm ref})=\ell.
\]
Every coefficient multiplying the count is deterministic.
The eventwise diagnostic averaging precedes this logarithmic
bound, so no terminal-logarithm/noise product is introduced.
Finiteness follows from the vector and diagnostic moments and
the curvature bound. The initial coefficient can be conditioned
on its initial history throughout these arguments.
\end{proof}

\begin{corollary}[Unknown population smoothness]
\label{cor:threshold-smooth}
If the population is $L$-smooth, put
$\ell_{\rm s}=\log_2(1+\mathcal H_{\rm s}/h_{\rm ref})$. Then
\begin{align}
 \E[f(x_N)-f(u)]\le{}&
 \frac{\E[h_{\rm init}R_u^2]}{2A_N}
 +\frac{[(28+2\ell_{\rm s})L+h_{\rm ref}]D^2}{N^2}
 \nonumber\\
 &+\frac{5c_{\rm vec}\sigma_0D}{\sqrt T}
 +\frac{4\sqrt2(1+\ell_{\rm s})d^{1/4}M_2D}
             {\sqrt N\sqrt T}.
 \label{eq:threshold-smooth-risk}
\end{align}
With $h_{\rm init}=h_{\rm ref}=h_0$ and
$\sigma_0^2\le CdM_2^2$, the order is
$\widetilde O((L+h_0)D^2/N^2+\sqrt dM_2D/\sqrt T)$.
Only the displayed logarithm is suppressed.
\end{corollary}

\begin{proof}
Use the smooth form of \eqref{eq:threshold-risk}.
The identities $A_N\ge N^2/2$ and $A_N/N^2\le1$ give
\[
 \frac{2D^2}{A_N}+\frac{3D^2}{N^2}\le7D^2/N^2,\qquad
 LD^2\ell_{\rm s}/A_N\le2LD^2\ell_{\rm s}/N^2.
\]
The vector-noise contribution is, more precisely,
\[
 \left(\frac{2D^2}{A_N}+\frac{3D^2}{N^2}\right)
 \frac{c_{\rm vec}\sigma_0A_N}{D\sqrt T}
 =\frac{c_{\rm vec}\sigma_0D}{\sqrt T}
       (2+3A_N/N^2)
 \le\frac{5c_{\rm vec}\sigma_0D}{\sqrt T}.
\]
Bound the remaining radial term by
\eqref{eq:threshold-public-moments}. This proves the displayed
inequality. For the dimensional specialization use $R_u\le D$,
$d^{1/4}/\sqrt N\le\sqrt d$ and
$\sigma_0\le\sqrt{Cd}M_2$. None of $L,M,M_2$ is an input.
The comparator can be the original population optimizer.
\end{proof}

\begin{corollary}[Unknown nonsmooth regularity]
\label{cor:threshold-nonsmooth}
Put $\ell_{\rm ns}=\log_2(1+\mathcal H_{\rm ns}/h_{\rm ref})$.
Then
\begin{align}
 \E[f(x_N)-f(u)]\le{}&
 \frac{\E[h_{\rm init}R_u^2]}{2A_N}
 +\frac{h_{\rm ref}D^2}{N^2}
 +[(42+14\sqrt2)d^{1/4}+4\ell_{\rm ns}]\frac{MD}{N}\nonumber\\
 &+\frac{28\sqrt dMD^2}{\gammaMax N^2}
 +\frac{5c_{\rm vec}\sigma_0D}{\sqrt T}
 +\frac{56d^{1/4}M_2D}{\sqrt T}\nonumber\\
 &+\frac{4\sqrt2(1+\ell_{\rm ns})d^{1/4}M_2D}
             {\sqrt N\sqrt T}.
 \label{eq:threshold-nonsmooth-risk}
\end{align}
If $h_{\rm init}=h_{\rm ref}=h_0$,
$\gammaMax\ge d^{1/4}D/N$ and $\sigma_0^2\le CdM_2^2$, the
same controller has order
$\widetilde O(h_0D^2/N^2+d^{1/4}MD/N+\sqrt dM_2D/\sqrt T)$.
\end{corollary}

\begin{proof}
Use
\[
 \Gamma^{-1}\le\gammaMax^{-1}+N/(d^{1/4}D),\qquad
 (4+2\sqrt2)d^{1/4}+2\le(6+2\sqrt2)d^{1/4},\qquad
 \omega_{\rm rad}\le4\sqrt2M_2N/\sqrt T
\]
in \eqref{eq:adaptive-main-Hns}.
Multiplication by at most $7D^2/N^2$ gives the respective
coefficients $28$, $42+14\sqrt2$ and $56$ for the evaluation domain,
deterministic curvature and diagnostic maximum.
The vector term uses the exact calculation in the preceding proof
and has coefficient $5c_{\rm vec}$.
The general Bregman remainder is at most $4MD\ell_{\rm ns}/N$.
The local diagnostic error uses
\eqref{eq:threshold-public-moments}. These are the displayed terms.
In the stated evaluation domain regime the evaluation-radius term itself is at most
$28d^{1/4}MD/N$. Use $R_u\le D$, $d^{1/4}\ge1$,
$d^{1/4}\le\sqrt d$ and the variance specialization to obtain
the stated order.
\end{proof}

Both corollaries concern the same threshold controller.
Its optimizer state and acceleration weights are never reset.
The cost of radius increases is included in the drift proof.
Known diameter, an evaluation domain, exact projection and radial sample
regularity remain explicit. The joint anchor guarantee uses this same
controller with an anchor-selected diameter. Positive corruption is
handled by the protected forward-secant controller.
The logarithms express learning relative to $h_{\rm ref}$.
For the smooth regime alone, $\gamma=\min\{\gammaMax,D/N\}$ and
vector batch sizes fixed before sampling attain the same order in $2N$ stages.
The simultaneous controller supplies observed thresholds and
radii across both regularity regimes.

\section{Forward secants and corruption-protected radial control}
\label{sec:app-corrupt-radial}

We use a forward secant of the ball-smoothed population to control the
radial affine-model charge. Its two endpoints yield a scalar corruption
bound independent of dimension. A radius safeguard chosen before sampling bounds the
vector bias. The scalar optimizer, acceleration weights and curvature
reference persist through every radius change.

Assume that $Q$ is compact and convex, $x^0\in Q$, and
$\operatorname{diam}_2(Q)\le D$ with $D>0$. The population is convex and
$M$-Lipschitz on the closed evaluation domain $Q+\gammaMax\mathbb B_2^d$.
The clean symmetric vector record has conditional mean $\nabla f_\gamma(z)$
and variance at most $\sigma_0^2$, uniformly at all positive radii and centers
used below, with $\sigma_0$ defined in Assumption~\ref{ass:vector}.
The clean forward differences satisfy Assumption~\ref{ass:forward-moments}.
Lemma~\ref{lem:corrupt-forward-secant} gives sufficient sample regularity
conditions for this scalar moment bound. Samples
need not be convex. At each sampling stage, functions and directions are drawn independently
of one another and of preceding observations. Each sampled function is
evaluated at both endpoints of its record.

The observed sample loss is
\begin{equation}
\widetilde F(x,\xi)=F(x,\xi)+\delta(x),\qquad
 \sup_{x\in Q+\gammaMax\mathbb B_2^d}|\delta(x)|\le\Delta,
 \label{eq:corrupt-spatial-model}
\end{equation}
where $\delta$ is a fixed Borel function. Its supplied envelope
$\Delta\ge0$ is an algorithm input. We retain
\[
 \mu_d=\frac d{d+1},\quad \nu_d=\frac d{d+2},\quad
 c_d=d\E|[e]_1|\le\sqrt d,
\]
and $f_\gamma(y)=\E_U f(y+\gamma U)$ for a uniform unit-ball $U$.
The radial quantity is $h_\gamma(y)=\gamma\partial_\gamma f_\gamma(y)$.

\subsection{A two-value upper secant for the radial charge}
\label{subsec:corrupt-forward}

\begin{lemma}[Forward radial secant]
\label{lem:corrupt-forward-secant}
Draw a uniform-ball $U$ independently of the sampled function indexed by $\xi$.
For $0<\gamma\le\gammaMax/2$, put
\begin{align}
 q_\gamma(y)&=f_{2\gamma}(y)-f_\gamma(y),\nonumber\\
 \widehat q_\gamma^0(y;U,\xi)
 &=F(y+2\gamma U,\xi)-F(y+\gamma U,\xi).
 \label{eq:corrupt-forward-record}
\end{align}
Then
\begin{equation}
 0\le h_\gamma(y)\le q_\gamma(y)\le\mu_dM\gamma.
 \label{eq:corrupt-forward-charge}
\end{equation}
Under the weak-gradient and segment-integrability hypotheses of
Lemma~\ref{lem:oracle-moments}, or a random sample Lipschitz envelope
with second moment at most $M_2^2$, the record satisfies
\begin{equation}
 \E \widehat q_\gamma^0(y)=q_\gamma(y),\qquad
 \E\bigl[\widehat q_\gamma^0(y)^2\bigr]\le\nu_dM_2^2\gamma^2.
 \label{eq:corrupt-forward-moment}
\end{equation}
The corrupted record $\widehat q_\gamma$ differs pointwise from $\widehat q_\gamma^0$
by at most $2\Delta$. For $0<\lambda\le1$,
\begin{equation}
 h_{\lambda\gamma}(y)\le\lambda q_\gamma(y).
 \label{eq:corrupt-forward-shrink}
\end{equation}
The affine lower model for every original comparator $u\in Q$ remains
\begin{equation}
 f(u)\ge f_\gamma(y)+\ip{\nabla f_\gamma(y)}{u-y}-h_\gamma(y).
 \label{eq:corrupt-original-model}
\end{equation}
If the population is $L$-smooth on the evaluation domain, then
$h_\gamma(y)\le L\nu_d\gamma^2$.
\end{lemma}

\begin{proof}
Lemma~\ref{lem:radial-affine} proves that $(y,s)\mapsto f_s(y)$ is
jointly convex and differentiable at positive $s$, including evaluation domain
boundary radii. For fixed $y$, write $a(s)=f_s(y)$. Convexity yields
$a(2\gamma)\ge a(\gamma)+\gamma a'(\gamma)$, proving
$h_\gamma(y)\le q_\gamma(y)$. The radial function is even and convex,
so both quantities are nonnegative. Lipschitzness gives
\[
 q_\gamma(y)\le\E M\gamma\norm U=\mu_dM\gamma.
\]
Its sample mean follows by taking the same-sample expectation at both
endpoints.

Under the gradient moment assumption, the fundamental theorem along the sampled
segment gives, almost surely,
\[
 \widehat q_\gamma^0
 =\gamma\int_1^2\ip{\nabla F(y+s\gamma U,\xi)}U\,ds.
\]
Jensen and Cauchy--Schwarz bound its squared magnitude by
$\gamma^2\norm U^2\int_1^2\norm{\nabla F(y+s\gamma U,\xi)}^2ds$.
The segment integrability assumption permits Fubini; integrating in
$\xi,U$ bounds the expectation by
$\gamma^2M_2^2\E\norm U^2=\nu_dM_2^2\gamma^2$.
At an endpoint on the boundary of the evaluation domain, sample continuity
and increasing inner radii give the same conclusion by Fatou. The
uniform second-moment bound supplies uniform integrability for the
mean. Under a random Lipschitz envelope, instead apply
$|\widehat q_\gamma^0|\le m(\xi)\gamma\norm U$ directly.

The two corruption values in the record differ by at most $2\Delta$.
Lemma~\ref{lem:radial-affine} gives
$h_{\lambda\gamma}(y)\le\lambda h_\gamma(y)$, proving
\eqref{eq:corrupt-forward-shrink}. Its joint convex affine lower model
at $(y,\gamma)$ applied to $(u,0)$ proves
\eqref{eq:corrupt-original-model}. Finally, in the smooth case,
\[
 \partial_\gamma f_\gamma(y)
 =\E\ip{\nabla f(y+\gamma U)-\nabla f(y)}U
 \le L\gamma\E\norm U^2=L\nu_d\gamma,
\]
where $\E U=0$ cancels the center gradient. Multiplication by $\gamma$
proves the last claim.
\end{proof}

The forward record is the adjacent-scale diagnostic of
Proposition~\ref{prop:adjacent-scale} with outer radius $2\gamma$.
Here the convex radial derivative inequality makes that record an
upper bound for the affine-model charge at the selected target.

\subsection{Persistent protected execution and exact costs}
\label{subsec:corrupt-radial-policy}

Let $a_t=t$, $A_t=t(t+1)/2$, $A=A_N$, and initialize the scalar method
once with $H_0=h_{\rm init}\ge h_{\rm ref}>0$. Put
\begin{equation}
 \mathsf K_t=h_{\rm ref}+H_t-h_{\rm init},\qquad
 \underline K_0=h_{\rm ref}.
 \label{eq:corrupt-adjusted}
\end{equation}
An observed initial coefficient is measurable in the initial history
and integrable, so its startup expectation is finite; preceding
observations have their separate cost. Choose a
safeguard $0\le g\le\gammaMax/2$ before sampling and set
\begin{equation}
 \Gamma_0=\min\{\gammaMax/2,d^{1/4}D/N\},\quad
 \Gamma=\max\{g,\Gamma_0\},\quad
 B^{\rm rad}=\floor{T/(8N)},\quad\gamma_0=\Gamma.
 \label{eq:corrupt-initial-radii}
\end{equation}
The two vector centers, projection and coefficient update are
\eqref{eq:center-y}--\eqref{eq:coefficient}; their observed records use
\eqref{eq:corrupt-spatial-model}.

Before the radial stage of iteration $t$, compute
\begin{align}
 I_t&=\mathbf1\{\mathsf K_{t-1}\ge2\underline K_{t-1}\},\nonumber\\
 \underline K_t&=
 \begin{cases}\mathsf K_{t-1},&I_t=1,\\
 \underline K_{t-1},&I_t=0,
 \end{cases}
 &\gamma_t^{\rm try}&=
 \begin{cases}\Gamma,&I_t=1,\\\gamma_{t-1},&I_t=0.
 \end{cases}
 \label{eq:corrupt-reference}
\end{align}
At the conditionally fixed center $y_t$ and trial radius, draw
$B^{\rm rad}$ conditionally independent corrupted forward records, and let their mean be
$\widehat q_t$. Set
\begin{equation}
 \tau_t=\underline K_tD^2/N^2,\qquad
 \gamma_t=\max\left\{g,
 \frac{\gamma_t^{\rm try}}
 {\max\{1,\pos{\widehat q_t-2\Delta}/\tau_t\}}\right\}.
 \label{eq:corrupt-radius-rule}
\end{equation}
Both vector update means use $\gamma_t$. Each radial, pilot or update
record uses a separate sampled function and evaluates it at exactly
two points. The optimizer state, weights and surplus
accounts persist throughout.

Fixed batch sizes require $T\ge8N$. Define
\begin{align}
 \mathsf U_{\rm pub}&=T-NB^{\rm rad},\nonumber\\
 B_t^{\rm pub}&=1+
 \floor{\frac{(\mathsf U_{\rm pub}-2N)t}{N(N+1)}},\quad
 m_{\rm pub}=\mathsf U_{\rm pub}-2\sum_{t=1}^NB_t^{\rm pub}.
 \label{eq:corrupt-public-counts}
\end{align}
Use these counts at both centers, adding the predetermined residual
$m_{\rm pub}$ to the final first-center mean. The version with adaptive batch sizes requires
$T\ge16N$, retains \PaidBatch\ and its persistent surplus, and uses
\begin{equation}
 \mathsf U_{\rm emp}=T-NB^{\rm rad}-4N,\qquad
 B_t^0=1+
 \floor{\frac{(\mathsf U_{\rm emp}-2N)t}{2N(N+1)}}.
 \label{eq:corrupt-paid-counts}
\end{equation}
The two pilot pairs at each center use the selected radius. The scalar
surplus reference $s_0>0$ is still chosen before sampling.

\begin{algorithm}[t]
\caption{Scalar acceleration with protected forward-secant smoothing}
\label{alg:corrupt-radial}
\small
\begin{algorithmic}[1]
\Require $Q,x^0,D$, $N\ge1$, $\Delta\ge0$, $0\le g\le\gammaMax/2$,
$h_{\rm init}\ge h_{\rm ref}>0$; fixed batch sizes with $T\ge8N$, or adaptive batch sizes with
$T\ge16N,s_0>0$
\State $x_0=v_0=x^0$, $H_0=h_{\rm init}$,
$\mathsf K_0=\underline K_0=h_{\rm ref}$
\State Compute $\Gamma,B^{\rm rad}$ by
\eqref{eq:corrupt-initial-radii}; set $\gamma_0=\Gamma$
\State Initialize fixed batch sizes \eqref{eq:corrupt-public-counts}
or adaptive-batch accounts \eqref{eq:corrupt-paid-counts} once
\For{$t=1,\ldots,N$}
 \State Set $a_t=t,A_t=t(t+1)/2$; compute $y_t$ by
 \eqref{eq:center-y}
 \State Compute the predictable reference and trial radius by
 \eqref{eq:corrupt-reference}
 \State Draw the mean $\widehat q_t$ of $B^{\rm rad}$ corrupted
 records \eqref{eq:corrupt-forward-record}
 \State Select $\gamma_t$ by \eqref{eq:corrupt-radius-rule}
 \State Select $B_t^y$; compute a corrupted symmetric vector mean from independent records
 at $(y_t,\gamma_t)$
 \State Compute $v_t,x_t$ by \eqref{eq:step}
 \State Select $B_t^x$; compute a corrupted symmetric vector mean from independent records
 at $(x_t,\gamma_t)$
 \State Update $H_t$ by \eqref{eq:curvature}--\eqref{eq:coefficient}
 \State Set $\mathsf K_t=h_{\rm ref}+H_t-h_{\rm init}$
\EndFor
\State \Return $x_N$
\end{algorithmic}
\end{algorithm}

\subsection{Diagnostics, inverse radii and biased vector errors}
\label{subsec:corrupt-controller-moments}

For analysis, let $\widehat q_t^0$ be the clean mean on the same
sampled functions and directions as $\widehat q_t$. Define
\begin{align}
 \mathsf S_N&=\max_{t\le N}
       \pos{\widehat q_t^0/\gamma_t^{\rm try}},\nonumber\\
 \omega&=\sqrt{\nu_d}M_2\sqrt{N/B^{\rm rad}},&
 \eta&=\sqrt{\nu_d}M_2\Gamma/\sqrt{B^{\rm rad}},
 \label{eq:corrupt-diagnostic-scales}\\
 \zeta&=\min\left\{g^{-1},\Gamma^{-1}+
       \frac{(\mu_dM+\omega)N^2}{h_{\rm ref}D^2}\right\},&
 \beta&=c_d\Delta\zeta,\qquad
 \bar\sigma=\sigma_0+d\Delta\zeta.
 \label{eq:corrupt-bias-scales}
\end{align}
The convention is $g^{-1}=+\infty$ when $g=0$. The second argument in
the minimum is finite, including when $\Delta=0$. Put
\[
 \mathsf W_{N,2}=\sum_{t=1}^Nt^2=\frac{N(N+1)(2N+1)}6,
 \qquad
 c_{\rm vec}=\begin{cases}8/\sqrt7,&\text{fixed batches},\\
                         4\sqrt2,&\text{adaptive batches}.
             \end{cases}
\]

\begin{lemma}[Protected radius and exact resource bounds]
\label{lem:corrupt-controller}
Algorithm~\ref{alg:corrupt-radial} uses exactly $T$ pairs and $2T$
values, $N$ projections, and $3N$ sampling stages with fixed batch sizes or $5N$ with adaptive batch sizes. Its diagnostic and radius satisfy
\begin{align}
 \|\mathsf S_N\|_{L_2}&\le\mu_dM+\omega,
 &\E\mathsf S_N&\le\mu_dM+\omega,
 \label{eq:corrupt-diagnostic-maximum}\\
 \gamma_t^{-1}
 &\le\min\{g^{-1},
        \max\{\Gamma^{-1},\mathsf S_NN^2/(\underline K_tD^2)\}\},
 \label{eq:corrupt-inverse-path}\\
 \left\|\max_t\gamma_t^{-1}\right\|_{L_2}&\le\zeta.
 \label{eq:corrupt-inverse-moment}
\end{align}
If $J=\sum_t I_t$, then
\begin{equation}
 J\le\log_2\left(1+\frac{H_N-h_{\rm init}}{h_{\rm ref}}\right),
 \qquad
 \sum_{t:I_t=1}\underline K_t\le2(H_N-h_{\rm init}).
 \label{eq:corrupt-reference-cost}
\end{equation}
\end{lemma}

\begin{proof}
The trial center and radius are fixed before sampling the diagnostic functions
and directions. The clean
normalized conditional mean is between zero and $\mu_dM$, and its
conditional variance is at most $\nu_dM_2^2/B^{\rm rad}$. Therefore
\[
 \mathsf S_N\le\mu_dM+
 \max_t\left|\widehat q_t^0/\gamma_t^{\rm try}
      -q_{\gamma_t^{\rm try}}(y_t)/\gamma_t^{\rm try}\right|.
\]
The squared maximum is at most the sum of the centered squared
normalized diagnostics. Taking expectations bounds its $L_2$ norm
by $\omega$. Minkowski and Cauchy--Schwarz prove
\eqref{eq:corrupt-diagnostic-maximum}.

The endpoint bound gives
$|\widehat q_t-\widehat q_t^0|\le2\Delta$ pathwise, hence
$\pos{\widehat q_t-2\Delta}\le\pos{\widehat q_t^0}$.
The exact reciprocal of the protected rule is
\[
 \gamma_t^{-1}=\min\left\{g^{-1},
 \max\left\{(\gamma_t^{\rm try})^{-1},
 \frac{\pos{\widehat q_t-2\Delta}}
 {\tau_t\gamma_t^{\rm try}}\right\}\right\}.
\]
At the first iteration of every reference block the trial radius is
$\Gamma$; afterward it is the previous selected radius. Induction,
with the constant threshold in each block, proves
\eqref{eq:corrupt-inverse-path}. Since $\underline K_t\ge h_{\rm ref}$,
the pathwise maximum is at most $g^{-1}$ and at most
$\Gamma^{-1}+\mathsf S_NN^2/(h_{\rm ref}D^2)$. Taking its $L_2$ norm
proves \eqref{eq:corrupt-inverse-moment}. These arguments remain valid
when $g=0$, because all selected radii are positive.

Monotonicity of $\mathsf K_t$ and the trigger give
$\underline K_t\le\mathsf K_{t-1}<2\underline K_t$.
Successive replaced reference values increase by a factor at least
two, and the last value is at most
$\mathsf K_N=h_{\rm ref}+H_N-h_{\rm init}$. This gives the count in
\eqref{eq:corrupt-reference-cost}. For consecutive reference values
$c_{j-1},c_j$, $c_j\ge2c_{j-1}$ implies
$c_j\le2(c_j-c_{j-1})$. Summation proves its reference sum.

Finally $B^{\rm rad}\ge T/(16N)$ and $NB^{\rm rad}\le T/8$.
With fixed batch sizes, $\mathsf U_{\rm pub}\ge7T/8\ge7N$; summing its floors gives
$2\sum_t B_t^{\rm pub}\le\mathsf U_{\rm pub}$. Thus the prescribed
residual is nonnegative and makes the two-center update counts exact.
Moreover $B_t^\iota\ge\mathsf U_{\rm pub}t/(2N(N+1))$ for $\iota\in\{y,x\}$.
With adaptive batch sizes $\mathsf U_{\rm emp}\ge5T/8\ge10N$; the adaptive
surplus calculation gives exactly $\mathsf U_{\rm emp}$ update pairs
with $B_t^\iota\ge\mathsf U_{\rm emp}t/(4N(N+1))$. Its pilots cost $4N$
pairs. Adding the $NB^{\rm rad}$ radial pairs gives exactly $T$ in
both modes. Each sampled function contributes two values. Every iteration uses one
radial and two vector stages with fixed batch sizes, or one radial and four
vector sampling stages, plus one projection.
\end{proof}

Write $\phi_t=f_{\gamma_t}$ and
$e_t^y=g_t^y-\nabla\phi_t(y_t)$, $e_t^x=g_t^x-\nabla\phi_t(x_t)$. Define
$\mathcal E_N=\sum_t t^2\norm{e_t^x-e_t^y}^2$ and retain
$Z_t=t\ip{e_t^y}{u-v_{t-1}}-t\ip{e_t^x}{v_t-v_{t-1}}$.

\begin{lemma}[Biased vector energy and inner products]
\label{lem:corrupt-vector-energy}
The actual corrupted vector means satisfy
\begin{equation}
 \frac{\sqrt{\E\mathcal E_N}}D
 \le\frac{c_{\rm vec}\bar\sigma A}{D\sqrt T}
            +\frac{2\beta\sqrt{\mathsf W_{N,2}}}D,
 \qquad
 \frac1A\left|\E\sum_tZ_t\right|\le2D\beta.
 \label{eq:corrupt-vector-energy}
\end{equation}
\end{lemma}

\begin{proof}
At a fixed center and selected radius, decompose a single corrupted
symmetric record into its conditional mean and centered part. The
conditional bias is $b_\gamma$. The dual-norm calculation gives
$\norm{b_\gamma}\le c_d\Delta/\gamma$.
The corruption vector has pointwise norm at most $d\Delta/\gamma$.
Its centered $L_2$ norm is therefore at most the same bound.
Minkowski with the clean centered record proves that the corrupted
centered standard deviation is at most $\sigma_0+d\Delta/\gamma$.
A mean of conditionally independent records has variance equal to the
single-record variance divided by its batch size.

For $\iota\in\{y,x\}$, write $e_t^\iota=b_t^\iota+n_t^\iota$,
where $n_t^\iota=\zeta_t^\iota$ in the notation of
\eqref{eq:conditional} and is centered conditionally on the history preceding
that update mean. The second history
contains the first centered mean, so
$\E\ip{n_t^x}{n_t^y}=0$. The expected centered energy is bounded
using the deterministic lower batch floors. Since
\eqref{eq:corrupt-inverse-moment} implies
\[
 \left\|\sigma_0+d\Delta\gamma_t^{-1}\right\|_{L_2}
 \le\bar\sigma,
\]
the floor sums give
\[
 \E\sum_t t^2\norm{n_t^x-n_t^y}^2\le
 \begin{cases}
 8\bar\sigma^2A^2/\mathsf U_{\rm pub}
          \le(64/7)\bar\sigma^2A^2/T,\\
 16\bar\sigma^2A^2/\mathsf U_{\rm emp}
          \le32\bar\sigma^2A^2/T.
 \end{cases}
\]
Here the randomness of $B_t^\iota$ is harmless because its deterministic
floor holds after the pilot observations choose it.
Both biases have norm at most $c_d\Delta\gamma_t^{-1}$, so
\[
 \E\sum_t t^2\norm{b_t^x-b_t^y}^2\le4\beta^2\mathsf W_{N,2}.
\]
Minkowski in the weighted vector $L_2$ space proves the energy bound.
This decomposition retains all bias/noise correlations; it cancels
only the cross term between the two centered update means.

The vector $u-v_{t-1}$ precedes $n_t^y$, and $v_t-v_{t-1}$ precedes
$n_t^x$. Their two centered inner products vanish separately.
Each of these displacements has norm at most $D$, and
$\E\norm{b_t^\iota}\le\beta$. Thus
$|\E Z_t|\le2tD\beta$. Summation uses $\sum_t t=A$ and proves
the second claim.
\end{proof}

\subsection{Original-comparator charge, coefficient growth and drift}
\label{subsec:corrupt-growth-drift}

For $\mathrm{reg}=\mathrm{ns}$ put $w_{\mathrm{reg}}=\mu_dMg$ and $q_{\mathrm{reg}}=2MDN/A$.
For an $L$-smooth population one may instead use
$\mathrm{reg}=\mathrm{s}$, $w_{\mathrm{reg}}=\nu_dLg^2$ and $q_{\mathrm{reg}}=LD^2/A$.

\begin{lemma}[Protected local charge]
\label{lem:corrupt-local-charge}
Conditionally on the history before the radial stage,
\begin{equation}
 \E[h_{\gamma_t}(y_t)\mid\mathcal F_t^{\rm rad}]
 \le\tau_t+\eta+4\Delta+w_{\mathrm{reg}}.
 \label{eq:corrupt-local-charge}
\end{equation}
\end{lemma}

\begin{proof}
Write $q=q_{\gamma_t^{\rm try}}(y_t)$,
$p=\pos{\widehat q_t-2\Delta}$, and
$\lambda=1/\max\{1,p/\tau_t\}$. The endpoint bound implies
\[
 q\le\widehat q_t^0+
      |\widehat q_t^0-q|
 \le p+4\Delta+|\widehat q_t^0-q|.
\]
Indeed $\widehat q_t^0\le\widehat q_t+2\Delta\le p+4\Delta$.
If the floor is inactive, Lemma~\ref{lem:corrupt-forward-secant}
gives $h_{\gamma_t}(y_t)\le\lambda q$. Since
$\lambda p\le\tau_t$ and $\lambda\le1$, this is at most
$\tau_t+4\Delta+|\widehat q_t^0-q|$.
If the floor is active, $h_{\gamma_t}(y_t)=h_g(y_t)\le w_{\mathrm{reg}}$.
Either case is bounded by their sum.
The conditional standard deviation of the clean mean is at most
$\sqrt{\nu_d}M_2\gamma_t^{\rm try}/\sqrt{B^{\rm rad}}\le\eta$.
Conditional Cauchy--Schwarz now proves
\eqref{eq:corrupt-local-charge}.
\end{proof}

Define
\begin{align}
 \mathcal V&=\frac{c_{\rm vec}\bar\sigma A}{D\sqrt T}
                         +\frac{2\beta\sqrt{\mathsf W_{N,2}}}D,\nonumber\\
 \mathcal H^\Delta_{\rm ns}
 &=\frac{4c_dM}{\Gamma}
 +\frac{2\sqrt{2c_dM(\mu_dM+\omega)}N}{D}
 +\frac{2MN}{D}+\mathcal V,\nonumber\\
 \mathcal H^\Delta_{\rm s}&=4L+\mathcal V,\qquad
 \ell_{\mathrm{reg}}^\Delta=\log_2(1+\mathcal H^\Delta_{\mathrm{reg}}/h_{\rm ref}),\qquad
 k_N=2/A+3/N^2.
 \label{eq:corrupt-curvature-envelopes}
\end{align}

\begin{lemma}[Protected coefficient and target drift]
\label{lem:corrupt-growth-drift}
For either applicable case $\mathrm{reg}$,
$\E(H_N-h_{\rm init})\le\mathcal H^\Delta_{\mathrm{reg}}$ and $\E J\le\ell_{\mathrm{reg}}^\Delta$.
For
\[
 \mathcal D_N^+=\sum_{t=2}^NA_{t-1}
 \pos{f_{\gamma_t}(x_{t-1})-f_{\gamma_{t-1}}(x_{t-1})},
\]
one has
\begin{equation}
 \frac{\E\mathcal D_N^+}A
 \le\frac{2D^2}{N^2}\mathcal H^\Delta_{\mathrm{reg}}
                      +(\eta+4\Delta+w_{\mathrm{reg}}+q_{\mathrm{reg}})\ell_{\mathrm{reg}}^\Delta.
 \label{eq:corrupt-drift}
\end{equation}
\end{lemma}

\begin{proof}
Fix a realization, and put
$a=c_dM/\Gamma$, $b=c_dM\mathsf S_NN^2/D^2$ and
$\mathfrak H=4a+2\sqrt{2b}$. The radius bound and
$\underline K_t>\mathsf K_{t-1}/2$ imply
\[
 L_t\le a+2b/\mathsf K_{t-1}.
\]
Moreover
$\mathfrak H^2-4a\mathfrak H-8b=8\sqrt2a\sqrt b\ge0$.
Thus $\mathsf K_{t-1}\ge\mathfrak H$ implies
$4L_t\le\mathfrak H$; if $\mathfrak H=0$, the population is
constant and this statement holds directly.

Use the unified shift $\mathfrak B=\mathfrak H+2NM/D$ and
$k_t=\pos{\mathsf K_t-\mathfrak B}$. The scalar balance identity
and $H_t\ge\mathsf K_t$ give, when
$\mathsf K_{t-1}\ge\mathfrak H$ and $k_t>0$,
\[
 (H_t-H_{t-1})D^2
 \le\pos{-k_tr_t^2/2+t\norm{e_t^x-e_t^y}r_t}
 \le\frac{t^2\norm{e_t^x-e_t^y}^2}{2k_t}.
\]
Here the observed secant is bounded by
$A_t\widehat\beta_t\le2L_tr_t^2+
t\norm{e_t^x-e_t^y}r_t$, exactly as in the original scalar-growth
proof. Since $0\le k_t-k_{t-1}\le H_t-H_{t-1}$, it follows that
\begin{equation}
 D^2(k_t^2-k_{t-1}^2)\le t^2\norm{e_t^x-e_t^y}^2.
 \label{eq:corrupt-square-recursion}
\end{equation}
The case $k_t=0$ is immediate. If
$\mathsf K_{t-1}<\mathfrak H$, both true gradients have norm at
most $M$, and $r_t\le D$ bounds the coefficient increment by
$t(2M+\norm{e_t^x-e_t^y})/D$. Then $k_{t-1}=0$ and
$k_t\le t\norm{e_t^x-e_t^y}/D$, proving the same recursion.
Summation gives
\[
 H_N-h_{\rm init}
 \le(\mathfrak H+2NM/D-h_{\rm ref})_+
                        +\sqrt{\mathcal E_N}/D.
\]
Take expectations, use
$\E\sqrt{M\mathsf S_N}\le\sqrt{M(\mu_dM+\omega)}$, and apply
Lemma~\ref{lem:corrupt-vector-energy}. This proves the nonsmooth
envelope. Under common population smoothness every target is
$L$-smooth; the original scalar-growth argument gives
$H_N-h_{\rm init}\le4L+\sqrt{\mathcal E_N}/D$, proving the
smooth envelope. Jensen and \eqref{eq:corrupt-reference-cost} yield
$\E J\le\ell_{\mathrm{reg}}^\Delta$.

When $I_t=0$, the protected radius cannot increase: its candidate is
at most the previous radius and its floor also is at most that radius.
Thus only reference events have positive target drift. On such an
event the original affine model and $f_{\gamma_{t-1}}\ge f$ give
\[
 \pos{f_{\gamma_t}(x_{t-1})-f_{\gamma_{t-1}}(x_{t-1})}
 \le h_{\gamma_t}(y_t)
 +D_{f_{\gamma_t}}(x_{t-1},y_t),
\]
where $D_\phi(x,y)=\phi(x)-\phi(y)-\ip{\nabla\phi(y)}{x-y}$.
The center identity gives
$\norm{x_{t-1}-y_t}\le2D/(t+1)$. The Bregman remainder is
at most $2M\norm{x_{t-1}-y_t}$, or
$L\norm{x_{t-1}-y_t}^2/2$ in the smooth case. After multiplication
by $A_{t-1}$ these bounds are respectively at most $2MDN$ and
$LD^2$, that is $Aq_{\mathrm{reg}}$.

The event and threshold are predictable before the radial
stage. Multiply \eqref{eq:corrupt-local-charge} by $I_tA_{t-1}$
before taking expectations. Since $A_{t-1}\le A$ and the reference
sum is at most $2(H_N-h_{\rm init})$, this gives
\[
 \frac1A\E\sum_t I_tA_{t-1}h_{\gamma_t}(y_t)
 \le\frac{2D^2}{N^2}\E(H_N-h_{\rm init})
                      +(\eta+4\Delta+w_{\mathrm{reg}})\E J.
\]
Adding the Bregman sum and applying the two envelope bounds proves
\eqref{eq:corrupt-drift}. No terminal event count is multiplied by
an earlier random diagnostic error.
\end{proof}

\subsection{The protected finite-budget guarantee}
\label{subsec:proof-corrupt-radial}

\begin{theorem}[Corruption-protected adaptive smoothing]
\label{thm:corrupt-radial}
Under the stated geometry, clean moment and fixed-corruption assumptions,
Algorithm~\ref{alg:corrupt-radial} satisfies, for every fixed $u\in Q$,
\begin{align}
 \E[f(x_N)-f(u)]\le{}&
 \frac{\E[h_{\rm init}R_u^2]}{2A}
 +k_ND^2\mathcal H^\Delta_{\mathrm{reg}}+\frac{h_{\rm ref}D^2}{N^2}\nonumber\\
 &+(\eta+4\Delta+w_{\mathrm{reg}})(1+\ell_{\mathrm{reg}}^\Delta)+q_{\mathrm{reg}}\ell_{\mathrm{reg}}^\Delta+2D\beta,
 \label{eq:corrupt-adaptive-risk}
\end{align}
where $R_u=\norm{u-x^0}$ and the quantities are
\eqref{eq:corrupt-diagnostic-scales}--\eqref{eq:corrupt-curvature-envelopes}.
The nonsmooth branch is always applicable; the smooth branch additionally
assumes population $L$-smoothness on the evaluation domain. The exact counts are
$T$ original pairs, $2T$ values and $N$ projections in $3N$ sequential sampling stages with fixed batch sizes or
$5N$ with adaptive batch sizes.
\end{theorem}

\begin{proof}
The scalar projection potential is a pathwise statement for arbitrary
observed vector records. Its derivation uses the same selected target
at both centers and the original-objective affine model. Those
conditions hold here. Therefore the changing-target potential is
\begin{align*}
 A[f(x_N)-f(u)]\le{}&
 \frac{h_{\rm init}R_u^2}{2}+2D^2(H_N-h_{\rm init})
 +\sum_t t h_{\gamma_t}(y_t)\\
 &+\mathcal D_N^++\sum_t Z_t.
\end{align*}
This is Theorem~\ref{thm:radial-potential} with its proved arbitrary
vector-error inner products retained; it requires no centering
assumption in the pathwise inequality.
Predictable local-charge averaging and
$\underline K_t\le\mathsf K_N$ give
\[
 \frac1A\sum_t t\E h_{\gamma_t}(y_t)
 \le\frac{D^2}{N^2}
 [h_{\rm ref}+\E(H_N-h_{\rm init})]+\eta+4\Delta+w_{\mathrm{reg}}.
\]
Add \eqref{eq:corrupt-drift} and the biased inner-product bound
\eqref{eq:corrupt-vector-energy}, and use
$\E(H_N-h_{\rm init})\le\mathcal H^\Delta_{\mathrm{reg}}$ and $\E J\le\ell_{\mathrm{reg}}^\Delta$.
The three local/drift coefficient contributions are
$2D^2/A+D^2/N^2+2D^2/N^2=k_ND^2$, giving exactly
\eqref{eq:corrupt-adaptive-risk}.
All terms are integrable by Lemmas~\ref{lem:corrupt-controller} and
\ref{lem:corrupt-vector-energy}, and the coefficient envelope.
An initial observed coefficient can be conditioned on its initial
history throughout. The cost claim is
Lemma~\ref{lem:corrupt-controller}.
\end{proof}

\subsection{A vector-bias allowance chosen before sampling}
\label{subsec:corrupt-bias-budget}

Put $\kappa_d=d/c_d$ for every $\Delta\ge0$.
The dimension factor satisfies
\begin{equation}
 \kappa_d^2\le\frac{3d^2}{d+2}\le3d.
 \label{eq:corrupt-kappa-dimension}
\end{equation}
To verify this, put $X=[e]_1$. Rotational symmetry gives
$\E X^2=1/d$ and $\E X^4=3/[d(d+2)]$. For the latter identity in
$d\ge2$, a rotation of $([e]_1,[e]_2)$ gives
$\E [e]_1^4=3\E([e]_1^2[e]_2^2)$; expanding
$\E(\sum_i [e]_i^2)^2=1$ determines both moments. The formula is
also direct in $d=1$. H\"older yields
$\E X^2\le(\E|X|)^{2/3}(\E X^4)^{1/3}$, hence
$c_d^2=d^2(\E|X|)^2\ge(d+2)/3$ and proves
\eqref{eq:corrupt-kappa-dimension}.

\begin{corollary}[Protected bias budget]
\label{cor:corrupt-bias-budget}
Fix $\rho>0$. Choose
\begin{equation}
 g=\frac{c_d\Delta D}{\rho}\le\gammaMax/2.
 \label{eq:corrupt-public-floor}
\end{equation}
For $\Delta>0$,
$D\beta\le\rho$ and
$\bar\sigma\le\sigma_0+\kappa_d\rho/D$.
The following explicit bounds hold, with the actual applicable
$\ell_{\mathrm{reg}}^\Delta$ of \eqref{eq:corrupt-curvature-envelopes}:
\begin{align}
 \E[f(x_N)-f(u)]\le{}&
 \frac{\E[h_{\rm init}R_u^2]}{2A}
 +\frac{[(28+2\ell_{\rm s}^\Delta)L+h_{\rm ref}]D^2}{N^2}
 +\frac{5c_{\rm vec}\sigma_0D}{\sqrt T}\nonumber\\
 &+\left(2+\frac{14}{\sqrt N}
              +\frac{5c_{\rm vec}\kappa_d}{\sqrt T}\right)\rho
 +(\eta+4\Delta+\nu_dLg^2)(1+\ell_{\rm s}^\Delta)
 \label{eq:corrupt-smooth-explicit}
\end{align}
in the smooth case, and
\begin{align}
 \E[f(x_N)-f(u)]\le{}&
 \frac{\E[h_{\rm init}R_u^2]}{2A}
 +\frac{h_{\rm ref}D^2}{N^2}
 +\frac{56c_dMD^2}{\gammaMax N^2}\nonumber\\
 &+\left(\frac{28c_d}{d^{1/4}}+14\sqrt{2c_d}
                       +14+4\ell_{\rm ns}^\Delta\right)\frac{MD}{N}
 +\frac{5c_{\rm vec}\sigma_0D}{\sqrt T}\nonumber\\
 &+\frac{28\sqrt{2c_d}M_2D}{\sqrt T}
 +\left(2+\frac{14}{\sqrt N}
            +\frac{5c_{\rm vec}\kappa_d}{\sqrt T}\right)\rho\nonumber\\
 &+(\eta+4\Delta+\mu_dMg)(1+\ell_{\rm ns}^\Delta)
 \label{eq:corrupt-nonsmooth-explicit}
\end{align}
generally. The diagnostic scale obeys
\begin{equation}
 \eta\le4\sqrt{\nu_d}M_2
 \left(\frac{d^{1/4}D}{\sqrt N\sqrt T}
                         +g\sqrt{N/T}\right).
 \label{eq:corrupt-floor-diagnostic-cost}
\end{equation}
When $\Delta=0$, take $g=0$; then $\beta=0$ and
$\bar\sigma=\sigma_0$, and all displayed $\rho$ terms may be removed.
\end{corollary}

\begin{proof}
For positive $\Delta$, the inverse-radius bound is at most $g^{-1}$,
so \eqref{eq:corrupt-public-floor} gives
$\beta\le\rho/D$ and
$\bar\sigma\le\sigma_0+(d/c_d)\rho/D$.
Since $A\ge N^2/2$, $A/N^2\le1$, and $\mathsf W_{N,2}\le N^3$,
\begin{align*}
 k_ND^2\frac{c_{\rm vec}\bar\sigma A}{D\sqrt T}
 &\le\frac{5c_{\rm vec}\bar\sigma D}{\sqrt T},\\
 k_ND^2\frac{2\beta\sqrt{\mathsf W_{N,2}}}D
 &\le\frac{14D\beta}{\sqrt N}.
\end{align*}
The direct inner-product term is at most $2\rho$.
The smooth deterministic envelope gives
$4Lk_ND^2\le28LD^2/N^2$ and
$q_{\rm s}\ell_{\rm s}^\Delta\le2L\ell_{\rm s}^\Delta D^2/N^2$.
These inequalities prove \eqref{eq:corrupt-smooth-explicit}.

For the nonsmooth branch,
\[
 \sqrt{M(\mu_dM+\omega)}\le\sqrt{M(M+\omega)}
                                     \le M+\omega/2.
\]
Also $k_N\le7/N^2$, $\Gamma\ge\Gamma_0$, and
\[
 \Gamma^{-1}\le2/\gammaMax+N/(d^{1/4}D),\qquad
 \omega\le4\sqrt{\nu_d}M_2N/\sqrt T\le4M_2N/\sqrt T.
\]
The deterministic coefficients are therefore bounded by
$56c_dMD^2/(\gammaMax N^2)$ and
$(28c_d/d^{1/4}+14\sqrt{2c_d}+14)MD/N$.
The $\omega$ term contributes at most
$28\sqrt{2c_d}M_2D/\sqrt T$, and
$q_{\rm ns}\ell_{\rm ns}^\Delta\le4M\ell_{\rm ns}^\Delta D/N$.
Combining these with the preceding vector bounds proves
\eqref{eq:corrupt-nonsmooth-explicit}.
Finally $B^{\rm rad}\ge T/(16N)$ and
$\Gamma=\max\{g,\Gamma_0\}\le g+d^{1/4}D/N$ prove
\eqref{eq:corrupt-floor-diagnostic-cost}.
For zero corruption, \eqref{eq:corrupt-bias-scales} gives exactly
$\beta=0$ and $\bar\sigma=\sigma_0$, so their additional terms vanish.
\end{proof}

At zero corruption the policy attains the clean smooth and nonsmooth
orders of the curvature-reference method. The secant assumptions use
$\nu_d\le1$ instead of the exact radial record's second-moment
constant, and the evaluation domain is halved to accommodate its outer point.
For positive corruption, \eqref{eq:corrupt-floor-diagnostic-cost}
retains the extra sampling cost of raising the trial radius to the
safeguard. It is explicit even for oversized optimizer budgets.
If $g\sqrt N\le\sqrt dD$, the additional diagnostic scale is bounded
by the same $\sqrt dM_2D/\sqrt T$ statistical order, with the displayed
reference logarithm retained. Equation~\eqref{eq:corrupt-kappa-dimension}
shows that the extra vector-variance restriction in
Corollary~\ref{cor:corrupt-accuracy} has order $d$.

\subsection{Sufficient accuracy and corruption levels}
\label{subsec:corrupt-accuracy}

For either branch, let $\mathcal C_{\mathrm{reg}}$ be the right-hand side of
\eqref{eq:corrupt-smooth-explicit} or
\eqref{eq:corrupt-nonsmooth-explicit} after removing its displayed
$\rho$ term and $(4\Delta+w_{\mathrm{reg}})(1+\ell_{\mathrm{reg}}^\Delta)$. Thus
$\mathcal C_{\mathrm{reg}}$ retains startup, deterministic terms, clean vector
variance, the nonsmooth radial-maximum variance term when present,
and the actual $\eta(1+\ell_{\mathrm{reg}}^\Delta)$.

\begin{corollary}[A sufficient corruption allowance]
\label{cor:corrupt-accuracy}
Fix $\varepsilon>0$ and choose $\rho=\varepsilon/64$ in
\eqref{eq:corrupt-public-floor}. Suppose $\mathcal C_{\mathrm{reg}}\le\varepsilon/2$,
the appropriate mode budget holds, and
\begin{equation}
 T\ge\left(\frac{5c_{\rm vec}\kappa_d}{4}\right)^2,
 \qquad
 \Delta\le\frac{\varepsilon\gammaMax}{128c_dD},
 \qquad
 \Delta\le\frac{\varepsilon}{64(1+\ell_{\mathrm{reg}}^\Delta)}.
 \label{eq:corrupt-accuracy-common}
\end{equation}
In the nonsmooth branch additionally assume
\begin{equation}
 \Delta\le
 \frac{\varepsilon^2}
 {1024c_d\mu_dMD(1+\ell_{\rm ns}^\Delta)}.
 \label{eq:corrupt-accuracy-nonsmooth}
\end{equation}
In the smooth branch this condition may instead be replaced by
\begin{equation}
 \Delta\le
 \frac{\varepsilon^{3/2}}
 {256c_dD\sqrt{\nu_dL(1+\ell_{\rm s}^\Delta)}}.
 \label{eq:corrupt-accuracy-smooth}
\end{equation}
A constraint with zero $M$ or zero $L$ in its denominator is omitted
because the corresponding floor charge is zero. Then
$\E[f(x_N)-f(u)]\le\varepsilon$, with the exact costs of
Theorem~\ref{thm:corrupt-radial}.
\end{corollary}

\begin{proof}
The legal-floor condition in \eqref{eq:corrupt-accuracy-common}
gives $g=64c_d\Delta D/\varepsilon\le\gammaMax/2$.
The scalar corruption term is at most
$4\Delta(1+\ell_{\mathrm{reg}}^\Delta)\le\varepsilon/16$.
For the nonsmooth floor, \eqref{eq:corrupt-accuracy-nonsmooth}
gives
\[
 \mu_dMg(1+\ell_{\rm ns}^\Delta)
 =\frac{64c_d\mu_dMD\Delta}{\varepsilon}
                                      (1+\ell_{\rm ns}^\Delta)
 \le\varepsilon/16.
\]
The smooth floor condition analogously gives
$\nu_dLg^2(1+\ell_{\rm s}^\Delta)\le\varepsilon/16$.
The direct bias and its curvature contribution obey
$(2+14/\sqrt N)\rho\le16\rho=\varepsilon/4$.
The additional vector variance is at most
$5c_{\rm vec}\kappa_d\rho/\sqrt T\le\varepsilon/16$ by
the first budget condition. Add these quantities to
$\mathcal C_{\mathrm{reg}}\le\varepsilon/2$: their sum is
$15\varepsilon/16<\varepsilon$. Zero corruption only removes terms.
\end{proof}

The conditions are sufficient interpretations of a finite-budget
expected-risk theorem. They do not provide a data-dependent accuracy
stopping certificate. The nonsmooth allowance has scale
$\varepsilon^2/(c_dMD)$ up to the displayed reference logarithm,
along with the scalar, evaluation domain and sampling conditions. Since
$c_d\le\sqrt d$, it attains the usual accelerated-smoothing
sufficient scale $\varepsilon^2/(\sqrt dMD)$. In the deterministic
Euclidean setting $M=M_2$, this agrees with the allowance obtained
from the bounded-bias analysis of \citet[Appendix A.11]{Gasnikov2022}.
Their stochastic notation uses $M_2$ for sample regularity; here the
smoothed population approximation uses its explicit constant $M$.
The positive result concerns a supplied corruption envelope and the
diameter-based controller. The larger fine-accuracy arbitrary-pair
frontier, original-distance optimality, and joint stage/value/corruption
optimality remain open.

\section{Target-linked initialization and sufficient budgets}
\label{sec:epsilon-initialization}

The initial scalar coefficient can be calibrated using the requested
accuracy, the known diameter and the iteration budget. This calibration
also bounds the curvature-reference logarithms whenever the population
scales are compatible with the chosen budgets. The policy uses no
population regularity parameter in its initialization.

We retain all geometry, conditional vector-variance, forward-secant
moment and evaluation-domain assumptions of
Theorem~\ref{thm:corrupt-radial}. In particular, $D>0$ is a supplied
diameter bound, $d,N,T$ are positive integers, and $M_2$ satisfies the uniform
forward-difference moment bound in Assumption~\ref{ass:forward-moments}.
This scalar moment assumption follows from the
sample regularity assumptions in
Lemma~\ref{lem:corrupt-forward-secant}; a vector-variance bound alone
does not assert it. Fixed batch sizes require $T\ge8N$; adaptive batch sizes require
$T\ge16N$. No preceding observations are used to select the initial
coefficient. Choose a target $\varepsilon>0$ and a constant
$\theta>0$ before sampling, and initialize once with
\begin{equation}
 h_{\rm init}=h_{\rm ref}=h_0
       =\theta\varepsilon N^2/D^2.
 \label{eq:epsilon-initialization}
\end{equation}
The optimizer and all batch and curvature-reference accounts then
persist for the entire execution.

For interpreting the budgets, define
\begin{align}
 r_{\rm s}&=LD^2/N^2,
 &r_{\rm ns}&=d^{1/4}MD/N,\nonumber\\
 a_{\rm ns}&=\sqrt dMD^2/(\gammaMax N^2),
 &r_{\rm st}&=(\sigma_0+\sqrt dM_2)D/\sqrt T,\nonumber\\
 \mathcal B_{\rm s}&=r_{\rm s}+r_{\rm st},
 &\mathcal B_{\rm ns}&=r_{\rm ns}+a_{\rm ns}+r_{\rm st}.
 \label{eq:epsilon-population-budgets}
\end{align}
The smooth branch additionally requires population $L$-smoothness on
the evaluation domain. These quantities enter the analysis; the policy does
not evaluate them.

\subsection{Calibrated clean risk}
\label{subsec:epsilon-clean}

\begin{lemma}[Clean coefficient and diagnostic scales]
\label{lem:epsilon-clean-scales}
For $\Delta=0$ and $g=0$, the envelopes and diagnostic scales of
Theorem~\ref{thm:corrupt-radial} satisfy
\begin{align}
 \mathcal H^\Delta_{\rm s}D^2/N^2
   &\le4r_{\rm s}+6r_{\rm st}\le6\mathcal B_{\rm s},\nonumber\\
 \mathcal H^\Delta_{\rm ns}D^2/N^2
   &\le(6+2\sqrt2)r_{\rm ns}+8a_{\rm ns}+6r_{\rm st}
    \le9\mathcal B_{\rm ns},\nonumber\\
 \eta&\le4r_{\rm st},\qquad
 q_{\rm s}\le2r_{\rm s},\qquad
 q_{\rm ns}\le4r_{\rm ns}.
 \label{eq:epsilon-clean-scales}
\end{align}
They hold in both batch modes, including the integer floors and
residual allocation in the stated execution.
\end{lemma}

\begin{proof}
Put $A=N(N+1)/2$. Both modes have
$B^{\rm rad}=\floor{T/(8N)}\ge T/(16N)$ and
$c_{\rm vec}\le6$. Consequently,
\[
 \omega\le4M_2N/\sqrt T,
 \qquad \eta\le
 4d^{1/4}M_2D/(\sqrt N\sqrt T)\le4r_{\rm st}.
\]
Here $d\ge1$ and $N\ge1$ give
$d^{1/4}/\sqrt N\le\sqrt d$. Also $A/N^2\le1$.
Thus the vector term in $\mathcal H^\Delta_{\mathrm{reg}}D^2/N^2$ is at most
$6\sigma_0D/\sqrt T$, and the smooth deterministic term is
$4r_{\rm s}$.

For the nonsmooth envelope, $\Gamma=\Gamma_0$ and
\[
 \Gamma^{-1}\le2/\gammaMax+N/(d^{1/4}D),\qquad
 \sqrt{M(\mu_dM+\omega)}\le M+\omega/2.
\]
The first term in the nonsmooth envelope is therefore at most
$8a_{\rm ns}+4r_{\rm ns}$, because $c_d\le\sqrt d$.
Its square-root term is at most
\[
 2\sqrt2r_{\rm ns}
   +4\sqrt2d^{1/4}M_2D/\sqrt T
 \le2\sqrt2r_{\rm ns}+6\sqrt dM_2D/\sqrt T.
\]
The $2MN/D$ term contributes at most $2r_{\rm ns}$.
Together with the vector term these inequalities prove the two
coefficient bounds. Since $6+2\sqrt2<9$, their final simplifications
follow. Finally $A\ge N^2/2$ and $d^{1/4}\ge1$ give the two
$q_{\mathrm{reg}}$ bounds. All floor estimates hold already for $T\ge8N$;
the adaptive-batch branch retains its additional $T\ge16N$ condition.
\end{proof}

\begin{theorem}[Target-linked clean initialization]
\label{thm:epsilon-clean}
Under the clean-oracle assumptions and initialization
\eqref{eq:epsilon-initialization}, define
\begin{equation}
 \lambda_{\rm s}^\varepsilon=\log_2\left(1+
             \frac{6\mathcal B_{\rm s}}{\theta\varepsilon}\right),
 \qquad
 \lambda_{\rm ns}^\varepsilon=\log_2\left(1+
             \frac{9\mathcal B_{\rm ns}}{\theta\varepsilon}\right).
 \label{eq:epsilon-calibrated-logs}
\end{equation}
For every fixed original comparator $u\in Q$, the same execution
satisfies every applicable branch of
\begin{align}
 \E[f(x_N)-f(u)]\le{}&2\theta\varepsilon
    +(28+2\lambda_{\rm s}^\varepsilon)r_{\rm s}
    +(46+4\lambda_{\rm s}^\varepsilon)r_{\rm st},
 \label{eq:epsilon-clean-smooth}\\
 \E[f(x_N)-f(u)]\le{}&2\theta\varepsilon
    +(42+14\sqrt2+4\lambda_{\rm ns}^\varepsilon)r_{\rm ns}
    +56a_{\rm ns}+(46+4\lambda_{\rm ns}^\varepsilon)r_{\rm st}.
 \label{eq:epsilon-clean-nonsmooth}
\end{align}
The counts remain exactly $T$ two-value pairs, $2T$ values and $N$
projections in $3N$ sampling stages with fixed batch sizes or $5N$ with adaptive batch sizes.
\end{theorem}

\begin{proof}
At zero corruption the floor, bias and scalar corruption terms in
Theorem~\ref{thm:corrupt-radial} vanish. Its startup and reference
terms obey, because $R_u\le D$ and $2A\ge N^2$,
\[
 \frac{h_0R_u^2}{2A}+\frac{h_0D^2}{N^2}
 \le2\theta\varepsilon.
\]
Lemma~\ref{lem:epsilon-clean-scales} and
\eqref{eq:epsilon-initialization} give
$\ell_{\mathrm{reg}}^\Delta\le\lambda_{\mathrm{reg}}^\varepsilon$. Also $k_N\le7/N^2$. Multiplying the
coefficient bounds of \eqref{eq:epsilon-clean-scales} by $k_NN^2\le7$,
adding $\eta(1+\ell_{\mathrm{reg}}^\Delta)+q_{\mathrm{reg}}\ell_{\mathrm{reg}}^\Delta$, and using their remaining
inequalities gives the displayed component bounds directly.
No new oracle observations are required for the initialization or
its interpretation, so all counts are unchanged.
\end{proof}

\begin{corollary}[Log-free sufficient clean budgets]
\label{cor:epsilon-clean-budgets}
Choose the initialization before sampling
\begin{equation}
 h_0=h_{\rm ref}=\frac{4\varepsilon N^2}{77D^2}.
 \label{eq:epsilon-initialization-77}
\end{equation}
If an applicable population budget satisfies
$\mathcal B_{\mathrm{reg}}\le\varepsilon/77$, then
$\E[f(x_N)-f(u)]<\varepsilon$ for every fixed $u\in Q$.
In particular, sufficient smooth budgets are
\begin{equation}
 N\ge\max\{1,\sqrt{154}\,D\sqrt{L/\varepsilon}\},
 \qquad
 T\ge\left(
  \frac{154(\sigma_0+\sqrt dM_2)D}{\varepsilon}\right)^2,
 \label{eq:epsilon-explicit-smooth-budgets}
\end{equation}
in addition to the appropriate mode condition. Sufficient nonsmooth
budgets are
\begin{align}
 N\ge\max\left\{1,
   \frac{231d^{1/4}MD}{\varepsilon},
   \sqrt{231}\,d^{1/4}D
                     \sqrt{\frac M{\gammaMax\varepsilon}}\right\},
 \qquad
 T\ge\left(
   \frac{231(\sigma_0+\sqrt dM_2)D}{\varepsilon}\right)^2.
 \label{eq:epsilon-explicit-nonsmooth-budgets}
\end{align}
Thus the sufficient orders contain no adaptation logarithm:
\begin{align}
 N_{\rm s}&=O\left(1+D\sqrt{L/\varepsilon}\right),\nonumber\\
 N_{\rm ns}&=O\left(1+\frac{d^{1/4}MD}{\varepsilon}
       +d^{1/4}D\sqrt{\frac M{\gammaMax\varepsilon}}\right),\nonumber\\
 T&=O\left(N+
       \frac{(\sigma_0^2+dM_2^2)D^2}{\varepsilon^2}\right).
 \label{eq:epsilon-clean-complexities}
\end{align}
These are sufficient conditions on budgets fixed before sampling. Accuracy is
interpreted through the population risk bound for the supplied horizon.
\end{corollary}

\begin{proof}
With $\theta=4/77$ and $\mathcal B_{\mathrm{reg}}\le\varepsilon/77$,
\eqref{eq:epsilon-calibrated-logs} gives
$\lambda_{\mathrm{reg}}^\varepsilon\le\log_2(13/4)<7/4$ in both branches. The last strict
inequality follows from $13^4=28561<32768=2^7\cdot4^4$.
The nonsmooth coefficients in
\eqref{eq:epsilon-clean-nonsmooth} are therefore bounded above by
$49+14\sqrt2$, $56$, and $53$. Since $\sqrt2<10/7$,
$49+14\sqrt2<69$. The corresponding smooth coefficients are at most
$63/2$ and $53$. The noninitial terms in either branch are strictly
less than $69\mathcal B_{\mathrm{reg}}$ whenever $\mathcal B_{\mathrm{reg}}>0$.
Their sum with $2\theta\varepsilon=8\varepsilon/77$ is strictly
less than $\varepsilon$. If $\mathcal B_{\mathrm{reg}}=0$, the initial term alone
is at most $8\varepsilon/77<\varepsilon$.

In \eqref{eq:epsilon-explicit-smooth-budgets}, each of $r_{\rm s}$
and $r_{\rm st}$ is at most $\varepsilon/154$. In
\eqref{eq:epsilon-explicit-nonsmooth-budgets}, each of
$r_{\rm ns},a_{\rm ns},r_{\rm st}$ is at most $\varepsilon/231$.
This proves the sufficient conditions. Taking ceilings for integer
budgets, imposing $T\ge8N$ or $T\ge16N$, and using
$(\sigma_0+\sqrt dM_2)^2\le2(\sigma_0^2+dM_2^2)$ gives
\eqref{eq:epsilon-clean-complexities}.
\end{proof}

\subsection{A supplied corruption envelope with bounded reference logs}
\label{subsec:epsilon-corruption}

\begin{corollary}[Target-linked protected corruption budgets]
\label{cor:epsilon-corruption-budgets}
Initialize by \eqref{eq:epsilon-initialization-77}, choose before sampling the
bias allowance $\rho=\varepsilon/64$, and supply $\Delta\ge0$.
Choose
\begin{equation}
 g=\frac{64c_d\Delta D}{\varepsilon},\qquad
 \Gamma=\max\{g,\Gamma_0\}.
 \label{eq:epsilon-corruption-floor}
\end{equation}
Let $\kappa_d=d/c_d$. In addition to the mode condition, assume
\begin{align}
 \mathcal B_{\mathrm{reg}}&\le\varepsilon/256,
 &T&\ge(5c_{\rm vec}\kappa_d/4)^2,\nonumber\\
 \Delta&\le\frac{\varepsilon\gammaMax}{128c_dD},
 &\Delta&\le\varepsilon/160,\nonumber\\
 T&\ge\frac{25600\nu_dM_2^2g^2N}{\varepsilon^2}.
 \label{eq:epsilon-corruption-common}
\end{align}
In the nonsmooth branch additionally assume
\begin{equation}
 \Delta\le\frac{\varepsilon^2}{2560c_d\mu_dMD};
 \label{eq:epsilon-corruption-nonsmooth}
\end{equation}
in the smooth branch it suffices instead that
\begin{equation}
 \Delta\le\frac{\varepsilon^{3/2}}
                  {64\sqrt{40}\,c_dD\sqrt{\nu_dL}}.
 \label{eq:epsilon-corruption-smooth}
\end{equation}
A condition with $M=0$ or $L=0$ in its denominator is omitted.
Then, for every fixed $u\in Q$,
\begin{equation}
 \E[f(x_N)-f(u)]
   <\left(\frac8{77}+\frac{49}{64}\right)\varepsilon
   <\varepsilon,
 \label{eq:epsilon-corruption-risk}
\end{equation}
with exactly the same pair, value, projection and stage counts.
The reference logarithm satisfies $\ell_{\mathrm{reg}}^\Delta<3/2$.
\end{corollary}

\begin{proof}
The finite-radius condition gives $g\le\gammaMax/2$. For positive
$\Delta$, Corollary~\ref{cor:corrupt-bias-budget} gives
$D\beta\le\rho$ and
$\bar\sigma\le\sigma_0+\kappa_d\rho/D$.
Since $\Gamma\ge\Gamma_0$, the clean deterministic envelope bounds
of Lemma~\ref{lem:epsilon-clean-scales} remain valid. The additional
vector contributions give
\begin{equation}
 \mathcal H^\Delta_{\mathrm{reg}}D^2/N^2
 \le9\mathcal B_{\mathrm{reg}}+
       \rho\left(\frac{c_{\rm vec}\kappa_d}{\sqrt T}
                       +\frac2{\sqrt N}\right).
 \label{eq:epsilon-corruption-envelope}
\end{equation}
Indeed, $A/N^2\le1$ bounds the added variance contribution,
and $\mathsf W_{N,2}\le N^3$ bounds the bias contribution.
The conditions on $\mathcal B_{\mathrm{reg}},T$, and $N\ge1$ imply
\[
 \frac{\mathcal H^\Delta_{\mathrm{reg}}D^2}{N^2}
 \le\varepsilon\left(\frac9{256}+
                  \frac1{64}\left(\frac45+2\right)\right)
 =\frac{101\varepsilon}{1280}.
\]
Dividing by $h_0D^2/N^2=4\varepsilon/77$ gives
$\mathcal H^\Delta_{\mathrm{reg}}/h_0\le7777/5120$.
Since $1+7777/5120=12897/5120<8/3<2^{3/2}$, this proves
$\ell_{\mathrm{reg}}^\Delta<3/2$. For $\Delta=0$ the same envelope bound holds with
its nonnegative additional term retained, so this conclusion is
also valid.

Write
$\eta_0=4d^{1/4}M_2D/(\sqrt N\sqrt T)\le4r_{\rm st}$ and
$\eta_g=4\sqrt{\nu_d}M_2g\sqrt{N/T}$.
Equation~\eqref{eq:corrupt-floor-diagnostic-cost} gives
$\eta\le\eta_0+\eta_g$.
Using $\ell_{\mathrm{reg}}^\Delta<3/2$ in the exact component bounds from the clean
proof, the clean noninitial terms are at most
$68\mathcal B_{\mathrm{reg}}$, strictly so when $\mathcal B_{\mathrm{reg}}>0$: in the
nonsmooth branch their coefficients are
at most $48+14\sqrt2<68$, $56$, and $52$, and in the smooth
branch they are at most $31$ and $52$. The initial and reference
terms remain at most $8\varepsilon/77$.

The additional vector bias and its curvature contribution are at
most $(2+14/\sqrt N)\rho\le\varepsilon/4$.
The added vector-variance contribution is at most
$5c_{\rm vec}\kappa_d\rho/\sqrt T\le\varepsilon/16$.
The scalar corruption contributes at most
$4\Delta(1+\ell_{\mathrm{reg}}^\Delta)\le10\Delta\le\varepsilon/16$.
Each branch condition gives $w_{\mathrm{reg}}\le\varepsilon/40$, and hence
$w_{\mathrm{reg}}(1+\ell_{\mathrm{reg}}^\Delta)\le\varepsilon/16$.
Finally the last budget condition in
\eqref{eq:epsilon-corruption-common} gives
$\eta_g(1+\ell_{\mathrm{reg}}^\Delta)\le10\sqrt{\nu_d}M_2g\sqrt{N/T}
\le\varepsilon/16$.
Their total is at most $\varepsilon/2$. The clean noninitial
contribution is strictly less than $68\varepsilon/256$, including
when $\mathcal B_{\mathrm{reg}}=0$. Adding the clean terms proves
\[
 \E[f(x_N)-f(u)]
 <\frac{8\varepsilon}{77}
       +\frac{68\varepsilon}{256}+\frac{\varepsilon}{2}
 =\left(\frac8{77}+\frac{49}{64}\right)\varepsilon.
\]
All positive-corruption contributions vanish or decrease when
$\Delta=0$. The floor and initialization use no extra observations,
so the original exact cost calculation applies.
\end{proof}

The supplied corruption envelope therefore permits sufficient
nonsmooth allowance of order $\varepsilon^2/(c_dMD)$, and smooth
allowance of order $\varepsilon^{3/2}/(c_dD\sqrt L)$, together with
the scalar and finite-radius restrictions displayed above.
Because $\kappa_d^2\le3d$, the added vector-variance condition has
order $d$. The extra forward-diagnostic budget is
$O(M_2^2g^2N/\varepsilon^2)$ and remains explicit; if
$g\sqrt N\le\sqrt dD$, it is bounded by the existing statistical
budget order. The logarithm is bounded by a universal constant in
these sufficient regimes. It bounds the expected number of reference
events; the realized count can be larger. A supplied target calibrates
an expected risk guarantee for fixed budgets, a known diameter and a
supplied corruption envelope. Observed stopping certification,
unknown-corruption adaptation and original-distance complexity remain
separate questions.

\section{Proofs of the main adaptive statements}
\label{sec:main-adaptive-proofs}

\subsection{Error from the initial estimate and selected ball}
\label{subsec:comparison-parameters}
The following quantities express the last row of Table~\ref{tab:comparison}.
They are used only in the error bound. The algorithm estimates a gradient
from its initial batch; it does not require the true gradient or its variance.
Here $A=N(N+1)/2$ and $B_0=\lfloor T/2\rfloor$.

For fixed $u\in Q$, define $a_{R_u}=(R_u-r_0)_+$,
$\Psi(R_u)=R_u^2/2+a_{R_u}^2/8$,
$G_{\rm a}=\norm{\nabla f_{\gamma_{\rm a}}(x^0)}$ and
$\beta_{\rm a}(u)=f_{\gamma_{\rm a}}(u)-f(u)$.
If the anchor record variance is at most $\sigma_{\rm a}^2$,
put $\kappa=h_{\rm ref}pr_0/A$ and
\begin{equation}
 z_{\rm a}=\sqrt{G_{\rm a}^2+\sigma_{\rm a}^2/B_0},\qquad
 S_2=r_0+\frac{2pr_0z_{\rm a}}{\kappa+z_{\rm a}}.
 \label{eq:joint-main-moments}
\end{equation}
The anchor geometry satisfies $\E S^2\le S_2^2$ and $\E S\le S_2$.
Set
\begin{align}
 \mathcal A_p(u)&=\frac{h_{\rm ref}\Psi(R_u)}{A}
 +\frac{G_{\rm a}\Psi(R_u)}{p r_0}+\beta_{\rm a}(u)\nonumber\\
 &\quad+\frac{\sigma_{\rm a}}{\sqrt{B_0}}
       \left(a_{R_u}+\frac{\Psi(R_u)}{p r_0}\right).
 \label{eq:joint-main-price}
\end{align}
\begin{align}
 \mathcal A_{\rm sat}(u)&=
 \frac{h_{\rm ref}\min\{R_u^2,S_2^2\}}{2A}
 +\frac{z_{\rm a}\min\{R_u^2,r_pS_2\}}{2pr_0}\nonumber\\
 &\quad+G_{\rm a}a_{R_u}+\beta_{\rm a}(u),
 \label{eq:joint-main-saturated-price}
\end{align}
and $\mathcal A_*(u)=\min\{\mathcal A_p(u),\mathcal A_{\rm sat}(u)\}$.

The coefficient $h_{\rm ref}$ controls only part of $\mathcal A_*$.
Indeed, $h_{\rm a}/A=h_{\rm ref}/A+\|g_0\|/(pr_0)$, so reducing
$h_{\rm ref}$ leaves a contribution from the initial gradient estimate.
The smoothing and search restriction can also contribute error.
For example, take $d=1$, $Q=[0,R]$, $x^0=0$, and
$F(x,\xi)=f(x)=-x$ throughout the evaluation domain,
with $\delta=0$.
Smoothing preserves $f$ and every initial record equals $-1$.
If $R>r_p$, every iterate satisfies $x_N\le S<r_p$, and hence
$f(x_N)-f(R)\ge R-r_p>0$ for all $N,T$ and $h_{\rm ref}>0$.
Thus fixed $p,r_0$ can leave a nonvanishing error despite arbitrarily small
$h_{\rm ref}$. Corollary~\ref{cor:run5-distance-reference-absorption}
isolates the reference terms that can be made smaller than a target accuracy.

\subsection{Proof of Theorem~\ref{thm:joint-main}}
\label{subsec:proof-joint-main}
Under the assumptions and budgets of Theorem~\ref{thm:joint-main}, the exact bound is
\begin{equation}
\begin{aligned}
 \E[f(x_N)-f(u)]&\le\mathcal A_*(u)\\
 &\quad+\min\{\mathcal J_{\rm s}(S_2),\mathcal J_{\rm ns}(S_2)\}.
\end{aligned}
 \label{eq:joint-main-risk}
\end{equation}
For each applicable case, let $\ell_{\mathrm{reg}}^0=\ell_{\mathrm{reg}}(2r_0)$.
For an $L$-smooth population on the evaluation domain, put
\begin{equation}
 \mathcal J_{\rm s}(S_2)=\mathcal U_{\rm s}(2S_2,\ell_{\rm s}^0).
 \label{eq:joint-main-smooth-cost}
\end{equation}
Otherwise set $\mathcal J_{\rm s}(S_2)=+\infty$.
If $\gammaMax\ge2d^{1/4}r_p/N$, put
$\mathcal J_{\rm ns}(S_2)=\mathcal U_{\rm ns}(2S_2,\ell_{\rm ns}^0)$.
Otherwise put
\begin{equation}
\begin{aligned}
 \mathcal J_{\rm ns}(S_2)&=\min\Big\{
       \mathcal U_{\rm ns}(2r_p,\ell_{\rm ns}^0),\\
       &\qquad\mathcal V_{\rm ns}(S_2,\ell_{\rm ns}^0)\Big\}.
\end{aligned}
 \label{eq:joint-main-nonsmooth-cost}
\end{equation}
The cost $\mathcal V_{\rm ns}$ is defined in
\eqref{eq:run5-distance-finite-halo-cost}.
The envelopes \eqref{eq:adaptive-main-Hns}--\eqref{eq:adaptive-main-U}
use $T'$ and $B^{\rm rad}=\floor{T'/(8N)}$.

In \eqref{eq:joint-main-order}, $\widetilde O$ suppresses universal constants
and at most the factor $1+\max_{\mathrm{reg}\text{ applicable}}\ell_{\mathrm{reg}}^0$.
The nonsmooth case is always applicable; the smooth case is included only
when $L<\infty$. These logarithms retain their dependence on the chosen
$h_{\rm ref}$.

\begin{proof}
The anchor rule \eqref{eq:joint-main-anchor} is
\eqref{eq:run5-distance-anchor}, with the same anchor batch,
$Q_S$, reserve $D_S=2S$ and initial coefficient $h_{\rm a}$.
Condition on the full anchor history. The execution of
Algorithm~\ref{alg:threshold} with budget $T'$ agrees with
Appendix~\ref{subsec:joint-offset-rule}: its projected updates use step size
$a_t/H_{t-1}$ and its radius selection tracks
$\mathsf K_{t-1}=h_{\rm ref}+H_{t-1}-h_{\rm a}$.
The conditional vector and radial moment assumptions hold at the selected
anchor-selected centers and radii under this conditioning.

Lemma~\ref{lem:run5-distance-moment} proves
$\E S^2\le S_2^2$ and $\E S\le S_2$ from the anchor second moment.
Lemma~\ref{lem:run5-distance-clipping} retains the clipped initial
comparator distance and proves both anchor prices, hence their minimum
\eqref{eq:joint-main-saturated-price}. Its full proof bounds the fourth
power of the bounded radius by $r_p^2 S^2$; it assumes no fourth
anchor-record moment. The prices are exactly
\eqref{eq:run5-distance-old-price}--\eqref{eq:run5-distance-best-price}.

Theorem~\ref{thm:run5-distance-joint}, whose proof starts from the
complete offset-controller potential, now applies. Its envelopes and
costs \eqref{eq:run5-distance-joint-envelopes}--\eqref{eq:run5-distance-joint-cost}
are \eqref{eq:adaptive-main-Hns}--\eqref{eq:adaptive-main-U}, with
budget $T'$ and $B^{\rm rad}=\floor{T'/(8N)}$.
Its logarithms are the deterministic $\ell_{\mathrm{reg}}(2r_0)$, frozen before
averaging any anchor-selected geometry. The smooth conclusion is
\eqref{eq:run5-distance-joint-smooth}. Under
$\gammaMax\ge2d^{1/4}r_p/N$, the evaluation domain is untruncated on the whole
possible radius range and \eqref{eq:run5-distance-joint-nonsmooth}
applies. Otherwise \eqref{eq:run5-distance-joint-finite-halo} gives
the stated minimum with \eqref{eq:run5-distance-finite-halo-cost}.
These give the costs \eqref{eq:joint-main-smooth-cost}--\eqref{eq:joint-main-nonsmooth-cost}.
Both applicable bounds hold for the same execution, so their minimum gives
\eqref{eq:joint-main-risk}. Without population smoothness, only the nonsmooth
bound is used, as encoded by $\mathcal J_{\rm s}(S_2)=+\infty$.

The initial batch consumes $B_0$ sampled functions in one stage and $T'\ge T/2$.
Thus the condition for fixed batch sizes gives $T'\ge8N$, and the adaptive-batch condition gives
$T'\ge16N$. Lemma~\ref{lem:threshold-accounting} spends exactly $T'$
later sampled functions in $3N$ or $5N$ stages. Adding the anchor gives $T$ pairs,
$2T$ values and $3N+1$ or $5N+1$ stages. Each optimizer step uses one
intersection projection. All states are initialized once and persist.

It remains to derive the explicit rate \eqref{eq:joint-main-order}.
The budgets give $B^{\rm rad}\ge T/(32N)$, and $\chi_d\le2$ implies
\[
 \omega\le\frac{8M_2N}{\sqrt T},\qquad
 \eta(2S_2)\le\frac{16d^{1/4}M_2S_2}{\sqrt{NT}}.
\]
Here $\omega$ and $\eta$ are defined in Appendix~\ref{sec:run5-distance};
the second inequality follows from
$\Gamma(2S_2)\le2d^{1/4}S_2/N$ for every $\gammaMax>0$.
Also $k_N\le7/N^2$, $k_NA\le5$, and
$c_{\rm vec}\le4\sqrt2$.
Substitution into the exact costs yields
\begin{align*}
 \mathcal U_{\rm s}(2S_2,\ell_{\rm s}^0)
 &\le\frac{4h_{\rm ref}S_2^2}{N^2}
      +\frac{80\sigma_0S_2}{\sqrt T}
      +(1+\ell_{\rm s}^0)\left(
        \frac{112LS_2^2}{N^2}
        +\frac{16d^{1/4}M_2S_2}{\sqrt{NT}}\right),\\
 \mathcal V_{\rm ns}(S_2,\ell_{\rm ns}^0)
 &\le\frac{4h_{\rm ref}S_2^2}{N^2}
      +\frac{80\sigma_0S_2}{\sqrt T}\\
 &\quad+(1+\ell_{\rm ns}^0)\left(
        \frac{112\sqrt d\,MS_2^2}{\gammaMax N^2}
        +\frac{128d^{1/4}MS_2}{N}
        +\frac{176d^{1/4}M_2S_2}{\sqrt T}\right).
\end{align*}
For the first inequality, the curvature and drift coefficients obey
$112+8\ell\le112(1+\ell)$ for $\ell\ge0$.
For the second, they obey
$84+28\sqrt2+8\ell\le128(1+\ell)$, while the radial-noise coefficient
satisfies $112\sqrt2+16(1+\ell)/\sqrt N\le176(1+\ell)$.
The nonsmooth bound using $\mathcal V_{\rm ns}$ holds for every
$\gammaMax>0$ by \eqref{eq:run5-distance-joint-finite-halo}.
Taking the minimum of the applicable bounds and bounding their logarithmic
factors by $1+\max_{\mathrm{reg}\text{ applicable}}\ell_{\mathrm{reg}}^0$
proves \eqref{eq:joint-main-order}. Without population smoothness, only
the nonsmooth bound is used.
\end{proof}

\subsection{Proof of Theorem~\ref{thm:protected-main}}
\label{subsec:proof-protected-main}
\begin{proof}
The initial state, predictable reference, forward records and protected
radius in Algorithm~\ref{alg:protected-main} agree exactly with
Algorithm~\ref{alg:corrupt-radial}. The fixed batch-size lower bounds and residual, or
pilot and update budgets, are \eqref{eq:corrupt-public-counts} and
\eqref{eq:corrupt-paid-counts}. Thus their executions coincide when
coupled to the same sampled functions and directions, and
Theorem~\ref{thm:corrupt-radial} applies.

Corollary~\ref{cor:corrupt-bias-budget} expands that finite theorem into
\eqref{eq:corrupt-smooth-explicit} and
\eqref{eq:corrupt-nonsmooth-explicit}. For $u=x^\star$, feasibility
implies $R\le D$, and $2A=N(N+1)\ge N^2$. Hence the startup and
reference terms have sum at most $2h_0D^2/N^2$.
The dimensional inequalities $c_d\le\sqrt d$,
$\sqrt{c_d}\le d^{1/4}$ and
$(d/c_d)^2\le3d$ bound all coefficients by the displayed dimension
orders. The uniform vector assumption gives
$\sigma_0\le\sqrt{Cd}M_2$. Since $d,N\ge1$,
$d^{1/4}/\sqrt N\le\sqrt d$. The diagnostic bound
\eqref{eq:corrupt-floor-diagnostic-cost} is therefore at most a
constant times $\sqrt dM_2D/\sqrt T+M_2g\sqrt{N/T}$.
Also $2+14/\sqrt N\le16$, so the bias-budget terms are bounded by a
constant times $\rho(1+\sqrt{d/T})$.

In the smooth branch the remaining deterministic terms have orders
$LD^2/N^2$ and $Lg^2$, multiplied only by
$1+\ell_{\rm s}^\Delta$. In the nonsmooth branch they have orders
$d^{1/4}MD/N$, $\sqrt dMD^2/(\gammaMax N^2)$ and $Mg$, multiplied
only by $1+\ell_{\rm ns}^\Delta$. Both share the previously bounded
sampling, initialization, direct $\Delta$ and bias-budget terms.
Taking the smaller applicable branch proves
\eqref{eq:protected-main-order}. Without smoothness only the nonsmooth
branch and its finite logarithm are used. At $\Delta=0$, the exact
finite bound has $g=\beta=0$, $\bar\sigma=\sigma_0$ and no extra
$\rho$ terms; this is the stated zero-corruption specialization.
Lemma~\ref{lem:corrupt-controller} supplies the exact resource counts.
\end{proof}

\subsection{Proof of Corollary~\ref{cor:protected-main-clean}}
\label{subsec:proof-protected-main-clean}
\begin{proof}
Take $\Delta=g=0$ in Theorem~\ref{thm:protected-main} and use its
exact zero-corruption specialization. The order is
$\widetilde O(h_0D^2/N^2+\mathcal B_{N,T})$ with only the applicable
reference logarithm suppressed. The assumed initialization gives
$2h_0D^2/N^2\le\mathcal B_{N,T}$, so the explicit sum is at most a
universal constant times $\mathcal B_{N,T}$. This proves the first
claim while retaining its logarithmic reference dependence.
The algorithm requires strict positivity because both its projected
step and radial threshold divide by positive coefficients, but none of
its assumptions imposes a further positive lower bound. Finally,
$h_0\le\varepsilon N^2/(8D^2)$ implies
$2h_0D^2/N^2\le\varepsilon/4$ directly. This choice made before sampling needs no
population smoothness, Lipschitz scale or variance, and does not remove
the logarithms in the finite theorem.
\end{proof}

\section{A tunable distance frontier}
\label{sec:frontier-distance}

We choose the initial scalar coefficient from an independent anchor
batch and retain one optimizer trajectory. A parameter $p$ chosen before sampling trades
deterministic affine progress for the cost of learning the distance.
The geometric primitive is exact projection onto
$Q\cap\overline{\mathbb B}(x^0,S)$ for each observed $S>0$.
This primitive is charged separately from projection onto $Q$.
Its arithmetic complexity depends on the feasible set.

For affine objectives, $p$ trades
deterministic progress for a noise coefficient of order $p$, matching the
lower bound for $p\ge5$ under the stated progress requirement
(Theorem~\ref{thm:frontier-pareto-lower}). The corresponding tradeoff for
curved objectives remains open. A stronger initial-progress requirement can
force an $N^2$ factor (Corollary~\ref{cor:startup-price}). These lower bounds hold even for methods that query arbitrary points and observe each sampled function in full.

\subsection{Anchor-dependent initialization}
\label{subsec:frontier-rule}

Let $Q$ be compact and convex, $x^0\in Q$, $N\ge1$ and
$A=A_N=N(N+1)/2$. Fix a legal radius $\gamma>0$ and write
$\phi=f_\gamma$. Assume that $\phi$ is convex and differentiable,
has an $L_\gamma$-Lipschitz gradient on $Q$, and satisfies
$f\le\phi$. Put $\beta_\gamma(u)=\phi(u)-f(u)$.
Uncorrupted gradient estimates have conditional mean $\nabla\phi(z)$ and conditional
variance at most $\sigma_\gamma^2$, uniformly over legal predictable
centers. Each estimate uses a newly sampled function and direction,
independent of the preceding draws.

Choose references $r_0,h_{\rm ref}>0$ and a tradeoff
$p\ge1$ before sampling. Draw $B_0=\lfloor T/2\rfloor$ anchor records at $x^0$,
and denote their mean by $g_0$. Set
\begin{align}
 T'&=T-B_0,\qquad
 h_{\rm a}=h_{\rm ref}+\frac{A\norm{g_0}}{p r_0},
 \qquad c_{\rm a}=2A/h_{\rm a},\nonumber\\
 S&=r_0+c_{\rm a}\norm{g_0},\qquad
 Q_S=Q\cap\overline{\mathbb B}(x^0,S),\qquad D_S=2S.
 \label{eq:frontier-anchor-rule}
\end{align}
Initialize $x_0=v_0=x^0$ and $H_0=h_{\rm a}$ once.
Use the scalar updates \eqref{eq:center-y}--\eqref{eq:coefficient}
on $Q_S$ with the fixed reserve $D_S$ and budget $T'$.
All later labels and directions are independent of the anchor draws.
Their centers and record values may depend on the anchor; their
uniform moment bounds hold conditionally on the preceding history.
The coefficient, reserve and feasible set are selected before the first
optimizer step, and the optimizer state persists throughout.

For fixed batch sizes require $T\ge8N$ and use
\begin{equation}
 B_t^y=B_t^x=1+\left\lfloor
 \frac{(T'-2N)t}{N(N+1)}\right\rfloor.
 \label{eq:frontier-public-batches}
\end{equation}
Add any unused labels to the final second-center batch. The execution
spends exactly $T$ pairs and $2T$ scalar values, with $2N+1$ sampling
stages and $N$ ball-intersection projections.
Alternatively, require $T\ge16N$ and use the adaptive batch allocation of Algorithm~\ref{alg:scalar},
with total budget $T'$, including its preliminary samples, and
$s_0>0$ chosen before sampling. The counts are $T$ pairs, $2T$ values, $4N+1$ stages and
$N$ ball-intersection projections. Neither version uses an input value
of $R,D,L,M,M_2$ or the variance.

The initialization satisfies the deterministic envelope
\begin{equation}
 0<S\le r_p:=r_0(1+2p),\qquad h_{\rm a}\ge h_{\rm ref}>0.
 \label{eq:frontier-radius-envelope}
\end{equation}
Indeed $A\norm{g_0}/h_{\rm a}\le p r_0$.
For a zero anchor, $S=r_0$ and $h_{\rm a}=h_{\rm ref}$.

\subsection{Conditional fixed-radius geometry}
\label{subsec:frontier-upper}

For fixed $u\in Q$, define
\begin{equation}
 R_u=\norm{u-x^0},\qquad a_{R_u}=(R_u-r_0)_+,
 \qquad \Psi(R_u)=R_u^2/2+a_{R_u}^2/8,
 \qquad G_0=\norm{\nabla\phi(x^0)}.
 \label{eq:frontier-distance-shape}
\end{equation}
Conditionally on the anchor, let
$u_S=x^0+\min\{1,S/R_u\}(u-x^0)$ for $R_u>0$ and $u_S=x^0$
otherwise. Convexity gives $u_S\in Q_S$.

\begin{lemma}[Conditional fixed-radius geometry]
\label{lem:frontier-conditional}
Under the preceding assumptions, set $\nu=8\sqrt2$ for fixed
batch sizes and $\nu=16\sqrt2$ for adaptive batch sizes. Then
\begin{align}
 \E[\phi(x_N)-\phi(u_S)\mid\text{anchor}]
 \le{}&\frac{h_{\rm a}\min\{R_u^2,S^2\}}{2A}
 +\frac{32L_\gamma S^2}{A}
 +\frac{\nu\sigma_\gamma S}{\sqrt{T'}}.
 \label{eq:frontier-conditional-risk}
\end{align}
\end{lemma}
\begin{proof}
Condition on the complete anchor batch. Its realized $h_{\rm a}$,
$S$, $Q_S$ and $u_S$ are fixed, positive or finite as appropriate.
Apply Theorem~\ref{thm:potential} and Lemma~\ref{lem:growth} to
$\phi$ itself, with zero comparator smoothing charge. The initial
potential is $h_{\rm a}\norm{u_S-x^0}^2/2$, where
$\norm{u_S-x^0}^2=\min\{R_u^2,S^2\}$.
Write $e_t^y=g_t^y-\nabla\phi(y_t)$, $e_t^x=g_t^x-\nabla\phi(x_t)$ and
$\mathcal E_N=\sum_tt^2\norm{e_t^x-e_t^y}^2$.
The scalar coefficient bound is
$H_N-h_{\rm a}\le4L_\gamma+\sqrt{\mathcal E_N}/D_S$.
The fixed batch-size lower bounds satisfy
$B_t^\iota\ge T't/[2N(N+1)]$ for $\iota\in\{y,x\}$ because $T'\ge4N$.
The first mean is known before the second batch is drawn, so conditional
centering gives $\E[\ip{e_t^x}{e_t^y}\mid\text{anchor}]=0$.
Thus $\E[\mathcal E_N\mid\text{anchor}]\le8\sigma_\gamma^2A^2/T'$.
The minimum adaptive batch sizes give instead $32\sigma_\gamma^2A^2/T'$.
Both error inner products in the potential have zero conditional
expectation, since their distance factors are known before the
corresponding batches are drawn. The remaining curvature contribution is
\[
 8L_\gamma D_S^2/A+
 2D_S\sqrt{\E[\mathcal E_N\mid\text{anchor}]}/A.
\]
Substitution of $D_S=2S$ proves \eqref{eq:frontier-conditional-risk}.
Compactness bounds the distance factors, and the positive reference
and scalar energy bound ensure integrability.

For the accounting with fixed batch sizes,
$2\sum_tB_t^y\le2N+2(T'-2N)\sum_tt/[N(N+1)]=T'$.
The added final labels spend the residual and preserve every floor.
There are two vector stages per optimizer step and one anchor stage.
The adaptive sample count follows Lemma~\ref{lem:paid-accounting} with budget
$T'$. Each record contributes one pair and two values.
\end{proof}

The clipping calculation in Lemma~\ref{lem:run5-distance-clipping}
combines this conditional estimate with the original comparator.
Theorem~\ref{thm:run5-distance-fixed} gives the resulting bound using
the anchor's radius moments and the smaller of two bounds on the
initial-estimate and search-restriction error.

\begin{corollary}[A uniform noise multiplier]
\label{cor:frontier-noise-multiplier}
Under Theorem~\ref{thm:run5-distance-fixed}, with
$\gamma_{\rm a}=\gamma$ and $\sigma_{\rm a}=\sigma_\gamma$, define
\begin{equation}
 \mathcal H_p(R)=R+r_0+R^2/(p r_0),\qquad
 P_{\rm up}(p)=\sqrt3+\nu\sqrt2(1+2p).
 \label{eq:frontier-noise-shape}
\end{equation}
Then
\begin{align}
 \E[f(x_N)-f(u)]\le{}&
 \frac{h_{\rm ref}\Psi(R_u)}{A}
 +\frac{G_0\Psi(R_u)}{p r_0}
 +\frac{32L_\gamma r_p^2}{A}\nonumber\\
 &+P_{\rm up}(p)\frac{\sigma_\gamma\mathcal H_p(R_u)}{\sqrt T}
 +\beta_\gamma(u).
 \label{eq:frontier-noise-risk}
\end{align}
\end{corollary}

\begin{proof}
Theorem~\ref{thm:run5-distance-fixed} gives
$\mathcal A_*(u)\le\mathcal A_p(u)$ and $S_2\le r_p$.
The admissible budgets give $B_0\ge T/3$ and $T'\ge T/2$.
Since $a_R\le R$ and $\Psi(R)\le5R^2/8$, the anchor noise in
$\mathcal A_p(u)$ is at most
$\sqrt3\sigma_\gamma\mathcal H_p(R)/\sqrt T$.
The remaining noise is at most
$\nu\sqrt2(1+2p)\sigma_\gamma r_0/\sqrt T$, and
$\mathcal H_p(R)\ge r_0$. This proves the claim.
\end{proof}

\subsection{A lower bound under an affine error guarantee}
\label{subsec:frontier-lower}

Before sampling, fix $r_0,h_{\rm ref}>0$, $p\ge1$, $N\ge1$ and an integer
$n\ge\max\{4,16N\}$. Set
\begin{equation}
 R_0=2p r_0/5,\quad S=r_0+R_0,\quad
 Q=[0,S],\quad x^0=r_0,\quad T=n,
 \qquad m\ge h_{\rm ref}p r_0/A.
 \label{eq:frontier-lower-inputs}
\end{equation}
The deterministic law returns $F_0(x)=m(S-x)$ under every label.
The rare law returns
\begin{equation}
 F(x;\xi)=\begin{cases}
 m(S-x),&\text{with probability }1-1/n,\\
 2nm x,&\text{with probability }1/n.
 \end{cases}
 \label{eq:frontier-rare-law}
\end{equation}
The rare population has gradient $m(1+1/n)$, optimizer zero and
distance $R=r_0$. Its variance and sample-gradient moment are
\begin{equation}
 \sigma_1^2=m^2(4n-3/n-1/n^2)\le4nm^2,
 \qquad M_{2,1}^2=m^2(4n+1-1/n).
 \label{eq:frontier-rare-moments}
\end{equation}
Every sample is globally affine, so smoothing bias and curvature vanish.

\begin{theorem}[The cost of the relaxed affine error guarantee]
\label{thm:frontier-pareto-lower}
Let any algorithm draw at most $n$ independent sampled functions and return
$\widehat x\in Q$, with the same known inputs and internal-randomization
law on both instances. Arbitrary adaptive query endpoints are allowed,
including observation of each entire realized sample function.
Suppose its deterministic-law guarantee is
\begin{equation}
 \E_0[m(S-\widehat x)]
 \le\frac{h_{\rm ref}\Psi(R_0)}{A}
 +\frac{m\Psi(R_0)}{p r_0}.
 \label{eq:frontier-affine-contract}
\end{equation}
Suppose also its rare-law guarantee, for some $P\ge0$, is
\begin{align}
 \E_1[f_1(\widehat x)-f_1(0)]\le{}&
 \frac{h_{\rm ref}\Psi(r_0)}{A}
 +\frac{m(1+1/n)\Psi(r_0)}{p r_0}\nonumber\\
 &+P\frac{\sigma_1\mathcal H_p(r_0)}{\sqrt n}.
 \label{eq:frontier-local-contract}
\end{align}
Writing $q_n=(1-1/n)^n$, necessarily
\begin{equation}
 P\ge\frac{[q_n(1+1/n)(1+p/5)-(2+1/n)/(2p)]_+}
                {2(2+1/p)}.
 \label{eq:frontier-exact-lower}
\end{equation}
In particular $P\ge p/88$ for $p\ge5$.
\end{theorem}

\begin{proof}
Since $h_{\rm ref}/A\le m/(p r_0)$ and $\Psi(R_0)\le5R_0^2/8$,
the deterministic-law error bound gives
\[
 \E_0(S-\widehat x)\le2\Psi(R_0)/(p r_0)
 \le5R_0^2/(4p r_0)=R_0/2.
\]
Thus $\E_0\widehat x\ge r_0(1+p/5)$.
Generate $n$ independent draws from the second law in advance,
including unused draws if the algorithm stops early. Couple its internal random seed across
the two instances. On the event that all draws select $F_0$,
the entire corresponding functions agree. Induction through the
adaptive transcript gives identical endpoints, batch allocations,
stopping decisions and output. This reasoning also permits repeated
evaluations under one label and full realized-function observation.
The event is independent of internal randomness and has probability
$q_n$. The rare gap is the nonnegative quantity
$m(1+1/n)\widehat x$. Consequently
\begin{equation}
 \E_1[f_1(\widehat x)-f_1(0)]
 \ge q_n m(1+1/n)r_0(1+p/5).
 \label{eq:frontier-coupled-lower}
\end{equation}

Since $\Psi(r_0)=r_0^2/2$, the deterministic terms in
\eqref{eq:frontier-local-contract} are at most
$m r_0(2+1/n)/(2p)$. Moreover
$\mathcal H_p(r_0)=(2+1/p)r_0$ and $\sigma_1\le2m\sqrt n$,
so its noise term is at most $2P m r_0(2+1/p)$.
Comparison with \eqref{eq:frontier-coupled-lower} proves
\eqref{eq:frontier-exact-lower}.

For completeness, let $g(x)=x\log(1-1/x)$ on $x\ge2$ and
$z=1/(x-1)$. Then
$g'(x)=z-\log(1+z)\ge0$, because
$\log(1+z)=\int_0^z(1+s)^{-1}\,ds\le z$.
Thus $q_n\ge\exp(g(2))=1/4$.
For $n\ge4$ the numerator of \eqref{eq:frontier-exact-lower}
is at least $1/4+p/20-9/(8p)$.
For $p\ge5$ its first and last terms have nonnegative sum,
and the denominator is at most $2(2+1/5)=22/5$.
Hence $P\ge(p/20)/(22/5)=p/88$.
\end{proof}

The upper method satisfies \eqref{eq:frontier-affine-contract} because
the deterministic affine law has zero variance, curvature and smoothing
bias. Corollary~\ref{cor:frontier-noise-multiplier} gives its rare-law
guarantee with $P=P_{\rm up}(p)=O(p)$. Thus the order of this multiplier
is attained under the displayed affine error guarantee. This conclusion
does not characterize the optimal frontier for general curved objectives.
The law has finite second moments at each horizon and supplies no common
known noise envelope across the two instances.

\section{Joint localization and radial selection}
\label{sec:joint-anchor-radial}

We combine anchor-dependent initialization with a persistent radial
controller. The distance tradeoff and the ball-intersection projector
are explicit. A scalar offset prevents the anchor coefficient from
creating a nonvanishing radial reference term.

Let $Q$ be compact and convex, $x^0\in Q$, and let a known evaluation domain of
width $\gammaMax>0$ be queryable. The population is convex and
$M$-Lipschitz on the convex evaluation domain. Clean vector records are conditionally
unbiased with variance at most $\sigma_0^2$, uniformly over the legal
positive radii and exposed centers. We assume the radial regularity of
Lemma~\ref{lem:radial-moment}, uniformly over these queries, with second
moment parameter $M_2^2$. In particular its same-sample radial record has
mean $h_\gamma(y)$ and second moment at most
$\chi_d M_2^2\gamma^2$, where
$\chi_d=2d^2/[(d+1)(d+2)]\le2$.
At each sampling stage, the sampled functions, directions and radial
variables are independent of the preceding observations. Exact projection onto each
$Q\cap\overline{\mathbb B}(x^0,S)$ is an additional projection
operation and is counted separately from projection onto $Q$.

Before sampling, fix $r_0,h_{\rm ref}>0$, $p\ge1$, $N\ge1$, $T$, and a
legal anchor radius $\gamma_{\rm a}>0$. Put $A=A_N$, draw
$B_0=\lfloor T/2\rfloor$ anchor records at $(x^0,\gamma_{\rm a})$,
and let $g_0$ be their mean. Their variance may be bounded by the sharper
$\sigma_{\rm a}^2\le\sigma_0^2$. Set
\begin{align}
 T'&=T-B_0,\qquad
 h_{\rm a}=h_{\rm ref}+A\norm{g_0}/(p r_0),\nonumber\\
 S&=r_0+2A\norm{g_0}/h_{\rm a},\qquad
 Q_S=Q\cap\overline{\mathbb B}(x^0,S),\qquad D_S=2S,
 \label{eq:joint-anchor-initialization}\\
 r_p&=r_0(1+2p),\qquad r_0\le S\le r_p.\nonumber
\end{align}
Initialize the scalar optimizer once with $H_0=h_{\rm a}$ and
$x_0=v_0=x^0$. Use the original scalar updates on $Q_S$ with reserve
$D_S$. The projected step size is $a_t/H_{t-1}$.

\subsection{The offset radial controller}
\label{subsec:joint-offset-rule}

Define an observed reference coefficient
\begin{equation}
 \mathsf K_t=H_t-h_{\rm a}+h_{\rm ref},\qquad
 \mathsf K_0=h_{\rm ref}.
 \label{eq:joint-offset-coefficient}
\end{equation}
It is positive, nondecreasing, has the same increments as $H_t$,
and satisfies $\mathsf K_t\le H_t$.
Set $\Gamma_S=\min\{\gammaMax,d^{1/4}D_S/N\}$,
$\gamma_0=\Gamma_S$, and $\underline K_0=h_{\rm ref}$.
At each iteration, compute the usual first center $y_t$ and, before sampling any
record in that iteration, set
\begin{align}
 I_t&=\boldsymbol1\{\mathsf K_{t-1}\ge2\underline K_{t-1}\},\nonumber\\
 \underline K_t&=\begin{cases}\mathsf K_{t-1},&I_t=1,\\\underline K_{t-1},&I_t=0,
 \end{cases}\qquad
 \gamma_t^{\rm try}=\begin{cases}\Gamma_S,&I_t=1,\\
 \gamma_{t-1},&I_t=0,
 \end{cases}\nonumber\\
 \tau_t&=\underline K_tD_S^2/N^2,\qquad
 \gamma_t=\gamma_t^{\rm try}/
          \max\{1,[\widehat h_t]_+/\tau_t\}.
 \label{eq:joint-radial-controller}
\end{align}
Here $\widehat h_t$ is the mean of
$B^{\rm rad}=\lfloor T'/(8N)\rfloor$ conditionally independent radial records at
$(y_t,\gamma_t^{\rm try})$. Use this selected radius in both
vector update means, make the original projected step and averaging update,
and perform the original scalar post-update. Then compute
\eqref{eq:joint-offset-coefficient}. All optimizer vectors, indices,
weights, coefficients and allocation accounts persist.

With adaptive batch sizes, require $T\ge32N$ and choose
$s_0>0$ before sampling. Use the original adaptive batch allocator with
$\mathsf U=T'-NB^{\rm rad}-4N$.
It spends exactly $T$ original pairs and $2T$ values in $5N+1$
sampling stages, including the anchor, with $N$ ball-intersection
projections. With fixed vector batch sizes require $T\ge16N$,
set $\mathsf U_{\rm pub}=T'-NB^{\rm rad}$, and use
\[
 B_t^y=B_t^x=1+\left\lfloor
 \frac{(\mathsf U_{\rm pub}-2N)t}{N(N+1)}\right\rfloor.
\]
Add the residual to the final first-center batch. The counts are exactly
$T$ pairs, $2T$ values, $3N+1$ sampling stages and the same $N$
projections. No problem constant $R,D,L,M,M_2$ or variance is supplied.

\subsection{A conditional controller bound}
\label{subsec:joint-conditional}

Write $c_{\rm vec}=4\sqrt2$ for adaptive batches and
$c_{\rm vec}=8/\sqrt7$ for fixed batch sizes, and define
\begin{align}
 \omega_{\rm rad}&=\sqrt{\chi_d}M_2\sqrt{N/B^{\rm rad}},\qquad
 \Gamma(D)=\min\{\gammaMax,d^{1/4}D/N\},\nonumber\\*
 \eta(D)&=\sqrt{\chi_d}M_2\Gamma(D)/\sqrt{B^{\rm rad}},
 \label{eq:joint-diagnostic-scales}\\
 \mathcal H_{{\rm j},\rm ns}(D)
 &=\frac{4\sqrt d M}{\Gamma(D)}
  +\frac{(2\sqrt2d^{1/4}+2)MN}{D}
  +\frac{\sqrt2d^{1/4}\omega_{\rm rad} N}{D}
  +\frac{c_{\rm vec}\sigma_0A}{D\sqrt{T'}},\nonumber\\
 \mathcal H_{{\rm j},\rm s}(D)
 &=4L+\frac{c_{\rm vec}\sigma_0A}{D\sqrt{T'}}.
 \label{eq:joint-curvature-envelopes}
\end{align}
The smooth option assumes population $L$-smoothness on the evaluation domain.

\begin{lemma}[The offset-controller potential]
\label{lem:joint-offset-potential}
Condition on the complete anchor batch. For either admissible envelope
$\mathcal H(D)$ in \eqref{eq:joint-curvature-envelopes}, put
$\ell(D)=\log_2(1+\mathcal H(D)/h_{\rm ref})$.
For every conditionally fixed $w\in Q_S$,
\begin{align}
 \E[f(x_N)-f(w)\mid\text{anchor}]\le{}&
 \frac{h_{\rm a}\norm{w-x^0}^2}{2A}
 +\left(\frac2A+\frac3{N^2}\right)D_S^2\mathcal H(D_S)
 +\frac{h_{\rm ref}D_S^2}{N^2}\nonumber\\
 &+\eta(D_S)[1+\ell(D_S)]
 +\frac{2MD_SN}{A}\ell(D_S).
 \label{eq:joint-conditional-risk}
\end{align}
For the smooth envelope replace the last term by
$LD_S^2\ell(D_S)/A$.
\end{lemma}

\begin{proof}
Under the anchor law the geometry and initial coefficients are fixed.
Let $\mathsf S_N$ be the maximum positive radial mean normalized by its
trial radius. Conditional normalized means lie in $[0,M]$, and their
centered variances are at most $\chi_dM_2^2/B^{\rm rad}$.
Thus $\norm{\mathsf S_N}_{L_2}\le M+\omega_{\rm rad}$ by bounding the squared
error maximum by the sum of squared errors and applying Minkowski.
No independence between iterations is required.

The reference satisfies $\underline K_t\le\mathsf K_{t-1}<2\underline K_t$.
Within a reference block, the reciprocal update is a maximum starting
from $\Gamma_S^{-1}$. Therefore
\begin{align}
 \gamma_t^{-1}&\le
 \max\{\Gamma_S^{-1},\mathsf S_NN^2/(\underline K_tD_S^2)\},\nonumber\\*
 L_t&\le\sqrt d M/\Gamma_S+
       2\sqrt dM\mathsf S_NN^2/(\mathsf K_{t-1}D_S^2).
 \label{eq:joint-reciprocal-control}
\end{align}
Fix a complete realization and define
\[
 \mathfrak K=4\sqrt dM/\Gamma_S+
             2\sqrt2d^{1/4}\sqrt{M\mathsf S_N}\,N/D_S.
\]
Write $a=\sqrt dM/\Gamma_S$ and
$b=\sqrt dM\mathsf S_NN^2/D_S^2$. Then
$\mathfrak K=4a+\sqrt{8b}$ and
$\mathfrak K^2-4a\mathfrak K-8b=4a\sqrt{8b}\ge0$.
Whenever $\mathsf K_{t-1}\ge\mathfrak K$,
\eqref{eq:joint-reciprocal-control} therefore gives
$4L_t\le4a+8b/\mathfrak K\le\mathfrak K$.
If $a=b=0$, every target is affine and the shifted-square argument
below applies directly with threshold zero.
Write $e_t^y=g_t^y-\nabla f_{\gamma_t}(y_t)$, $e_t^x=g_t^x-\nabla f_{\gamma_t}(x_t)$,
$\delta_t=H_t-H_{t-1}$ and
$\mathcal E_N=\sum_tt^2\norm{e_t^x-e_t^y}^2$.
The scalar balance and secant bound give
\begin{align*}
 \delta_tD_S^2
 &\le[-(H_t-4L_t)r_t^2/2+t\norm{e_t^x-e_t^y}r_t]_+\\
 &\le[-(\mathsf K_t-\mathfrak K)r_t^2/2
                   +t\norm{e_t^x-e_t^y}r_t]_+.
\end{align*}
Here $H_t\ge\mathsf K_t$ is essential.
Because $\mathsf K_t-\mathsf K_{t-1}=\delta_t$, completion of
a square gives the positive-part square recurrence after this crossing.
Before it, the gradient-norm envelope and $r_t\le D_S$ give
$\delta_t\le t(2M+\norm{e_t^x-e_t^y})/D_S$.
There is at most one crossing. Its noisy jump and all later energy
combine using $\sqrt{(a+\sqrt q)^2+b}\le a+\sqrt{q+b}$.
Consequently
\[
 H_N-h_{\rm a}=\mathsf K_N-h_{\rm ref}
 \le\mathfrak K+2NM/D_S+\sqrt{\mathcal E_N}/D_S.
\]
If $h_{\rm ref}\ge\mathfrak K$, the square recurrence holds throughout
and gives a stronger bound.
The lower bounds on the vector batch sizes yield
$\sqrt{\E[\mathcal E_N\mid\text{anchor}]}\le
c_{\rm vec}\sigma_0A/\sqrt{T'}$.
Also $\E\sqrt{M\mathsf S_N}\le M+\omega_{\rm rad}/2$.
Thus the conditional expected coefficient increase is bounded by
$\mathcal H_{{\rm j},\rm ns}(D_S)$.
Under smoothness, Lemma~\ref{lem:growth} instead gives
$\mathcal H_{{\rm j},\rm s}(D_S)$.

Each nonzero replacement at least doubles the reference. If $J$ is
their number, then
\[
 J\le\log_2(\mathsf K_N/h_{\rm ref}),\qquad
 \sum_{t:I_t=1}\underline K_t\le2(\mathsf K_N-h_{\rm ref})=2(H_N-h_{\rm a}).
\]
Lemma~\ref{lem:radial-affine} bounds the selected local charge,
conditionally on the history preceding its diagnostic, by
$\underline K_tD_S^2/N^2+\eta(D_S)$ in expectation.
Its weighted average is at most
$D_S^2\E[\mathsf K_N\mid\text{anchor}]/N^2+\eta(D_S)$.
At a reexpansion event the positive target drift is at most this new
local charge plus the Bregman remainder at $(x_{t-1},y_t)$.
Since $\norm{x_{t-1}-y_t}\le2D_S/(t+1)$, the weighted remainder
is at most $2MD_SN$, or $LD_S^2$ under smoothness.
The event is known before its radial stage, so its charge may be averaged
conditionally first. The expected radial drift is at most
$2D_S^2\E[H_N-h_{\rm a}\mid\text{anchor}]/N^2
+\eta(D_S)\E[J\mid\text{anchor}]$.

Theorem~\ref{thm:radial-potential} retains this positive target drift
and charges $2D_S^2\E[H_N-h_{\rm a}\mid\text{anchor}]/A$ for
scalar growth. Combining the preceding terms gives
\eqref{eq:joint-conditional-risk}, after Jensen bounds the event count
by $\ell(D_S)$. Every diagnostic error is averaged at its predictable
event; no product with a future random logarithm is bounded directly.
The moment estimates also ensure integrability.

For adaptive-batch accounting, $T'\ge16N$ gives the original floors and
$\mathsf U\ge5T'/8$. For the accounting with fixed batch sizes, $T'\ge8N$ gives
$B^{\rm rad}\ge T'/(16N)$ and
$\mathsf U_{\rm pub}\ge7T'/8\ge7N\ge4N$.
Its floors yield energy at most $64\sigma_0^2A^2/(7T')$.
Allocating the remaining samples and including the initial batch proves
all stated counts.
\end{proof}

The unconditional original-comparator guarantee is
Theorem~\ref{thm:run5-distance-joint}. It combines this conditional
potential with the shared clipping price and the two anchor-radius
moments. The deterministic logarithm is evaluated at $2r_0$ before
any anchor-dependent geometry is averaged.

\section{Anchor moment geometry and curved objectives}
\label{sec:run5-distance}

We retain the initialization and exact ball-intersection projection of
Section~\ref{sec:frontier-distance}. Put $A=A_N=N(N+1)/2$ for $N\ge1$.
The parameters chosen before sampling are $r_0,h_{\rm ref}>0$, $p\ge1$, a legal anchor
radius $\gamma_{\rm a}$ and the sample budget. Let $g_0$ be the mean of
$B_0\ge1$ independent clean vector records at $x^0$, with mean
$\nabla f_{\gamma_{\rm a}}(x^0)$ and centered second moment at most
$\sigma_{\rm a}^2/B_0$. The algorithm initializes once with
\begin{align}
 h_{\rm a}&=h_{\rm ref}+\frac{A\norm{g_0}}{p r_0},&
 S&=r_0+\frac{2A\norm{g_0}}{h_{\rm a}},&
 r_p&=r_0(1+2p),\label{eq:run5-distance-anchor}\\*
 Q_S&=Q\cap\overline{\mathbb B}(x^0,S),&
 D_S&=2S,&H_0&=h_{\rm a}.\nonumber
\end{align}
Projection onto $Q_S$ is an additional projection operation. All later
optimizer vectors, coefficients, weights, indices and allocation
accounts persist. Define the analysis quantities
\begin{align}
 G_{\rm a}&=\norm{\nabla f_{\gamma_{\rm a}}(x^0)},&
 z_{\rm a}&=\sqrt{G_{\rm a}^2+\sigma_{\rm a}^2/B_0},&
 \kappa&=\frac{h_{\rm ref}p r_0}{A},\nonumber\\
 s(z)&=r_0+\frac{2p r_0z}{\kappa+z},&
 S_2&=s(z_{\rm a}).
 \label{eq:run5-distance-moment-radius}
\end{align}
The method receives none of $G_{\rm a},\sigma_{\rm a},z_{\rm a},S_2$.

\subsection{Moment geometry and saturated clipping}
\label{subsec:run5-distance-moment}

\begin{lemma}[Anchor-radius second moment]
\label{lem:run5-distance-moment}
Under the preceding moment assumption,
\begin{equation}
 \E S^2\le S_2^2,\qquad \E S\le S_2<r_p.
 \label{eq:run5-distance-radius-moments}
\end{equation}
\end{lemma}
\begin{proof}
Put $a=r_0$, $b=2p r_0$, $k=\kappa$ and
$F(w)=[a+b\sqrt w/(k+\sqrt w)]^2$ for $w\ge0$. It is increasing,
bounded and continuous at zero. At $z=\sqrt w>0$,
\[
 F''(w)=-\frac{bk[3(a+b)z+4ak+ak^2/z]}{2z^2(k+z)^4}<0.
\]
Indeed, writing $s(z)=a+bz/(k+z)$, differentiation gives
$F''=[(s')^2+ss''-ss'/z]/(2z^2)$,
$s'=bk/(k+z)^2$ and $s''=-2bk/(k+z)^3$.
After factoring $bk/(k+z)^4$, the numerator is
$bk-s(k+z)(3+k/z)=-3(a+b)z-4ak-ak^2/z$.
Continuity extends concavity to $[0,\infty)$. Centering gives
$\E\norm{g_0}^2=G_{\rm a}^2+
\E\norm{g_0-\nabla f_{\gamma_{\rm a}}(x^0)}^2\le z_{\rm a}^2$.
Jensen and monotonicity therefore yield
$\E S^2=\E F(\norm{g_0}^2)\le F(z_{\rm a}^2)=S_2^2$.
Cauchy--Schwarz gives $\E S\le S_2$. If $z_{\rm a}=0$, then
$g_0=0$ almost surely and $S=r_0$. The strict upper bound follows
from $\kappa>0$ and finite $z_{\rm a}$.
\end{proof}

For a fixed original comparator $u\in Q$, put
\begin{align}
 R_u&=\norm{u-x^0},&a_{R_u}&=(R_u-r_0)_+,&
 \Psi(R_u)&=\frac{R_u^2}{2}+\frac{a_{R_u}^2}{8},\nonumber\\
 u_S&=x^0+\min\{1,S/R_u\}(u-x^0),&
 \beta_{\rm a}(u)&=f_{\gamma_{\rm a}}(u)-f(u),
 \label{eq:run5-distance-comparator}
\end{align}
where $u_S=x^0$ if $R_u=0$. Define
\begin{align}
 \mathcal A_p(u)={}&\frac{h_{\rm ref}\Psi(R_u)}A
 +\frac{G_{\rm a}\Psi(R_u)}{p r_0}
 +\frac{\sigma_{\rm a}}{\sqrt{B_0}}
       \left(a_{R_u}+\frac{\Psi(R_u)}{p r_0}\right)+\beta_{\rm a}(u),
 \label{eq:run5-distance-old-price}\\
 \mathcal A_{\rm sat}(u)={}&
 \frac{h_{\rm ref}\min\{R_u^2,S_2^2\}}{2A}
 +\frac{z_{\rm a}\min\{R_u^2,r_p S_2\}}{2p r_0}
 +G_{\rm a}a_{R_u}+\beta_{\rm a}(u),
 \label{eq:run5-distance-saturated-price}\\
 \mathcal A_*(u)={}&\min\{\mathcal A_p(u),\mathcal A_{\rm sat}(u)\}.
 \label{eq:run5-distance-best-price}
\end{align}

\begin{lemma}[Saturated startup and clipping price]
\label{lem:run5-distance-clipping}
If $f\le f_{\gamma_{\rm a}}$ and $f_{\gamma_{\rm a}}$ is convex, then
\begin{equation}
 \E\left[\frac{h_{\rm a}\norm{u_S-x^0}^2}{2A}
       +f_{\gamma_{\rm a}}(u_S)-f(u)\right]\le\mathcal A_*(u).
 \label{eq:run5-distance-clipping-price}
\end{equation}
The same bound holds with $f(u_S)$ in place of
$f_{\gamma_{\rm a}}(u_S)$.
\end{lemma}
\begin{proof}
Write $\phi=f_{\gamma_{\rm a}}$, $z=\norm{g_0}$ and
$e_0=g_0-\nabla\phi(x^0)$. If $R_u>S$, set $\theta=S/R_u$.
Segment convexity and the tangent inequality at $x^0$ give
\[
 \phi(u_S)-\phi(u)
 \le(1-\theta)[\phi(x^0)-\phi(u)]\le G_{\rm a}(R_u-S).
\]
If $R_u\le S$, the difference is zero. Thus
$\phi(u_S)-f(u)\le G_{\rm a}(R_u-S)_++\beta_{\rm a}(u)$.
Since $|G_{\rm a}-z|\le\norm{e_0}$ and
$S=r_0+(2A/h_{\rm a})z$,
\[
 G_{\rm a}(R_u-S)_+\le
 z[a_{R_u}-(2A/h_{\rm a})z]_++a_{R_u}\norm{e_0}
 \le\frac{h_{\rm a}a_{R_u}^2}{8A}+a_{R_u}\norm{e_0}.
\]
The last inequality maximizes $z(a_{R_u}-cz)$ with $c=2A/h_{\rm a}$;
it holds for every realized positive $h_{\rm a}$.
Bounding the startup distance by $R_u$ gives the charge
$h_{\rm a}\Psi(R_u)/A+a_{R_u}\norm{e_0}+\beta_{\rm a}(u)$.
The inequalities $\E z\le G_{\rm a}+\sigma_{\rm a}/\sqrt{B_0}$ and
$\E\norm{e_0}\le\sigma_{\rm a}/\sqrt{B_0}$ prove
\eqref{eq:run5-distance-old-price}.

For the other price retain
$V=\norm{u_S-x^0}^2=\min\{R_u^2,S^2\}$.
Lemma~\ref{lem:run5-distance-moment} gives
$\E V\le\min\{R_u^2,S_2^2\}$. Pointwise $V^2\le R_u^4$ and
$V^2\le S^4\le r_p^2S^2$, so
$\E V^2\le\min\{R_u^4,r_p^2S_2^2\}$.
Cauchy--Schwarz consequently gives
\[
 \E[zV]\le\sqrt{\E z^2\,\E V^2}
       \le z_{\rm a}\min\{R_u^2,r_p S_2\}.
\]
Expand $h_{\rm a}$ and use $S\ge r_0$ in the clipping term to obtain
\eqref{eq:run5-distance-saturated-price}. Both prices bound the same
expected charge, so their minimum applies. The original-function
version follows from $f(u_S)\le\phi(u_S)$. The fourth power here is
that of the bounded radius; no fourth anchor moment is assumed.
\end{proof}

\subsection{Fixed-radius and joint guarantees}
\label{subsec:run5-distance-guarantees}

\begin{theorem}[Fixed-radius moment localization]
\label{thm:run5-distance-fixed}
Use the fixed-radius execution of Section~\ref{subsec:frontier-rule} with
$\gamma_{\rm a}=\gamma$. Assume uniform clean conditional vector
variance $\sigma_\gamma^2$, anchor variance
$\sigma_{\rm a}^2\le\sigma_\gamma^2$, and $L_\gamma$-smoothness of
$f_\gamma$. Put $B_0=\lfloor T/2\rfloor$, $T'=T-B_0$.
Then
\begin{equation}
 \E[f(x_N)-f(u)]\le\mathcal A_*(u)+\frac{32L_\gamma S_2^2}A
       +\frac{\nu\sigma_\gamma S_2}{\sqrt{T'}},
 \label{eq:run5-distance-fixed-risk}
\end{equation}
where $\nu=8\sqrt2$ for fixed batch sizes and $16\sqrt2$ for adaptive batches.
With fixed batch sizes, $T\ge8N$ uses exactly $T$ pairs, $2T$ values, $2N+1$ sampling
stages and $N$ ball-intersection projections. With adaptive batch sizes, $T\ge16N$ gives
$4N+1$ stages and the same other counts.
\end{theorem}
\begin{proof}
Condition on the complete anchor. Lemma~\ref{lem:frontier-conditional}
gives
\[
 \E[f_\gamma(x_N)-f_\gamma(u_S)\mid\text{anchor}]
 \le\frac{h_{\rm a}\min\{R_u^2,S^2\}}{2A}
       +\frac{32L_\gamma S^2}A
       +\frac{\nu\sigma_\gamma S}{\sqrt{T'}}.
\]
Add $f_\gamma(u_S)-f(u)$ and use $f(x_N)\le f_\gamma(x_N)$.
Taking expectation, Lemma~\ref{lem:run5-distance-clipping} bounds the
sum of the startup and comparator charges by $\mathcal A_*(u)$.
Lemma~\ref{lem:run5-distance-moment} gives $\E S^2\le S_2^2$ and
$\E S\le S_2$, proving \eqref{eq:run5-distance-fixed-risk}.
The allocation is unchanged: the anchor and two stages with fixed batch sizes or four with adaptive
batch sizes per optimization iteration use all $T$ sampled functions, with two values for each
sampled function and one intersection projection per iteration. The required budgets
ensure $T'\ge4N$ or $8N$, respectively, as used in the conditional lemma.
\end{proof}

\paragraph{Fixed-radius choices made before sampling.}
The smooth specialization of \eqref{eq:run5-distance-fixed-risk} uses
$L_\gamma\le L$ and
$\beta_{\rm a}(u)\le\nu_dL\gamma^2/2$, where $\nu_d=d/(d+2)$.
The nonsmooth specialization uses
$L_\gamma\le\sqrt dM/\gamma$ and $\beta_{\rm a}(u)\le M\gamma$.
In either case retain the actual $\mathcal A_*(u)$ and $S_2$.
If $G_{\rm a}\le M$ and $\sigma_\gamma\le\sigma_0$ uniformly, the
following choices made before sampling give convergence for fixed problem and
reference quantities:
\begin{align}
 p&=\min\{\sqrt N,T^{1/4}\},&
 \gamma&=\min\{\gammaMax,d^{1/4}r_0/p\}
 &&\text{(nonsmooth)},\label{eq:frontier-nonsmooth-parameters}\\
 p&=\min\{N^{2/3},T^{1/4}\},&
 \gamma&=\min\{\gammaMax,r_0/p\}
 &&\text{(smooth)}.\label{eq:frontier-smooth-parameters}
\end{align}
Indeed, $\mathcal A_*\le\mathcal A_p$, $S_2\le r_p\le3pr_0$,
$A\ge N^2/2$, $B_0\ge T/3$ and $T'\ge T/2$.
The anchor-gradient charge is $O(p^{-1})$ and its noise charge is
$O(T^{-1/2})$; the optimizer noise is $O(p/\sqrt T)$.
For the nonsmooth choice,
$1/\gamma\le1/\gammaMax+p/(d^{1/4}r_0)$, so the curvature terms are
$O(p^2/N^2+p^3/N^2)=O(N^{-1}+N^{-1/2})$ and the smoothing bias is
$O(p^{-1})$. For the smooth choice, these terms are
$O(p^2/N^2)=O(N^{-2/3})$ and $O(p^{-2})$, respectively.
Finally $p\to\infty$ and $p/\sqrt T\le T^{-1/4}$, proving convergence.
The constants in these consequences depend only on the fixed problem
and reference quantities. The bound itself remains the moment bound
\eqref{eq:run5-distance-fixed-risk}; neither choice uses the unknown
problem constants.

For the joint controller of Section~\ref{sec:joint-anchor-radial},
retain its uniform clean vector and radial moment assumptions. Define
\begin{align}
 B^{\rm rad}&=\lfloor T'/(8N)\rfloor,&
 \Gamma(D)&=\min\{\gammaMax,d^{1/4}D/N\},\nonumber\\
 \omega&=\sqrt{\chi_d}M_2\sqrt{N/B^{\rm rad}},&
 \eta(D)&=\frac{\sqrt{\chi_d}M_2\Gamma(D)}{\sqrt{B^{\rm rad}}},
 \nonumber\\
 v&=\frac{c_{\rm vec}\sigma_0 A}{\sqrt{T'}},&
 k_N&=\frac2A+\frac3{N^2},\nonumber\\
 C_{\rm ns}&=(2\sqrt2d^{1/4}+2)MN+\sqrt2d^{1/4}\omega N+v,
 \nonumber\\
 \mathcal H_{\rm s}(D)&=4L+v/D,&
 \mathcal H_{\rm ns}(D)&=\frac{4\sqrt d M}{\Gamma(D)}
       +\frac{C_{\rm ns}}D,
 \label{eq:run5-distance-joint-envelopes}
\end{align}
where $c_{\rm vec}=8/\sqrt7$ for fixed batch sizes and $4\sqrt2$ for adaptive
batches. Put
\begin{align}
 \ell_{\mathrm{reg}}^0&=\log_2\left(1+\frac{\mathcal H_{\mathrm{reg}}(2r_0)}{h_{\rm ref}}\right),&
 q_{\rm s}(D)&=LD^2/A,&q_{\rm ns}(D)&=2MDN/A,\nonumber\\
 \mathcal U_{\mathrm{reg}}(D,\ell)&=k_ND^2\mathcal H_{\mathrm{reg}}(D)
       +\frac{h_{\rm ref}D^2}{N^2}+\eta(D)(1+\ell)+q_{\mathrm{reg}}(D)\ell,
 \label{eq:run5-distance-joint-cost}\\
 \mathcal V_{\rm ns}(s,\ell)&=
 k_N\left[\frac{16\sqrt d M s^2}{\gammaMax}
       +2s(4d^{1/4}MN+C_{\rm ns})\right]
       +\frac{4h_{\rm ref}s^2}{N^2}\nonumber\\
 &\hspace{1em}+\eta(2s)(1+\ell)+\frac{4MsN}A\ell.
 \label{eq:run5-distance-finite-halo-cost}
\end{align}

\begin{theorem}[Joint anchor moment refinement]
\label{thm:run5-distance-joint}
For a smooth population,
\begin{equation}
 \E[f(x_N)-f(u)]\le\mathcal A_*(u)
       +\mathcal U_{\rm s}(2S_2,\ell_{\rm s}^0).
 \label{eq:run5-distance-joint-smooth}
\end{equation}
If $\gammaMax\ge2d^{1/4}r_p/N$, the nonsmooth branch satisfies
\begin{equation}
 \E[f(x_N)-f(u)]\le\mathcal A_*(u)
       +\mathcal U_{\rm ns}(2S_2,\ell_{\rm ns}^0).
 \label{eq:run5-distance-joint-nonsmooth}
\end{equation}
At every finite permitted evaluation radius it also satisfies
\begin{equation}
 \E[f(x_N)-f(u)]\le\mathcal A_*(u)+
 \min\{\mathcal U_{\rm ns}(2r_p,\ell_{\rm ns}^0),
             \mathcal V_{\rm ns}(S_2,\ell_{\rm ns}^0)\}.
 \label{eq:run5-distance-joint-finite-halo}
\end{equation}
With fixed batch sizes, $T\ge16N$ uses exactly $T$ pairs, $2T$ values, $3N+1$ sampling
stages and $N$ ball-intersection projections. With adaptive batch sizes, $T\ge32N$ gives
$5N+1$ stages and the same other counts.
\end{theorem}
\begin{proof}
Conditional on the full anchor, Lemma~\ref{lem:joint-offset-potential}
gives the startup charge plus $\mathcal U_{\mathrm{reg}}(D_S,\ell_{\mathrm{reg}}(D_S))$, where
$\ell_{\mathrm{reg}}(D)=\log_2(1+\mathcal H_{\mathrm{reg}}(D)/h_{\rm ref})$.
Add clipping and apply Lemma~\ref{lem:run5-distance-clipping}, in its
original-function version. Since $D_S\ge2r_0$ and $\mathcal H_{\mathrm{reg}}$ is
decreasing, $\ell_{\mathrm{reg}}(D_S)\le\ell_{\mathrm{reg}}^0$. Its coefficients are nonnegative;
replace the logarithm by this deterministic bound before averaging.

For the smooth envelope, $D^2\mathcal H_{\rm s}(D)=4LD^2+vD$.
Lemma~\ref{lem:run5-distance-moment} gives
$\E D_S^2\le4S_2^2$ and $\E D_S\le2S_2$.
The reference and smooth target-drift terms are multiples of $D^2$.
The function $\eta(D)$ is increasing and concave, as a positive
multiple of the minimum of a constant and a linear map. Jensen gives
$\E\eta(D_S)\le\eta(\E D_S)\le\eta(2S_2)$.
These substitutions prove \eqref{eq:run5-distance-joint-smooth}.

Under the untruncated-evaluation domain condition, every anchor outcome has
$\Gamma(D_S)=d^{1/4}D_S/N$. Thus
$D^2\mathcal H_{\rm ns}(D)=(4d^{1/4}MN+C_{\rm ns})D$,
and this term, $q_{\rm ns}$ and $\eta$ are linear on the whole observed
range. Their first moments and the reference term's second moment
give \eqref{eq:run5-distance-joint-nonsmooth}.

For an arbitrary finite evaluation domain, pointwise
\[
 \frac{D^2}{\Gamma(D)}
 =\max\left\{\frac{D^2}{\gammaMax},\frac{ND}{d^{1/4}}\right\}
 \le\frac{D^2}{\gammaMax}+\frac{ND}{d^{1/4}}.
\]
Use both radius moments here and in the other linear and quadratic
terms, and use concavity for $\eta$. The result is exactly
$\mathcal V_{\rm ns}(S_2,\ell_{\rm ns}^0)$.
Independently, $D_S\le2r_p$ and monotonicity of each cost factor at
fixed logarithm give $\mathcal U_{\rm ns}(2r_p,\ell_{\rm ns}^0)$.
Each cost bound can be combined with either anchor-price bound,
because they bound separate expected charges. Taking both minima
proves \eqref{eq:run5-distance-joint-finite-halo}.
All algorithmic stages and allocation rules remain those of the
joint controller, including the anchor and each three-stage iteration with fixed batch sizes
or five-stage iteration with adaptive batch sizes.
\end{proof}

\subsection{A convergent joint choice made before sampling}
\label{subsec:joint-convergence}

Take
\[
 p=\min\{\sqrt N,T^{1/4}\},\qquad
 \gamma_{\rm a}=\min\{\gammaMax,d^{1/4}r_0/N\}.
\]
We verify convergence for fixed problem and reference quantities as
$N,T\to\infty$ subject to the stated budget condition.
Use Theorem~\ref{thm:run5-distance-joint} with
$\mathcal A_*\le\mathcal A_p$ and $S_2\le r_p$.
The anchor bias is at most $Md^{1/4}r_0/N$ and $G_{\rm a}\le M$.
Since $r_p\le3pr_0$ and $1/p\le N^{-1/2}+T^{-1/4}$, the anchor
price tends to zero, including its variance terms of order $T^{-1/2}$.
The nonsmooth deterministic and evaluation domain costs obey
$p/N\le N^{-1/2}$ and $p^2/N^2\le N^{-1}$; the smooth costs
also use the latter inequality. The optimizer stochastic terms satisfy
$p/\sqrt{T'}\le\sqrt2T^{-1/4}$.
Moreover $\eta(2r_p)\le8\sqrt2d^{1/4}M_2r_p/(\sqrt N\sqrt{T'})$.
At $D=2r_0$ and $T'\ge cN$, both curvature envelopes are
$O(1+N^{3/2})$ with constants depending only on the fixed problem
and references. Hence $\ell_{\mathrm{reg}}^0=O(1+\log N)$.
The drift terms tend to zero, and the diagnostic term does so because
$(1+\log N)/\sqrt N\to0$ and $T^{-1/4}\to0$.
Thus every term in the original-distance bound converges, with the
full $R_u/r_0$ dependence retained.

The method selects smoothing radii, scalar curvature and, in its adaptive-batch
version, batch sizes from samples without the unknown problem scales.
Its original-distance guarantee pays the displayed anchor price and
uses the separately charged intersection geometry. Optimal simultaneous
leading rates and the positive spatial-corruption allowance remain open.

\subsection{Original stationary comparators}
\label{subsec:run5-distance-stationary}

Assume the original population is $L$-smooth on the convex evaluation domain and
the fixed original comparator satisfies $\nabla f(u)=0$. Let
$\nu_d=d/(d+2)$ and define
\begin{equation}
 G_{R_u}=L\sqrt{R_u^2+\nu_d\gamma_{\rm a}^2},\qquad
 z_{R_u}=\sqrt{G_{R_u}^2+\sigma_{\rm a}^2/B_0},\qquad S_{R_u}=s(z_{R_u}).
 \label{eq:run5-distance-stationary-radius}
\end{equation}
Define $\mathcal A_*^{\rm stat}(R_u)$ by substituting
$G_{R_u},z_{R_u},S_{R_u},L\nu_d\gamma_{\rm a}^2/2$ for
$G_{\rm a},z_{\rm a},S_2,\beta_{\rm a}(u)$, respectively, in the two
prices defining \eqref{eq:run5-distance-best-price}.

\begin{corollary}[Original stationary distance]
\label{cor:run5-distance-stationary}
Under the smooth assumptions of Theorem~\ref{thm:run5-distance-joint},
\begin{equation}
 \E[f(x_N)-f(u)]\le\mathcal A_*^{\rm stat}(R_u)
       +\mathcal U_{\rm s}(2S_{R_u},\ell_{\rm s}^0).
 \label{eq:run5-distance-stationary-risk}
\end{equation}
If $\nabla f$ is affine on the evaluation domain, replace $G_{R_u}$ by $LR_u$ and $z_{R_u}$
by $\sqrt{L^2R_u^2+\sigma_{\rm a}^2/B_0}$ in these formulas. The smoothing
excess bound remains $L\nu_d\gamma_{\rm a}^2/2$.
\end{corollary}
\begin{proof}
Let $U$ be uniform on the unit ball. Smoothness permits differentiation
under the bounded ball integral. Jensen, original stationarity and
gradient Lipschitz continuity give
\begin{align*}
 G_{\rm a}^2
 &=\norm{\E_U\nabla f(x^0+\gamma_{\rm a}U)}^2\\
 &\le\E_U\norm{\nabla f(x^0+\gamma_{\rm a}U)-\nabla f(u)}^2\\
 &\le L^2\E_U\norm{x^0-u+\gamma_{\rm a}U}^2
 =L^2(R_u^2+\nu_d\gamma_{\rm a}^2).
\end{align*}
Here $\E U=0$ and $\E\norm U^2=\nu_d$.
Ball averaging of the smooth upper model at $u$ yields
$\beta_{\rm a}(u)\le L\nu_d\gamma_{\rm a}^2/2$.
Both prices are increasing in the nonnegative quantities
$G_{\rm a},z_{\rm a},S_2,\beta_{\rm a}(u)$, and the smooth cost at a
fixed nonnegative logarithm is increasing in $D$. Substitution into
\eqref{eq:run5-distance-joint-smooth} proves the result.
If the gradient is affine, symmetry gives
$\nabla f_{\gamma_{\rm a}}(x^0)=\nabla f(x^0)$, and original
stationarity yields $G_{\rm a}\le LR_u$. No distance to a smoothed
minimizer is assumed.
\end{proof}

\subsection{Absorbing a reference chosen before sampling}
\label{subsec:run5-distance-reference-absorption}

Define costs with their explicit reference terms removed:
\[
 \mathcal U_{\mathrm{reg}}^0(D,\ell)=\mathcal U_{\mathrm{reg}}(D,\ell)-h_{\rm ref}D^2/N^2,
 \qquad
 \mathcal V_{\rm ns}^0(s,\ell)
   =\mathcal V_{\rm ns}(s,\ell)-4h_{\rm ref}s^2/N^2.
\]
Let $\mathcal J_{\mathrm{reg}}^0(S_2)$ be the cost
\eqref{eq:joint-main-smooth-cost}--\eqref{eq:joint-main-nonsmooth-cost}
with $\mathcal U_{\mathrm{reg}},\mathcal V_{\rm ns}$ replaced by
$\mathcal U_{\mathrm{reg}}^0,\mathcal V_{\rm ns}^0$, respectively.
Every logarithm, envelope and moment radius in these costs retains its
actual $h_{\rm ref}$ dependence.

\begin{corollary}[Absorbing a reference chosen before sampling]
\label{cor:run5-distance-reference-absorption}
Under Theorem~\ref{thm:run5-distance-joint}, for each applicable regularity case $\mathrm{reg}$, the saturated-price branch gives
\begin{align}
 \E[f(x_N)-f(u)]\le{}&
 \frac{z_{\rm a}\min\{R_u^2,r_p S_2\}}{2pr_0}
 +G_{\rm a}a_{R_u}+\beta_{\rm a}(u)
 +\mathcal J_{\mathrm{reg}}^0(S_2)
 +\frac{5h_{\rm ref}r_p^2}{N^2}.
 \label{eq:run5-distance-absorbed-risk}
\end{align}
For any target $\varepsilon>0$ chosen before sampling, the choice
$0<h_{\rm ref}\le\varepsilon N^2/(20r_p^2)$ makes the final term
at most $\varepsilon/4$. It needs no supplied original distance,
domain diameter, smoothness or variance.
\end{corollary}
\begin{proof}
Theorem~\ref{thm:run5-distance-joint} permits the saturated price instead
of its minimum with $\mathcal A_p(u)$. Its explicit startup reference
term satisfies
\[
 \frac{h_{\rm ref}\min\{R_u^2,S_2^2\}}{2A}
 \le\frac{h_{\rm ref}r_p^2}{N^2},
\]
because $S_2\le r_p$ and $2A=N(N+1)\ge N^2$.
In the smooth or untruncated nonsmooth cost, the other explicit
reference term is $4h_{\rm ref}S_2^2/N^2$, at most
$4h_{\rm ref}r_p^2/N^2$.
For a finite nonsmooth evaluation domain, the two candidate geometry bounds have
reference terms $4h_{\rm ref}r_p^2/N^2$ and
$4h_{\rm ref}S_2^2/N^2$. Both are bounded by the former.
For any real $b_1,b_2$ and $q_1,q_2\le q$,
$\min\{b_1+q_1,b_2+q_2\}\le\min\{b_1,b_2\}+q$.
Apply this identity to the finite-radius minimum. In every regime the
remaining geometry cost is bounded by
$\mathcal J_{\mathrm{reg}}^0(S_2)+4h_{\rm ref}r_p^2/N^2$.
Adding the startup bound proves
\eqref{eq:run5-distance-absorbed-risk}.
Finally $r_p=r_0(1+2p)$ depends only on inputs chosen before sampling, so the specified
positive reference choice makes its displayed polynomial charge at
most $\varepsilon/4$ without knowledge of the problem scales.
\end{proof}

The anchor gradient and sampling charges in
\eqref{eq:run5-distance-absorbed-risk} remain. Changing $h_{\rm ref}$
also changes $S_2$ and the reference logarithms; all must be evaluated
at the chosen reference. This absorption concerns the explicit
polynomial reference terms, and does not remove the other prices of
original-distance adaptation.

\subsection{Sufficient accuracy budgets for the known-diameter policy}
\label{subsec:run5-distance-accuracy-budgets}

The following calculation concerns the known-diameter execution based on
the radial derivative in
Corollaries~\ref{cor:threshold-smooth} and~\ref{cor:threshold-nonsmooth}, with
$h_{\rm init}=h_{\rm ref}=h_0=\varepsilon N^2/(8D^2)>0$.
It interprets that fixed-budget guarantee and does not substitute
original distance for its diameter bound $D$. The forward-secant method
in Corollary~\ref{cor:epsilon-clean-budgets} has sufficient budgets without
these logarithmic factors.
Assume $\sigma_0^2\le CdM_2^2$, with fixed $C\ge0$.
The explicit bounds in these corollaries, together with
$R_u\le D$ and $A\ge N^2/2$, imply the following sufficient orders:
\begin{align}
 N_\varepsilon&=\widetilde O\left(
      1+D\sqrt{L/\varepsilon}\right)
      &&\text{in the smooth case},\label{eq:run5-distance-smooth-budget}\\
 N_\varepsilon&=\widetilde O\left(
      1+\frac{d^{1/4}MD}{\varepsilon}
      +d^{1/4}D\sqrt{\frac{M}{\gammaMax\varepsilon}}\right)
      &&\text{in the nonsmooth case},\label{eq:run5-distance-ns-budget}\\
 T_\varepsilon&=\widetilde O\left(
      N_\varepsilon+\frac{dM_2^2D^2}{\varepsilon^2}\right).
      \label{eq:run5-distance-value-budget}
\end{align}
The suppressed factors depend on the displayed reference logarithms
and fixed problem and reference quantities. The sample counts are rounded to
integers and obey $T_\varepsilon\ge8N_\varepsilon$ for fixed batch sizes
or $T_\varepsilon\ge16N_\varepsilon$ for adaptive batches.

The startup and explicit reference charges together are at most
$2h_0D^2/N^2=\varepsilon/4$. This choice made before sampling thus absorbs both
without knowledge of smoothness or variance. The remaining smooth
deterministic term is bounded by a constant times
$(1+\ell_{\rm s})LD^2/N^2$.
Assigning a fixed fraction of $\varepsilon$ to it requires
$N$ of the first displayed order. In the nonsmooth bound, the
remaining deterministic terms are bounded by constants times
$(1+\ell_{\rm ns})d^{1/4}MD/N$ and
$\sqrt dMD^2/(\gammaMax N^2)$. The last term needs
$N$ proportional to
$d^{1/4}D\sqrt{M/(\gammaMax\varepsilon)}$, which retains the finite
evaluation domain explicitly. Assigning fixed fractions of the accuracy to all
two terms gives \eqref{eq:run5-distance-ns-budget}.

In either regime, the vector term is bounded by a constant times
$\sqrt dM_2D/\sqrt T$. The radial term obeys
\[
 \frac{(1+\ell_{\mathrm{reg}})d^{1/4}M_2D}{\sqrt N\sqrt T}
 \le\frac{(1+\ell_{\mathrm{reg}})\sqrt dM_2D}{\sqrt T},
\]
since $d,N\ge1$. The additional nonsmooth term
$56d^{1/4}M_2D/\sqrt T$ has the same upper order.
Squaring these stochastic coefficients and retaining the allocation
floor proves \eqref{eq:run5-distance-value-budget}.

These logarithmic requirements are compatible with finite sufficient
budgets. Indeed, if $T\ge bN$ with $b=8$ or $16$, then
$B^{\rm rad}\ge1$, $A\le N^2$ and the curvature envelopes in
\eqref{eq:adaptive-main-Hns}--\eqref{eq:adaptive-main-Hs} are bounded
by $c_0+c_1N+c_2N^{3/2}$, with nonnegative constants depending only
on the fixed problem, dimension, evaluation domain, diameter and mode.
For the nonsmooth envelope, use
$1/\Gamma(D)\le1/\gammaMax+N/(d^{1/4}D)$ and
$\omega\le\sqrt{\chi_d}M_2\sqrt N$; for its vector term use
$A/\sqrt T\le N^{3/2}/\sqrt b$. The smooth envelope uses the latter
inequality directly. Consequently
$\mathcal H_{\mathrm{reg}}(D)\le(c_0+c_1+c_2)N^2$ for $N\ge1$.
Recomputing the logarithm at the reference chosen before sampling gives
\[
 \ell_{\mathrm{reg}}\le\log_2\left(1+
       \frac{8D^2(c_0+c_1+c_2)}{\varepsilon}\right).
\]
This finite envelope is independent of $N$ and $T$.
Increasing the displayed orders by its corresponding logarithmic
factors enforces each fixed accuracy fraction. This proves sufficiency
without a separate polynomial cost for an arbitrarily fixed reference.

Each execution uses exactly $2T_\varepsilon$ scalar values and
$3N_\varepsilon$ sampling stages with fixed batch sizes or $5N_\varepsilon$ with adaptive batch sizes.
These budgets use population quantities to interpret a risk theorem.
The actual execution accepts only $N,T$ and positive references chosen before sampling;
an accuracy stopping certificate requires additional information about
the unknown stochastic law.

\subsection{A curved limitation of the existing initialization}
\label{subsec:run5-distance-curved-limit}

Fix $N\ge1$, $A=A_N$, $p\ge1$, $r_0,h_{\rm ref},G>0$ and put
\begin{equation}
 h=h_{\rm ref}+\frac{AG}{p r_0},\qquad q=G/h,\qquad
 S=r_0+2Aq,\qquad Q=[0,S],\qquad x^0=0.
 \label{eq:run5-distance-quadratic-inputs}
\end{equation}
Choose $0<\gamma\le\min\{\gammaMax,q/(2A)\}$. For each
$L\ge\max\{2Ah,G/r_0\}$, let
\begin{equation}
 R=G/L,\qquad u=R,\qquad
 F_L(x,\xi)=f_L(x)=\frac L2(x-R)^2
 \label{eq:run5-distance-quadratic}
\end{equation}
for every sampled function.

\begin{proposition}[An exact curved trajectory]
\label{prop:run5-distance-curved}
The fixed-radius anchor optimizer satisfies
\begin{equation}
 v_1=x_1=q,\qquad v_t=0,\quad x_t=q/A_t\quad(2\le t\le N),
 \label{eq:run5-distance-quadratic-trajectory}
\end{equation}
and hence
\begin{equation}
 f_L(x_N)-f_L(u)=\frac L2(q/A-G/L)^2\ge\frac{Lq^2}{8A^2}.
 \label{eq:run5-distance-quadratic-risk}
\end{equation}
The joint radial controller has the same optimizer trajectory at every
selected legal radius. Vector and anchor variances are zero.
\end{proposition}
\begin{proof}
In one dimension, for either $e\in\{-1,1\}$ and every radius,
\[
 \frac{f_L(z+\gamma e)-f_L(z-\gamma e)}{2\gamma}e=L(z-R).
\]
All vector means are exact. The anchor is $-G$ and selects
$h_{\rm a}=h$ and the displayed $S$. The first projected vector is
$q\in Q$, with $x_1=q$. For $D_S=2S$ its post coefficient is
\[
 H_1=h+\kappa_q(L-h/2),\qquad
 \kappa_q=\frac{q^2}{D_S^2+q^2/2}<\frac1{16A^2}\le\frac1{16}.
\]
The second center is $y_2=q$. Its tentative vector
$q-2(Lq-G)/H_1$ is nonpositive when
$(2-\kappa_q)L\ge(3-\kappa_q/2)h$.
The assumption $L\ge2Ah\ge2h$ suffices because
$2(2-\kappa_q)-(3-\kappa_q/2)=1-3\kappa_q/2>0$.
Thus $v_2=0$, $x_2=q/A_2$. If the claim holds at $t-1\ge2$, then
$y_t=q/A_t\ge q/A\ge2G/L=2R$, so its gradient is positive.
From $v_{t-1}=0$ projection sends the tentative negative point to zero,
and averaging gives $x_t=q/A_t$. This proves the trajectory by
induction; for $N=1$ the first step alone suffices. Since
$R\le q/(2A)$, the risk bound follows.

Every selected-radius vector gradient is identical in the joint
controller. Its projected vectors, averaged iterates, scalar inner
products and post coefficients are consequently identical; radial
stages change only radii and references. A deterministic sample
Lipschitz envelope on an open neighborhood of the finite evaluation domain is
$M_{2,L}=L(S+\gammaMax+r_0)+G$. This supplies all radial moments at
legal radii. Independent draws from this point-mass distribution satisfy the two-value
same-sample oracle model.
\end{proof}

\begin{corollary}[Failure of comparator-only curvature transfer]
\label{cor:run5-distance-curved-transfer}
For this fixed-radius family there are no finite $C,E_0\ge0$ depending
only on the fixed known inputs and $G$ such that, uniformly in $L$,
\begin{equation}
 f_L(x_N)-f_L(u)\le\frac{h_{\rm ref}\Psi(R)}A+
 \frac{G\Psi(R)}{p r_0}+\frac{CLR^2}A+\beta_\gamma(u)+E_0.
 \label{eq:run5-distance-invalid-transfer}
\end{equation}
This remains false if $C$ is multiplied by any fixed polynomial in
$\log(1+L/h_{\rm ref})$.
\end{corollary}
\begin{proof}
Since $R\le r_0$, $\Psi(R)=R^2/2$, and the first two proposed charges
sum to $hR^2/(2A)$. Quadratic ball smoothing in dimension one gives
$\beta_\gamma(u)=L\gamma^2/6\le Lq^2/(24A^2)$.
Proposition~\ref{prop:run5-distance-curved} therefore gives
\[
 f_L(x_N)-f_L(u)-\beta_\gamma(u)\ge\frac{Lq^2}{12A^2}.
\]
The remaining proposed right side is
$hG^2/(2AL^2)+CG^2/(AL)+E_0$, which remains bounded as $L\to\infty$.
The left side diverges. A fixed logarithmic polynomial times $1/L$
still tends to zero, proving the last assertion.
\end{proof}

For the joint controller, integer budgets
$T_L\ge\max\{32N,2(L/h_{\rm ref})^4\}$ make the radial diagnostic
scale $O(M_{2,L}/\sqrt{T_L})=O(1/L)$. Its product with
$\log(1+4L/h_{\rm ref})$ tends to zero, while vector and anchor
variances remain zero. Retaining that diagnostic error therefore
cannot repair a transfer that replaces every geometry-dependent
curvature term by $LR^2/A$ up to fixed logarithmic factors.
This is a limitation of the displayed initialization and averaging
rule. The construction has $R<r_0$; a supplied promise $r_0\le R$
imposes a different scope.

\section{Pre-step scalar models and the price of unobserved curvature}
\label{sec:prestep-distance}

A fixed Euclidean proximal coefficient and variable acceleration
weights yield a potential in the distance to any original comparator.
The weight is selected before a candidate step is accepted. This is the
similar-triangles construction underlying the adaptive methods of
\citet{GasnikovNesterov2018}. We include the observation errors explicitly and use a
two-value function-difference record. In a stochastic execution, the
acceptance decision changes the distribution of the retained errors.
The resulting rejection term is retained explicitly.

\subsection{An admissible function-difference record}
\label{subsec:prestep-relative-record}

Let $Q\subset\mathbb R^d$ be nonempty, closed and convex, and suppose
the projections below exist. For moment and expectation statements,
assume $Q$ is compact. Before sampling, fix $x^0\in Q$, $h_0,\mathcal L_0>0$,
and a legal smoothing radius $\gamma>0$. Let
$\phi=f_\gamma$ be the ball-smoothed population. The vector record at
$y$ is the usual symmetric two-value estimator for $\nabla\phi(y)$.
An independent sampled function supplies the function-difference record
\begin{equation}
 P_\gamma(x,y;U,\xi)
     =F(x+\gamma U,\xi)-F(y+\gamma U,\xi),
 \label{eq:prestep-value-record}
\end{equation}
where $U$ is a uniform unit-ball vector independent of $\xi$.
Both endpoints use the same sampled function, independently of the
samples used for the vector estimate.

\begin{lemma}[Relative-function moments]
\label{lem:prestep-relative-moments}
Suppose the population expectations exist at the queried points.
Assume the weak-gradient, segment-integrability and uniform
sample-gradient second-moment bound with envelope $M_2^2$ on the
evaluation domain, or a random sample Lipschitz envelope with second
moment at most $M_2^2$. Then, at every fixed $x,y\in Q$,
\begin{equation}
 \E P_\gamma(x,y)=\phi(x)-\phi(y),\qquad
 \E P_\gamma(x,y)^2\le M_2^2\norm{x-y}^2.
 \label{eq:prestep-value-moments}
\end{equation}
\nopagebreak[4]
An average of $B^q$ conditionally independent records consequently has
centered conditional variance at most
$M_2^2\norm{x-y}^2/B^q$.
\end{lemma}

\begin{proof}
Taking expectations of the same sampled function at both endpoints, followed by the
ball average, gives the mean. Under the gradient assumptions, the
fundamental theorem along the shifted segment gives
\[
 P_\gamma(x,y)=\int_0^1
 \ip{\nabla F(y+\gamma U+s(x-y),\xi)}{x-y}\,ds
\]
almost surely. The shifted segment remains in the convex evaluation
domain. Jensen and Cauchy--Schwarz give a bound by
$\norm{x-y}^2\int_0^1
\norm{\nabla F(y+\gamma U+s(x-y),\xi)}^2ds$
for the squared record. The segment-integrability assumption permits
Fubini and yields the claimed second moment. At a largest legal
boundary radius, sample continuity and increasing inner radii yield
the same bound by Fatou; uniform second moments give uniform
integrability for the mean. Under the alternative Lipschitz assumption,
the pointwise inequality
$|P_\gamma(x,y)|\le m(\xi)\norm{x-y}$ proves the bound directly.
Conditional independence and centering of the $B^q$ records divide their
variance by $B^q$.
\end{proof}

\subsection{One persistent trajectory with rejected trials}
\label{subsec:prestep-execution}

Initialize $A_0=0$, $x_0=v_0=x^0$, and keep the proximal coefficient
$h_0$ fixed. The current accepted state is denoted by
$(A^-,x^-,v^-)$. At trial $j$, the scalar curvature candidate
$\mathcal L_j$ and a nonnegative model allowance $\chi_j$ are known
before the vector batch. Let $a_j>0$ be the unique positive solution
and define the candidate weight and center by
\begin{align}
 \mathcal L_j a_j^2&=h_0(A^-+a_j),\qquad A_j^+=A^-+a_j,
 \nonumber\\*
 a_j&=\frac{h_0+\sqrt{h_0^2+4h_0\mathcal L_jA^-}}
                   {2\mathcal L_j},\qquad
 y_j=\frac{A^-x^-+a_jv^-}{A_j^+}.
 \label{eq:prestep-weight}
\end{align}
Draw samples independently of the preceding history, form their
vector mean $g_j$ at $y_j$, and set
\begin{equation}
 v_j^+=\operatorname{Proj}_Q(v^--a_jg_j/h_0),\qquad
 x_j^+=\frac{A^-x^-+a_jv_j^+}{A_j^+},\qquad
 w_j=v_j^+-v^- .
 \label{eq:prestep-candidate}
\end{equation}
Using independent samples, draw a mean $\widehat q_j$ for
$\phi(x_j^+)-\phi(y_j)$ using
\eqref{eq:prestep-value-record}. Accept precisely when
\begin{equation}
 A_j^+\bigl[\widehat q_j-\ip{g_j}{x_j^+-y_j}\bigr]
     \le h_0\norm{w_j}^2/2+\chi_j.
 \label{eq:prestep-acceptance}
\end{equation}
An accepted trial commits $(A_j^+,x_j^+,v_j^+)$ as the next state.
A rejected trial leaves that state fixed and doubles the curvature
candidate. The next trial uses independent sampled functions. A success carries its
curvature candidate into the next round. There is one accepted
trajectory, and the scalar coefficient, weights and state are never
reinitialized.

A cap of $K_{\rm tr}$ executed trials, chosen before sampling,
ensures a finite budget.
The execution stops after $N$ accepted trials or after this cap,
whichever occurs first, and returns its last accepted averaged point.
If there is no accepted trial, it returns $x^0$. Let $\mathcal A$
and $\mathcal R$ denote the accepted and rejected trial indices.
Their union contains every executed trial.

Choose batch sizes $B^g,B^q\ge1$ before sampling and use them in every trial.
The exact cost is
$(B^g+B^q)|\mathcal A\cup\mathcal R|$ pairs, twice that many values,
$2|\mathcal A\cup\mathcal R|$ sequential sampling stages and
$|\mathcal A\cup\mathcal R|$ projections. For $B^g=B^q=B$ it is at
most $2BK_{\rm tr}$ pairs. The rejected projections and their two
sampling stages are included. Predictable variable batches have the
exact pair cost $\sum_j(B_j^g+B_j^q)$; any pilot batches have their
additional pair and stage costs. No sampled function is evaluated more than twice.

\begin{theorem}[Original-comparator accepted-trial potential]
\label{thm:prestep-potential}
Let $\phi$ be convex and differentiable on $Q$. For arbitrary vector
means and relative-value records define
\begin{equation}
 e_j=g_j-\nabla\phi(y_j),\qquad
 \zeta_j^q=\widehat q_j-[\phi(x_j^+)-\phi(y_j)],\qquad
 Z_j(u)=a_j\ip{e_j}{u-v^-}-A_j^+\zeta_j^q .
 \label{eq:prestep-errors}
\end{equation}
For every fixed $u\in Q$, if $n$ trials have been accepted, then
\begin{equation}
 A_n[\phi(x_n)-\phi(u)]
       +\frac{h_0}{2}\norm{u-v_n}^2
 \le\frac{h_0}{2}R_u^2
       +\sum_{j\in\mathcal A}\chi_j
       +\sum_{j\in\mathcal A}Z_j(u),
 \qquad R_u=\norm{u-x^0}.
 \label{eq:prestep-pathwise-potential}
\end{equation}
If $\phi=f_\gamma\ge f$, the same fixed original comparator satisfies
\begin{align}
 A_n[f(x_n)-f(u)]+\frac{h_0}{2}\norm{u-v_n}^2
 \le{}&\frac{h_0}{2}R_u^2+
       \sum_{j\in\mathcal A}[\chi_j+Z_j(u)]
       +A_n\beta_\gamma(u),\nonumber\\
 &\beta_\gamma(u)=f_\gamma(u)-f(u).
 \label{eq:prestep-original-potential}
\end{align}
The comparator need not minimize the smoothed population.
\end{theorem}

\begin{proof}
Projection optimality gives, at any trial,
\[
 a_j\ip{g_j}{v_j^+-u}
 \le\frac{h_0}{2}\left(
 \norm{u-v^-}^2-\norm{u-v_j^+}^2-\norm{w_j}^2\right).
\]
The center identity and convexity at $y_j$ give
\[
 A^-\phi(x^-)+a_j\phi(u)
 \ge A_j^+\phi(y_j)
             +a_j\ip{\nabla\phi(y_j)}{u-v^-}.
\]
Also $A_j^+(x_j^+-y_j)=a_jw_j$. Combining these identities, writing
$\nabla\phi(y_j)=g_j-e_j$, and adding the distance potential yields
\begin{align*}
 &A_j^+[\phi(x_j^+)-\phi(u)]
                    +\frac{h_0}{2}\norm{u-v_j^+}^2\\
 &\quad\le A^-[\phi(x^-)-\phi(u)]
                    +\frac{h_0}{2}\norm{u-v^-}^2\\
 &\qquad\quad+
 A_j^+\bigl[\phi(x_j^+)-\phi(y_j)
                       -\ip{g_j}{x_j^+-y_j}\bigr]
       -\frac{h_0}{2}\norm{w_j}^2
       +a_j\ip{e_j}{u-v^-}.
\end{align*}
On an accepted trial, substitute the observed model inequality
\eqref{eq:prestep-acceptance} and
$\phi(x_j^+)-\phi(y_j)=\widehat q_j-\zeta_j^q$.
The last three terms are bounded by $\chi_j+Z_j(u)$.
Rejected trials do not change the accepted state. Summing only the
accepted transitions, using $A_0=0$, proves
\eqref{eq:prestep-pathwise-potential}. Finally,
$f(x_n)-f(u)\le\phi(x_n)-\phi(u)+\beta_\gamma(u)$ proves
\eqref{eq:prestep-original-potential}. At $n=0$ the potential is its
initial value and both claims hold as equalities or inequalities
directly.
\end{proof}

\begin{proposition}[Expected error from rejected trials]
\label{prop:prestep-rejection-charge}
Suppose $Q$ is compact, $\phi$ is continuously differentiable on $Q$,
$K_{\rm tr}<\infty$, and all vector means and relative-value means
satisfy uniform finite conditional second-moment envelopes under
the histories preceding their respective batches. Assume their conditional means are respectively
$\nabla\phi(y_j)$ and $\phi(x_j^+)-\phi(y_j)$.
The curvature candidate is nondecreasing and at least
$\mathcal L_0>0$. Suppose additionally that the nonnegative
allowances have integrable total. Then all potential and error terms
below are integrable, and
\begin{align}
 \E\left[A_n[\phi(x_n)-\phi(u)]
                  +\frac{h_0}{2}\norm{u-v_n}^2\right]
 \le{}&\frac{h_0}{2}R_u^2
        +\E\sum_{j\in\mathcal A}\chi_j
        -\E\sum_{j\in\mathcal R}Z_j(u).
 \label{eq:prestep-rejection-expectation}
\end{align}
This is a weighted guarantee with a random terminal $A_n$.
The rejected-error term records the dependence introduced by model
selection.
\end{proposition}

\begin{proof}
We first justify integrability under a uniform second-moment
envelope, as in the stated oracle assumptions. Let
$c=h_0/\mathcal L_0$. From \eqref{eq:prestep-weight},
$a_j\le c+\sqrt{cA^-}$. By induction on the number $t$ of accepted
states, $A_t\le c(t+1)^2$: the bound at $t=0$ is immediate, and
$A_{t+1}\le c[(t+1)^2+(t+1)+1]\le c(t+2)^2$.
Thus every executed candidate has
$a_j\le c(K_{\rm tr}+2)$ and
$A_j^+\le c(K_{\rm tr}+2)^2$.
The accepted and tentative points belong to the compact set $Q$;
$\phi$ and all distances are bounded there. Uniform conditional
second moments, these weight bounds and the finite trial cap make
every $Z_j(u)$ integrable.

The execution indicator of each trial, its weight, first center and
$u-v^-$ are known before its vector samples are drawn. Therefore the
expectation of its first error inner product is zero. Its tentative
second center and $A_j^+$ are known before the function-difference samples are drawn, so the expectation of the second error term is zero separately.
Summing over the at most $K_{\rm tr}$ executed trials gives
\[
 \E\sum_{j\in\mathcal A\cup\mathcal R}Z_j(u)=0.
\]
No accepted or rejected indicator was included in these separate
centering steps. Since their sums partition every executed trial,
$\E\sum_{\mathcal A}Z_j=-\E\sum_{\mathcal R}Z_j$.
Taking expectations of \eqref{eq:prestep-pathwise-potential} proves
\eqref{eq:prestep-rejection-expectation}.
\end{proof}

For an original comparator, the right side of
\eqref{eq:prestep-rejection-expectation} additionally contains
$\E[A_n\beta_\gamma(u)]$ as in
\eqref{eq:prestep-original-potential}. Dividing by random $A_n$ and
then dropping a centered all-trial sum is not justified by this
proposition. A finite trial cap can also end before any acceptance.
An unweighted finite-second-moment risk theorem therefore needs
explicit control of the rejection charge, the terminal denominator
and the incomplete-execution event.

\subsection{The exact-record original-distance bound}
\label{subsec:prestep-exact}

\begin{theorem}[Exact-record scalar backtracking]
\label{thm:prestep-exact-distance}
Suppose $\phi$ is convex and $L$-smooth on $Q$, each $g_j$ equals
$\nabla\phi(y_j)$, each $\widehat q_j$ equals
$\phi(x_j^+)-\phi(y_j)$, and $\chi_j=0$.
Let the first curvature candidate be $\mathcal L_0$ and carry it
across successes, doubling it only on rejections. Define
\[
 J_\star=\left\lceil
       \log_2\max\{1,L/\mathcal L_0\}\right\rceil,\qquad
 \overline{\mathcal L}=\max\{\mathcal L_0,2L\}.
\]
A trial cap $K_{\rm tr}\ge N+J_\star$ guarantees $N$ acceptances.
There are at most $J_\star$ rejected trials, every curvature
candidate is at most $\overline{\mathcal L}$, and
\begin{equation}
 A_N\ge\frac{h_0N^2}{4\overline{\mathcal L}},\qquad
 \phi(x_N)-\phi(u)
 \le\frac{h_0R_u^2}{2A_N}
 \le\frac{2\overline{\mathcal L}R_u^2}{N^2}
 \label{eq:prestep-exact-rate}
\end{equation}
for every fixed $u\in Q$. The coefficient and the accepted trajectory
persist throughout. Each rejected trial has the full pair, stage and
projection costs stated in the execution.
\end{theorem}

\begin{proof}
Smoothness and the center identity give
\begin{align*}
 A_j^+\bigl[\phi(x_j^+)-\phi(y_j)
             -\ip{\nabla\phi(y_j)}{x_j^+-y_j}\bigr]
 &\le \frac{LA_j^+}{2}\norm{x_j^+-y_j}^2\\*
 &=\frac{La_j^2}{2A_j^+}\norm{w_j}^2
 =\frac{h_0L}{2\mathcal L_j}\norm{w_j}^2.
\end{align*}
Every candidate $\mathcal L_j\ge L$ therefore passes. Each rejection
must occur below $L$, so its doubling yields a candidate below $2L$.
Monotonicity of the candidate then proves the uniform bound and the
total rejection count. In particular, $N+J_\star$ trials suffice.

For an accepted weight,
$a_t=\sqrt{h_0A_t/\mathcal L_t}$ and $A_t=A_{t-1}+a_t$.
Consequently,
\[
 \sqrt{A_t}-\sqrt{A_{t-1}}
 =\frac{a_t}{\sqrt{A_t}+\sqrt{A_{t-1}}}
 \ge\frac12\sqrt{h_0/\mathcal L_t}
 \ge\frac12\sqrt{h_0/\overline{\mathcal L}}.
\]
Sum from $A_0=0$ to obtain the lower bound on $A_N$.
All retained errors and allowances vanish in
\eqref{eq:prestep-pathwise-potential}. Discarding its nonnegative
terminal distance and dividing by $A_N>0$ proves the risk bounds.
\end{proof}

This theorem is an exact-record statement. Relative smoothed-function
records generally have noise even for deterministic sample losses,
and a two-value vector record is generally random in $d>1$.
One exact two-value implementation exists for deterministic
one-dimensional quadratics: use the central pair $y\pm\gamma$ for
$g_j$, and an independent sampled function for the comparison pair $x_j^+,y_j$.
Both records are exact for the original quadratic population at every
legal $\gamma>0$. With one record of each type, the exact cost is
two pairs, four values, two stages and one projection per trial.
For a quadratic with $R=G/L$, the bound in
\eqref{eq:prestep-exact-rate} is $O(G^2/(LN^2))$ when
$L\ge\mathcal L_0$, overcoming the corresponding post-step
initialization trajectory while using the same original comparator.

\subsection{Unobserved curvature under arbitrary two-value queries}
\label{subsec:prestep-rare-curvature}

Rare-observation transcript couplings also appear in parameter-free
stochastic-gradient lower bounds
\citep[Section~4.4]{AttiaKoren2024}.
The limits of simultaneous parameter adaptation are studied by
\citet{khaled2024tuning}. The following construction treats an entire
sample function, so its coupling covers both function values at any
adaptive pair of endpoints. Every proof is given directly.

\begin{theorem}[A rare-curvature original-distance obstruction]
\label{thm:prestep-rare-curvature}
Fix $D,G,\gammaMax>0$ and a positive integer pair budget $T$.
Let $Q=[0,D]$, $x^0=0$, and let an arbitrary randomized adaptive
algorithm draw at most $T$ independent sampled functions, each evaluated at exactly two
endpoints in $[-\gammaMax,D+\gammaMax]$, and return a point in $Q$.
All samples from rejected trials, pilot batches and diagnostics are included in this
budget. The algorithm receives no supplied population smoothness,
sample regularity, variance or original-distance parameter.
Its other supplied inputs are held fixed.

For the deterministic sample function $F_0(x)=-Gx$, let $W\in Q$
be its output and suppose $w=(\E W^2)^{1/2}>0$.
For every $L\ge\max\{G/D,2G/w\}$ put
\begin{equation}
 p=\frac1{4T},\qquad
 F_L(x,\xi)=-Gx+\frac{L\xi}{2p}x^2,\qquad
 \xi\sim\operatorname{Bernoulli}(p).
 \label{eq:prestep-rare-sample}
\end{equation}
This oracle has a convex globally $L$-smooth population
$f_L(x)=-Gx+Lx^2/2$ with original stationary minimizer
$x^\star=R=G/L\in Q$. Its samples are convex and have finite
second moments on the entire evaluation domain. For the standard
one-dimensional symmetric vector record, uniform center variance
and sample regularity envelopes may be taken as
\begin{equation}
 \sigma_0=LD\sqrt{\frac{1-p}{p}},\qquad
 M_2=G+\frac{L(D+\gammaMax)}{\sqrt p}.
 \label{eq:prestep-rare-moments}
\end{equation}
Nevertheless the output $X_L$ obeys
\begin{align}
 \E[f_L(X_L)-f_L(x^\star)]
 &\ge\frac{3Lw^2}{32},\nonumber\\
 \frac{\sigma_0R}{\sqrt T}&\le2GD,\qquad
 \frac{M_2R}{\sqrt T}
 \le\frac{G^2}{L\sqrt T}+2G(D+\gammaMax).
 \label{eq:prestep-rare-risk}
\end{align}
The statement allows arbitrary nonsymmetric pairs and adaptive
endpoint locations. It imposes no restriction on the optimizer
states, restarts or sequential sampling stages.
\end{theorem}

\begin{proof}
The population mean and its minimizer follow directly from
$\E\xi=p$. Completing the square gives
$f_L(x)-f_L(x^\star)=L(x-G/L)^2/2$.
The lower bound on $L$ places the stationary minimizer inside $Q$.

Each sample is a convex polynomial. Put $b=D+\gammaMax$.
Its derivative is $-G+L\xi x/p$.
On the evaluation domain, a sample Lipschitz envelope is
$m(\xi)=G+Lb\xi/p$ and Minkowski gives
$(\E m(\xi)^2)^{1/2}\le G+Lb/\sqrt p$.
Thus the displayed $M_2$ is a valid uniform sample-gradient and
sample-Lipschitz second-moment envelope. The same envelope holds on
an open neighborhood of the evaluation domain. Indeed, choose
$0<\eta_{\rm dom}\le\min\{D/2,Gp/L\}$. On
$(-\gammaMax-\eta_{\rm dom},b+\eta_{\rm dom})$, the rare sample
has derivative supremum
$L(b+\eta_{\rm dom})/p-G$: its positive endpoint dominates the
negative endpoint because $LD/p-2G\ge2G$. The squared second moment
of this supremum, including the common sample, is
\[
 G^2-2GL(b+\eta_{\rm dom})
       +\frac{L^2(b+\eta_{\rm dom})^2}{p}.
\]
Subtracting it from $(G+Lb/\sqrt p)^2$ and using
$L\eta_{\rm dom}\le Gp$ gives a difference at least
$2GLb/\sqrt p+GL\eta_{\rm dom}>0$.
Each function value also has a finite second moment throughout
this bounded neighborhood. In dimension one the symmetric vector
record at any legal radius is exactly
$-G+L\xi z/p$ at center $z$, because the quadratic central
difference equals its derivative. Its centered variance is
$L^2z^2(1-p)/p\le L^2D^2(1-p)/p$.
This proves all moment assumptions without restricting the radius.

Couple the algorithm's internal random seed with its deterministic
execution on $F_0$. Draw $T$ independent Bernoulli variables in
advance, using them in order as sampled functions are requested.
Let $E$ be the event that they are all zero. It is independent of the
algorithm seed and has probability
$(1-p)^T\ge1-Tp=3/4$.
On $E$, every sampled loss function equals
$F_0$ at every point. Induction through every adaptive query and
internal decision shows that both values, both endpoint locations,
the stopping decision and the output have the deterministic-oracle
transcript. This argument also covers any samples unused by an early
stop. In particular, on $E$, $X_L=W$ under the common seed.

Since the loss gap is nonnegative everywhere,
\[
 \E[f_L(X_L)-f_L(x^\star)]
 \ge\frac{3L}{8}\E(W-R)^2.
\]
Cauchy--Schwarz gives $\E W\le w$, so
$\E(W-R)^2=w^2-2R\E W+R^2\ge(w-R)^2$.
The assumption $L\ge2G/w$ implies $R\le w/2$, proving the
first risk inequality. Finally $pT=1/4$ and $R=G/L$ give
\[
 \frac{\sigma_0R}{\sqrt T}
 =GD\sqrt{\frac{1-p}{pT}}\le2GD,\qquad
 \frac{M_2R}{\sqrt T}
 =\frac{G^2}{L\sqrt T}+\frac{Gb}{\sqrt{pT}},
\]
which proves the remaining bounds.
\end{proof}

\begin{corollary}[A necessary additional adaptation price]
\label{cor:prestep-rare-price}
Under Theorem~\ref{thm:prestep-rare-curvature},
\begin{equation}
 \E[f_L(X_L)-f_L(x^\star)]
 \ge\frac3{64}\frac{w^2}{D^2}
                      \frac{\sigma_0D}{\sqrt T}.
 \label{eq:prestep-rare-diameter-price}
\end{equation}
If the deterministic linear instance satisfies
$\E[F_0(W)-\min_QF_0]\le GD/2$, the right side is at least
$3\sigma_0D/(256\sqrt T)$.
For fixed supplied inputs and any finite constant $C$, an algorithm
with $w>0$ cannot have, uniformly over this oracle family,
\begin{equation}
 \E[f_L(X_L)-f_L(x^\star)]
 \le C\left(\frac{(L+h_0)R^2}{N^2}
           +\frac{(\sigma_0+M_2)R}{\sqrt T}\right).
 \label{eq:prestep-pure-distance-template}
\end{equation}
Multiplying its right side by any fixed finite product of powers of
logarithms of $L,\sigma_0,M_2,R^{-1}$ also cannot restore this uniform
claim. A separate adaptation charge is necessary in the stated
unknown-scale finite-moment model.
\end{corollary}

\begin{proof}
Equation~\eqref{eq:prestep-rare-moments} gives
$\sigma_0D/\sqrt T=LD^2\sqrt{4-1/T}\le2LD^2$.
Substitute this bound into the first inequality of
\eqref{eq:prestep-rare-risk}, proving
\eqref{eq:prestep-rare-diameter-price}.
For the linear-progress condition,
$G(D-\E W)\le GD/2$ implies $\E W\ge D/2$;
thus $w^2\ge(\E W)^2\ge D^2/4$.

In \eqref{eq:prestep-pure-distance-template},
$(L+h_0)R^2/N^2=G^2/(LN^2)+h_0G^2/(L^2N^2)$ tends to zero
as $L$ grows, and both statistical products remain bounded by
\eqref{eq:prestep-rare-risk}. Its right side is therefore bounded
for fixed supplied inputs, whereas the lower bound grows linearly in
$L$. This is a contradiction for sufficiently large $L$.
Along this family $\sigma_0,M_2,R^{-1}$ all grow at most linearly
in $L$, and every fixed product of their logarithmic powers grows
more slowly than $L$. The same contradiction applies.
\end{proof}

The construction identifies a necessary price for unknown scales
and progress on the linear instance. It does not specify the
optimal form of that price over all functions. The deterministic
pre-step theorem removes the curvature-transfer failure of the
studied anchor execution. The complete stochastic original-distance
guarantee additionally requires explicit sampling costs for acceptance tests and
control of smoothing charges at the original comparator.

\section{Stochastic pre-step acceleration with confidence bounds}
\label{sec:preconf}

We retain one accepted similar-triangles trajectory and keep its Euclidean
proximal coefficient fixed. A predictable model allowance controls both
the acceptance test and the retained errors. The construction uses a
supplied diameter and second-moment upper bounds. Its sample and
curvature-reference costs are stated explicitly.

\subsection{Paired records and finite-moment aggregation}
\label{subsec:preconf-records}

Let $Q\subset\mathbb R^d$ be nonempty, compact and convex, with exact
Euclidean projection and supplied diameter bound $D>0$. Fix $x^0\in Q$
and a legal radius $\gamma>0$. Write $\phi=f_\gamma$, and assume that
$\phi$ is convex and $L_\phi$-smooth on $Q$, with unknown $L_\phi\ge0$.
For a fixed original comparator $u\in Q$ put
\[
 R_u=\norm{u-x^0},\qquad \beta_\gamma(u)=\phi(u)-f(u)\ge0.
\]
We assume $f\le\phi$ and that $f$ is $M$-Lipschitz on $Q$.

At any center $y$ chosen before sampling, a symmetric paired record has
conditional mean $\nabla\phi(y)$ and centered second moment at most
the supplied positive bound $\bar\sigma^2$. At a pair
$x,y$ chosen before sampling, an independent sampled function supplies
\begin{equation}
 P_\gamma(x,y;U,\xi)
 =F(x+\gamma U,\xi)-F(y+\gamma U,\xi),
 \qquad U\sim\operatorname{Unif}(\mathbb B_2^d),
 \label{eq:preconf-relative-record}
\end{equation}
with $U$ independent of $\xi$. Its conditional mean is
$\phi(x)-\phi(y)$ and its conditional second moment is at most
$\overline M_2^2\norm{x-y}^2$, where $\overline M_2>0$ is supplied.
The weak-gradient and segment assumptions or random sample Lipschitz
assumption of Lemma~\ref{lem:prestep-relative-moments} gives this
relative record. Every sampled function is evaluated at exactly two points.
The vector and comparison batches use separate independent samples.

Median-based confidence amplification under finite second moments is
classical \citep{Minsker2015}. We use the following explicit block rule.
For an odd number $k$ of vector blocks, let $z_i$ be the mean of the
records in block $i$. For each $i$, take the median of
$\{\norm{z_i-z_l}:1\le l\le k\}$, including the zero distance to
itself. Return the block mean with smallest median, breaking ties by index.
For scalar blocks, return the ordinary median of their means.

\begin{lemma}[Robust records under finite second moments]
\label{lem:preconf-robust-means}
For positive error tolerances, suppose the vector observations are independent
conditional on the preceding history, with common mean
$\mu$ and centered second moment at most $\bar\sigma^2$ at the
center determined by that history. If each of $k$ odd blocks has at least
$72\bar\sigma^2/\eta_g^2$ records, the selected vector above
satisfies
\[
 \Pr\{\norm{g-\mu}>\eta_g\mid\text{preceding history}\}
 \le2^{-k/2}.
 \label{eq:preconf-vector-confidence}
\]
Under the same conditional independence assumption, suppose scalar
observations have conditional mean $q$ and
centered variance at most $\overline M_2^2s^2$, with $s\ge0$ known
before these observations. If each block has at least
$8\overline M_2^2/\eta_q^2$ records, their median satisfies
\[
 \Pr\{|\widehat q-q|>\eta_qs\mid\text{preceding history}\}
 \le2^{-k/2}.
 \label{eq:preconf-scalar-confidence}
\]
The vector and scalar aggregates need not be unbiased.
\end{lemma}

\begin{proof}
For a vector block mean $\bar z$ of size $b$, independence and
centering give $\E\norm{\bar z-\mu}^2\le\bar\sigma^2/b$.
With $r=\eta_g/3$, Chebyshev's inequality therefore bounds the
probability of $\norm{\bar z-\mu}>r$ by $1/8$.
If a strict majority of the block means lie within $r$ of $\mu$,
any one of these good blocks has median distance at most $2r$.
The selected block thus has at least $(k+1)/2$ blocks within $2r$.
This set intersects the good majority, so the triangle inequality
places the selected block within $3r=\eta_g$ of $\mu$.
Failure requires at least $(k+1)/2\ge k/2$ bad blocks.
Independence and a union bound over subsets give
$2^k(1/8)^{k/2}=2^{-k/2}$.

For scalar blocks and $s>0$, Chebyshev with radius $\eta_qs$ gives
the same $1/8$ bound. A good majority forces the scalar median
inside $[q-\eta_qs,q+\eta_qs]$. The same subset argument applies.
For $s=0$, variance zero makes every scalar observation equal to
its mean almost surely, so the conclusion holds directly.
All arguments are conditional on the history fixing the center,
scale and block size.
\end{proof}

\subsection{One execution and the sum of trial weights}
\label{subsec:preconf-execution}

Before sampling, choose $h_0,\mathcal L_0,\varepsilon>0$, integers $N\ge1$ and
$K_{\rm tr}\ge N$, and $\delta_{\rm conf}\in(0,1)$.
Initialize $A_0=0$, $x_0=v_0=x^0$, and curvature candidate
$\mathcal L_0$. At a trial $j$ with accepted state
$(A^-,x^-,v^-)$, let
\begin{align}
 \mathcal L_j a_j^2&=h_0(A^-+a_j),\qquad A_j^+=A^-+a_j,\nonumber\\
 a_j&=\frac{h_0+\sqrt{h_0^2+4h_0\mathcal L_jA^-}}
                    {2\mathcal L_j},\qquad
 y_j=\frac{A^-x^-+a_jv^-}{A_j^+}.
 \label{eq:preconf-weight}
\end{align}
Choose positive predictable $\eta_j^g,\eta_j^q$, put
$\eta_j=\eta_j^g+\eta_j^q$, and set
\[
 \chi_j=\frac{a_j^2\eta_j^2}{h_0}.
 \label{eq:preconf-allowance}
\]
Draw a robust vector $g_j$ and form
\begin{equation}
 v_j^+=\operatorname{Proj}_Q(v^--a_jg_j/h_0),\qquad
 x_j^+=\frac{A^-x^-+a_jv_j^+}{A_j^+},\qquad
 w_j=v_j^+-v^-.
 \label{eq:preconf-candidate}
\end{equation}
Draw a distinct robust relative-value record $\widehat q_j$ at
$(x_j^+,y_j)$. Accept precisely when
\begin{equation}
 A_j^+\bigl[\widehat q_j-\ip{g_j}{x_j^+-y_j}\bigr]
 \le\frac{h_0}{2}\norm{w_j}^2+\chi_j.
 \label{eq:preconf-test}
\end{equation}
A success commits the candidate state and carries its curvature
into the next iteration. A rejection leaves the accepted state fixed
and doubles only the curvature candidate. Subsequent records use
samples independent of the preceding history.
Stop after $N$ accepted trials or $K_{\rm tr}$ executed trials and
return the last accepted point, using $x^0$ if no trial was accepted.
The accepted trajectory and proximal coefficient are never
reinitialized.

\begin{lemma}[Sum of the executed trial weights]
\label{lem:preconf-weight-account}
For every record realization, let $n\le N$ be the number of accepted
trials and let $J_{\rm exec}\le K_{\rm tr}$ be the number executed.
Put $c=h_0/\mathcal L_0$. Then
\begin{align}
 A_n&\le cn^2,\nonumber\\
 \sum_{j\ {\rm executed}}a_j
 &\le A_n+2c+(2+\sqrt2)\sqrt{cA_n}\nonumber\\
 &\le c\bigl[N^2+(2+\sqrt2)N+2\bigr].
 \label{eq:preconf-weight-account}
\end{align}
These bounds include rejected trials and executions stopped before any acceptance.
\end{lemma}

\begin{proof}
For an accepted increment, $a_t^2\le cA_t$ because
$\mathcal L_t\ge\mathcal L_0$.
Write $s_t=\sqrt{A_t/c}$. Then
$s_t^2-s_{t-1}^2=a_t/c\le s_t$, so
$s_t-s_{t-1}\le s_t/(s_t+s_{t-1})\le1$.
The denominator is positive for an accepted step.
Starting from $s_0=0$ gives $s_n\le n$ and $A_n\le cn^2$.

Each rejection visits a different dyadic curvature level
$\mathcal L_0\,2^r$ because the candidate doubles on rejection
and never decreases. Its current accepted mass is at most $A_n$.
The explicit positive-root formula in \eqref{eq:preconf-weight}
therefore gives
$a_j\le c2^{-r}+\sqrt{cA_n}\,2^{-r/2}$.
Sum these bounds over the subset of levels actually rejected, and
then over all $r\ge0$. The two geometric sums are $2$ and
$(1-2^{-1/2})^{-1}=2+\sqrt2$.
Accepted increments sum to $A_n$, proving the first account.
Substitute $A_n\le cN^2$ for the second.
\end{proof}

\subsection{A simultaneous-error acceleration bound}
\label{subsec:preconf-error-contract}

\begin{theorem}[Pre-step acceleration on a simultaneous error event]
\label{thm:preconf-error-contract}
Suppose that at every executed trial
\begin{equation}
 \norm{g_j-\nabla\phi(y_j)}\le\eta_j^g,\qquad
 |\widehat q_j-[\phi(x_j^+)-\phi(y_j)]|
       \le\eta_j^q\norm{x_j^+-y_j}.
 \label{eq:preconf-error-event}
\end{equation}
Put
\[
 J_\star=\ceil{\log_2\max\{1,2L_\phi/\mathcal L_0\}},\qquad
 \overline{\mathcal L}=\max\{\mathcal L_0,4L_\phi\}.
\]
Every candidate at least $2L_\phi$ passes. There are at most $J_\star$
rejections and no candidate exceeds $\overline{\mathcal L}$.
If $K_{\rm tr}\ge N+J_\star$, the execution completes $N$ accepted
steps and $A_N\ge h_0N^2/(4\overline{\mathcal L})$.

Define $e_j=g_j-\nabla\phi(y_j)$,
$\zeta_j^q=\widehat q_j-[\phi(x_j^+)-\phi(y_j)]$, and
\[
 Z_j(u)=a_j\ip{e_j}{u-v^-}-A_j^+\zeta_j^q.
\]
For every fixed $u\in Q$ and every accepted prefix,
\begin{equation}
 A_n[f(x_n)-f(u)]+\frac{h_0}{2}\norm{u-v_n}^2
 \le\frac{h_0R_u^2}{2}
       +\sum_{j\ {\rm accepted}}(\chi_j+Z_j(u))
       +A_n\beta_\gamma(u).
 \label{eq:preconf-potential}
\end{equation}
If
\begin{equation}
 \eta_j=\min\left\{\frac{\varepsilon}{4D},
                 \sqrt{\frac{h_0\varepsilon}{4a_j}}\right\},
 \label{eq:preconf-tolerance}
\end{equation}
then the completed execution satisfies the unweighted bound
\begin{equation}
 f(x_N)-f(u)
 \le \frac{2\overline{\mathcal L}R_u^2}{N^2}
               +\frac{\varepsilon}{2}+\beta_\gamma(u).
 \label{eq:preconf-unweighted}
\end{equation}
\end{theorem}

\begin{proof}
The center identity is
$A_j^+(x_j^+-y_j)=a_jw_j$.
Smoothness and the two error bounds imply
\[
 A_j^+[\widehat q_j-\ip{g_j}{x_j^+-y_j}]
 \le\frac{h_0L_\phi}{2\mathcal L_j}\norm{w_j}^2
                   +a_j\eta_j\norm{w_j}.
\]
At $\mathcal L_j\ge2L_\phi$ the first term is at most
$h_0\norm{w_j}^2/4$. The elementary square inequality gives
$a_j\eta_j\norm{w_j}\le
h_0\norm{w_j}^2/4+a_j^2\eta_j^2/h_0$.
Thus this trial passes. Rejections occur only below $2L_\phi$;
their doubling gives a candidate below $4L_\phi$.
Monotonicity proves the rejection count and ceiling, including $L_\phi=0$.
The cap condition then ensures completion.
For accepted increments, $a_t=\sqrt{h_0A_t/\mathcal L_t}$, so
\[
 \sqrt{A_t}-\sqrt{A_{t-1}}
 =\frac{a_t}{\sqrt{A_t}+\sqrt{A_{t-1}}}
 \ge\frac12\sqrt{h_0/\overline{\mathcal L}}.
\]
Summing proves the denominator bound.

Projection optimality gives
\[
 a_j\ip{g_j}{v_j^+-u}
 \le\frac{h_0}{2}
   [\norm{u-v^-}^2-\norm{u-v_j^+}^2-\norm{w_j}^2].
\]
Convexity at $y_j$ gives
$A^-\phi(x^-)+a_j\phi(u)\ge
A_j^+\phi(y_j)+a_j\ip{\nabla\phi(y_j)}{u-v^-}$.
Combine these two inequalities and the center identity. The increase
of the potential
$A[\phi(x)-\phi(u)]+h_0\norm{u-v}^2/2$ is at most
\[
 A_j^+[\phi(x_j^+)-\phi(y_j)-\ip{g_j}{x_j^+-y_j}]
 -\frac{h_0}{2}\norm{w_j}^2
 +a_j\ip{e_j}{u-v^-}.
\]
Substitute the accepted model test and the definition of $\zeta_j^q$.
The increase is at most $\chi_j+Z_j(u)$.
Only accepted states change the potential, so these transitions
telescope. The inequality
$f(x_n)-f(u)\le\phi(x_n)-\phi(u)+\beta_\gamma(u)$ proves
\eqref{eq:preconf-potential}.

Under \eqref{eq:preconf-tolerance},
$\chi_j\le a_j\varepsilon/4$.
Both $u-v^-$ and $w_j$ have norm at most $D$, hence
$|Z_j(u)|\le a_jD(\eta_j^g+\eta_j^q)
\le a_j\varepsilon/4$.
Accepted weights sum to $A_N$. Discard the nonnegative terminal
distance in \eqref{eq:preconf-potential}, divide by the positive
$A_N$, and use its deterministic lower bound. This proves
\eqref{eq:preconf-unweighted} without centering selected errors.
\end{proof}

\subsection{An implemented finite-moment guarantee and exact resources}
\label{subsec:preconf-paid-risk}

Choose the smallest positive odd $k_{\rm mom}$ satisfying
\begin{equation}
 k_{\rm mom}\ge2\log_2(2K_{\rm tr}/\delta_{\rm conf}).
 \label{eq:preconf-block-number}
\end{equation}
Let
\begin{align}
 c_g&=72\bar\sigma^2,\qquad c_q=8\overline M_2^2,\nonumber\\
 \theta_g&=\frac{c_g^{1/3}}{c_g^{1/3}+c_q^{1/3}},
 \qquad \theta_q=1-\theta_g,\qquad
 \Lambda_{\rm mom}=(c_g^{1/3}+c_q^{1/3})^3 .
 \label{eq:preconf-moment-split}
\end{align}
Use \eqref{eq:preconf-tolerance},
$\eta_j^g=\theta_g\eta_j$ and $\eta_j^q=\theta_q\eta_j$.
Before the vector batch select integer block sizes and complete
batch counts
\begin{equation}
 b_j^g=\ceil{\frac{c_g}{\theta_g^2\eta_j^2}},\qquad
 b_j^q=\ceil{\frac{c_q}{\theta_q^2\eta_j^2}},\qquad
 B_j^g=k_{\rm mom}b_j^g,\quad B_j^q=k_{\rm mom}b_j^q .
 \label{eq:preconf-block-sizes}
\end{equation}
All vector blocks use the same first center, and all comparison
blocks use the same fixed tentative pair.

\begin{theorem}[Expected risk with confidence bounds]
\label{thm:preconf-paid-risk}
Under the stated moment assumptions and supplied upper bounds, this paired-oracle execution
has a simultaneous error event of probability at least $1-\delta_{\rm conf}$.
If $K_{\rm tr}\ge N+J_\star$, then for every fixed $u\in Q$ its
feasible returned point $\widehat x$ satisfies
\begin{equation}
 \E[f(\widehat x)-f(u)]
 \le\frac{2\overline{\mathcal L}R_u^2}{N^2}
       +\frac{\varepsilon}{2}+\beta_\gamma(u)+\delta_{\rm conf} MD.
 \label{eq:preconf-expected-risk}
\end{equation}
For every realization, the exact resources are
\begin{equation}
 T_{\rm exec}=k_{\rm mom}
       \sum_{j\ {\rm executed}}(b_j^g+b_j^q),\qquad
 2T_{\rm exec}\ \text{values},\quad
 2J_{\rm exec}\ \text{stages},\quad
 J_{\rm exec}\ \text{projections}.
 \label{eq:preconf-exact-resources}
\end{equation}
These counts include all rejected trials. The sample count satisfies
\begin{equation}
 T_{\rm exec}\le k_{\rm mom}\left[
 2J_{\rm exec}+\Lambda_{\rm mom}
 \sum_{j\ {\rm executed}}\max\left\{
       \frac{16D^2}{\varepsilon^2},
       \frac{4a_j}{h_0\varepsilon}\right\}\right].
 \label{eq:preconf-realized-budget}
\end{equation}
For every record realization, including incomplete executions,
\begin{equation}
 T_{\rm exec}\le k_{\rm mom}\left[
 2K_{\rm tr}+\frac{16\Lambda_{\rm mom}K_{\rm tr}D^2}{\varepsilon^2}
 +\frac{4\Lambda_{\rm mom}}{\mathcal L_0\varepsilon}
             [N^2+(2+\sqrt2)N+2]\right].
 \label{eq:preconf-public-budget}
\end{equation}
On the simultaneous event the stage and projection counts are at
most $2(N+J_\star)$ and $N+J_\star$, respectively.
\end{theorem}

\begin{proof}
Condition on the history preceding each batch. Its center, accuracy
bound and integer block size are fixed before its samples are drawn.
Lemma~\ref{lem:preconf-robust-means} and
\eqref{eq:preconf-block-number} bound each of the at most
$2K_{\rm tr}$ conditional failure probabilities by
$\delta_{\rm conf}/(2K_{\rm tr})$.
An executed-trial indicator is predictable before its first batch;
the tentative comparison pair is known before its second batch.
Taking expectations of these conditional bounds and summing gives
failure probability at most $\delta_{\rm conf}$. This union is over the
executed sequence, rather than hypothetical center locations.

On the joint event, Theorem~\ref{thm:preconf-error-contract}
guarantees completion and \eqref{eq:preconf-unweighted}.
On its complement the output remains feasible and
$f(\widehat x)-f(u)\le MD$.
The good-event upper bound is nonnegative, since $\beta_\gamma(u)\ge0$.
Splitting the expectation over the two events proves
\eqref{eq:preconf-expected-risk}. Neither conditional centering of the
robust aggregate nor division by a random failed-execution weight
is invoked.

The vector and scalar blocks use separate independent sampled
functions, each evaluated at exactly two points. All vector records can be drawn
in one parallel stage. After their aggregate determines the
projection, all comparison records can be drawn in the next stage.
There is one projection per trial. This proves the exact counts.
For any split $\theta\in(0,1)$, the continuous block-size bound is
$c_g/\theta^2+c_q/(1-\theta)^2$. Its derivative vanishes only at
$\theta=\theta_g$, and its second derivative is positive.
Substitution gives the minimum $\Lambda_{\rm mom}$.
The ceilings in \eqref{eq:preconf-block-sizes} therefore imply
$b_j^g+b_j^q\le2+\Lambda_{\rm mom}/\eta_j^2$.
The reciprocal of the squared tolerance is
$\max\{16D^2/\varepsilon^2,4a_j/(h_0\varepsilon)\}$,
proving \eqref{eq:preconf-realized-budget}.
Finally bound each maximum by its two-term sum, use
$J_{\rm exec}\le K_{\rm tr}$ and apply
Lemma~\ref{lem:preconf-weight-account}. This proves the deterministic
budget bound determined before sampling. The good-event rejection count proves the last
stage and projection bounds.
\end{proof}

Computing the selected vector block uses $O(dk_{\rm mom}^2)$
arithmetic per trial, besides averaging the records. Thus the total
arithmetic outside function evaluation and projection is
$O(dT_{\rm exec}+dk_{\rm mom}^2J_{\rm exec})$.
Only scalar curvature is maintained. These confidence costs include
an executed-trial factor in the $\varepsilon^{-2}$ term and a
term depending on the initial curvature reference; they do not give the tuned joint
Gasnikov work bound.

\begin{corollary}[A trial cap chosen before sampling for ball smoothing]
\label{cor:preconf-public-cap}
Suppose the supplied $\overline M_2$ also bounds the sample
weak-gradient and segment assumptions, or the random sample Lipschitz
assumption, on the evaluation domain. Then one may set
\begin{align}
 K_{\rm tr}&=N+
  \ceil{\log_2\max\{1,2\sqrt d\,\overline M_2/
                               (\gamma\mathcal L_0)\}},\nonumber\\
 \delta_{\rm conf}&=\min\{1/2,\varepsilon/(4\overline M_2D)\}.
 \label{eq:preconf-fully-public-cap}
\end{align}
No population smoothness or original-distance input is used. The
expected risk obeys
\begin{equation}
 \E[f(\widehat x)-f(u)]
 \le\frac{2\max\{\mathcal L_0,4L_\phi\}R_u^2}{N^2}
                  +\frac{3\varepsilon}{4}+\beta_\gamma(u),
 \label{eq:preconf-fully-public-risk}
\end{equation}
where $L_\phi$ can be any valid smoothness constant of $\phi$.
For an originally $L$-smooth population, one may take $L_\phi\le L$.
Consequently the population-dependent sufficient condition
$2\max\{\mathcal L_0,4L_\phi\}R_u^2/N^2+\beta_\gamma(u)
\le\varepsilon/4$ gives expected gap at most $\varepsilon$.
This condition is used in the analysis and does not define a stopping
rule. The exact and deterministic budget bounds of
Theorem~\ref{thm:preconf-paid-risk} apply.
\end{corollary}

\begin{proof}
Under the sample Lipschitz assumption, taking expectations of sample
differences gives population Lipschitz constant at most
$\E m(\xi)\le\overline M_2$.
Under the weak-gradient and segment assumptions, integrate along a
fixed segment and use
$\E\norm{\nabla F(z,\xi)}
\le(\E\norm{\nabla F(z,\xi)}^2)^{1/2}\le\overline M_2$.
Its stated integrability and continuity give the same inequality,
including boundary endpoints. Hence $M\le\overline M_2$ may be used.
The ball-smoothing bound supplies the valid constant
$\sqrt d\,\overline M_2/\gamma$.
For any other valid $L_\phi$, their minimum is also valid.
Apply Theorem~\ref{thm:preconf-paid-risk} with this minimum;
the cap chosen before sampling in \eqref{eq:preconf-fully-public-cap} is then at
least $N+J_\star$, and its curvature ceiling is at most
$\max\{\mathcal L_0,4L_\phi\}$. Also
$\delta_{\rm conf} MD\le\varepsilon/4$.
Apply \eqref{eq:preconf-expected-risk} and the sufficient condition.
The cap and confidence parameters depend only on the displayed
supplied inputs and moment bound.
\end{proof}

\section{The cost of relative-mean certification}
\label{sec:adaptive-certification}

We study upper confidence bounds for the mean of a function-value
difference, allowing stopping rules with no predetermined sample limit.
Without a supplied second-moment bound, a uniformly valid finite upper
bound is impossible. With such a bound, a median-of-means estimator
achieves the necessary sampling order.

\subsection{A necessary cost of adaptive mean certification}
\label{subsec:observable-certification-lower}

Let $Z$ be a scalar random variable, sampled independently on each call,
and define the sample function
$F(x,Z)=Zx$ on $\mathbb R$. Every sample and its population are affine
and convex. For any fixed $x-y=\ell>0$, its same-sample smoothed
relative-function record satisfies
\[
 \frac{F(x+\gamma U,Z)-F(y+\gamma U,Z)}{\ell}=Z,
\]
at every legal $\gamma$, independently of the ball direction $U$.
Thus $\mu=\mathbb E Z$ is a normalized relative-function mean, and
$\mathbb E Z^2<\infty$ is a sample-gradient and sample-Lipschitz
second-moment envelope.

Consider a randomized adaptive procedure that samples functions independently,
evaluates each at exactly two points, and either reports an upper
bound $\widehat U$ or continues sampling. The lower bound also holds
if each sampled function is revealed in full. Let $\tau\in\{0,1,\ldots\}\cup\{\infty\}$
be the total number of sampled functions when it reports; a report of
$+\infty$ is permitted. Every sample used for a pilot, a vector estimate
or another diagnostic is included in $\tau$. The procedure has the same supplied inputs and
internal seed law on every population under comparison.
For $\delta_{\rm conf}\in(0,1)$, its one-sided validity requirement is
\begin{equation}
 \Pr(\tau<\infty,\ \mu>\widehat U)\le\delta_{\rm conf}.
 \label{eq:observable-certification-validity}
\end{equation}

\begin{theorem}[Adaptive certification cost on the common history]
\label{thm:observable-certification-cost}
Under the null law $Z=0$, let
$C=\{\tau<\infty,\widehat U\le\eta\}$, where $\eta>0$, and let
$q=\Pr_0(C)>0$. For $\mu>\eta$ and $0<p<1$, define the alternative
\[
 Z=\frac{\mu}{p}\,\xi,\qquad \xi\sim\operatorname{Bernoulli}(p).
\]
Its mean is $\mu$ and its second moment is $\mu^2/p$.
The alternative false-report probability satisfies
\begin{equation}
 \Pr_1(\tau<\infty,\mu>\widehat U)
 \ge\mathbb E_0[\mathbf1_C(1-p)^\tau]
 \ge q\exp\left(
   \frac{\log(1-p)}{q}\,
   \mathbb E_0[\tau\mathbf1_C]\right).
 \label{eq:observable-certification-coupling}
\end{equation}
The second bound is interpreted as zero if the restricted expectation
is infinite. If \eqref{eq:observable-certification-validity} holds on
this alternative and $q>\delta_{\rm conf}$, then
\begin{equation}
 \mathbb E_0[\tau\mathbf1_C]
 \ge\frac{q\log(q/\delta_{\rm conf})}{-\log(1-p)}.
 \label{eq:observable-certification-price}
\end{equation}
The result covers arbitrary endpoints, adaptive batches and stopping,
and any number of sequential sampling stages.
\end{theorem}

\begin{proof}
Couple the internal seed on the two populations, including all
independent direction choices. Under the null, every realized
function is identically zero, so the entire transcript, terminal
bound and sample count are functions of this seed. Prepare an infinite
independent Bernoulli sequence for the alternative. For each null
seed in $C$, the event that its first $\tau$ Bernoulli variables
vanish has probability $(1-p)^\tau$. On this event every complete
sample function agrees with the null, at every possible endpoint.
Induction through the adaptive transcript therefore identifies all
queries, batches, decisions and the terminal report with the null
execution. The reported bound is at most $\eta<\mu$, so it is false.
Integration over the seed proves the first inequality in
\eqref{eq:observable-certification-coupling}. This argument uses
only the samples drawn before the null report and requires no
deterministic horizon.

Conditional on $C$, the function $t\mapsto(1-p)^t$ is convex.
If $\mathbb E_0[\tau\mid C]<\infty$, Jensen gives the second bound.
If that expectation is infinite, its stated zero bound is immediate.
When validity holds and $q>\delta_{\rm conf}$, comparison with $\delta_{\rm conf}$, taking
logarithms and dividing by $-\log(1-p)>0$ gives
\eqref{eq:observable-certification-price}; an infinite restricted
expectation also satisfies it.
\end{proof}

\begin{corollary}[No finite upper confidence bound without a moment bound]
\label{cor:observable-certification-no-cap}
Suppose \eqref{eq:observable-certification-validity} is required
uniformly over all laws with a finite second moment, and no upper
moment or range bound is supplied. Under the null law,
\begin{equation}
 \Pr_0(\tau<\infty,\widehat U<+\infty)\le\delta_{\rm conf}.
 \label{eq:observable-certification-no-cap}
\end{equation}
This holds even when $\tau$ has an unbounded distribution or infinite
expectation.
\end{corollary}

\begin{proof}
For each integer $k\ge1$, let
$C_k=\{\tau<\infty,\widehat U\le k\}$.
Choose the alternative mean $\mu=k+1$ and let $p\downarrow0$.
Every alternative has finite second moment, so validity and
Theorem~\ref{thm:observable-certification-cost} imply
$\mathbb E_0[\mathbf1_{C_k}(1-p)^\tau]\le\delta_{\rm conf}$.
On $C_k$, $\tau$ is finite. Dominated convergence therefore gives
$\Pr_0(C_k)\le\delta_{\rm conf}$. The increasing union of the $C_k$ is the
event in \eqref{eq:observable-certification-no-cap}.
Continuity from below completes the proof.
\end{proof}

The conclusion concerns a uniformly valid upper confidence bound for the population mean.
Algorithms can still use finite observed batches and prove
population-dependent expected-risk guarantees without such a report.

\begin{theorem}[The certification scale under a supplied moment cap]
\label{thm:observable-certification-capped}
Suppose validity is required on the class $\mathbb E Z^2\le \overline M_2^2$,
where the common cap $\overline M_2>0$ is supplied. If $0<\eta<\overline M_2$ and
$C=\{\tau<\infty,\widehat U\le\eta\}$ and
$q=\Pr_0(C)>\delta_{\rm conf}$, then
\begin{equation}
 \mathbb E_0[\tau\mathbf1_C]
 \ge\frac{q\log(q/\delta_{\rm conf})}
 {-\log(1-\eta^2/\overline M_2^2)}.
 \label{eq:observable-certification-capped-lower}
\end{equation}
In particular, when $\overline M_2^2\ge2\eta^2$,
\begin{equation}
 \mathbb E_0[\tau\mathbf1_C]
 \ge\frac{q \overline M_2^2}{2\eta^2}\log(q/\delta_{\rm conf}).
 \label{eq:observable-certification-capped-order}
\end{equation}

Conversely, for any $\delta_{\rm conf}\in(0,1)$, set
$b=\lceil8\overline M_2^2/\eta^2\rceil$ and let $k$ be the smallest positive
odd integer with $k\ge2\log_2(1/\delta_{\rm conf})$.
Draw $bk$ records in $k$ disjoint blocks of size $b$, let $m$ be
the median of their block averages, and report
$\widehat U=m+\eta$. This bound has validity at least $1-\delta_{\rm conf}$
uniformly over the stated class, and under the null it satisfies
$\widehat U=\eta$ surely. It uses exactly $bk$ independent pairs,
$2bk$ values, and one sampling stage at a fixed pair of centers.
For $\overline M_2^2\ge2\eta^2$ and $\delta_{\rm conf}\le1/4$, this matches the necessary
$\overline M_2^2\eta^{-2}\log(1/\delta_{\rm conf})$ order for $q=1$, even though the
null law has zero actual variance.
\end{theorem}

\begin{proof}
For every $\mu\in(\eta,\overline M_2)$ choose $p=\mu^2/\overline M_2^2$ in
Theorem~\ref{thm:observable-certification-cost}. The resulting law
has second moment exactly $\overline M_2^2$. Let $\mu\downarrow\eta$ in its
necessary cost inequality to obtain
\eqref{eq:observable-certification-capped-lower}.
For $x=\eta^2/\overline M_2^2\le1/2$,
$-\log(1-x)=\int_0^x(1-t)^{-1}\,dt\le2x$, proving
\eqref{eq:observable-certification-capped-order}.

For the upper construction, the variance of each block average is
at most $\overline M_2^2/b\le\eta^2/8$. Chebyshev therefore bounds its
probability of deviating by more than $\eta$ from the mean by $1/8$.
If the median deviates by more than $\eta$, at least $(k+1)/2$
blocks are bad. Their independence and a union bound over subsets
give failure probability at most
\[
 2^k(1/8)^{(k+1)/2}
 =2^{-k/2-3/2}\le2^{-k/2}\le\delta_{\rm conf}.
\]
Thus $m+\eta$ is an upper mean bound with the stated coverage.
Under the null, every block average is zero, giving the asserted
report and $q=1$. All $bk$ functions can be sampled in parallel and evaluated at the
same fixed pair of centers. Each sampled function is evaluated twice.
Finally, $b\le8\overline M_2^2/\eta^2+1$ and
$k<2\log_2(1/\delta_{\rm conf})+2$ show the matching order.
\end{proof}

The supplied cap in the matching construction is an additional input.
It is not estimated by the null data. Replacing $F(x,Z)$ by
$-Gx+Zx$ and subtracting its known linear contribution gives the
same normalized record, with a nonzero common gradient and sample
Lipschitz root-second-moment at most $G+\overline M_2$. Thus the necessary
sampling lower bound also applies to relative model residuals
around a nonzero common first step.

\section{A smooth frontier for unknown variance and original distance}
\label{sec:universal-distance-frontier}

We give a cross-instance expected-risk relation with a supplied diameter
and a supplied common smoothness constant. The lower bounds allow
the entire sampled function to be revealed and
therefore apply to arbitrary pairs of evaluations of that function. The matching upper
construction is restricted to the explicitly specified two-law family.

\begin{lemma}[A common-curvature family]
\label{lem:universal-distance-family}
Fix integers $T\ge2$, $N\ge1$, constants $D,L>0$, and $0<r<D$.
Let $Q=[0,D]$, $x^0=r$, and let the evaluation interval be
$I=[-D,2D]$. Define on the real line
\begin{equation}
 F_0(x)=\frac L2(x-D)^2,\qquad
 F_*(x)=F_0(x)+LDTx.
 \label{eq:universal-distance-samples}
\end{equation}
Under $\mathbb P_0$ every label returns $F_0$. Under $\mathbb P_1$
independent draws return $F_*$ with probability $1/T$ and $F_0$
otherwise. Both laws have the same supplied $Q,x^0,D,L,N,T$.
Their population objectives, stationary minimizers, original distances
and uniform centered gradient variances satisfy
\begin{align}
 f_0(x)&=\frac L2(x-D)^2,&x_0^\star&=D,&R_0&=D-r,&\sigma_0&=0,
 \nonumber\\
 f_1(x)&=\frac L2x^2+\frac L2D^2,&x_1^\star&=0,&R_1&=r,
 &\sigma_1^2&=L^2D^2(T-1).
 \label{eq:universal-distance-parameters}
\end{align}
All sample and population functions are convex and globally
$L$-smooth. On the open neighborhood $J=(-2D,3D)$ of $I$, both
populations admit the common Lipschitz bound $3LD$, and valid uniform
sample-gradient second-moment envelopes are
\begin{equation}
 M_{2,0}=3LD,\qquad M_{2,1}=LD\sqrt{T+8}.
 \label{eq:universal-distance-moments}
\end{equation}
The smallest uniform envelopes on the closed interval $I$ itself are
$2LD$ and $LD\sqrt{T+3}$.
\end{lemma}

\begin{proof}
The mean of the added linear term under $\mathbb P_1$ is $LDx$.
Expanding the square proves the population formulas and the two
stationary minimizers. If $\xi$ indicates a rare label, the centered
gradient under $\mathbb P_1$ is $LD(T\xi-1)$. Its variance is
$L^2D^2(T-1)$, independently of the query point. Consequently
\[
 \E_0|F_0'(x)|^2=L^2(x-D)^2,\qquad
 \E_1|F'(x)|^2=L^2x^2+L^2D^2(T-1).
\]
On $J$, $|x-D|<3D$ and $|x|<3D$, proving the stated Lipschitz
and moment bounds. Maximizing the two displayed expressions over
$[-D,2D]$ gives the exact closed-interval envelopes. All second
derivatives equal $L$.
\end{proof}

\begin{theorem}[A universal cross-instance risk relation]
\label{thm:universal-distance-pareto}
Consider any randomized adaptive procedure drawing at most $T$ independent
sampled functions from the family in Lemma~\ref{lem:universal-distance-family}
and returning a feasible output $X\in Q$. Its supplied inputs and
internal seed law are the same on the two populations. Each draw may
reveal the entire sampled function; queries, stopping decisions and
sampling schedules may be adaptive. Let $X_0$ denote its output under
$\mathbb P_0$ and put
\[
 \varepsilon_i=\E_i[f_i(X)-f_i(x_i^\star)],\qquad
 q_T=(1-1/T)^T.
\]
Then $q_T\ge1/4$ and
\begin{equation}
 \varepsilon_1\ge q_T\frac L2\E_0X_0^2
 \ge q_T\frac L2
       \left(D-\sqrt{2\varepsilon_0/L}\right)_+^2.
 \label{eq:universal-distance-pareto}
\end{equation}
The statement includes arbitrary nonsymmetric two-value queries and
does not depend on their number of sequential sampling stages.
\end{theorem}

\begin{proof}
Couple the algorithm's internal random seed on the two laws. Prepare
$T$ independent rare indicators for the alternative law, assigning
them to sampled functions in their order of use; pad an early stop with
unused indicators. The event $E$ that all $T$ indicators are zero has
probability $q_T$ and is independent of the seed. On $E$, every
revealed function is exactly $F_0$. Induction through the adaptive
transcript therefore identifies all queries, internal decisions and
the final output with the deterministic-law execution. This also
covers endpoints chosen adaptively within a same-sample pair.

The alternative excess risk is $LX^2/2$ and is nonnegative on every
transcript. Restricting its expectation to $E$ proves the first
inequality. Under the deterministic law,
$2\varepsilon_0/L=\E_0(D-X_0)^2$. Since $0\le X_0\le D$,
Cauchy--Schwarz gives
$\E_0X_0\ge D-\sqrt{2\varepsilon_0/L}$, while Jensen gives
$\E_0X_0^2\ge(\E_0X_0)^2$. This proves the second inequality.
Compactness ensures all moments used here are finite.

For the probability estimate, set $g(x)=x\log(1-1/x)$ for $x\ge2$.
With $u=1/(x-1)$ its derivative is $u-\log(1+u)\ge0$, because
$\log(1+u)=\int_0^u(1+s)^{-1}\,ds\le u$.
Thus $q_T=\exp(g(T))\ge\exp(g(2))=1/4$.
\end{proof}

\begin{theorem}[An exact competitive factor on the two-law family]
\label{thm:universal-distance-menu}
For deterministic benchmarks $b_0,b_1>0$, the smallest factor $C$
for which one procedure can satisfy $\varepsilon_i\le Cb_i$ on both
laws is
\begin{equation}
 C_* =\frac{q_TLD^2}
 {2\bigl(\sqrt{b_1}+\sqrt{q_Tb_0}\bigr)^2}.
 \label{eq:universal-distance-menu-factor}
\end{equation}
The lower bound permits full-function access. A procedure that knows
the specified two-law family and both benchmark values attains this
factor using one sampling stage, exactly $T$ same-sample pairs and
$2T$ function values, one feasible output, and no projections or
optimizer restarts.
\end{theorem}

\begin{proof}
Put $z=\E_0X_0$. Jensen and
Theorem~\ref{thm:universal-distance-pareto} give
\[
 \varepsilon_0\ge\frac L2(D-z)^2,\qquad
 \varepsilon_1\ge q_T\frac L2z^2.
\]
Hence $\varepsilon_i\le Cb_i$ implies
\[
 D-z\le\sqrt{2Cb_0/L},\qquad
 z\le\sqrt{2Cb_1/(q_TL)}.
\]
Adding and squaring proves the lower bound in
\eqref{eq:universal-distance-menu-factor}.

For attainment, query endpoints $0,D$ for each of $T$ independently
sampled functions. For the common function, the difference is $-LD^2/2$; for the
rare function, the difference is $-LD^2/2+LD^2T$. Thus any occurrence of the rare function is
identified exactly. If at least one is observed, output zero.
Otherwise output the feasible point
\[
 z_* =\frac{D\sqrt{b_1}}
 {\sqrt{b_1}+\sqrt{q_Tb_0}}.
\]
The deterministic-law risk is $L(D-z_*)^2/2$ and the alternative-law
risk is $q_TLz_*^2/2$. Substitution shows that both benchmark ratios
are $C_*$. All pairs can be sampled in one stage.
\end{proof}

\begin{corollary}[A local accelerated variance benchmark]
\label{cor:universal-distance-variance}
For
\[
 b_i^{\rm var}=\frac{LR_i^2}{N^2}
                    +\frac{\sigma_iR_i}{\sqrt T},
\]
every procedure satisfying $\varepsilon_i\le Cb_i^{\rm var}$ on
both laws, with $C\ge0$, must have
\begin{equation}
 C\ge\frac1{64}\min\{N^2,D/r\}.
 \label{eq:universal-distance-variance-price}
\end{equation}
For $0<r\le D/2$, the exact factor on the disclosed family obeys
\begin{equation}
 \frac1{64}\min\{N^2,D/r\}
 \le C_*^{\rm var}\le2\min\{N^2,D/r\}.
 \label{eq:universal-distance-variance-match}
\end{equation}
\end{corollary}

\begin{proof}
The two benchmarks are
\[
 b_0^{\rm var}=\frac{L(D-r)^2}{N^2},\qquad
 b_1^{\rm var}=\frac{Lr^2}{N^2}
                  +LDr\sqrt{1-1/T}\le2LDr.
\]
If $C\ge N^2/8$, the first branch of the asserted minimum follows.
Otherwise $\varepsilon_0<LD^2/8$, so
\eqref{eq:universal-distance-pareto} gives
$\varepsilon_1\ge q_TLD^2/8\ge LD^2/32$.
Dividing by $b_1^{\rm var}\le2LDr$ gives $C\ge D/(64r)$.
These alternatives prove
\eqref{eq:universal-distance-variance-price}.

Formula~\eqref{eq:universal-distance-menu-factor} is bounded above
by $LD^2/(2b_0^{\rm var})$ and by
$q_TLD^2/(2b_1^{\rm var})$. For $r\le D/2$ the first bound is
at most $2N^2$. For $T\ge2$,
$b_1^{\rm var}\ge LDr/\sqrt2$, so the second bound is at most
$D/r$. This proves the upper bound in
\eqref{eq:universal-distance-variance-match}.
\end{proof}

\begin{corollary}[A local second-moment benchmark]
\label{cor:universal-distance-moment}
Assume $T\ge4$ and use the neighborhood envelopes in
\eqref{eq:universal-distance-moments}. For
\[
 b_i^{\rm mom}=\frac{LR_i^2}{N^2}
                    +\frac{M_{2,i}R_i}{\sqrt T},
\]
every procedure satisfying $\varepsilon_i\le Cb_i^{\rm mom}$ on
both laws, with $C\ge0$, must have
\begin{equation}
 C\ge\frac1{96}\min\{N^2,\sqrt T,D/r\}.
 \label{eq:universal-distance-moment-price}
\end{equation}
For $0<r\le D/2$, the exact competitive factor also satisfies
\begin{equation}
 C_*^{\rm mom}\le2\min\{N^2,\sqrt T,D/r\}.
 \label{eq:universal-distance-moment-match}
\end{equation}
\end{corollary}

\begin{proof}
Set $a=N^{-2}+3T^{-1/2}$. Since
$b_0^{\rm mom}\le LD^2a$, either $C\ge1/(8a)$, which gives
$C\ge\min\{N^2,\sqrt T\}/32$, or
$\varepsilon_0<LD^2/8$. In the latter case
\eqref{eq:universal-distance-pareto} gives
$\varepsilon_1\ge LD^2/32$. For $T\ge4$,
\[
 b_1^{\rm mom}=\frac{Lr^2}{N^2}
       +LDr\sqrt{1+8/T}\le(1+\sqrt3)LDr\le3LDr,
\]
so $C\ge D/(96r)$. These alternatives prove the lower bound.

For $r\le D/2$,
\[
 b_0^{\rm mom}\ge\frac{LD^2}{4N^2},\qquad
 b_0^{\rm mom}\ge\frac{3LD^2}{2\sqrt T},\qquad
 b_1^{\rm mom}\ge LDr.
\]
Apply these three inequalities to the two denominator bounds from
\eqref{eq:universal-distance-menu-factor} to obtain
$C_*^{\rm mom}\le\min\{2N^2,\sqrt T/3,D/(2r)\}$, proving
the stated upper bound.
\end{proof}

\begin{corollary}[A supplied initial coefficient]
\label{cor:universal-distance-startup}
Fix a supplied $h_0>0$ and choose $L\ge h_0$. If the same procedure
satisfies on both laws
\[
 \varepsilon_i\le C\left(
 \frac{(L+h_0)R_i^2}{N^2}
       +\frac{\sigma_iR_i}{\sqrt T}\right),\qquad C\ge0,
\]
then $C\ge\min\{N^2/16,D/(96r)\}$.
\end{corollary}

\begin{proof}
The deterministic benchmark is at most $2LD^2/N^2$. If
$C<N^2/16$, the deterministic risk is below $LD^2/8$, and
\eqref{eq:universal-distance-pareto} gives alternative risk at least
$LD^2/32$. Its benchmark is at most
$2Lr^2/N^2+LDr\le3LDr$, proving $C\ge D/(96r)$.
The other alternative is $C\ge N^2/16$.
\end{proof}

The two populations share the supplied diameter, smoothness constant
and population Lipschitz envelope. Their intrinsic variance and
original optimizer distance are unknown. Supplying the exact
law-dependent distance, variance or moment parameter changes the
supplied inputs and does not satisfy the common-input premise. A common
upper envelope may be supplied, but its benchmark differs from the
intrinsic benchmarks in these corollaries.

A diameter-based statistical term can absorb the alternative loss.
Thus these results are compatible with diameter-based upper bounds.
They establish no additional price specifically caused by an unknown
diameter, and no sampling-stage lower bound. The exact upper match
concerns the specified two-law family; a matching guarantee on the full
finite-second-moment class remains open.

\subsection{Proof of Corollary~\ref{cor:universal-distance-main}}
\label{subsec:proof-universal-distance-main}
\begin{proof}
Use the two sample-function laws in
Lemma~\ref{lem:universal-distance-family}, with the common supplied
inputs stated there. The neighborhood bounds
\eqref{eq:universal-distance-moments} verify the second-moment
assumptions on the permitted evaluation interval. Apply
Corollary~\ref{cor:universal-distance-moment}. Its full-function
coupling applies to every arbitrary-pair procedure and gives the
stated factor. Both laws have convex samples and the same supplied
population smoothness constant. No original-distance, variance or
law-specific moment parameter is supplied.
\end{proof}

The factor is tight for this pair: a method that knows both laws and their benchmark values attains it (Theorem~\ref{thm:universal-distance-menu}); matching bounds for the full class remain open. Unknown curvature can require a diameter-scale statistical error term as $R\to0$ (Corollary~\ref{cor:prestep-rare-price}). 

\section{Protected radii, trajectory bounds and corruption}
\label{sec:protected-extensions}
\subsection{A complete observed-radius execution}

The same two-value interface allows the radius to depend on observations.
Define the range before sampling
$\bar\gamma=\min\{\gammaMax,d^{1/4}D/N\}$,
$\underline\gamma=\bar\gamma/2$, and $\gamma_0=\bar\gamma$.
Before the vector pilots at $y_t$, draw one sampled function and a uniform
ball direction $U_t$, independently of each other and of the preceding
history, and compute
\begin{align}
 V_t&=\widetilde F(y_t+\gamma_{t-1}U_t,\xi_t)
       -\widetilde F(y_t+\gamma_{t-1}U_t/2,\xi_t),\nonumber\\
 q_t&=\frac{[V_t]_+}{s_\gamma+[V_t]_+},\qquad
 \gamma_t=\max\{\underline\gamma,\gamma_{t-1}/(1+q_t)\}.
 \label{eq:observed-radius}
\end{align}
Here $s_\gamma>0$ is a value reference chosen before sampling. For exact evaluations, the conditional mean of $V_t$ is
$f_{\gamma_{t-1}}(y_t)-f_{\gamma_{t-1}/2}(y_t)$, which lies between
one half and one times
$f_{\gamma_{t-1}}(y_t)-f(y_t)$
(Proposition~\ref{prop:adjacent-scale}). The statistic uses two values
at the two chosen endpoints of the same sampled function.

The complete variant of Algorithm~\ref{alg:scalar} requires $T\ge10N$.
Initialize the reserve with $\mathsf U=T-5N$ in
\eqref{eq:reserve}--\eqref{eq:surplus}, and use
\eqref{eq:observed-radius} immediately after computing $y_t$.
Every pilot and update record at both centers in iteration $t$ then uses
$G_{\gamma_t}$. The batch allocation, projection and scalar coefficient
remain \eqref{eq:pilot}--\eqref{eq:extras} and
\eqref{eq:step}--\eqref{eq:coefficient}. This specifies all count and
radius decisions before the corresponding update means.
The state persists across iterations.

The adjacent-scale mean is a local smoothing diagnostic, and its role
differs from a certificate at an unknown optimizer. In dimension one,
take $f(x)=|x|$ and a center $x\ge\gamma$. The two ball averages both
equal $x$, so $f_\gamma(x)-f_{\gamma/2}(x)=0$. At the optimizer,
however, $f_\gamma(0)-f(0)=\gamma/2$. The algorithm therefore retains
the weighted comparator bias in its potential instead of inferring it
from a small diagnostic at the current center.

Only boundedness of the selection signal $0\le q_t<1$ is needed for
the protected execution and its risk theorem. The moment guarantee
for the diagnostic in Proposition~\ref{prop:adjacent-scale} uses the
additional sample-gradient regularity in
Lemma~\ref{lem:oracle-moments}. Conditional on the completed diagnostic,
the selected radius is fixed for both update means; hence its
data dependence preserves their conditional vector moment assumptions. The same
reasoning permits noisy or corrupted diagnostic observations without
turning them into confidence bounds.

\begin{algorithm}[H]
\caption{Protected observed smoothing with adaptive batches}
\label{alg:observed-radius}
\small
\begin{algorithmic}[1]
\Require $Q,x^0,D$, $N\in\mathbb N_+$, integer $T\ge10N$;
$h_0,s_0,s_\gamma,\gammaMax>0$
\State Set $v_0=x_0=x^0$, $H_0=h_0$, $\Lambda_0=0$
\State $\bar\gamma\gets\min\{\gammaMax,d^{1/4}D/N\}$;
$\underline\gamma\gets\bar\gamma/2$; $\gamma_0\gets\bar\gamma$
\State Set $\mathsf U=T-5N$ and compute $B_t^0,m_1$ by
\eqref{eq:reserve}--\eqref{eq:surplus}
\For{$t=1,\ldots,N$}
 \State $a_t=t$, $A_t=t(t+1)/2$; compute $y_t$ by \eqref{eq:center-y}
 \State Obtain the diagnostic pair $V_t$ in \eqref{eq:observed-radius}
 \State Compute $q_t,\gamma_t$ by \eqref{eq:observed-radius}
 \State $B_t^y\gets\PaidBatch(y_t,t,2t-1,\gamma_t)$
 \State $g_t^y\gets\FreshMean(y_t,B_t^y,\gamma_t)$
 \State Compute $v_t,x_t$ by \eqref{eq:step}
 \State $B_t^x\gets\PaidBatch(x_t,t,2t,\gamma_t)$
 \State $g_t^x\gets\FreshMean(x_t,B_t^x,\gamma_t)$
 \State Update $r_t,\widehat\beta_t,H_t$ by \eqref{eq:curvature}--\eqref{eq:coefficient}
\EndFor
\State \Return $x_N$
\end{algorithmic}
\end{algorithm}
The procedures in Algorithm~\ref{alg:observed-radius} are those defined
for Algorithm~\ref{alg:scalar}, with the displayed radius supplied to
every estimate. Each subsequent batch uses sampled functions and
directions independent of the preceding history, which determines its
center and radius. Each iteration has one radius diagnostic followed
by four vector sampling rounds. Queries within each round run in parallel.

\begin{theorem}[Protected observed radii]
\label{thm:learned-radius}\label{thm:protected-smoothing}
The specified variant spends exactly $T$ pairs and $2T$ values in
$5N$ sequential stages and uses $N$ projections. Define
\begin{align*}
 \mathcal S_N=\frac1{A_N}\sum_{t=1}^Nt\,
      \E[f_{\gamma_t}(x^\star)-f(x^\star)]&\le M\bar\gamma,\\
 L_*&=\sqrt d\,M/\underline\gamma.
\end{align*}
Its clean expected risk is at most
\begin{equation}
 \frac{h_0R^2}{2A_N}+\frac{16L_*D^2}{N^2}
       +\frac{8\sqrt2\sigma_0D}{\sqrt T}+\mathcal S_N.
 \label{eq:learned-clean}
\end{equation}
Under bounded value corruption, replace the stochastic term by
\[
 \frac{16\sigma_0D}{\sqrt T}
 +\frac{6c_dD\Delta}{\underline\gamma}
 +\frac{16dD\Delta}{\underline\gamma\sqrt T}.
\]
Here $\sigma_0^2$ uniformly bounds the clean single-record variance at
the selected radii. The full proof is in
Appendix~\ref{subsec:proof-learned-radius}.
\end{theorem}

The radius is selected from data and decreases by at most a
factor of two over the entire computation. This protection gives the
same leading nonsmooth orders as Corollary~\ref{cor:horizon}, with
constants $33$ and $32$ in its first two regularity terms.
For an $L$-smooth population,
$L_*=\min\{L,\sqrt dM/\underline\gamma\}$ is valid and
$\mathcal S_N\le
Ld(2(d+2)A_N)^{-1}\sum_tt\,\E\gamma_t^2$.
This variant retains an additional dimension factor in its smooth bound.
The forward-secant method in Section~\ref{sec:protected-main} obtains both
regularity orders under the additional scalar moment assumption.

\begin{corollary}[Stages, values and admissible corruption]
\label{cor:resources}
There is a numerical constant $c$ such that, for any target
$\varepsilon>0$, Algorithm~\ref{alg:observed-radius} has expected gap
at most $\varepsilon$ whenever $\sigma_0^2\le CdM_2^2$ and the supplied
integer budgets satisfy
\begin{align}
 N&\ge c\left(1+\frac{\sqrt{h_0}R}{\sqrt\varepsilon}
 +\frac{d^{1/4}MD}{\varepsilon}
 +\frac{d^{1/4}D}{\gammaMax}\right),\nonumber\\
 T&\ge c\left(N+d+\frac{CdM_2^2D^2}{\varepsilon^2}\right),
 \qquad \Delta\le\frac{\varepsilon}{c\,d^{1/4}N}.
 \label{eq:resource-conditions}
\end{align}
The cost is $5N$ sequential sampling stages and $2T$ scalar values.
The full proof is in Appendix~\ref{subsec:proof-resources}.
\end{corollary}
These are sufficient budget conditions in terms of unknown analysis
parameters. The computation itself uses only the supplied budgets,
geometry and references chosen before sampling. They do not constitute a stopping rule
that certifies $\varepsilon$ from observations. When $N$ is chosen at
order $d^{1/4}MD/\varepsilon$, the corruption allowance has order
$\varepsilon^2/(\sqrt d\,MD)$. Initialization and a restricted evaluation domain can
increase the required horizon and reduce this allowance. Corollary~\ref{cor:noise-frontier} supplies the corresponding pair-budget
obstruction in a specified accuracy regime.

\subsection{Curvature along the generated trajectory}
The worst-case smoothness bound can be replaced by a charge for the
secants encountered along the actual execution. Set
$\mathcal E_N=\sum_t t^2\norm{e_t^x-e_t^y}^2$ and, for $r_t>0$, define
\begin{align}
 \kappa_t&=\frac{A_t}{r_t^2}
  \ip{\nabla\phi_t(x_t)-\nabla\phi_t(y_t)}{x_t-y_t},\nonumber\\
 \mathcal C_N&=\sum_t\frac{r_t^2}{D^2}
        [\kappa_t-H_{t-1}/4]_+.
 \label{eq:secant-charge-main}
\end{align}
Set $\kappa_t=0$ when $r_t=0$. These are analysis quantities;
the algorithm uses the observed coefficient update. They record
movement and true secant curvature on its own trajectory.

\begin{theorem}[Trajectory-sensitive risk]\label{thm:trajectory}
Let $f\le\phi_{t+1}\le\phi_t$ be predictably chosen differentiable convex targets with
common gradient Lipschitz bound $L_*<\infty$. Assume that the total
errors $e_t^y,e_t^x$ are integrable. For $\iota\in\{y,x\}$, define
$b_t^\iota=\E[e_t^\iota\mid\mathcal F_t^\iota]$ and
$\zeta_t^\iota=e_t^\iota-b_t^\iota$, where $v_{t-1}$ is
$\mathcal F_t^y$-measurable and $v_t-v_{t-1}$ is
$\mathcal F_t^x$-measurable. For fixed $u\in Q$,
put $R_u=\norm{u-x^0}$ and
$\mathcal S_N(u)=A_N^{-1}\sum_t t\,\E[\phi_t(u)-f(u)]$.
Assume $\mathcal S_N(u)<\infty$ and $\E\mathcal E_N<\infty$.
Then
\begin{align}
 &\E[f(x_N)-f(u)]
 \le\frac{h_0R_u^2}{2A_N}+\mathcal S_N(u)\nonumber\\*
 &\quad+\frac1{A_N}\min\left\{
 \begin{aligned}
  &4D^2\E\mathcal C_N+2\sqrt2D\sqrt{\E\mathcal E_N},\\
  &8L_*D^2+2D\sqrt{\E\mathcal E_N}
 \end{aligned}\right\}\nonumber\\*
 &\quad+\frac D{A_N}\sum_t t\,\E(\norm{b_t^y}+\norm{b_t^x}).
 \label{eq:trajectory-risk-main}
\end{align}
The full proof is in Appendix~\ref{subsec:proof-trajectory}.
\end{theorem}
Lemma~\ref{lem:secant-growth} proves
$H_N-h_0\le2\mathcal C_N+\sqrt{2\mathcal E_N}/D$ and the separate
uniform envelope $4L_*+\sqrt{\mathcal E_N}/D$.
Taking the smaller complete expected bound in
\eqref{eq:trajectory-risk-main} retains both guarantees. This
refinement does not require a confidence bound for a true secant.

\begin{corollary}[Weighted variance on the actual trajectory]
\label{cor:weighted-variance}
Under the hypotheses of Theorem~\ref{thm:trajectory}, let each update
estimate be the average of records that are conditionally independent
given its pre-update history. Each batch size and center is measurable
with respect to that history, and $\mathcal F_t^x$ contains
$\mathcal F_t^y$ and the first update batch. Assume each record has conditional mean
$\nabla\phi_t(y_t)$ or $\nabla\phi_t(x_t)$ at its corresponding center.
In particular, exact records of \eqref{eq:record} satisfy this condition
when $\phi_t=f_{\gamma_t}$. Let $V_t^y$ and $V_t^x$ be the conditional
single-record variances at $y_t$ and $x_t$, respectively, at the selected
radius and before the corresponding update records. Define
\begin{equation}
 \mathcal V_N=\sum_t t^2\E\left(
       \frac{V_t^y}{B_t^y}+\frac{V_t^x}{B_t^x}\right).
 \label{eq:weighted-variance-main}
\end{equation}
Then $\E\mathcal E_N=\mathcal V_N$ and the risk bound
\eqref{eq:trajectory-risk-main} holds with no bias term and with
$\mathcal V_N$ as its error energy. The full proof is in
Appendix~\ref{subsec:proof-weighted-variance}.
\end{corollary}
The conditional variances and selected counts in
\eqref{eq:weighted-variance-main} belong to the generated trajectory.
The protected floors give the stated envelope based on the supplied budget for this
sum. A smaller realized variance statistic alone does not establish a
smaller objective risk: the counts also change the centers and their
curvature charges. A comparison with allocation using future variances
would need to account for these trajectory changes and for the
observations required to learn those variances.

For an $L$-smooth population the actual comparator bias is quadratic:
\begin{equation}
 \mathcal S_N(u)\le\frac{L\nu_d}{2A_N}
       \sum_t t\,\E\gamma_t^2,
 \qquad \nu_d=\frac d{d+2}.
 \label{eq:smooth-comparator-main}
\end{equation}
Moreover every ball smoothing remains $L$-smooth.
Proposition~\ref{prop:protected-smoothing-smooth} gives the full proof
without assumptions on higher derivatives. The protected range alone
provides the upper bound $\bar\gamma^2\le\sqrt dD^2/N^2$ for its
weighted squared radius. It therefore preserves acceleration in $N$
while leaving a possible $\sqrt d$ factor in this smoothing term.

\subsection{A pair-budget corruption frontier}
The sufficient allowance in \eqref{eq:resource-conditions} can be compared
with a lower bound that permits arbitrary adaptive endpoints, including
points outside the optimization domain. The construction uses fixed
bounded spatial corruption and deterministic sample functions.

\begin{corollary}[Noise at a fixed pair budget]\label{cor:noise-frontier}
Let $T\in\mathbb N_+$, $\ell_T=\log(8(2T+1))$, $d\ge8\ell_T$,
$R,M>0$, and $0<\varepsilon\le MR/8$.
If
\begin{equation}
 \Delta\ge\frac{8\varepsilon\ell_T}{d}
   +\frac{16\sqrt2\varepsilon^2\sqrt{\ell_T}}{MR\sqrt d},
 \label{eq:noise-frontier-main}
\end{equation}
then every randomized method using at most $T$ original pairs has an
admissible convex $M$-Lipschitz instance on the ball $Q=R\mathbb B_2$
with $x^0=0$, initial solution distance $R$, and expected gap at least
$3\varepsilon/2$. The bound allows all sequential sampling schedules
and all query points in $\R^d$. The full proof is in
Appendix~\ref{subsec:proof-noise-frontier}.
\end{corollary}

Here $D=2R$. When $\varepsilon/(MR)$ is at least of order
$\sqrt{\ell_T/d}$, the quadratic term in
\eqref{eq:noise-frontier-main} dominates. In that regime it agrees,
up to logarithms and constants, with the sufficient scale
$\varepsilon^2/(\sqrt dMD)$ when the initialization and evaluation-radius terms
permit $N$ of order $d^{1/4}MD/\varepsilon$. At finer accuracy the
linear term $\varepsilon\ell_T/d$ leaves a gap to the smoothing
allowance. The lower bound constrains total pair work and applies
even when arbitrarily many sequential stages are permitted;
it does not establish joint optimality in all three resources.

\section{Distance to a smoothed minimizer}
\label{sec:app-localization-limits}

Ball smoothing controls objective values at each fixed comparator,
but it need not preserve distance to the optimizer. The following
example separates these two statements.

\label{subsec:proof-smoothed-distance}

\begin{proposition}[Ball smoothing and minimizer distance]
\label{prop:smoothed-distance}
Let $d=2$, $\gamma>0$ and $0<\theta\le1/1000$. Define
\[
 f_\theta(x,y)=\max\{\theta^2|x|,|y|,2y-\theta x\},
 \qquad
 Q_\theta=[-3\gamma/\theta,3\gamma/\theta]\times[-3\gamma,3\gamma],
 \qquad x^0=(0,\gamma).
\]
The objective is globally $\sqrt5$-Lipschitz and convex, and its
unique minimizer on $Q_\theta$ is $(0,0)$, with $R=\gamma$.
Every minimizer $(x_\gamma,y_\gamma)$ on $Q_\theta$ of its uniform
ball smoothing satisfies
\begin{equation}
 x_\gamma>\gamma/(8\theta).
 \label{eq:smoothed-minimizer-far}
\end{equation}
Consequently the distance from $x^0$ to a smoothed minimizer cannot
be bounded by a universal multiple of $R+\gamma$, even for
deterministic convex Lipschitz samples on a compact feasible set.
\end{proposition}
\begin{proof}
Each displayed function in the maximum is convex. Their Lipschitz
constants are $\theta^2$, one and $\sqrt{4+\theta^2}$, so their
maximum is $\sqrt5$-Lipschitz. Since
$f_\theta\ge\max\{\theta^2|x|,|y|\}$, it is nonnegative and
equals zero only at $(0,0)$, which belongs to $Q_\theta$.
The smoothed function is finite and continuous and therefore has
a minimum on the compact set $Q_\theta$.

Let $U$ be uniform on the unit disk. Its second-coordinate density is
\[
 p(u)=\frac2\pi\sqrt{1-u^2},\qquad -1\le u\le1.
\]
This follows by integrating the constant disk density $1/\pi$
over the horizontal section of length $2\sqrt{1-u^2}$.
Set $b(y)=\E|y+\gamma U_2|$. This function is convex, even and
minimized at zero. Differentiating its one-dimensional integral
gives $b'(0)=0$ and
$b''(y)=2p(y/\gamma)/\gamma$ for $|y|<\gamma$.
On $0\le y\le\gamma/4$ the density is at least
$\sqrt{15}/(2\pi)$. Integration of the second derivative and
monotonicity of the even convex function thus give
\begin{equation}
 b(y)-b(0)\ge\frac{\sqrt{15}}{32\pi}\gamma
 \quad\text{whenever }|y|\ge\gamma/4.
 \label{eq:smoothed-base-shift}
\end{equation}

Suppose $x\le\gamma/(8\theta)$ and $|y|<\gamma/4$.
On the event $U_2\in[1/2,3/4]$, put
$z=y+\gamma U_2$ and $s=\theta(x+\gamma U_1)$. Since
$|U_1|\le1$ and $\theta\le1/16$,
\[
 z\ge\gamma/4,\qquad
 s\le\gamma/8+\theta\gamma\le3\gamma/16.
\]
Therefore $2z-s\ge|z|+\gamma/16$ on this event. Its probability
is at least
\[
 \int_{1/2}^{3/4}p(u)\,du\ge\frac{\sqrt7}{8\pi}.
\]
The objective always dominates $|z|$, so
\begin{equation}
 f_{\theta,\gamma}(x,y)
 \ge b(y)+\frac{\sqrt7}{128\pi}\gamma
 \ge b(0)+\frac{\sqrt7}{128\pi}\gamma.
 \label{eq:smoothed-far-lower}
\end{equation}
For $|y|\ge\gamma/4$, inequality
\eqref{eq:smoothed-base-shift} gives a stronger bound. Hence
\eqref{eq:smoothed-far-lower} holds for every $y$ whenever
$x\le\gamma/(8\theta)$.

The comparison point $(2\gamma/\theta,0)$ belongs to $Q_\theta$.
For every $U$ in the unit disk, there
\[
 s=\theta(2\gamma/\theta+\gamma U_1)\ge(2-\theta)\gamma\ge\gamma,
 \qquad z=\gamma U_2\in[-\gamma,\gamma].
\]
Thus $2z-s\le|z|$, while
$\theta^2|2\gamma/\theta+\gamma U_1|\le(2\theta+\theta^2)\gamma$.
The inequality $\max\{a,b\}\le a+b$ for $a,b\ge0$ yields
\begin{equation}
 f_{\theta,\gamma}(2\gamma/\theta,0)
 \le b(0)+(2\theta+\theta^2)\gamma.
 \label{eq:smoothed-far-upper}
\end{equation}
Here $2\theta+\theta^2\le2001/10^6<1/256$, whereas
$\sqrt7/(128\pi)>1/256$ because $\sqrt7>2$ and $\pi<4$.
Consequently \eqref{eq:smoothed-far-upper} is strictly smaller
than \eqref{eq:smoothed-far-lower}; no constrained minimizer
can have $x\le\gamma/(8\theta)$.
Its distance from $x^0=(0,\gamma)$ is at least its positive
first coordinate. Dividing this distance by $R+\gamma=2\gamma$
and letting $\theta$ decrease to zero proves the final assertion.
\end{proof}

This example retains the comparator-value bound
$f_\gamma(u)-f(u)\le M\gamma$.
It concerns minimizer distance for an originally nonsmooth population;
the compact smoothing comparison does not assert an obstruction for
every originally smooth population. The accelerated potential uses
the original fixed comparator directly and therefore does not require
distance to an optimizer of the smoothed objective.

\section{The cost of preserving an affine error bound}
\label{sec:anchor-localization}
\label{subsec:startup-price}

Requiring the deterministic affine error to decrease at an accelerated rate
can force a polynomial increase in the stochastic error when both the
original solution distance and the noise scale are unknown. We quantify
this tradeoff by coupling two sample laws with the same supplied inputs.

Fix $r_0>0$, $N\ge2$, $A=A_N$ and an integer $n\ge\max\{4,8N\}$.
Use common supplied inputs
\begin{equation}
 Q=[0,S],\quad S=r_0(1+A),\quad x^0=r_0,
 \quad h_0=1/r_0,
 \quad T=n.
 \label{eq:startup-common-inputs}
\end{equation}
Instance zero always returns $F_0(x)=S-x$. Instance one returns the
globally affine sample
\begin{equation}
 F(x;\xi)=
 \begin{cases}
  S-x,&\text{with probability }1-1/n,\\
  2nx,&\text{with probability }1/n.
 \end{cases}
 \label{eq:startup-rare-law}
\end{equation}
This is an affine specialization of the rare-sample construction of
\citet[Section~1.1 and Theorem~3]{CarmonHinder2024}. The second population
has gradient $1+1/n$, optimizer zero and distance $R=r_0$. In dimension one,
the two-point gradient estimator of an affine sample equals its slope. Its clean
single-record variance and sample-gradient moment are
\begin{equation}
 \sigma_1^2=4n-3/n-1/n^2\le4n,
 \qquad M_{2,1}^2=4n+1-1/n.
 \label{eq:startup-rare-moments}
\end{equation}
The population has $L=0$, every ball smoothing has zero bias, and all
query-evaluation domain requirements hold on the global extensions.

\begin{theorem}[The cost of preserving the affine error bound]
\label{thm:startup-noise-price}
Let any algorithm draw at most $n$ independent sampled functions, choose arbitrary
adaptive query endpoints, and return $\widehat x\in Q$. Its supplied inputs
and internal-randomization law are the same on both instances. The result
also allows observation of the entire sample function under each label.
If for some $b_0\ge0$ its deterministic-instance gap satisfies
\begin{equation}
 \E_0[S-\widehat x]
 \le\frac{h_0(S-r_0)^2}{2A}+b_0
 =\frac{r_0A}{2}+b_0,
 \label{eq:startup-affine-contract}
\end{equation}
then
\begin{equation}
 \E_1[f_1(\widehat x)-f_1(0)]
 \ge(1-1/n)^n(1+1/n)[r_0(1+A/2)-b_0]_+.
 \label{eq:startup-affine-lower}
\end{equation}
\end{theorem}

\begin{proof}
The deterministic error bound gives
$\E_0\widehat x\ge r_0(1+A/2)-b_0$. Generate $n$ independent sampled functions
in advance, including unused functions if the algorithm stops early. Couple
the algorithm's internal randomness across the two instances. On the
event that every function from the second law is $F_0$, the full functions agree.
Induction through the adaptive transcript therefore gives the same
endpoints, batch allocations, stopping decisions and output. This event
is independent of the algorithm's internal randomness and has probability
$q=(1-1/n)^n$. The rare gap is $(1+1/n)\widehat x\ge0$, so restriction
to this event proves \eqref{eq:startup-affine-lower}.

\end{proof}

\begin{samepage}
\begin{corollary}[The required stochastic error multiplier]
\label{cor:startup-price}
Under Theorem~\ref{thm:startup-noise-price}, a proposed rare-instance
guarantee
\begin{equation}
 \E_1[f_1(\widehat x)-f_1(0)]
 \le\frac{h_0R^2}{2A}
 +P\frac{\sigma_1(R+r_0)}{\sqrt n},\qquad P\ge0,
 \label{eq:startup-local-contract}
\end{equation}
requires $P\ge A/32$ if the deterministic error bound
\eqref{eq:startup-affine-contract} has $b_0=0$.
Allowing $b_0\le r_0A/4$ still requires $P\ge A/64$.
In particular this applies with $n=T\ge8N$ and arbitrary adaptive
endpoints, even when the algorithm observes each full sample function.
\end{corollary}
\end{samepage}
\begin{proof}
Theorem~\ref{thm:startup-noise-price} bounds the rare-instance risk
with $q=(1-1/n)^n$.
To bound the event probability directly, let
$g(x)=x\log(1-1/x)$ for $x\ge2$ and $z=1/(x-1)$. Then
\[
 g'(x)=z-\log(1+z)\ge0,
 \qquad \log(1+z)=\int_0^z(1+s)^{-1}\,ds\le z.
\]
Thus $q\ge\exp(g(2))=1/4$.

Since $R=r_0$ and $\sigma_1\le2\sqrt n$, the upper bound in
\eqref{eq:startup-local-contract} is at most
$r_0/(2A)+4Pr_0$. For $b_0=0$, the lower bound is at least
$r_0(1/4+A/8)$, yielding
\[
 4P\ge A/8+1/4-1/(2A)\ge A/8,
\]
because $A\ge3$. Hence $P\ge A/32$. For $b_0\le r_0A/4$,
the lower risk is at least $r_0(1/4+A/16)$; the same calculation
gives $P\ge A/64$.
\end{proof}

For the deterministic law, the solution distance is
$R_0=S-r_0=r_0A$ and its Lipschitz constant is one. An additional
common-instance term $C(R_0+r_0)/N$ has
$b_0=Cr_0(A+1)/N$ and meets the weaker residual condition for
$N\ge6C$, since $A\ge3$ gives $4(A+1)/A\le16/3<6$.
A supplied instance-specific noise envelope can change the common-input
hypothesis or the deterministic error bound. The conclusion does not rule
out a smaller noise multiplier after relaxing the deterministic error bound.
The tunable-anchor construction makes this relaxation explicit in
Theorem~\ref{thm:frontier-pareto-lower} and
Corollary~\ref{cor:frontier-noise-multiplier}.

\section{The work required under bounded spatial corruption}
\label{sec:app-corruption-frontier}

We give an explicit corruption obstruction for the chosen-pair oracle.
The angular hiding mechanism follows the critical-angle construction of
\citet[Section~4]{RisteskiLi2016}. Their polynomial-query result motivates
the two accuracy regimes. We prove the support-function specialization
and its finite-budget constants here.

Throughout this appendix, $d\ge2$, $M,R>0$, $Q=R\Ball_2^d$,
$D=2R$, and $x^0=0$. Every sample function in the construction is
deterministic. Thus each original label still permits exactly two
values and costs one pair, while all observations are values of a fixed
function $f+\delta$. The functions and corruptions are defined on
all of $\R^d$. Consequently any legal evaluation domain is covered.

\subsection{Angular concentration}
\label{subsec:angular-concentration}

\begin{lemma}[Spherical hiding]\label{lem:sphere-hiding}
Let $\vartheta$ be uniform on the unit sphere in $\R^d$ and let $w$
be a deterministic unit vector. For $0<\beta<1$,
\[
 \Prb\bigl(|\ip{\vartheta}{w}|>\beta\bigr)
 \le2\exp(-d\beta^2/2).
\]
\end{lemma}
\begin{proof}
By rotational invariance, take $w=(1,0,\ldots,0)$. Let $Z$ have independent
standard normal coordinates. Polar coordinates in its
rotation-invariant density show that $Z/\norm Z$ is uniform on the
sphere and independent of $\norm Z$. Integration by parts in the
one-dimensional normal density gives
$\E Z_1^{2k}=(2k-1)!!$ for integers $k\ge1$. The radial density is
proportional to $r^{d-1}e^{-r^2/2}$; integration by parts gives
$\E\norm Z^{2k}=d(d+2)\cdots(d+2k-2)$. Independence therefore yields
\[
 \E\vartheta_1^{2k}
 =\frac{(2k-1)!!}{d(d+2)\cdots(d+2k-2)}
 \le\frac{(2k-1)!!}{d^k}.
\]
Symmetry implies
$\E e^{\lambda\vartheta_1}=\E\cosh(\lambda\vartheta_1)$.
Expand its nonnegative even-power series and use monotone convergence:
\[
 \E e^{\lambda\vartheta_1}
 \le\sum_{k=0}^\infty
       \frac{\lambda^{2k}(2k-1)!!}{(2k)!\,d^k}
 =e^{\lambda^2/(2d)}.
\]
The $k=0$ term is one. Markov's inequality with $\lambda=d\beta$
gives $\Prb(\vartheta_1>\beta)\le e^{-d\beta^2/2}$.
Apply symmetry and add the two tail bounds.
\end{proof}

\subsection{A globally convex hidden-direction family}
\label{subsec:hidden-direction-family}

Choose $0<\eta\le MR/2$ and $0<\beta\le1/2$. Set
\begin{equation}
 a=\eta/R,\qquad r_c=R\beta+\eta/M.
 \label{eq:corruption-family-scales}
\end{equation}
The quantity $r_c$ is a radius of the hard family, distinct from the
algorithm's length reference $r_0$ chosen before sampling. For a unit vector
$\vartheta$, define
\begin{align}
 h_\vartheta(z)
 &=\sup\{\ip{w}{z}:\norm w=1,\ \ip{w}{\vartheta}\le\beta\},
 \label{eq:truncated-support}\\
 f_\vartheta(z)
 &=\max\{M[h_\vartheta(z)-r_c],-a\ip{\vartheta}{z}\},
 \label{eq:hidden-objective}\\
 g(z)&=\max\{M(\norm z-r_c),0\}.
 \label{eq:common-reference-oracle}
\end{align}
Define the angular band
\[
 \mathcal B_\vartheta
 =\{z:|\ip{\vartheta}{z}|\le\beta\norm z\},
\]
and the corruption amplitude
\begin{equation}
 \Delta_{\eta,\beta}
 =\eta\beta^2+\frac{\eta^2\beta}{MR}.
 \label{eq:exact-band-amplitude}
\end{equation}

\begin{proposition}[Hidden-direction geometry]\label{prop:hidden-band}
The objective $f_\vartheta$ is globally convex and $M$-Lipschitz.
Its minimum over $Q$ is $-\eta$, attained at $R\vartheta$, and
$f_\vartheta(0)=0$. For every $z\in\mathcal B_\vartheta$, regardless
of its norm,
\[
 |f_\vartheta(z)-g(z)|\le\Delta_{\eta,\beta}.
\]
This amplitude is attained at a point of $Q\cap\mathcal B_\vartheta$.
Every $z\in Q\cap\mathcal B_\vartheta$ has clean objective gap at least
$\eta(1-\beta)$.
\end{proposition}
\begin{proof}
The support set in \eqref{eq:truncated-support} is nonempty and compact.
Its support function is convex and one-Lipschitz. The two terms in
\eqref{eq:hidden-objective} are convex and respectively $M$- and
$a$-Lipschitz. Since $a\le M/2$, their maximum is $M$-Lipschitz.

At $R\vartheta$, the support supremum is $\beta R$. It is attained by
$w=\beta\vartheta+\sqrt{1-\beta^2}\,w_\perp$, for any unit
$w_\perp\perp\vartheta$, which exists because $d\ge2$.
The two terms in \eqref{eq:hidden-objective} then both equal $-\eta$.
For every $z\in Q$, its linear term is at least $-aR=-\eta$,
so this is the minimum. At zero the maximum is zero.

For $z\ne0$ in the band, $w=z/\norm z$ belongs to the support set;
thus $h_\vartheta(z)=\norm z$. Put
$r=\norm z$, $b=M(r-r_c)$ and $\ell=-a\ip{\vartheta}{z}$.
Then
\[
 f_\vartheta(z)=\max\{b,\ell\},\qquad
 g(z)=\max\{b,0\},\qquad |\ell|\le a\beta r.
\]
If $r\le r_c$, then $g(z)=0$, while
$-a\beta r_c\le\ell\le\max\{b,\ell\}\le[\ell]_+\le a\beta r_c$.
If $r\ge r_c$, then $g(z)=b\ge0$. When $\ell\le b$, the
difference is zero. Otherwise it satisfies
\[
 \ell-b
 \le a\beta r-M(r-r_c)
 =a\beta r_c-(M-a\beta)(r-r_c)
 \le a\beta r_c.
\]
Here $M-a\beta\ge0$. Since
$a\beta r_c=\Delta_{\eta,\beta}$, this proves the bound uniformly
over all query norms; zero also satisfies it.

Take
$z=r_c[-\beta\vartheta+\sqrt{1-\beta^2}\,w_\perp]$.
Because $r_c\le R$, this point lies in $Q$, belongs to the band,
and has $b=0$ and $\ell=a\beta r_c$. It attains the amplitude.
Finally, on $Q\cap\mathcal B_\vartheta$,
$f_\vartheta(z)\ge-a\ip{\vartheta}{z}\ge-\eta\beta$.
Subtract the minimum $-\eta$ to obtain the gap bound.
\end{proof}

We use one fixed spatial corruption,
\begin{equation}
 \delta_\vartheta(z)
 =\operatorname{clip}_{[-\Delta_{\eta,\beta},\Delta_{\eta,\beta}]}
                    [g(z)-f_\vartheta(z)].
 \label{eq:fixed-hidden-corruption}
\end{equation}
This function is continuous and globally bounded by
$\Delta_{\eta,\beta}$. Proposition~\ref{prop:hidden-band} gives
$f_\vartheta+\delta_\vartheta=g$ throughout the angular band.
The corruption is fixed before the interaction.

\subsection{Adaptive transcripts and total observations}
\label{subsec:work-noise-proof}

\begin{theorem}[Work under spatial corruption]\label{thm:work-noise}
Fix an integer $T\ge0$, $0<\eta\le MR/2$ and $0<\beta\le1/2$.
For every randomized algorithm making at most $T$ original pairs and
returning a point of $Q$, there is an objective
$f_\vartheta$ and a fixed corruption of magnitude at most
$\Delta_{\eta,\beta}$ such that
\begin{equation}
 \E[f_\vartheta(\widehat x)-\min_Q f_\vartheta]
 \ge\eta(1-\beta)
      [1-2(2T+1)e^{-d\beta^2/2}]_+.
 \label{eq:work-noise-lower}
\end{equation}
The objective is globally convex and $M$-Lipschitz. The result allows
arbitrary endpoints and any number of sequential stages, and remains
valid even when arbitrary points of $\R^d$ may be queried.
\end{theorem}
\begin{proof}
Draw the hidden direction $\vartheta$ uniformly and independently of
the algorithm's internal seed $\omega$. Run the algorithm against the
deterministic reference oracle $g$. Conditional on $\omega$, this
execution determines every endpoint and its final output independently
of $\vartheta$. Pad early stopping with endpoints at zero to obtain
$2T$ reference endpoints. Denote its output by $x_g(\omega)$.

Apply Lemma~\ref{lem:sphere-hiding} to every nonzero endpoint direction
and to the output direction. Zero lies in the band automatically.
A union bound shows, conditional on every seed, that all these
$2T+1$ points lie in $\mathcal B_\vartheta$ with probability at least
$1-2(2T+1)e^{-d\beta^2/2}$.

On this event the fixed oracle
$f_\vartheta+\delta_\vartheta$ agrees with $g$ at every reference
endpoint. Induction over the algorithm's decisions therefore shows
that its actual transcript, stopping time and output coincide with the
reference execution. The induction allows arbitrary pairs, parallel
groups selected from the preceding transcript, and even selection of
each next scalar endpoint after the preceding answer.

Proposition~\ref{prop:hidden-band} gives actual clean gap at least
$\eta(1-\beta)$ on the event. Outside it the gap is nonnegative.
Average over $\vartheta,\omega$ and replace a negative probability
bound by zero. At least one fixed direction attains the resulting
average lower bound on its expected gap. This objective and its fixed
corruption prove \eqref{eq:work-noise-lower}.
\end{proof}

The theorem counts every scalar observation, including diagnostics and
comparisons. It applies to nonsymmetric same-label pairs because the
hard sample functions are deterministic. The finite-budget obstruction
also holds with parameters $M$ and $R$ supplied. It constrains total
work independently of a particular finite-difference estimator.

\subsection{Proof of the main noise corollary}
\label{subsec:proof-noise-frontier}

\paragraph{Proof of Corollary~\ref{cor:noise-frontier}.}
Put $\ell_T=\log(8(2T+1))$ and
$\beta_T=\sqrt{2\ell_T/d}$. The dimensional condition
$d\ge8\ell_T$ gives $\beta_T\le1/2$. The failure term in
\eqref{eq:work-noise-lower} is at most $1/4$, so its gap bound is
at least $3\eta/8$. The exact corruption amplitude is
\begin{equation}
 \Delta_{\eta,\beta_T}
 =\frac{2\eta\ell_T}{d}
  +\frac{\sqrt2\,\eta^2\sqrt{\ell_T}}{MR\sqrt d}.
 \label{eq:noise-two-branches}
\end{equation}
For $0<\varepsilon\le MR/8$, select $\eta=4\varepsilon$.
This is at most $MR/2$, and the expected gap lower bound is
$3\varepsilon/2$. Equation~\eqref{eq:noise-two-branches} becomes
$8\varepsilon\ell_T/d+
16\sqrt2\,\varepsilon^2\sqrt{\ell_T}/(MR\sqrt d)$.
Any allowed larger corruption includes this fixed constructed
corruption. This proves the stated impossibility. Since $D=2R$,
the second term also equals
$32\sqrt2\,\varepsilon^2\sqrt{\ell_T}/(MD\sqrt d)$.
\hfill$\square$

The quadratic branch dominates only when
$\varepsilon/(MR)$ is at least of order $\sqrt{\ell_T/d}$.
The comparison with the accelerated sufficient noise scale is
therefore restricted to that accuracy regime. In the finer-accuracy
regime the separate linear branch is larger. The result supplies a
total-work obstruction and imposes no independent lower bound on
sequential stages.

\subsection{An inverse work bound}
\label{subsec:inverse-corruption-work}

\begin{samepage}
\begin{corollary}[Exponential work requirement]
\label{cor:exponential-work}
Fix $0<\eta\le MR/2$. For an allowed $\Delta>0$, let
\[
 \beta_\Delta=\min\left\{\tfrac12,\sqrt{\Delta/(2\eta)},
                               MR\Delta/(2\eta^2)\right\}.
\]
If
\[
 T\le\frac1{16}\exp(d\beta_\Delta^2/2)-\frac12,
\]
some allowed fixed corrupted instance has expected gap at least
$3\eta/8$. Thus a uniform expected gap strictly below $3\eta/8$
requires strictly more pairs than this bound, and more than twice this
number of scalar values.
\end{corollary}
\end{samepage}
\begin{proof}
The definition gives
$\eta\beta_\Delta^2\le\Delta/2$ and
$\eta^2\beta_\Delta/(MR)\le\Delta/2$.
Consequently \eqref{eq:exact-band-amplitude} is at most $\Delta$.
The stated pair condition is equivalent to
$2(2T+1)e^{-d\beta_\Delta^2/2}\le1/4$.
Apply Theorem~\ref{thm:work-noise} and $1-\beta_\Delta\ge1/2$.
\end{proof}

Up to numerical constants, this exponent is
$d\min\{1,\Delta/\eta,M^2R^2\Delta^2/\eta^4\}$.
The result allows unlimited sequential adaptivity. It does not
identify the stage budget of a matching adaptive first-order method.

\subsection{Ambiguity in upper bounds on uncorrupted smoothing bias}
\label{subsec:smoothing-bias-ambiguity}

\begin{proposition}[Ambiguity in uncorrupted smoothing bias]\label{prop:bias-ambiguity}
Fix $\alpha\in(0,1)$. Suppose the query domain is contained in $R_h\Ball_2^d$,
$R_h>0$, and the smoothing ball of radius $0<\gamma\le R_h$ is legal at
zero. For $0<a\le M$ with $aR_h\le\Delta$, the deterministic instances
\[
 f_0(z)=0,\quad\delta_0(z)=0,\qquad
 f_1(z)=a\norm z,\quad
 \delta_1(z)=-\min\{a\norm z,\Delta\}
\]
give identical observations at every legal endpoint. Their clean
smoothing biases at zero are respectively zero and $a\mu_d\gamma$.
Every upper confidence bound with coverage $1-\alpha$ for that clean bias
on both instances must therefore satisfy
\[
 \Prb_0(U_{\mathrm{bias}}\ge a\mu_d\gamma)\ge1-\alpha.
\]
This statement is independent of the number of original pairs and
sequential stages.
\end{proposition}
\begin{proof}
Both objectives are convex and $a$-Lipschitz, and both corruptions are
globally bounded by $\Delta$. On the query domain,
$a\norm z\le aR_h\le\Delta$, so the observed functions are
identically zero. Every adaptive transcript and every output
statistic thus have the same law. For uniform ball $U$,
$f_{1,\gamma}(0)-f_1(0)=a\gamma\E\norm U=a\mu_d\gamma$,
whereas the bias on the base instance vanishes. Apply the required
upper coverage on instance one and transfer the identical statistic
law to instance zero.
\end{proof}

Zero minimizes both objectives in this last example. The proposition
concerns upper confidence bounds on uncorrupted smoothing bias; the hidden-direction
family establishes the separate optimization obstruction. Together
these results distinguish the sufficient accelerated noise guarantee
from the requirements of a complete adaptive stage, work and
corruption frontier.

\section{Notation}\label{sec:app-notation}
Reused symbols are distinguished by their algorithm or setting.
\begin{longtable}{@{}p{0.19\textwidth}p{0.76\textwidth}@{}}
\toprule
Symbol & Meaning\\
\midrule
\endhead
$d$ & Euclidean dimension.\\
$Q,D$ & Feasible set and a supplied upper bound on its diameter.\\
$x^0=x_0=v_0$ & Initial feasible point.\\
$x^\star,R$ & Fixed population minimizer and $R=\norm{x^0-x^\star}$.\\
$u,R_u$ & Fixed comparator and $R_u=\norm{x^0-u}$.\\
$M,M_2$ & Population Lipschitz constant and sample weak-gradient moment bound $\E_\xi\|\nabla F(z,\xi)\|^2\le M_2^2$; $M_2$ is distinct from the gradient-estimate variance bounds.\\
$m(\xi)$ & Sufficient sample Lipschitz bound, with $\E m(\xi)^2\le M_2^2$.\\
$L,L_\gamma$ & Gradient Lipschitz constants of the mean loss, when smooth, and its ball smoothing.\\
$L_\circ,L_*$ & Common gradient Lipschitz bounds for a sequence of objectives; $L_*$ is used for the protected-radius range.\\
$N,T,2T$ & Optimization iterations, sampled losses evaluated at two points each, and function values. Sequential sampling rounds are counted separately.\\
$a_t,A_t$ & Default scalar weights $a_t=t$, $A_t=\sum_{s\le t}a_s=t(t+1)/2$. The pre-step execution chooses variable weights through \eqref{eq:prestep-weight}.\\
$B_t^y,B_t^x$ & Sizes of the gradient-estimation batches at the two sequential centers.\\
$j,B_j,B_t^0$ & Center index $j=2t-1,2t$, corresponding batch size, and its minimum value fixed before sampling.\\
$\mathsf U,m_{\mathrm{sur}}$ & Update-pair budget after batch and radius selection, and surplus above minimum batch sizes.\\
$m_j,k_j$ & Remaining surplus and additional samples allocated at center $j$.\\
$\widehat v_j,\lambda_j,\Lambda_j$ & Variance estimate from two preliminary samples, weighted batch-allocation signal, and its cumulative sum.\\
$h_0,H_t$ & Positive initial step coefficient chosen before sampling and its updated value.\\
$h_{\rm init},h_{\rm ref}$ & Initial optimization coefficient and baseline for radius selection, with $h_{\rm init}\ge h_{\rm ref}>0$; both equal $h_0$ in the default known-diameter execution.\\
$\mathsf K_t,\underline K_t,I_t$ & Adjusted observed coefficient $h_{\rm ref}+H_t-h_{\rm init}$, stored reference and doubling test based on prior observations. The projected step size remains $a_t/H_{t-1}$.\\
$\Gamma(D),\tau_t$ & Maximum trial radius $\min\{\gammaMax,d^{1/4}D/N\}$ from the diameter bound, and radius-selection threshold $\underline K_tD^2/N^2$.\\
$c_{\rm vec},k_N,q_{\mathrm{reg}}(D)$ & Vector-error factor $8/\sqrt7$ or $4\sqrt2$, potential coefficient $2/A_N+3/N^2$ and error bound when smoothing radius increases $2MDN/A_N$ or $LD^2/A_N$.\\
$\mathcal H_{\mathrm{reg}}(D),\ell_{\mathrm{reg}}(D)$\par$\mathcal U_{\mathrm{reg}}(D,\ell)$ & Bound on expected coefficient increase, logarithmic factor and adaptive-smoothing error bound, for $\mathrm{reg}\in\{\mathrm{s},\mathrm{ns}\}$.\\
$h_{\rm a},S,Q_S,D_S$ & Initial-batch coefficient, selected ball radius, intersection with the feasible set, and diameter bound $2S$.\\
$p,r_p,\Psi(R_u),a_{R_u}$ & Distance parameter chosen before sampling, maximum selected radius $r_0(1+2p)$, distance factor $R_u^2/2+(R_u-r_0)_+^2/8$ and $(R_u-r_0)_+$.\\
$\mathcal A_p(u),\beta_{\rm a}(u),\sigma_{\rm a}^2$ & Initial-estimate and search-restriction error bound, initial smoothing bias, and variance of one uncorrupted initial gradient estimate.\\
$P,\mathcal H_p(R)$ & Noise multiplier in the affine-objective error requirement and its distance factor $R+r_0+R^2/(pr_0)$; $\mathcal H_p$ differs from the curvature envelope $\mathcal H_{\mathrm{reg}}$.\\
$y_t,v_t,x_t$ & Extrapolated center, projected point, and accumulated output average.\\
$r_t,\widehat\beta_t$ & Projected step length and inner product measuring the observed gradient change.\\
$\gammaMax,\gamma_t$ & Permitted evaluation distance defining $Q+\gammaMax\mathbb B_2^d$, and smoothing radius. Forward differences require $2\gamma_t\le\gammaMax$.\\
$\bar\gamma,\underline\gamma$ & Upper and lower radii chosen before sampling for the protected interval.\\
$U,e,[e]_i$ & Uniform ball and sphere directions, independent of the sampled loss, and the $i$th coordinate of $e$.\\
$V_t,q_t$ & Value difference at two smoothing radii and its bounded radius-selection signal.\\
$s_0,s_\gamma$ & Positive constants chosen before sampling to scale batch allocation and radius selection.\\
$\xi,\delta(z),\Delta$ & Independent random sample, fixed spatial evaluation error, and its magnitude bound.\\
$\iota\in\{y,x\}$ & Label for the update batch at $y_t$ or $x_t$, respectively.\\
$g_t^\iota,e_t^\iota$ & Batch-averaged gradient estimate and total error: $e_t^y=g_t^y-\nabla\phi_t(y_t)$ and $e_t^x=g_t^x-\nabla\phi_t(x_t)$.\\
$b_t^\iota,\zeta_t^\iota$ & Conditional mean error $b_t^\iota=\E[e_t^\iota\mid\mathcal F_t^\iota]$ and centered fluctuation $\zeta_t^\iota=e_t^\iota-b_t^\iota$. For exact batch means of \eqref{eq:record} with targets $\phi_t=f_{\gamma_t}$, both biases vanish.\\
$\sigma_\gamma^2,\sigma_0^2$ & Clean single-estimate variance bound at radius $\gamma$ and its supremum over all radii permitted by the method, including initialization.\\
$\mu_d,\nu_d,c_d$ & Ball radial moments $d/(d+1),d/(d+2)$ and sphere constant $d\E|[e]_1|$.\\
$\phi_t,\mathcal S_N(u)$ & Current convex target and its weighted expected smoothing bias at a fixed comparator.\\
$\mathcal E_N,W$ & Weighted sum of squared differences between the two gradient errors, and its expectation.\\
$V_t^\iota,\mathcal V_N$ & Uncorrupted single-estimate conditional variance at the corresponding center and its expected weighted sum divided by actual batch sizes.\\
$\kappa_t,\mathcal C_N$ & Coefficient computed from true gradient differences and accumulated curvature along the actual iterates.\\
$r_0$ & Minimum selected radius, chosen before the initial batch.\\
$\vartheta,\eta,\beta,r_c$ & Hidden unit direction, objective depth, angular band width and radial threshold in the corruption construction.\\
$\ell_T$ & Pair-budget logarithm $\log(8(2T+1))$.\\
$R_\phi$ & Initial distance to a minimizer of the queried objective, including smoothing when two-point estimates replace gradients.\\
$\Gamma,\tau,B^{\rm rad}$ & Initial smoothing radius, positive threshold chosen before sampling, and batch size for radius selection.\\
$r$ & Independent radial scalar with density $dr^{d-1}$ on $[0,1]$, distinct from the displacement $r_t$.\\
$h_\gamma(y),\widehat h_t,J_\gamma$ & Smoothing radius times the radius derivative of the smoothed objective, its batch mean, and the two-value estimate \eqref{eq:radial-estimate-main}.\\
$\chi_d$ & Radial second-moment factor $2d^2/((d+1)(d+2))\le2$.\\
$\omega_{\mathrm{rad}},\eta_{\mathrm{rad}}$ & Moment bounds for the maximum normalized diagnostic error and the local smoothing-error contribution.\\
$\mathfrak X_0,\mathfrak X,\tau_-$ & Reciprocal-radius $L_2$ bounds for exact and corrupted evaluations, and the corruption margin $\tau_-=\tau-2d\Delta$.\\
$L_{\mathrm{pop}}$ & $L$ when the mean loss is smooth and $+\infty$ for the nonsmooth alternative in the radial minimum bound.\\
$\mathcal D_N,\mathcal D_N^+$ & Signed target-change sum and sum of its positive increments.\\
$\Gamma_0,\Gamma,g$ & Baseline trial radius $\min\{\gammaMax/2,d^{1/4}D/N\}$, maximum trial radius $\max\{g,\Gamma_0\}$ and minimum radius chosen before sampling for the forward-difference method. The derivative-estimate method uses the separate maximum radius $\Gamma(D)$.\\
$\widehat q_\gamma^0,\widehat q_\gamma,q_\gamma$ & Exact and corrupted two-value differences, and their mean without corruption $f_{2\gamma}-f_\gamma$. Both values use the same sampled loss.\\
$\rho,\kappa_d$ & Objective-error tolerance chosen before sampling, with $D\beta\le\rho$, and dimension ratio $d/c_d$, with $\kappa_d^2\le3d$.\\
$\mathsf S_N,\zeta,\beta,\bar\sigma$ & Maximum normalized uncorrupted diagnostic, root mean square reciprocal-radius bound, gradient-bias and standard-deviation bounds for forward differences.\\
$\mathsf W_{N,2}$ & Power sum $\sum_{t=1}^Nt^2$, distinct from the root mean square radius bound $S_2$.\\
$\mathcal H_{\mathrm{reg}}^\Delta,\ell_{\mathrm{reg}}^\Delta$ & Coefficient bounds and logarithmic factors for forward differences. These differ from the derivative-estimate bounds $\mathcal H_{\mathrm{reg}}(D),\ell_{\mathrm{reg}}(D)$ used by the initial-batch variant.\\
$\mathcal P_{\mathrm{reg}}^\Delta,\mathcal B_{N,T}$ & Forward-difference error bound and exact-evaluation rate after absorbing a sufficiently small initial-coefficient contribution.\\
$z_{\rm a},\kappa,s(z),S_2$ & Root mean square bound on the initial batch mean, $h_{\rm ref}pr_0/A_N$, its increasing radius map, and the deterministic radius bounding $\E S^2$ by $S_2^2$.\\
$\mathcal A_{\rm sat},\mathcal A_*$ & Alternative initial-estimate and search-restriction bound, and its minimum with $\mathcal A_p$. Smoothing and search-restriction errors can remain when $h_{\rm ref}$ is small.\\
$\mathcal J_{\mathrm{reg}}(S_2),\mathcal V_{\rm ns}(s,\ell)$ & Regularity-specific terms in the exact bound \eqref{eq:joint-main-risk}, and its nonsmooth expansion for a finite evaluation radius; $\mathcal J_{\rm s}=+\infty$ without population smoothness.\\
$G_{R_u},z_{R_u},S_{R_u}$ & Gradient, initial-estimate moment and radius bounds under stationarity at the original comparator; no distance to a smoothed minimizer is substituted.\\
$\varepsilon,\theta$ & Supplied target accuracy and initialization constant chosen before sampling in $h_0=\theta\varepsilon N^2/D^2$; $\theta=4/77$ gives the calibrated sufficient budgets.\\
$r_{\rm s},r_{\rm ns},a_{\rm ns},r_{\rm st}$ & Smooth and nonsmooth iteration terms, finite-evaluation-radius correction, and statistical term in \eqref{eq:epsilon-population-budgets}.\\
$\mathcal B_{\rm s},\mathcal B_{\rm ns},\lambda_{\mathrm{reg}}^\varepsilon$ & Population budget sums and calibrated upper bounds on $\ell_{\mathrm{reg}}^\Delta$; these differ from the observed batch-allocation signal $\lambda_j$.\\
$\omega_j,V_j,\overline{\mathcal V}_N$ & Center weight $\omega_{2t-1}=\omega_{2t}=t$, conditional single-estimate variance at the actual center, and $\E\sum_j\omega_jV_j/(2A_N)$ in the fixed-target analysis.\\
$b_t,S_{\rm sur},\mathcal C_\star^{\rm tr},\mathfrak R_{\rm batch}$ & Minimum adaptive batch size, total surplus, minimum weighted variance over allocations at the same queried centers, and exact worst-case variance ratio $1+S_{\rm sur}/b_1$. These quantities describe weighted variance, not optimization error.\\
$\mathcal L_j,A_j^+,a_j,\chi_j$ & Pre-step curvature candidate, candidate accumulated and current acceleration weights, and nonnegative model allowance. Here $a_j$ solves \eqref{eq:prestep-weight} and is distinct from the default $a_t=t$.\\
$P_\gamma,\widehat q_j,\zeta_j^q$ & Two-value estimate of an objective difference, its batch mean and centered error for a fixed smoothed objective.\\
$\mathcal A,\mathcal R,K_{\rm tr},J_\star$ & Accepted and rejected trial index sets, executed-trial cap chosen before sampling, and regime-specific rejection bound. Costs include samples, rounds and projections from rejected trials.\\
$Z_j(u)$ & Pre-step error term in \eqref{eq:prestep-errors}; acceptance decisions can affect the mean of its retained sum.\\
$L_\phi,\overline{\mathcal L}$ & Smoothness bound of the fixed target in the confidence-controlled pre-step execution and maximum curvature candidate $\max\{\mathcal L_0,4L_\phi\}$.\\
$\bar\sigma,\overline M_2,\delta_{\rm conf}$ & Supplied standard-deviation bound for gradient estimates, root mean square bound per unit distance for objective-difference estimates, and failure probability. The confidence notation differs from spatial corruption $\delta(z)$.\\
$\eta_j^g,\eta_j^q,\eta_j$ & Gradient-error and normalized objective-difference tolerances chosen from past observations, and their sum in the pre-step confidence allowance.\\
$k_{\rm mom},b_j^g,b_j^q,\Lambda_{\rm mom}$ & Number of robust blocks, samples per gradient/comparison block, and continuous allocation coefficient in \eqref{eq:preconf-moment-split}.\\
$J_{\rm exec},T_{\rm exec}$ & Executed pre-step trials and their total pair count, including rejected trials.\\
$\tau,\widehat U,C,q$ & Sample count when the procedure stops, reported upper mean bound, event of a small report under the null law, and its probability.\\
$r,q_T,b_i,C_*$ & Initial point and alternative distance in the two-law family, probability of identical sampled functions, deterministic benchmarks, and exact competitive factor.\\
$\mathcal F_t^y,\mathcal F_t^x$ & Histories immediately before the respective update batches, including the selected center, radius, batch size and preceding selection observations. The $x$ history also contains the current $y$ batch, $g_t^y$, $v_t$ and $x_t$. General trajectory results require only their stated measurability and error moments.\\
\bottomrule
\end{longtable}

\stopcontents[appendices]
\end{document}